\documentclass[11pt,reqno]{amsart}
\usepackage{graphicx} 
\usepackage{hyperref}
\usepackage{fourier}
\usepackage{fullpage}
\usepackage{amsmath,amssymb,amsthm}
\usepackage[shortlabels]{enumitem}
\usepackage[mathscr]{euscript}
\usepackage{dutchcal}
\usepackage{upgreek}
\usepackage{comment}
\usepackage{csquotes}

\makeatletter
\def\resetMathstrut@{%
  \setbox\z@\hbox{%
    \mathchardef\@tempa\mathcode`\(\relax
    \def\@tempb##1"##2##3{\the\textfont"##3\char"}%
    \expandafter\@tempb\meaning\@tempa \relax
  }%
  \ht\Mathstrutbox@1.2\ht\z@ \dp\Mathstrutbox@1.2\dp\z@
}
\makeatother

\newtheorem{theorem}{Theorem}
\newtheorem{lemma}[theorem]{Lemma}
\newtheorem{prop}[theorem]{Proposition}
\newtheorem{claim}[theorem]{Claim}

\newtheorem{corollary}[theorem]{Corollary}

\newtheorem{definition}[theorem]{Definition}
\newtheorem{terminology}[theorem]{Terminology}
\newtheorem{construction}[theorem]{Construction}
\newtheorem{observation}[theorem]{Observation}

\newtheorem{fact}[theorem]{Fact}

\newtheorem{remark}[theorem]{Remark}

\newtheorem{question}[theorem]{Question}

\newcommand{\cF}{\mathcal{F}}

\newcommand{\cS}{\mathcal{S}}

\newcommand{\R}{\mathbb{R}}
\newcommand{\Z}{\mathbb{Z}}
\newcommand{\N}{\mathbb{N}}
\newcommand{\E}{\mathbb{E}}

\DeclareMathOperator{\Lip}{Lip}

\renewcommand{\le}{\leqslant}
\renewcommand{\ge}{\geqslant}

\renewcommand{\setminus}{\smallsetminus}
\renewcommand{\subset}{\subseteq}
\newcommand{\cc}{\mathsf{c}}
\newcommand{\ee}{\mathsf{e}}
\renewcommand{\ae}{\mathsf{ae}}

\newcommand{\eqdef}{\stackrel{\mathrm{def}}{=}}
\newcommand{\cZ}{\mathcal{Z}}
\newcommand{\cA}{\mathcal{A}}
\newcommand{\cB}{\mathcal{B}}

\newcommand{\vol}{\mathrm{vol}}
\newcommand{\ck}{\mathcal{k}}
\newcommand{\fp}{\mathfrak{p}}

\newcommand{\NN}{\mathcal{N}}
\newcommand{\MM}{\mathcal{M}}
\newcommand{\TT}{\mathscr{T}}

\renewcommand{\SS}{\mathscr{S}}
\newcommand{\sub}{\mathscr{C}}
\newcommand{\bX}{\mathbf{X}}
\newcommand{\bH}{\mathbf{H}}
\newcommand{\bY}{\mathbf{Y}}
\newcommand{\bZ}{\mathbf{Z}}
\DeclareMathOperator{\prob}{\mathbb{P}}
\newcommand{\e}{\varepsilon}
\newcommand{\f}{\varphi}
\renewcommand{\d}{\delta}
\newcommand{\ud}[0]{\,\mathrm{d}}
\newcommand{\sfG}{\mathsf{G}}

\newcommand{\sfU}{\mathsf{U}}
\newcommand{\sfH}{\mathsf{H}}

\newcommand{\sfF}{\mathsf{F}}
\newcommand{\sfZ}{\mathsf{z}}
\newcommand{\sfX}{\mathsf{X}}
\newcommand{\sfR}{\mathsf{R}}
\newcommand{\sfA}{\mathsf{A}}
\newcommand{\sfB}{\mathsf{B}}
\newcommand{\M}{\mathsf{M}}
\newcommand{\sfW}{\mathsf{W}}
\newcommand{\n}{\{1,\ldots,n\}}
\newcommand{\sfC}{\mathsf{C}}
\newcommand{\sfD}{\mathsf{D}}

\newcommand{\SDP}{\mathrm{SDP}}
\newcommand{\GL}{\mathsf{GL}}
\newcommand{\MMV}{\mathsf{MMV}}

\newcommand{\bbQ}{\mathbb{Q}}
\newcommand{\bbP}{\mathbb{P}}
\newcommand{\1}{\mathbf{1}}
\newcommand{\diam}{\mathrm{diam}}
\newcommand{\st}{\mathsf{t}}

\newcommand{\si}{\mathsf{i}}
\renewcommand{\nu}{\upnu}
\newcommand{\UU}{\mathcal{U}}
\newcommand{\sV}{\mathscr{V}}
\newcommand{\fd}{\mathfrak{d}}
\newcommand{\bd}{\boldsymbol{\delta}}

\title{Random zero sets with local growth guarantees}
\date{}

\author{Alan Chang}
\address{Department of Mathematics, Washington University in St. Louis, St. Louis MO 63130-4899}
\email{alanchang@math.wustl.edu}

\author{Assaf Naor}
\address{Department of Mathematics, Princeton NJ 08544-1000}
\email{naor@math.princeton.edu}

\author{Kevin Ren}
\address{Department of Mathematics, Princeton NJ 08544-1000}
\email{kr5621@princeton.edu}

\thanks{A.~C. was supported by NSF grant DMS-2247233. A.~N. was supported by NSF grant DMS-2054875, BSF grant 2018223, and the Simons Investigator award. K.~R. was supported by an NSF GRFP fellowship.\smallskip \\  \hangindent=0.4cm An extended abstract announcing some of the algorithmic aspects of this work, and titled ``Optimal rounding for Sparsest Cut,'' will appear in the proceedings of the 57th Annual ACM Symposium on Theory of Computing (STOC 2025).}

\usepackage{xcolor}
\definecolor{maroon}{HTML}{AF3235}

\begin{document}

\maketitle

\begin{abstract} We prove that if $(\MM,d)$ is an  $n$-point metric space that embeds quasisymmetrically  into a Hilbert space, then for every $\tau>0$ there is a random subset  $\cZ$ of $\MM$ such that for any pair of points $x,y\in \MM$ with $d(x,y)\ge \tau$, the  probability  that both $x\in \cZ$ and  $d(y,\cZ)\ge \beta\tau/\sqrt{1+\log (|B(y,\kappa \beta \tau)|/|B(y,\beta \tau)|)}$ is $\Omega(1)$, where  $\kappa>1$ is a universal constant and $\beta>0$ depends only on the modulus of the quasisymmetric embedding. The proof relies on a  refinement of the Arora--Rao--Vazirani rounding technique. Among the applications of this result is that the largest possible Euclidean distortion of an $n$-point subset of $\ell_1$ is  $\Theta(\sqrt{\log n})$, and the integrality gap of the Goemans--Linial semidefinite program for the Sparsest Cut problem  on inputs of size $n$  is  $\Theta(\sqrt{\log n})$. Multiple further applications are given. 
\end{abstract}

\setcounter{secnumdepth}{3}
\setcounter{tocdepth}{4}

\tableofcontents

\section{Introduction}

The field of {\em metric embeddings} is a mathematical discipline with interest in its own right, but it is also notable for its rich and multifaceted interactions with a broad swath of pure and applied mathematics, ranging from its origins in geometry and functional analysis, through its profound connections to group theory, topology, harmonic analysis, extension of functions, probability, combinatorics, statistics,  approximation algorithms, computational complexity, and  data structures. The main contribution of the present work resolves a longstanding question in metric embeddings, which has implications for several of the aforementioned directions. Given that the natural readership has varied backgrounds, the results that are obtained herein are presented below in a self-contained manner that introduces all of the relevant background, and explains  in context the motivations and the meanings of what is accomplished. This inevitably leads to a longer and more gradual Introduction. We will therefore open with a  brief  section (Section~\ref{sec:nontechnical}) that quickly covers a selection of the main results assuming familiarity with some history, background, notation, and terminology, while postponing the presentation of precise technical details   and the discussion of the history to subsequent sections. After this initial overview, we will restart the Introduction in Section~\ref{sec:detailed} while presenting our results more formally and with relevant background, and we will also later describe (mainly in Section~\ref{sec:history}) the history and key ideas of our proofs. 

\subsection{Informal overview}\label{sec:nontechnical} Bourgain's  embedding theorem~\cite{Bou85} asserts that any $n$-point metric space embeds into a Hilbert space with (bi-Lipschitz) distortion $O(\log n)$. This is optimal up to  the value of the implicit universal constant, as discovered independently by Linial-- London--Rabinovich~\cite{LLR95} and Aumann--Rabani~\cite{AR98}; the ``culprits'' which exhibit this optimality are $n$-vertex expander graphs. 

When a metric space $\MM$ does not admit a bi-Lipschitz embedding into a Hilbert space (which is typically the case), one can quantify  the extent to which $\MM$ is non-Euclidean by studying its Euclidean distortion growth sequence $\{\cc_2^n(\MM)\}_{n=1}^\infty$, which is defined by letting $\cc_2^n(\MM)$ be  the supremum of the Euclidean distortion   $\cc_2(\sub)$ over all   $n$-point subsets $\sub$ of $\MM$.  The aforementioned embedding theorem of Bourgain says that $\cc_2^n(\MM)\lesssim \log n$ for every metric space $\MM$. As any $n$-point metric space is isometric to a subset of $\ell_\infty^{n}$, the aforementioned nonembeddability result of  Linial-- London--Rabinovich and Aumann--Rabani says that $\cc_2^n(\ell_\infty)\asymp \log n$. Of course, one also obviously has $\cc_2^n(\ell_2)=1$. Prior to the present work and despite major efforts over the past decades,  the growth rate of $\cc_2^n(\bX)$ as $n\to \infty$ was not known up to positive constant factors for {\em any} infinite dimensional normed space $\bX$ other than the above trivial cases of $\ell_2$ and $\ell_\infty$, and obvious variants thereof (namely, $\cc_2^n(\bX)=O(1)$ if (and only if)  $\bX$ is $O(1)$-isomorphic to a Hilbert space, and for $\cc_2^n(\bX)\asymp \log n$ to hold it suffices that $\{\ell_\infty^n\}_{n=1}^\infty$ embed into $\bX$ with distortion $O(1)$, which is equivalent to $\bX$ not having finite cotype by a theorem of Maurey--Pisier~\cite{MP73}). 

Here we will prove that $\cc_2^n(\ell_1)\asymp \sqrt{\log n}$. This resolves a famous folklore question in metric geometry and functional analysis that goes back to Enflo's seminal work~\cite{Enf69}, though to the best of our knowledge it was first asked in print by Goemans~\cite{Goe97}. The new result herein is the positive embedding statement $\cc_2^n(\ell_1)\lesssim \sqrt{\log n}$, while multiple  $n$-point subsets of $\ell_1$ have long been known to have Euclidean distortion at least a positive constant multiple of $\sqrt{\log n}$, including the Hamming cubes~\cite{Enf69} and the Laakso and diamond graphs~\cite{Laa00,LP01,NR03}; the fact that our estimate $\cc_2^n(\ell_1)\lesssim \sqrt{\log n}$ is saturated by substantially different $n$-point subsets of $\ell_1$ indicates the difficulty of proving such an embedding result, which, as we will soon explain, holds in much greater generality and thus provides a positive answer within a large class of metric spaces to an old and influential question of Johnson--Lindenstrauss~\cite{JL82} (reiterated in~\cite[page~47]{Bou85}) on the validity of a natural nonlinear version of  John's theorem~\cite{Joh48}. 

   Given $0<s,\e<1$, say that a metric space $(\MM,d_\MM)$ is  $(s,\e)$-quasisymmetrically Hilbertian if there is a one-to-one function $f$ from $\MM$ to a Hilbert space $(\bH,\|\cdot\|_\bH)$ such that $\|f(x)-f(y)\|_\bH\le (1-\e)\|f(x)-f(z)\|_\bH$ for every $x,y,z\in \MM$ satisfying $d_\MM(x,y)\le s d_\MM(x,z)$.   Call $\MM$  quasisymetrically Hilbertian if there are $0<s,\e<1$ for which it is $(s,\e)$-quasisymmetrically Hilbertian.  The above requirement is satisfied trivially for every metric space $\MM$ when $\e=0$ and any $s>0$ (by  assigning points in $\MM$ to an orthonormal system). Thus, we are asking here for an embedding into a Hilbert space that is better---however slightly---than this trivial embedding, a requirement that holds when $\MM$ is quasisymmetrically equivalent to a subset of a Hilbert space in the sense of Ahlfors--Beurling~\cite{BA56} and Tukia--V\"ais\"al\"a~\cite{TK80}, which is a classical and extensively studied notion. The class of quasisymmetrically Hilbertian metric spaces encompasses disparate geometries, though not all metric spaces (e.g.,  any space that contains arbitrarily large expander graphs is not quasisymmetrically Hilbertian). One of our main embedding results (which includes the above statement on $\cc_2^n(\ell_1)$ as a special case), is that if  $\MM$ is  $(s,\e)$-quasisymmetrically Hilbertian, then $\cc_2^n(\MM)\lesssim_{s,\e} \sqrt{\log n}$. This  more general setup is geometrically  valuable in its own right, but also it includes all metric spaces of negative type, which is crucial for the following  application to computer science.

   The Sparsest Cut problem is a central and extensively studied topic in approximation algorithms due to its intrinsic interest, its  deep connections to pure mathematics, and its varied use as a subroutine for many other algorithmic tasks~\cite{Shm96}. It takes as its input an integer  $n$ (the ``universe size''), and two collections of symmetric nonnegative pairwise weights $\{C_{ij}\}_{(i,j)\in \{1,\ldots,n\}\times \{1,\ldots,n\}}$ and $\{D_{ij}\}_{(i,j)\in \{1,\ldots,n\}\times \{1,\ldots,n\}}$ (called ``capacities'' and ``demands,'' respectively), and its goal is to compute (or approximate) in polynomial time the minimum of $(\sum_{i\in S}\sum_{j\in \{1,\ldots,n\}\smallsetminus S} C_{ij})/\sum_{i\in S}\sum_{j\in \{1,\ldots,n\}\smallsetminus S} D_{ij}$ over all possible nonempty proper subsets $S$ of $\{1,\ldots ,n\}$ (this ratio should be thought of as the total capacity that crosses the boundary of the bipartition  $(S,\{1,\ldots,n\}\smallsetminus S)$ of $\{1,\ldots,n\}$, divided by the total demand that crosses its boundary).  In the mid-1990s, Goemans and Linial (independently) proposed an algorithm for Sparsest Cut, called the Goemans--Linial semidefinite program, but despite intensive efforts  prior to the present work the asymptotic growth rate as $n\to \infty$ (up to positive universal constant factors) of the performance of this famous algorithm remained unknown  (no other algorithm is currently conjectured to have asymptotically better performance than  the Goemans--Linial semidefinite program). Letting $\alpha_\mathsf{GL}(n)$ denote the smallest $\alpha\ge 1$ such that the Goemans--Linial semidefinite program outputs a number that is at most $\alpha$ times the aforementioned minimum of the ratio of total capacity to total demand across the boundary of a bipartition, we deduce herein that $\alpha_\mathsf{GL}(n)$ is bounded from above and from below by positive universal constant multiples of $\sqrt{\log n}$. Our new contribution is the positive algorithmic statement $\alpha_\mathsf{GL}(n)\lesssim \sqrt{\log n}$, as the matching impossibility result was already obtained (quite laboriously) in~\cite{naor2018vertical}.

Prior to passing to an overview of our main geometric structural result, which answers a question from~\cite{ALN05-STOC}, we will mention one more of its consequences; more applications will be described later in the Introduction.  Gromov classically associated a quantity called the observable diameter to a metric probability space $(\MM,d_\MM,\mu)$. Briefly and informally (the precise definition is recalled  in Section~\ref{sec:math applications}), the observable diameter quantifies the extent to which one could measure the size of the metric space  using sufficiently smooth real-valued functions as observations, while accounting for possible observational errors by discarding at most a fixed fraction (with respect to the given probability  measure $\mu$) of the observations. L\'evy's spherical isoperimetric theorem~\cite{Lev51} implies that the ratio between the observable diameter of the Euclidean $n$-sphere  to its actual (metric) diameter is of order $1/\sqrt{n}$. Here we prove that the Euclidean $n$-sphere is extremal in this regard among all nondegenerate $n$-dimensional metric probability spaces that are quasisymmetrically Hilbertian (both ``nondegenerate'' and ``$n$-dimensional'' need to be suitably defined here; see Section~\ref{sec:math applications} for the precise formulation), namely, up to dimension-independent positive constant factors, the Euclidean $n$-sphere has the smallest possible observable diameter among all such  spaces (this is new even for arbitrary Borel probability measures on $\R^n$).

\subsubsection{Random zero sets} Consider the unit integer grid in $\R^2$, colored black and white as a checker board. By randomly rotating and shifting this grid (say, rotating by a uniformly random angle and shifting by a vector chosen uniformly from a large bounded region), and then considering the union of all the black squares, one obtains a random subset of $\R^2$ which is an archetypical example of  what is called (following the terminology of~\cite{ALN05-STOC}) a random zero set (at unit scale; by dilating the grid one obtains the corresponding object at any given scale). There are multiple ways to perform such constructions in  $\R^n$ for any $n\in \N$ (e.g., flip an independent coin for each  integer cell and consider the union of all of those cells whose outcome is ``tails''). Another example of a useful random zero set for the Euclidean space $\R^n$ is as follows. Choose a direction $v\in S^{n-1}$ uniformly at random (i.e., $v$ is distributed according to the normalized surface area measure on the unit Euclidean sphere in $\R^n$), and consider the slabs $\big\{\{x\in \R^n:\ k\le \langle x,v\rangle<k+1\}:\ k\in \Z\big\}$. After coloring  those slabs back and white in an alternating fashion, shifting them randomly in the direction of $v$ and considering the union of the white slabs, one obtains another archetypical example of a random zero set. Constructions of this type are commonly used for organizing Euclidean point sets in harmonic analysis, geometry, algorithms and data structures. 

The above procedures rely on properties of $\R^n$ that are not purely metric (coordinate structure, random rotations). We wish to obtain meaningful analogues of such random zero sets in an arbitrary quasisymmetrically Hilbertian metric measure space $(\MM,d_\MM,\mu)$. We will consider only finite $\MM$  to circumvent  measurability issues, and only $\mu$ of full support (e.g., the counting measure on $\MM$). The most involved and innovative part of the present work is a construction when $\MM$ is $(s,\e)$-quasisymmetrically Hilbertian for any scale $\tau>0$ a random subset $\cZ_\tau$ of $\MM$ such that for {\em every} pair of points $x,y\in \MM$ with $d_\MM(x,y)\ge \tau$ and any $0<\lambda\le 1/2$, the probability that both  $x$ belongs to $\cZ_\tau$ and the distance of $y$ to $\cZ_\tau$ is at least $\lambda\tau$ is greater than a positive universal constant multiple of  $(\mu(B(y,19\beta\tau))/\mu(B(y,\beta\tau)))^{-\lambda^2}$, where $\beta=\beta(s,\e)$ depends only on $s,\e$ and $B(y,r)$ denotes the closed ball in $\MM$ of radius $r>0$ centered at $z\in \MM$. 

A crucial feature of the above random zero set is that the stated super-Gaussian probability lower bound depends (optimally) on the local growth of the measure $\mu$ near $y$ at scale $\tau$: If the measure of the  ball $B(y,19\beta\tau)$ is not much larger than the measure of $B(y,\beta\tau)$, then the bound improves. This  is critical for our applications as it leads to a  cancellation that yields estimates that are sharp up to constant factors, while previously known bounds included a redundant lower order factor that tends to $\infty$ as $|\MM|\to \infty$. 

Random zero sets as above were previously available with the local growth $\mu(B(y,19\beta\tau))/\mu(B(y,\beta\tau))$ replaced by the larger quantity  $\mu(\MM)/\min_{z\in \MM}\mu(\{z\})$ by the deep work~\cite{ARV04} of Arora--Rao--Vazirani in combination with the clever use of duality by Chawla--Gupta--R\"acke~\cite{CGR05}; see~\cite{ALN05-STOC}. The utility to metric embeddings of possibly improving this to random zero sets with the above local growth guarantee  was realized long ago, due to its compatibility with the measured descent embedding method~\cite{KLMN}. Indeed, the question whether  one could construct such random zero sets  was posed in~\cite{ALN05-STOC} while expressing skepticism that this  is possible  due to the nonlocal nature of the Arora--Rao--Vazirani reasoning. The fact that they do exist is an unexpected turn of events which leads to a conceptual   change of perspective on a major line  work in the literature on metric embeddings, as will be clarified in subsequent sections. Our proof introduces novel enhancements of the Arora--Rao--Vazirani  approach to rounding semidefinite programs (via the perspective that has been subsequently developed by Rothvoss~\cite{rothvoss2016lecture}), and we expect that more applications of this will be found beyond those that are obtained herein.

\subsection{Detailed technical  statements and background}\label{sec:detailed}

\noindent Unless stated otherwise, all the metric spaces that are discussed herein will be tacitly assumed to contain at least two points.  Our main result is the following geometric structural theorem for metric spaces:

\begin{theorem}\label{thm:random zero} There is a universal constant $\kappa>1$ with the following property. Given $\eta:[0,\infty)\to [0,\infty)$ and a finite metric space $(\MM,d)$ that has a modulus-$\eta$ quasisymmetric embedding into a Hilbert space, there is $\beta=\beta(\eta)>0$  that depends only on $\eta$ such that for every nondegenerate measure $\mu$ on $\MM$ and for every $\tau>0$ there exists a probability measure $\prob=\prob^{\tau,\MM}$ on the nonempty subsets $2^\MM\setminus\{\emptyset\}$ of $\MM$ that satisfies\footnote{We will use throughout the ensuing text the following (standard) conventions for asymptotic notation, in addition to  the usual $O(\cdot),o(\cdot),\Omega(\cdot), \Theta(\cdot)$ notation. Given $a,b>0$, by writing
$a\lesssim b$ or $b\gtrsim a$ we mean that $a\le \kappa b$ for some
universal constant $\kappa>0$, and $a\asymp b$
stands for $(a\lesssim b) \wedge  (b\lesssim a)$. Thus, we are asserting in~\eqref{eq:in main thm} that the probability is at least some positive universal constant. When we will need to allow for dependence on parameters, we will indicate it by subscripts. For example, in the presence of auxiliary objects $q,U,\phi$, the notation $a\lesssim_{q,U,\phi} b$ means that $a\le \kappa(q,U,\phi)b$, where $\kappa(q,U,\phi)>0$ may depend only on $q,U,\phi$, and similarly for the notations $a\gtrsim_{q,U,\phi} b$ and $a\asymp_{q,U,\phi} b$. Thus, in Theorem~\ref{thm:random zero} we can also write $\beta\asymp_\eta 1$.}
\begin{equation}\label{eq:in main thm}
\forall x,y\in \MM, \qquad d(x,y)\ge \tau \implies \prob\Bigg[ \emptyset\neq  \cZ\subset \MM:\  d(y,\cZ) \ge \frac{\beta \tau}{
\sqrt{1+\log\frac{\mu(B(y,\kappa\beta\tau))}{\mu(B(y,\beta\tau))}}}\quad  \mathrm{and}\quad      x\in \cZ\Bigg]\gtrsim 1.
\end{equation}
\end{theorem}

The statement of Theorem~\ref{thm:random zero} uses the following (standard) conventions for notation and terminology, which we will also maintain throughout what follows. Balls in a metric space $(\MM,d_\MM)$  will always be closed balls, namely $B(x,r)= \{y\in \MM:\ d_\MM(x,y)\le r\}$ for $x\in \MM$ and $r\ge 0$.  The distance of a point $y\in \MM$ from a nonempty subset $\cZ$ of $\MM$ is  $d(y,\cZ)=\inf_{z\in \cZ} d_\MM(y,z)$. A Borel measure on $\MM$ is nondegenerate if $0<\mu(B(x,r))<\infty$ for every $x\in \MM$ and every $r>0$. Finally (and mainly), one says that $(\MM,d_\MM)$ embeds quasisymmetrically into a metric space $(\NN,d_\NN)$ if there is an injection $\f:\MM\to \NN$ and  an increasing modulus $\eta:[0,\infty)\to [0,\infty)$ with $\lim_{s\to 0}\eta(s)=0$, for which  every distinct $x,y,z\in \MM$ satisfy
\begin{equation}\label{eq:def quasi}
\frac{d_\NN(\f(x),\f(y))}{d_\NN(\f(x),\f(z))}\le \eta \bigg(\frac{d_\MM(x,y)}{d_\MM(x,z)}\bigg).
\end{equation}
One calls such   $\f$ a modulus-$\eta$ quasisymmetric embedding of $\MM$ into $\NN$.

The notion of quasisymmetric embedding can be traced back to the classical work~\cite{BA56}, though it has been formally introduced in~\cite{TK80}. This important concept  has been investigated extensively; see for examples the survey~\cite{Vai99} and the monograph~\cite{Hei01} for an indication of the large body of work on this topic.  It suffices to say here that~\eqref{eq:def quasi} encodes the fact that $\f$ preserves the ``thinness'' of triangles, and hence it roughly preserve ``shape.'' In contrast, a bi-Lipschitz  embedding (see~\eqref{eq:def distortion} below) roughly preserves ``size,'' hence a fortiori it preserves shape. The requirement that $(\MM,d_\MM)$ embeds quasisymmetrically into $(\NN,d_\NN)$ is therefore quite weak,\footnote{As an indication of this,  it has long been unknown (see~\cite[problem~8.3.1]{Vai99}) whether {any} two separable infinite dimensional Banach spaces are quasisymmetrically equivalent to each other, i.e., there exists  a quasisymmetric bijection between them, and this possibility has been ruled out only relatively recently~\cite{Nao12}. As another illustration of the subtle restrictions  that a quasisymmmetry can impose, by~\cite{DS89} there are two closed $4$-dimensional manifolds that are homeomorphic but not quasisymmetrically homeomorphic. The discussion herein provides more obstructions to the existence of a quasisymmetry.} so  Theorem~\ref{thm:random zero} applies to a rich class of metric spaces, including notably those which have a snowflake that admits a bi-Lipschitz embedding into a Hilbert space, and in particular metric spaces of negative type (we will discuss these examples further below).

Theorem~\ref{thm:random zero} settles a problem that was considered by experts (for reasons that will soon become evident) ever since the appearance of the Arora--Rao--Vazirani  rounding algorithm~\cite{ARV04} and the measured descent embedding technique~\cite{KLMN}. The question whether  Theorem~\ref{thm:random zero} holds   was first  posed in the literature in~\cite{ALN05-STOC} (specifically, see the paragraph just before Section~3 there), though~\cite{ALN05-STOC}  expresses implicit doubt that   Theorem~\ref{thm:random zero} could be valid. The availability of Theorem~\ref{thm:random zero} is indeed a somewhat surprising development as it leads to a  change of perspective on (and a sharp improvement of) a line of work that straddles metric geometry, functional analysis and computer science. We will return to this in Section~\ref{sec:history}, which is devoted to a discussion of the history and prior results that were obtained in this direction, and a description of the conceptual  innovations that occur in our proof of Theorem~\ref{thm:random zero}. An explanation of the meaning and origin of~\eqref{eq:in main thm}, which is commonly called a ``random zero set'' requirement, appears in Section~\ref{sec:math applications} and Section~\ref{sec:history}. We will next explain some corollaries of Theorem~\ref{thm:random zero}.


\subsection{Geometric and algorithmic consequences of Theorem~\ref{thm:random zero}}

Here, we will explain some applications of Theorem~\ref{thm:random zero}; doing this will also demonstrate why it was sought-after  over the past 20 years by researchers in this area. Much of what we will describe belongs to the well-known intended consequences of this line of research (see Section~\ref{sec:history} for the history). Other parts are indeed corollaries of Theorem~\ref{thm:random zero} that have a short justification (modulo the available literature), but to the best of our knowledge they were not previously anticipated. We will recall all of the needed terminology when it will become relevant as we go along, thus making the ensuing discussion  self-contained.

\subsubsection{Quasisymmetrically Hilbertian metric spaces}

If $(\MM,d)$ is a metric space and $\{e_x\}_{x\in \MM}$ denotes the standard coordinate basis of $\ell_2(\MM)$,\footnote{The functional-analytic notation that is used herein is entirely standard, as in e.g.~the classical treatise~\cite{LT77}.} then for every increasing function $\eta:[0,\infty)\to [0,\infty)$ that satisfies $\lim_{s\to 0^+}\eta(s)\ge 1$ and for every distinct $x,y,z\in \MM$ we have
$$
\frac{\|e_x-e_y\|_{\ell_2(\MM)}}{\|e_x-e_z\|_{\ell_2(\MM)}}=\frac{\sqrt{2}}{\sqrt{2}}=1\le \lim_{s\to 0^+}\eta(s)< \eta \bigg(\frac{d(x,y)}{d(x,z)}\bigg).
$$
This observation shows that the existence of an embedding for which~\eqref{eq:def quasi} holds with $\NN$ a Hilbert space is a vacuous requirement (it is always satisfied) whenever $\lim_{s\to 0^+}\eta(s)\ge 1$.  Of course, $\lim_{s\to 0^+}\eta(s)=0$ is part of the definition of a quasisymmetric embedding, so this   trivial setting is excluded by the assumption of Theorem~\ref{thm:random zero}. Nevertheless, our proof of Theorem~\ref{thm:random zero} will show that its conclusion holds even if we require only that $\lim_{s\to 0^+}\eta(s)<1$. In other words, existence of an embedding into a Hilbert space that satisfies anything beyond the above trivial guarantee suffices for our purposes. It seems worthwhile to introduce the following terminology for this (perhaps somewhat deceptively) non-stringent    requirement:

\begin{definition}[quasisymmetrically Hilbertian metric space]\label{def:quasisym metric space} For $0<s,\e<1$, say that a metric space $(\MM,d)$ is $(s,\e)$-quasisymmetrically Hilbertian if there is an injection $\f:\MM\to \bH$ into a Hilbert space $\bH$ such that
\begin{equation}\label{eq:def quasi Hilbert}
\forall x,y,z\in \MM,\qquad d(x,y)\le s d(x,z)\implies  \|\f(x)-\f(y)\|_{\bH}\le (1-\e)\|\f(x)-\f(z)\|_{\bH}.
\end{equation}
Say that $(\MM,d)$ is quasisymmetrically Hilbertian if there exists such a $\f$ for some $0<s,\e<1$.
\end{definition}

Using the terminology of Definition~\ref{def:quasisym metric space}, we can now state the following strengthening of Theorem~\ref{thm:random zero}:

\begin{theorem}[generalization of Theorem~\ref{thm:random zero}]\label{thm:random zero general} There is a universal constant $\alpha>1$ with the following property. Given $0<s,\e\le 1/2$, denote $\beta=s^{\alpha/\e}$. Suppose that a finite metric space $(\MM,d)$ is $(s,\e)$-quasisymmetrically Hilbertian and that $\mu$ is a nondegenerate measure on $\MM$. Then, for every $\tau>0$ there is a probability measure $\prob^\tau$ on  $2^\MM\setminus\{\emptyset\}$ such that for every $0<\lambda\le 1/2$ and every $x,y\in \MM$ with $d(x,y)\ge\tau$ we have\footnote{The  constant  $19$ that appears in~\eqref{eq:in main thm-general} is included for concreteness as it arises from the ensuing proof, but it is neither claimed to be sharp nor is it important for the applications of~\eqref{eq:in main thm-general} that are obtained herein. The same is true for the constants in~\eqref{eq:super exponential general metric}.}
\begin{equation}\label{eq:in main thm-general}
 \prob^\tau\big[ \emptyset\neq  \cZ\subset \MM:\  d(y,\cZ) \ge \lambda\beta \tau   \quad  \mathrm{and} \quad x\in \cZ  \big]\gtrsim  \left(\frac{\mu(B(y,19\beta\tau))}{\mu(B(y,\beta\tau))}\right)^{-\lambda^2}.
\end{equation}
\end{theorem}

Theorem~\ref{thm:random zero} coincides with  the special case $\lambda\asymp 1/\sqrt{1+ \log (\mu(B(y,19\beta\tau))/\mu(B(y,\beta\tau)))}$ of   Theorem~\ref{thm:random zero general}, now with $\kappa$ and the dependence of $\beta$ on $\eta$  given explicitly. To the best of our knowledge, the distributional estimate~\eqref{eq:in main thm-general} in the rest of the range $0<\lambda\le 1/2$ was not previously conjectured in the literature.

In contrast to Theorem~\ref{thm:random zero general}, for general metric spaces the best analogous result is the following theorem:

\begin{theorem}\label{thm:general doubling exponential} Let $(\MM,d)$ be a finite metric space and $\mu$ a nondegenerate measure on $\MM$. For every $\tau>0$ there is a probability measure $\prob^\tau$ on  $2^\MM\setminus\{\emptyset\}$ such that for every $0<\lambda\le 1/8$ and every $x,y\in \MM$ with $d(x,y)\ge\tau$,
\begin{equation}\label{eq:super exponential general metric}
 \prob^\tau\big[ \emptyset\neq  \cZ\subset \MM:\ d(y,\cZ) \ge \lambda \tau \quad  \mathrm{and}\quad    x\in \cZ   \big]\ge \frac14\left(\frac{\mu(B(y,\frac58\tau))}{\mu(B(y,\frac18\tau))}\right)^{-8\lambda}.
 \end{equation}
\end{theorem}

Even though Theorem~\ref{thm:general doubling exponential} does not appear in the literature, it follows from a generalization of the proof of the main result of~\cite{MN07}, combined with an idea from~\cite{Rao99}; we will include the proof   of Theorem~\ref{thm:general doubling exponential} in Section~\ref{sec:prove loose ends}. The fact that Theorem~\ref{thm:general doubling exponential} cannot be improved follows from Remark~\ref{rem:general doubling} below. 

In summary, Theorem~\ref{thm:random zero general} states that the additional information that $\MM$ is quasisymmetrically Hilbertian  can be used  to improve the linear dependence on $\lambda$ in the exponent in right hand side of~\eqref{eq:super exponential general metric} to the quadratic dependence on $\lambda$ in the exponent in the right hand side of~\eqref{eq:in main thm-general}, which we will soon see is the best one can hope for.

\subsubsection{Observable diameter, embeddings, nonlinear spectral gaps, and Lipschitz extension}\label{sec:math applications}

It is natural to start by recalling  the following definition, which was introduced in~\cite[Definition~2.2]{ALN05-STOC}:

\begin{definition}[spreading random zero set]\label{def:random zero} Fix $\zeta>0$ and $0<\d\le 1$. A finite metric space $(\MM,d)$  is said to admit a  random zero set which is $\zeta$-spreading with probability $\d$ if for every $\tau>0$ there is a probability measure $\prob^\tau$ on the nonempty subsets $2^\MM\setminus\{\emptyset\}$ of $\MM$ such that
\begin{equation}\label{eq:def  zero set}
\forall x,y\in \MM, \qquad d(x,y)\ge \tau\implies \prob^\tau\Big[ \emptyset \neq  \cZ\subset \MM:\  d(y,\cZ) \ge \frac{\tau}{
\zeta}  \quad  \mathrm{and}\quad  x\in \cZ   \Big]\ge \d.
\end{equation}
\end{definition}

The requirement~\eqref{eq:def  zero set} of Definition~\ref{def:random zero} is simpler than the conclusion~\eqref{eq:in main thm} of Theorem~\ref{thm:random zero},  which involves a spreading parameter $\zeta=\zeta(y,\tau)$ that  depends on the point $y$ and the distance scale $\tau$, suitably encoding what one should think of as the ``local growth rate of the measure $\mu$ near $y$ at scale $\tau$.''   In a similar vein, the requirement~\eqref{eq:def  zero set} of Definition~\ref{def:random zero} is simpler than the conclusion~\eqref{eq:in main thm-general}  of Theorem~\ref{thm:random zero general}, which allows the lower bound $\d$ on the ``separation probability'' to depend on the local growth rate of the measure $\mu$ near $y$ at scale $\tau$. Our first example of the utility of such more refined geometric information deduces  the existence of a traditional random zero set in sense of Definition~\ref{def:random zero}, when the metric space is doubling.

Given $K\in \N$, a metric space $(\MM,d)$ is said to be  $K$-doubling if for every $x\in \MM$ and every $r\ge 0$ there are $x_1,\ldots,x_K\in \MM$ such that $B(x,2r)\subset B(x_1,r)\cup\cdots\cup B(x_K,r)$; note that necessarily $K\ge 2$ if $\MM$ contains at least two points. By~\cite{VK84}, if $\sub\subset \MM$ is finite (it suffices to only assume here is that $(\sub,d)$ is complete~\cite{LS98}), then there is a measure $\mu$ whose support is equal to  $\sub$ such that for every $x\in \sub$ and every $r\ge 0$,
\begin{equation}\label{eq:doubling measure}
0<\mu\big(B(x,2r)\big)\le K^{O(1)}\mu\big(B(x,r)\big).
\end{equation}

By applying Theorem~\ref{thm:random zero general} to a measure $\mu$ as in~\eqref{eq:doubling measure}, we obtain the following statement:

\begin{theorem}\label{cor:spreading doubling} There is a universal constant $\kappa>1$ with the following property. Fix $0<s,\e\le \frac12$, $p\ge 1$, $K\in \N$. Let $(\MM,d)$ be a $K$-doubling metric space that is  $(s,\e)$-quasisymmetrically Hilbertian. Then, every finite subset of $\MM$ admits a random zero set that is $(1/s)^{\kappa/\e}\!\!\sqrt{\max\{1,(\log K)/p\}}$-spreading with probability $e^{-\kappa p}$.
\end{theorem}

\begin{proof}[Proof of Theorem~\ref{cor:spreading doubling} assuming Theorem~\ref{thm:random zero general} ]  If $\sub\subset \MM$ is finite, then use~\cite{VK84} to get a doubling measure   $\mu$ satisfying~\eqref{eq:doubling measure} whose support is equal to $\sub$. Let $\alpha>1$ and $\beta=s^{\alpha/\e}$  be as in Theorem~\ref{thm:random zero general}.  By iterating~\eqref{eq:doubling measure} five times one sees that  $\mu(B(y,19\beta \tau))\le \mu(B(y,2^5\beta\tau))\le K^{O(1)}\mu(B(y,\beta\tau))$ for every $\tau>0$ and $y\in \sub$. Theorem~\ref{thm:random zero general} therefore provides a probability measure $\prob^\tau$ on $2^{\sub}\setminus\{\emptyset\}$ and a universal constant  $\gamma>1$ such that for every $\lambda \in [1/\sqrt{\log K},1/2]$ and every $x,y\in \sub$ with $d(x,y)\ge\tau$ we have
\begin{equation}\label{eq:doubling version gaussian}
 \prob^\tau\big[ \emptyset\neq  \cZ\subset \sub:\  x\in \cZ \quad  \mathrm{and}\quad  d(y,\cZ) \ge \lambda\beta \tau   \big]\ge  K^{-\gamma\lambda^2}.
\end{equation}
If $1\le p\le (\log K)/4$, i.e., $\sqrt{p/\log K}\in [1/\sqrt{\log K},1/2]$, we may use~\eqref{eq:doubling version gaussian} with $\lambda=\sqrt{p/\log K}$ to see that $\prob^\tau$ is $\zeta$-spreading with probability $e^{-\gamma p}$, where $\zeta= 1/(\lambda \beta)= (1/s)^{\alpha/\e}\sqrt{(\log K)/p}$.  If $p> (\log K)/4$, then use~\eqref{eq:doubling version gaussian}  with $\lambda=1/2$ to get that   $\prob^\tau$ is $\zeta$-spreading with probability $K^{-\gamma/4}> e^{-\gamma p}$, where $\zeta= 2/\beta=2(1/s)^{\alpha/\e}$.  The desired conclusion now follows provided  $\kappa>1$ is  a sufficiently large universal constant.
\end{proof}

We will next describe some consequences of Theorem~\ref{cor:spreading doubling}, starting with resolving a problem that was left open in~\cite{naor2005quasisymmetric}. The geometric phenomenon that~\cite{naor2005quasisymmetric} aimed to  establish is: 
\medskip
\begin{displayquote}
{\em ``Up to dimension-independent constant factors, the $n$-dimensional Euclidean sphere has the smallest possible observable diameter  among  all the normalized $e^{O(n)}$-doubling  metric probability spaces that are quasisymmetrically Hilbertian.''}
\end{displayquote}
\medskip

For a metric space $(\MM,d)$, a Borel probability measure $\mu$ on $\MM$, and  $\theta>0$ the $\theta$-observable diameter $$\mathrm{ObsDiam}_\mu^{(\MM,d)}(\theta)$$ of the metric measure space $(\MM,d,\mu)$ is defined~\cite{Gro07} to be the supremum over all possible $1$-Lipschitz functions $f:\MM\to \R$ of the infimum of $\diam_\R(f(\cS))=\sup\{|f(x)-f(y)|:\ x,y\in \cS\}$ over all possible Borel subsets $\cS\subset \MM$ with $\mu(\cS)\ge 1-\theta$. If the metric is clear from the context, then we will use the  notation 
$$\mathrm{ObsDiam}_\mu^\MM(\theta)=\mathrm{ObsDiam}_\mu^{(\MM,d)}(\theta).$$

The observable diameter is   extensively-studied; see e.g.~\cite{Led01,Gro07,Shi16,BF22}. Its qualitative interpretation (e.g.~\cite[page~336]{Ber00}) is to view each real-valued $1$-Lipschitz function as a way to ``observe'' the size of a given metric space while accounting for possible observational errors by discarding at most a $\theta$-fraction (with respect to a given probability measure) of the observations. As explained in~\cite{Gro07}, for every $n\in \N$ and every  $0<\theta<1$, the $\theta$-observable diameter of the unit Euclidean sphere in $\R^n$ equipped with its normalized surface area measure is at most $1/\sqrt{n}$ times a positive factor that depends only on $\theta$.


For $n\in \N$, denote  the standard  unit Euclidean sphere in $\R^n$ by $S^{n-1}=\{x\in \R^n:\ \|x\|_2=1\}$. Throughout what follows, $S^{n-1}$ will be equipped with the metric that is inherited from $\ell_2^n$, and  $\sigma^{n-1}$ will denote the normalized surface-area (probability) measure on $S^{n-1}$. Since $\ell_2^n$ is  $K_n$-doubling for some $K_n\le 5^n$ (e.g.~\cite[Section~2.2]{Sem01}; a better upper bound on $K_n$ appears in~\cite{Ver05}), also  $S^{n-1}\subset \ell_2^n$ is $5^n$-doubling. 

If $(\MM,d)$ is a metric space and $\mu$ is a Borel probability measure on $\MM$, then we say that $\mu$ is normalized if the median of $d(x,y)$ is equal to $1$ when $(x,y)\in \MM\times \MM$ is distributed according to $\mu\times \mu$. Note that $\sigma^{n-1}$ is not normalized per this terminology, but one could normalize it by rescaling the metric by a factor that is bounded from above and from below by positive universal constants.

\begin{theorem}\label{thm:observable} For every $0<s,\e<1$ and every  $n\in \N$, if  $(\MM,d)$ is a  metric space that is $(s,\e)$-quasisymmetrically Hilbertian and $5^n$-doubling, then  every normalized Borel probability measure $\mu$ on $\MM$ satisfies 
\begin{equation}\label{eq:observable lower} 
\forall 0<\theta\le \theta_0,\qquad \mathrm{ObsDiam}_\mu^{\MM}(\theta)\gtrsim_{s,\e} \mathrm{ObsDiam}_{\sigma^{n-1}}^{S^{n-1}}(\theta),
\end{equation}
where $\theta_0>0$ is a universal constant. 
\end{theorem}

Theorem~\ref{thm:observable}   establishes that the above extremal property of the observable diameter of the Euclidean sphere indeed holds; its (simple) derivation from Theorem~\ref{cor:spreading doubling} appears in Section~\ref{sec:obs}.

It would be worthwhile to investigate the extent to which the phenomenon that is expressed in Theorem~\ref{thm:observable}  holds for metric spaces that need not be quasisymmetrically Hilbertian. For example, we do not know if when $2<p<\infty$, any $e^{O(n)}$-doubling subset $\MM$ of $\ell_p$ satisfies~\eqref{eq:observable lower} (for any Borel probability measure $\mu$ on $\MM$), with the implicit dependence on $s,\e$ now replaced by a dependence on $p$; if this is not the case, then it would be interesting to determine the asymptotic rate (as $n\to \infty$) at which the smallest possible observable diameter of such spaces tends to $0$. Theorem~\ref{thm:observable}  does not cover this setting because $\ell_p$ is not quasisymmetrically Hilbertian when $p>2$; this follows from~\cite{Nao12}, as explained in Remark~\ref{rem:path} below.  It is simple to see that some assumption on the $e^{O(n)}$-doubling metric  space is needed for the conclusion~\eqref{eq:observable lower}  of Theorem~\ref{cor:spreading doubling} to hold with a dimension-independent constant factor (i.e., Theorem~\ref{cor:spreading doubling}  does not hold for subsets of $\ell_\infty$); indeed, expander graphs demonstrate this, as explained in Remark~\ref{rem:expander obs} below. 

Theorem~\ref{thm:observable}  is a manifestation of the following  impossibility result for super-Gaussian concentration of any quasisymmetrically Hilbertian doubling metric probability space.  Given a metric space $(\MM,d)$ and a Borel probability measure $\mu$ on $\MM$,  the isoperimetric (or concentration) function of the metric probability space $(\MM,d,\mu)$ is defined as follows, where $\mathcal{Bor}{(\MM,d)}\subset 2^\MM$ denotes the Borel subsets of $\MM$: 
\begin{equation}\label{eq:def:isoperimetric function}
\forall t\ge 0,\qquad I_\mu^{(\MM,d)}(t)\eqdef \sup_{\substack{\emptyset \neq \sub\in \mathcal{Bor}{(\MM,d)}\\\mu(\sub)\ge \frac12}} \mu\big(\{x\in \MM:\ d(x,\sub)\ge t\}\big).
\end{equation}
We will use the  notation $I_\mu^\MM(t)=I_\mu^{(\MM,d)}(t)$ when the  metric is clear from the context. 

Concentration of measure is related to  the observable diameter   through the following  relations, which hold for every complete metric space $(\MM,d)$ and every Borel probability measure $\mu$ on $\MM$; their simple derivations can be found in e.g.~\cite[Proposition~1.12]{Led01} and~\cite[Remark~2.28]{Shi16}. 
\begin{equation}\label{eq:relations obs iso}
\forall \theta,t>0,\qquad I_\mu^\MM(t)>\theta \implies \mathrm{ObsDiam}_\mu^{\MM}(\theta)\ge t  \qquad\mathrm{and}\qquad  I_\mu^\MM(t)\le \theta\implies \mathrm{ObsDiam}_\mu^{\MM}(2\theta)\le 2t. 
\end{equation}

The solution~\cite{Lev51} of the isoperimetric problem for the Euclidean sphere implies that for any $n\in \N$, 
\begin{equation}\label{eq:levy}
\forall t>0,\qquad I^{S^{n-1}}_{\sigma^{n-1}}(t)\lesssim e^{-O(n)t^2}.
\end{equation}
Alternatively, by the solution~\cite{borell1975brunn,sudakov1978extremal} of the isoperimetric problem for the Gaussian measure, if we rescale the metric on $\ell_2^n$ by a factor of order $1/\sqrt{n}$ so that the standard Gaussian measure on $\ell_2^n$ will be normalized, the isoperimetric function of the resulting metric probability space is also bounded from above by the right hand side of~\eqref{eq:levy}. More such examples are available in the literature; see  e.g.~\cite{Mau79,Gro80,MS86,GM87}, and a systematic treatment can be found in~\cite{Led01}. Inspired by~\cite{Cha14},\footnote{In~\cite{Cha14}, even weaker estimates about variance rather than concentration are said to be super-concentrated.} we say that a metric probability space exhibits the super-concentration phenomenon if its isoperimetric function is asymptotically smaller (as $n\to \infty$) than the right hand side of~\eqref{eq:levy}. This phenomenon is useful when available, as demonstrated in~\cite{Cha14}, but Theorem~\ref{thm:no-super} below, which is also a simple consequence of Theorem~\ref{cor:spreading doubling} that we will derive in Section~\ref{sec:obs},  asserts the following ``no measure is super-concentrated'' phenomenon: If $(\MM,d)$ is a quasisymmetrically Hilbertian metric space that is $e^{O(n)}$-doubling, then $(\MM,d,\mu)$   does not have super-concentration for any Borel probability measure $\mu$ on $\MM$ whatsoever. 

\begin{theorem}\label{thm:no-super} For every $0<s,\e<1$ there exists $c=c(s,\e)>0$ with the following property.\footnote{An inspection of the ensuing proofs reveals that if $0<s,\e\le \frac12$, then both the constant $c$ in~\eqref{eq:isoperimetric lower intro} and the implicit constant factor in~\eqref{eq:observable lower} can be taken to be a positive universal constant multiple of $s^{\kappa/\e}$, where $\kappa$ is the  constant in Theorem~\ref{cor:spreading doubling}.} If $n\in \N$ and $(\MM,d)$ is a  metric space that is $(s,\e)$-quasisymmetrically Hilbertian and $5^n$-doubling, then  every normalized Borel probability measure $\mu$ on $\MM$ satisfies
\begin{equation}\label{eq:isoperimetric lower intro}
\forall 0<\phi\le 1,\qquad I_\mu^\MM(c\phi)\gtrsim  e^{-n\phi^2}.
\end{equation}
\end{theorem}


Following~\cite{Rab08,Nao21}, given $p,D\ge 1$ we say that a metric space $(\MM,d_\MM)$ embeds with $p$-average distortion $D$ into a metric space $(\NN,d_\NN)$ if for every Borel probability measure $\mu$ on $\MM$  there are $s>0$ and $f=f_\mu:\MM\to \NN$ that is $sD$-Lipschitz, namely, $d_\NN(f(x),f(y))\le sDd_\MM(x,y)$ for every $x,y\in \MM$, such that
$$
\left(\iint_{\MM\times \MM} d_\NN(f(x),f(y))^p\ud \mu(x)\ud\mu(y)\right)^{\frac{1}{p}}\ge s\left(\iint_{\MM\times \MM} d_\MM(x,y)^p\ud\mu(x)\ud\mu(y)\right)^{\frac{1}{p}}.
$$
If this occurs for $p=1$, then one omits $p$ by saying that $\MM$ embeds with average distortion $D$ into  $\NN$, and if this occurs for $p=2$, then one  says that $\MM$ embeds with quadratic average distortion $D$ into  $\NN$.

By substituting Theorem~\ref{cor:spreading doubling} into~\cite[Proposition~7.1]{Nao14}, and then substituting the resulting statement into~\cite[Lemma~60]{Nao21}, we arrive at the following corollary of Theorem~\ref{thm:random zero}:

\begin{theorem}\label{coro:doubling average embedding} There exists a universal constant $\kappa>1$ with the following property. Suppose that  $K\in \N$ and $0<s,\e\le 1/2$. Let $(\MM,d)$ be a $K$-doubling metric space that is $(s,\e)$-quasisymmetrically Hilbertian. Then, for every $p\ge 1$ one can embed $(\MM,d)$ into the real line  $\R$ with $p$-average distortion at most
\begin{equation}\label{eq:average distortion bound}
\left(\frac{1}{s}\right)^{\frac{\kappa}{\e}}\max\left\{\sqrt{\frac{\log K}{p}},1\right\}.
\end{equation}
\end{theorem}

The upper bound~\eqref{eq:average distortion bound} on the average distortion in Theorem~\ref{coro:doubling average embedding}  is new  even for subsets of a Hilbert space.\footnote{Nevertheless, we expect that for the (very) special case of subsets of Hilbert space one could prove Theorem~\ref{coro:doubling average embedding} using more standard (and simpler) chaining-style reasoning \`a la Fernique--Talagrand, along the lines of the proof in~\cite{IN07} which treats a related, but different, setting. The main value of Theorem~\ref{coro:doubling average embedding} is its applicability to metric spaces that are far from Euclidean.} Furthermore, Theorem~\ref{coro:doubling average embedding} is optimal in the entire range of possible values of $K\in \N$ and  $p\ge 1$, up to the dependence on $s,\e$ (which we did not investigate in the present work), as seen by the Fact~\ref{prop: Rn} below, which we will prove in Section~\ref{sec:Rn into line}; this also demonstrates the optimality of Theorem~\ref{thm:random zero general}.

\begin{fact}\label{prop: Rn} For  $n\in \N$ and $p\ge 1$, the smallest $D\ge 1$ such that $\ell_2^n$ embeds into $\R$ with $p$-average distortion $D$ is bounded from above and from below by positive universal constant multiples of $\sqrt{\max\{1,n/p\}}$.
\end{fact}

By standard volumetric reasoning~\cite{CW71,Hei01}, every $n$-dimensional normed space is $K_n$-doubling for some $K_n=e^{O(n)}$. As the quantity  $\sqrt{\max\{1,n/p\}}$ in Fact~\ref{prop: Rn}  is bounded from above and from below by positive universal constant multiples of $\sqrt{\max\{1,(\log K_n)/p\}}$, Fact~\ref{prop: Rn}  shows that Theorem~\ref{coro:doubling average embedding} is sharp.

Theorem ~\ref{eq:type 2 p dependence}  below shows that Theorem~\ref{coro:doubling average embedding} is, in fact, sharp in the following much stronger sense: Even if one relaxes the  stringent goal of Theorem~\ref{coro:doubling average embedding} to embed  $\MM$ into $\R$ by allowing the embedding to be into any Banach space of type $2$, then   it is still impossible to do so with $p$-average distortion that is  asymptotically better (as $p,K\to \infty$) than what is provided by Theorem~\ref{coro:doubling average embedding}. For this (as well as for ensuing discussions), we recall the (standard) terminology  that a Banach space $(\bX,\|\cdot\|)$ is said~\cite{Hof74,MP76}  to have (Rademacher) type $1\le r\le 2$ if there is $T>0$ such that for every $n\in\N$, every $x_1,\ldots,x_n\in \bX$ satisfy
\begin{equation}\label{eq:def type r}
\E \Big[\big\|\sum_{i=1}^n \e_i x_i\big\|^r\Big]\le T^r\sum_{i=1}^n \|x_i\|^r,
\end{equation}
where the expectation in~\eqref{eq:def type r} is with respect to i.i.d.~symmetric Bernoulli random variables $\e_1,\ldots,\e_n$, i.e., they are independent and $\Pr[\e_i=1]=\Pr[\e_i=-1]=1/2$ for every $i\in [n]=\n$. The infimum over those $T$ for which the above requirement holds is denoted $T_r(\bX)$.  We then have, for example,  $T_q(L_q)\asymp 1$ when $1\le q\le 2$, and $T_2(L_q)\asymp\sqrt{q}$ when $q\ge 2$, as explained in, say, Chapter~6 of the textbook~\cite{AK16}. 

The proof of the following theorem was shown to us by Alexandros Eskenazis; we thank him for allowing us to include it in Section~\ref{sec:Rn into line}, which also treats target spaces of type $r$ for any $1\le r\le 2$.

\begin{theorem}\label{eq:type 2 p dependence} Suppose that $(\bX,\|\cdot\|)$ is a Banach space of type $2$. For every $n\in \N$ and $p\ge 1$, if $\ell_1^n$  embeds into $\bX$ with $p$-average distortion $D$, then necessarily $D\gtrsim \sqrt{n}/\max\{\sqrt{p},T_2(\bX)\}$.
\end{theorem}

Theorem~\ref{eq:type 2 p dependence} implies the aforementioned strong optimality of Theorem~\ref{coro:doubling average embedding}  as $\ell_1$ is  an instance of a metric space that embeds quasisymmetrically into a Hilbert space (two different examples of closed-form formulae of embeddings that exhibit this appear in e.g.~\cite[Remark~5.10]{MN04} and~\cite[Section~3]{Nao10}).

\begin{remark}\label{rem:general doubling} By repeating the above reasoning mutatis mutandis while using Theorem~\ref{thm:general doubling exponential}  instead of Theorem~\ref{thm:random zero general}, one sees that for every $p\ge 1$, any finite $K$-doubling metric space admits a random zero set that is $O(\max\{1,(\log K)/p\})$-spreading with probability $e^{-p}$. Consequently, any $K$-doubling metric space embeds with $p$-average distortion  $O(\max\{1,(\log K)/p\})$ into $\R$. This cannot be improved even if one allows the embedding to be  into any target normed space $\bX$ that has type $r>1$. Indeed, by combining the deep result of~\cite{Laf09} with the first bound in equation~(131) of~\cite{Nao21} and the reasoning in~\cite{LLR95,Mat97}, one sees that  for arbitrarily large $n\in \N$ there is an $n$-point metric space  such if it embeds with $p$-average distortion $D$ into $\bX$, then necessarily $D\gtrsim_{r,T_r(\bX)} \max\{1,(\log n)/p\}$. It  remains to note that any $n$-point metric space is (vacuously) $n$-doubling. Theorem~\ref{coro:doubling average embedding} thus shows that for doubling metric spaces that are also quasisymmetrically Hilbertian the aforementioned upper bound on the $p$-average distortion into $\R$  can be (sharply) improved to the square root of the best that one can hope for in the setting of arbitrary metric spaces.


\end{remark}

Given $n\in \N$, let  $\bigtriangleup^{\!n-1}=\{\pi=(\pi_1,\ldots,\pi_n)\in [0,1]^n:\ \sum_{i=1}^n\pi_i=1\}$ denote the simplex of probability measures on $[n]$. When we say that a matrix $\sfA=(a_{ij})\in \M_n(\R)$ is stochastic we always mean row-stochastic, i.e., $(a_{i1},\ldots,a_{in})\in  \bigtriangleup^{\!n-1}$ for every $i\in \n$. Given  $\pi\in \bigtriangleup^{\!n-1}$, a stochastic matrix $\sfA=(a_{ij})\in \M_n(\R)$ is $\pi$-reversible if $\pi_ia_{ij}=\pi_ja_{ji}$ for every $i,j\in \n$. In this case, $\sfA$ is a self-adjoint contraction on $L_2(\pi)$ and the decreasing rearrangement of the eigenvalues of $\sfA$ is denoted $1=\lambda_1(\sfA)\ge\ldots\ge \lambda_n(\sfA)\ge -1$. The spectral gap of $\sfA$ is    $1-\lambda_2(\sfA)$. The conductance (with respect to $\sfA, \pi$) of a subset   $\emptyset \neq S\subsetneq [n]$ is
$$
\Phi_\sfA(S)\eqdef \frac{1}{\pi(S)\pi([n]\setminus S)}\sum_{i\in S}\sum_{j\in [n]\setminus S}\pi_i a_{ij},
$$
where $\pi(S)=\sum_{i\in S} \pi_i$. The Cheeger constant of $\sfA$ is defined by
$$
h(\sfA)\eqdef \min_{\emptyset \neq S\subsetneq [n]} \Phi_\sfA(S).
$$

For $p\ge 1$, the  $p$-Poincar\'e constant  (reciprocal of the $p$-spectral gap) of $\sfA$ with respect to a metric space $(\MM,d)$, denoted $\gamma(\sfA,d^p)$, is the infimum over  $\gamma\in [0,\infty]$ such that for every $x_1,\ldots,x_n\in \MM$ we have:
\begin{equation}\label{eq:def nonlinear gap}
\sum_{i=1}^n \sum_{j=1}^n \pi_i\pi_j d(x_i,x_j)^p\le \gamma\sum_{i=1}^n\sum_{j=1}^n \pi_i a_{ij} d(x_i,x_j)^p.
\end{equation}
The works~\cite{Mat97,Laf08,NS11,MN14,Nao14,Nao18,ANNRW18,LS21,Nao21,Esk22,Nao24-mixing} include information  on nonlinear spectral   gaps and their applications. Because it is simple to check that $\gamma(\sfA,d_\R^2)=1/(1-\lambda_2(\sfA))$ and $\gamma(\sfA,d_\R)=1/h(\sfA)$, where $d_\R$ is the standard metric on $\R$, Theorem~\ref{coro:doubling average embedding} implies:

\begin{theorem}\label{thm:gap estimate} Fix  $n,K\in \N$ and $0<s,\e\le 1/2$. Suppose that $(\MM,d)$ is a  metric space that is both $K$-doubling and $(s,\e)$-quasisymmetrically Hilbertian. Then, every stochastic matrix $\sfA\in \M_n(\R)$ satisfies
\begin{equation}\label{eq:nonlinear gap quasi}
\gamma(\sfA,d)\lesssim_{s,\e} \frac{\sqrt{\log K}}{h(\sfA)}\qquad\mathrm{and}\qquad \gamma(\sfA,d^2)\lesssim_{s,\e} \frac{\log K}{1-\lambda_2(\sfA)}.
\end{equation}
Furthermore, the implicit constants in~\eqref{eq:nonlinear gap quasi} can be taken to be $(1/s)^{O(1/\e)}$.
\end{theorem}
Theorem~\ref{thm:gap estimate} complements the bounds in~\cite[Section~7.1.1]{Nao14} by showing that they can be improved if one considers quasisymmetrically Hilbertian doubling metric spaces rather than arbitrary doubling metric spaces; in the latter case, the dependence on $K$ in the two estimates in~\eqref{eq:nonlinear gap quasi} can be taken to be, respectively, $\log K$ and $(\log K)^2$, and this is optimal. The second inequality in~\eqref{eq:nonlinear gap quasi} is optimal, as seen by taking $\sfA=\mathsf{Q}_n$ to be the transition matrix of the standard random walk on the Hamming cube, because it is straightforward to check that $\mathsf{Q}_n$ satisfies the following bounds:
$$
\frac{1}{1-\lambda_1(\mathsf{Q}_n)}\asymp n\qquad \mathrm{yet}\qquad \gamma(\mathsf{Q}_n,d_{\ell_1^n}^2) \asymp n^2.
$$

Whether the first inequality in~\eqref{eq:nonlinear gap quasi} can be improved is perhaps the most important question that remains open in the context of the present work; see Question~\ref{Q:main} below. The best lower bound that is currently available is the example  in~\cite{KM13}, which  shows that one cannot replace the $\sqrt{\log K}$ term in~\eqref{eq:nonlinear gap quasi} by anything smaller than $\exp(c\sqrt{\log\log K})$ for some universal constant $c>0$. If it were possible to improve the dependence on $K$ in the first inequality in~\eqref{eq:nonlinear gap quasi}, then this would have major algorithmic implications; we will explain this matter after presenting the relevant algorithmic setup in Section~\ref{sec:alg}.

\begin{question}\label{Q:main} For $0<s,\e\le 1/2$ and $K\in \N$, what is the growth rate as $K\to \infty$ (up to constant factors that depend on $s,\e$) of the smallest $\alpha=\alpha(K,s,\e)$ such that $\gamma(\sfA,d)\le \alpha/h(\sfA)$ for every $(s,\e)$-quasisymmetrically Hilbertian $K$-doubling   metric space $(\MM,d)$, every $n\in \N$ and every stochastic matrix $\sfA\in \M_n(\R)$?
\end{question}

Among the consequences of the nonlinear spectral gap bounds of Theorem~\ref{thm:gap estimate}  is that they quickly imply the following bound on the cutting modulus, which is a parameter that was introduced in~\cite{ANNRW18} as a tool for  hierarchically partitioning metric spaces. It was used in~\cite{ANNRW18}  to construct a data structure for approximate nearest neighbor search; we will return to this algorithmic setting  in Section~\ref{sec:alg}.

Given $\Phi>0$, the cutting modulus $\Xi_\MM(\Phi)$ of a metric space $(\MM,d)$ was defined in~\cite{ANNRW18} to be the infimum over those $\Xi\in (0,\infty]$ with the following property. For every $n\in \N$, every $r>0$, every $\pi\in \Delta^{n-1}$, every stochastic $\pi$-reversible matrix $\sfA=(a_{ij})\in \M_n(\R)$, and every $x_1,\ldots,x_n\in \MM$ that satisfy
\begin{equation}\label{eq:geometric graph r}
\forall i,j\in [n],\qquad a_{ij}>0\implies d(x_i,x_j)\le r,
\end{equation}
at least one of the following  two (non-dichotomic) scenarios must occur:
\begin{itemize}
\item Either there exists $\emptyset \neq S\subsetneq [n]$ such that $\Phi_\sfA(S)\le \Phi$,
\item or, there exists $i\in [n]$ such that $\pi\big(\{j\in [n]:\ d(x_i,x_j)\le \Xi r\}\big)\ge \frac12$.
\end{itemize}

\begin{theorem}\label{thm:thm:cutting mudulus} Fix  $K\in \N$ and $0<s,\e\le 1$. Suppose that $(\MM,d)$ is a  metric space that is both $K$-doubling and $(s,\e)$-quasisymmetrically Hilbertian. Then, for every $\Phi>0$ we have
$$
\Xi_\MM(\Phi)\lesssim_{s,\e} \frac{\sqrt{\log K}}{\Phi}.
$$
\end{theorem}

\begin{proof}[Proof of Theorem~\ref{thm:thm:cutting mudulus}  assuming Theorem~\ref{thm:gap estimate}] Fix $\Phi>0$ and take any $0<\Xi<\Xi_\MM(\Phi)$. By the definition of  $\Xi_\MM(\Phi)$ as an infimum, this means that there are  $n\in \N$ and $r>0$ for which we can find a probability vector $\pi\in \Delta^{n-1}$, a  stochastic $\pi$-reversible matrix $\sfA=(a_{ij})\in \M_n(\R)$, and points $x_1,\ldots,x_n\in \MM$ that satisfy~\eqref{eq:geometric graph r}   yet
\begin{equation}\label{eq:negation of ball cut}
\min_{\emptyset \neq S\subsetneq [n]}\Phi_\sfA(S)> \Phi\qquad\mathrm{and} \qquad \min_{j\in [n]}\pi\big(\{j\in [n]:\ d(x_i,x_j)> \Xi r\}\big)>\frac12.
\end{equation}
The first assertion in~\eqref{eq:negation of ball cut} is the same as saying that $h(\sfA)> \Phi$. Hence, $\gamma(\sfA,d)\lesssim_{s,\e}\sqrt{\log K}/\Phi$ by the first inequality  in~\eqref{eq:nonlinear gap quasi}. An application of the definition of $\gamma(\sfA,d)$ to the points $x_1,\ldots,x_n$ now shows that
\begin{equation*}
\frac{\Xi r}{2}\stackrel{\eqref{eq:negation of ball cut}}{\le} \sum_{i=1}^n\sum_{j=1}^n \pi_i\pi_j d(x_i,x_j)\lesssim_{s,\e}\frac{\sqrt{\log K}}{\Phi} \sum_{i=1}^n\sum_{j=1}^n \pi_i a_{ij} d(x_i,x_j)\stackrel{\eqref{eq:geometric graph r} }{\le} \frac{\sqrt{\log K}}{\Phi}r.\tag*{\qedhere}
\end{equation*}
\end{proof}

\begin{remark}\label{rem:use chegger bound instead of cheeger inequality}The above proof  of Theorem~\ref{thm:thm:cutting mudulus}   used the first inequality  in~\eqref{eq:nonlinear gap quasi}. One can also deduce Theorem~\ref{thm:thm:cutting mudulus}  from Theorem~\ref{thm:gap estimate} by repeating this reasoning mutatis mutandis while using the second inequality  in~\eqref{eq:nonlinear gap quasi} and incorporating Cheeger's inequality for Markov chains~\cite{JS88,LS88}. Indeed, the cutting modulus was bounded in~\cite{ANNRW18} using bounds on quadratic nonlinear spectral gaps in this way. However, it is worthwhile to deduce Theorem~\ref{thm:thm:cutting mudulus} using only the first inequality  in~\eqref{eq:nonlinear gap quasi} since the second inequality  in~\eqref{eq:nonlinear gap quasi} is sharp yet the first inequality  in~\eqref{eq:nonlinear gap quasi} could possibly be improved, as expressed in Question~\ref{Q:main}. Algorithmic ramifications of this possibility will be discussed in Section~\ref{sec:alg}.
\end{remark}

Following~\cite{LLR95}, the (bi-Lipschitz) distortion of a finite metric space $(\MM,d_\MM)$ in an infinite metric space $(\NN,d_\NN)$, which is denoted $\cc_{(\NN,d_\NN)}(\MM,d_\MM)$ or simply $\cc_\NN(\MM)$ when the underlying metrics are clear from the context, is the infimum over those $D>0$ for which there are $f:\MM\to \NN$ and $s>0$ such that
\begin{equation}\label{eq:def distortion}
\forall x,y\in \MM,\qquad sd_\MM(x,y)\le d_\NN\big(f(x),f(y)\big)\le Dsd_\MM(x,y).
\end{equation}

If $p\ge 1$ and $\NN$ is an infinite dimensional $L_p(\mu)$ space, then $\cc_{L_p(\mu)}(\MM)$ does  not depend on $\mu$ (in fact,  this holds if $\dim (L_p(\mu))\ge |\MM|(|\MM|-1)/2$ by~\cite{Bal90}), so one commonly uses the shorter notation:
\begin{equation}\label{eq:cp notation}
\cc_p(\MM)\eqdef \cc_{L_p(\mu)}(\MM).
\end{equation}
The especially important and useful parameters $\cc_2(\MM)$ and $\cc_1(\MM)$ are naturally called, respectively, the Euclidean distortion of $\MM$ and the $L_1$ distortion of $\MM$.

The classical Fr\'echet embedding $\Phi_{(\MM,d_\MM)}$  of $\MM$ into the  set  $\R^{2^\MM\setminus \{\emptyset\}}$ of real-valued functions on the  nonempty subsets of $\MM$ is defined by setting for every $x\in \MM$ and $\emptyset \neq \cZ\subset \MM$,
\begin{equation}\label{eq:def Frechet embedding}
\Phi_{(\MM,d_\MM)}(x)(\cZ)\eqdef d_\MM(x,\cZ).
\end{equation}
Given $p,D\ge 1$, one says that $(\MM,d_\MM)$ embeds into $L_p$ with distortion $D$ via the Fr\'echet embedding if there exists a probability measure
$\prob$ on  $2^\MM\setminus\{\emptyset\}$ such that for every $x,y\in \MM$ we have
\begin{equation}\label{eq:frechet condition}
d_\MM(x,y)\le D\|\Phi_{(\MM,d_\MM)}(x)-\Phi_{(\MM,d_\MM)}(y)\|_{L_p(\prob)}.
\end{equation}
For example, the famous embedding of~\cite{Bou85} is of this form. This terminology is consistent with what we recalled above because~\eqref{eq:frechet condition}  implies that $\cc_p(\MM)\le D$. Indeed, contrast~\eqref{eq:frechet condition} with the trivial estimate
\begin{equation}\label{eq:trivial frechet}
\|\Phi_{(\MM,d_\MM)}(x)-\Phi_{(\MM,d_\MM)}(y)\|_{L_p(\prob)}\le \|\Phi_{(\MM,d_\MM)}(x)-\Phi_{(\MM,d_\MM)}(y)\|_{L_\infty(\prob)}\le d_\MM(x,y),
\end{equation}
where the first step of~\eqref{eq:trivial frechet} holds as $\prob$ is a probability measure and the second step of~\eqref{eq:trivial frechet} is a straightforward consequence of the triangle inequality for $d_\MM$. However, knowing that $(\MM,d_\MM)$ embeds into $L_p$ with distortion $D$ via the Fr\'echet embedding provides more information than mere embeddability into $L_p$, which is sometimes needed in applications (e.g.~\cite{FHL08,MM16}); furthermore, the former embedding may not be possible  when it is known  that the latter embedding does exist (e.g.~\cite{MR01,BLMN06}).

The  measured descent embedding technique~\cite{KLMN05} yields the following result:

\begin{theorem}[special case of measured descent~\cite{KLMN05}]\label{thm:quote descent} Fix $n\in \N$, $\alpha\ge \beta>0$, and $0<\e,\d,\theta\le 1$. Let $(\MM,d)$ be an $n$-point metric space with the property that for every $\tau>0$ there is a probability measure $\prob^\tau$ on the nonempty  subsets of $\MM$ such that for every $x,y\in \MM$ that satisfy $\tau \le d(x,y)\le (1+\theta)\tau$ we have
\begin{equation}\label{eq:cadinality version of zero set-intro}
\prob^\tau\Bigg[\emptyset\neq   \cZ\subset \MM:\  d(y,\cZ) \ge \frac{\e\tau}{
\sqrt{1+\log\frac{|B(y,\alpha\tau)|}{|B(y,\beta\tau)|}}}\quad  \mathrm{and}\quad      x\in \cZ\Bigg]\ge \d.
\end{equation}
Then,
\begin{equation}\label{eq:quote descent}
\cc_2(\MM)\lesssim_{\e,\d,\theta,\alpha,\beta} \sqrt{\log n}.
\end{equation}
Furthermore, the bound~\eqref{eq:quote descent} on the Euclidean distortion of $\MM$ is obtained via the Fr\'echet embedding.
\end{theorem}

Even though those who are familiar with~\cite{KLMN05}  will recognize Theorem~\ref{thm:quote descent} as a special case of measured descent, Theorem~\ref{thm:quote descent} does not appear in~\cite{KLMN05}  as a standalone statement. Instead, Theorem~\ref{thm:quote descent}  follows from {part of\footnote{The proof of \cite[Lemma~1.8]{KLMN05} derives  further facts that are relevant to the setup of~\cite{KLMN05}, but not for Theorem~\ref{thm:quote descent}.} the  proof of} Lemma~1.8 in~\cite{KLMN05}: the random zero set appears in the second displayed equation on page 847 of~\cite{KLMN05} and the corresponding Fr\'echet embedding appears in the line immediately following it.  As Theorem~\ref{thm:quote descent}  is crucial for the  applications of Theorem~\ref{thm:quote descent} that we will explain next, we will include in Section~\ref{sec:descent revisited}  a self-contained  proof of Theorem~\ref{thm:quote descent} which builds on the above ideas while incorporating further enhancements so as to yield both a more general statement and the best dependence that we currently have of the implicit constant factor in~\eqref{eq:quote descent} on $\e,\d,\theta,\alpha,\beta$; see~\eqref{eq:quote descent with explicit constants}.


\smallskip

In the  influential work~\cite{JL82}, Johnson and Lindenstrauss asked if $\cc_2(\MM)\lesssim \sqrt{\log n}$ for any $n$-point metric space $(\MM,d)$. If this were true, then it would have been a satisfactory analogue of  John's theorem~\cite{Joh48} that $\cc_2(\bX)\le \sqrt{\dim \bX}$ for any finite-dimensional normed space $\bX$, in accordance with the predictions of the  Ribe program~\cite{Bou86,Kal08,Nao12-ribe,Bal13,Ost13,God17,Nao18}. Bourgain famously answered~\cite{Bou85} the aforementioned Johnson--Lindenstrauss question negatively;  an asymptotically stronger (sharp) impossibility result  was subsequently obtained (by another method) in~\cite{LLR95,AR98}. Hence, a positive answer would necessitate imposing restrictions on the metric space $(\MM,d)$. The following theorem follows by substituting Theorem~\ref{thm:random zero} into Theorem~\ref{thm:quote descent}. It answers the Johnson--Lindenstrauss problem positively within the class of quasisymmetrically Hilbertian  metric spaces.

\begin{theorem}[Johnson--Lindenstrauss problem/nonlinear John theorem]\label{thm:quasisymmetric John} For every  $0<s,\e\le 1/2$ and  $n\in \N$, if $(\MM,d)$ is an $n$-point $(s,\e)$-quasisymmetrically Hilbertian metric space, then
$
\cc_2(\MM)\lesssim_{s,\e} \sqrt{\log n}.
$
\end{theorem}

The class of metric spaces that admit a quasisymmetric embedding into a Hilbert space encompasses metric spaces of finite Assouad--Nagata dimension~\cite{LS05}, or more generally~\cite{NS11} metric spaces that admit a padded stochastic decomposition (thus, it includes, say, planar graphs~\cite{KPR93} and doubling metric spaces~\cite{Ass83,KPR03}), as well as $L_p(\mu)$ spaces for $1\le p\le 2$~\cite{BDK65,WW75}, and the infinite dimensional   Heisenberg group $\mathbb{H}^\infty$~\cite{LN06}.

\begin{remark} Ignoring much of the information that is provided by Theorem~\ref{coro:doubling average embedding},  one can consider the real line as a subset of $\ell_2$ to deduce that for any $K\in \N$ and $0<s,\e\le 1/2$, a $K$-doubling $(s,\e)$-quasisymmetrically Hilbertian metric space embeds into $\ell_2$ with average distortion $O(\sqrt{\log K})$, which is  sharp\footnote{We warn of the following somewhat confusing pitfall here. Assouad famously proved~\cite{Ass83} that any doubling metric space embeds quasisymmetrically into $\ell_2$ (even into $\R^n$ for some $n\in \N$). So, superficially it might seem that our assumption that the space is quasisymmetrically Hilbertian is redundant. However, Assouad's theorem  yields a quasisymmetric embedding whose modulus (necessarily) depends on the doubling constant. } by Fact~\ref{prop: Rn}.  This nonlinear version of John's theorem is quite satisfactory for the following reasons. Firstly, it does not hold for general $K$-doubling metric spaces, for which the best bound that one can get on the Euclidean average distortion is $O(\log K)$, as follows from the proofs in~\cite{LLR95,AR98,Mat97}; a stronger  impossibility result appears in Remark~\ref{rem:general doubling}. Secondly, unlike Theorem~\ref{thm:quasisymmetric John}, such a result does not hold for bi-Lipschitz embeddings, so one must somehow relax the requirement from the embedding, as we did here by considering average distortion.  Indeed, by~\cite{Laa00,LP01} there exists a doubling metric space $\MM$ that embeds quasisymmetrically into $\ell_2$ by~\cite{Ass83}, yet $\MM$ does not admit any bi-Lipschitz embedding into $\ell_2$; the subsequent works~\cite{CK06,LN06}, which are natural extensions to infinite dimensional targets of the important contributions~\cite{Pan89,Sem96}, show that one can also take $\MM$ here to be the Heisenberg group.  
\end{remark}

 A longstanding open question  asks for the growth rate of the largest possible Euclidean distortion of a finite subset of $\ell_1$. This question arose from Enflo's influential proof~\cite{Enf69} that the $d$-dimensional Hamming cube $\sub_d=\{0,1\}^d\subset \ell_1^d$ has Euclidean distortion  $\sqrt{d}$, i.e., it is of order $\sqrt{\log n}$ for $n=2^d=|\sub_d|$.

 To the best of our knowledge, the aforementioned question remained a famous unpublished folklore problem for many years, though it was eventually published   by Goemans in~\cite[page~157]{Goe97}, who stated the elegant conjecture that for every $n\in \N$, any $n$-point subset of $\ell_1$ embeds into $\ell_2$ with distortion $O(\sqrt{\log n})$; see also Linial's Open Problem~4 in~\cite{Lin02}. By contrasting the conjectural asymptotic upper bound in~\cite{Goe97} with the lower bound of~\cite{Enf69}, it would then follow that the $d$-dimensional Hamming cube has the asymptotically  largest possible growth rate among all subsets of $\ell_1$ of size $2^d$.

 As an indication of the  difficulty here,  we mention that arbitrarily large $n$-point subsets of $\ell_1$ that are different from the Hamming cube have been shown to have Euclidean distortion  $\Theta(\sqrt{\log n})$. E.g.,~these can be planar graphs, which can even be taken to be   $O(1)$-doubling~\cite{Laa00,LP01,NR03}. Thus, the above question has asymptotic extremizers that are markedly different from each other.

Theorem~\ref{thm:quasisymmetric John} provides the following positive resolution of the above conjecture:

\begin{theorem}[finite subsets of $\ell_1$ that are furthest from being Euclidean]\label{thm:l1} For every $n\in \N$, the maximal Euclidean distortion of an $n$-point subset of $\ell_1$ is bounded from above and from below by positive universal constant multiples of $\sqrt{\log n}$.
\end{theorem}

For $0<\theta\le 1$ and a metric space $(\MM,d)$, it is common to call (in reference to the  Koch snowflake; see e.g.~\cite{DS97}) the metric space $(\MM,d^\theta)$ the $\theta$-snowflake of $(\MM,d)$, or simply the $\theta$-snowflake of $\MM$ if the metric is clear from the context. If the $\frac12$-snowflake of $\MM$ embeds isometrically into a Hilbert space, then one says that $\MM$ is a metric space of negative type; see e.g.~the monograph~\cite{DL10} or the survey~\cite{Nao10} for more on this important and useful notion, including the reason for the nomenclature.

In~\cite[page~158]{Goe97}, Goemans conjectured that for every $n\in \N$, any $n$-point  metric space of negative type embeds into $\ell_2$ with distortion $O(\sqrt{\log n})$. By definition, a metric space of negative type has a modulus-$\eta$ quasisymmetric emebdding into $\ell_2$  with $\eta(t)=\sqrt{t}$ for every $t\ge 0$, and with~\eqref{eq:def quasi} holding as equality. Hence,  Theorem~\ref{thm:quasisymmetric John}  provides the following positive resolution of Goemans' conjecture (the corresponding lower bound follows by considering the Hamming cube $\{0,1\}^d\subset \ell_1^d$ again and using~\cite{Enf69}, as its image in $\ell_2^d$ under the formal identity mapping exhibits that it is a metric space of negative type):

\begin{theorem}[finite metrics of negative type that are furthest from being Euclidean]\label{thm:neg} For every $n\in \N$, the largest possible Euclidean distortion of an $n$-point metric space of negative type is bounded from above and from below by positive universal constant multiples of $\sqrt{\log n}$.
\end{theorem}

In fact, we also deduce from Theorem~\ref{thm:quasisymmetric John} the following result that answers the major problem  of evaluating the growth rate of the largest possible $L_1$ distortion of a finite metric space of negative type (this formulation appears as Open Problem~3 in~\cite{Lin02}, and see also  e.g.~\cite{Goe97,Mat02}):

\begin{theorem}[finite metrics of negative type that are furthest from subsets of $L_1$]\label{thm:negl1} For every $n\in \N$, the largest possible $L_1$ distortion of an $n$-point metric space of negative type is bounded from above and from below by positive universal constant multiples of $\sqrt{\log n}$.
\end{theorem}

The upper bound on the distortion  in Theorem~\ref{thm:negl1} follows from Theorem~\ref{thm:quasisymmetric John} since  $\cc_1(\MM)\le \cc_2(\MM)$ for every finite metric space $(\MM,d_\MM)$; there are multiple ways to justify this, but perhaps the quickest (though overkill) is to apply Dvoretzky's theorem~\cite{Dvo60}. The matching lower bound  is due to~\cite{naor2018vertical}, where it is exhibited by considering a sufficiently dense net in the unit ball of the $5$-dimensional Heisenberg group $\mathbb{H}^5$, equipped with a suitable equivalent  metric of negative type that was constructed in~\cite{LN06} (by~\cite{NY22}, the $3$-dimensional Heisenberg group $\mathbb{H}^3$ is insufficient for this purpose).

\begin{remark} Until its landmark negative answer in~\cite{KV05}, it was a major open question (due, independently, to Goemans and Linial)  whether every metric space of negative type admits a bi-Lipschitz embedding into $L_1$. Theorem~\ref{thm:negl1} completes the efforts by many researchers in mathematics and theoretical computer science (including~\cite{LN06,ALN08,chawla2008embeddings,KR09,CK10,CK10-monotonicity,CKN10,KV15,naor2018vertical}, as well as multiple unpublished works) to understand the  rate (for $n$-point metric spaces of negative type, as $n\to \infty$) of the failure of this possibility. It is worthwhile to recall in this context the (still open) conjecture of~\cite{ANV10} that any invariant  metric of negative type on an Abelian group  admits a bi-Lipschitz embedding into $L_1$.
\end{remark}

Thus far, we applied Theorem~\ref{thm:random zero} by combining it with  Theorem~\ref{thm:quote descent} without utilizing the further knowledge from Theorem~\ref{thm:quote descent} that the corresponding distortion bound~\eqref{eq:quote descent} is obtained  via the  Fr\'echet embedding. We will next present an application to the Lipschitz extension problem that incorporates this additional information. In order to do so, we first need to recall basic notation that originates  in~\cite{Mat90}.

Suppose that $(\SS,d_\SS)$ is a (source) metric space and  $(\TT,d_\TT)$ is a (target) metric space. For a  nonempty subset $\MM$ of $\SS$, let $\ee(\SS,\MM;\TT)$ denote the infimum over those $K\in [1,\infty]$ such that for every $L\ge 0$ and every $L$-Lipschitz function  $f:\MM\to \TT$ there is a $KL$-Lipschitz function $F:\SS\to \NN$ that extends $f$.

The following observation is a straightforward consequence of the above definition:

\begin{observation}\label{observation:unravel definition} Fix $\alpha,\beta\ge 0$. Suppose that $(\SS,d_\SS)$ is a metric space and that $\emptyset \neq \MM\subset \SS$. Let $(\NN,d_\NN)$ be a metric space such that there is an $\alpha$-Lipschitz function $\f:\SS\to \NN$ satisfying  $d_\NN(\f(x),\f(y))\ge d_\SS(x,y)/\beta $ for every $x,y\in \MM$. Then, $\ee(\SS,\MM;\TT)\le \alpha\beta \ee(\NN,\f(\MM);\TT)$ for every metric space $(\TT,d_\TT)$.
\end{observation}

The Lipschitz extension modulus $\ee(\SS;\TT)$ of  a  pair $(\SS,\TT)$ of metric spaces is  defined to be the supremum of $\ee(\SS,\MM;\TT)$ over all the nonempty subsets $\MM$ of $\SS$.
Analogously, following~\cite{LN04} we define the absolute Lipschitz extension modulus $\ae(\MM;\TT)$ of  a  pair $(\MM,\TT)$ of metric spaces to be the supremum of $\ee(\SS,\MM;\TT)$ over all the possible metric spaces $\SS$ that (isometrically) contain $\MM$ as a subset.

The study of the above moduli has a rich history due to their intrinsic geometric and analytic interest as well as their applications and connections to multiple areas of mathematics; a  (partial) description of this background appears in the monographs~\cite{BB12,BB12-2,CMN19}, and in~\cite{MN13,NR17,Nao24}.   Nevertheless, basic problems about these moduli have stubbornly resisted many efforts, including the following example of a  longstanding  open question due to Ball~\cite{Bal92}, to which we will return later in an algorithmic context, following the insights of Makarychev and Makarychev in~\cite{MM16}:

\begin{question}[Ball's extension problem]\label{Q:ball} Is it true that  $\ee(\ell_2;\ell_1)<\infty$?
\end{question}

More generally, we pose the following question which we think is of fundamental importance:

\begin{question}\label{Q:targets ext} Characterize the class $\TT\!(\ell_2)$ of those target metric spaces $(\TT,d_\TT)$ for which $\ee(\ell_2;\TT)<\infty$.

\end{question}

Kirszbraun's theorem~\cite{Kir34} says that $\ell_2\in \TT\!(\ell_2)$; in fact, $\ee(\ell_2;\ell_2)=1$. Question~\ref{Q:ball} is of course the special case of Question~\ref{Q:targets ext} that asks whether $\ell_1\in  \TT\!(\ell_2)$. One could also aim to characterize for any source metric space $(\SS,d_\SS)$ the class $\TT\!(\SS)$ of those target metric spaces $(\TT,d_\TT)$ for which $\ee(\SS,\TT)<\infty$. This seems to be a challenging though potentially fruitful research direction. As a start, we do not know if given two Banach spaces  $\bX,\bY$, the equality  $\TT\!(\bX)=\TT\!(\bY)$ implies that $\bX$ and $\bY$ are bi-Lipschitz equivalent.

When $\ee(\SS;\TT)=\infty$, which is typically the case, one still studies the Lipschitz extension problem for the pair $(\SS,\TT)$ by considering the asymptotic growth rate of the following finitary moduli (which are  finite if a quite mild assumption on the target space $\TT$ holds; see~\cite{JLS86} and Section~5.3 in~\cite{Nao24}):
$$
\forall n\in \N,\qquad \ee_n(\SS;\TT)\eqdef \sup_{\substack {\emptyset \neq\MM\subset \SS\\ |\MM|\le n}} \ee(\SS,\MM;\TT).
$$
The introduction of~\cite{NR17} (particularly Section~1.3 there) surveys known bounds on  these parameters.\footnote{Note   that~\cite{NR17} failed to state the best known bound on $\ee_n(\ell_1,\ell_p)$ for $1<p<2$, which is $\ee_n(\ell_1,\ell_p)\lesssim \sqrt{(\log n)/(p-1)}$. This is a special case of~\cite[Theorem~2.1]{MN06}, i.e., it was actually available at the time~\cite{NR17}  was written but overlooked there.} The following consequence of Theorem~\ref{thm:random zero} provides a modest amount of further information on this topic:

\begin{corollary}\label{thm:ext}  There exists a universal constant $C>0$ with the following property. Fix $0<s,\e\le 1/2$ and an integer  $n\ge 2$. Suppose that  $(\MM,d_\MM)$ is an $n$-point metric space that is also $(s,\e)$-quasisymmetrically Hilbertian. Then, for every $0<\d\le 1$ and every target metric space $(\TT,d_\TT)$ we have
\begin{equation}\label{eq:our extension bound}
\ae(\MM;\TT)\lesssim_{s,\e}  e^{\frac{C}{\d}}\ee_n\big(\ell_2^{\lceil \d \log n\rceil} ;\TT\big)\sqrt{\log n}.
\end{equation}
\end{corollary}

The main result of~\cite{JL82} is that $\ae(\MM;\ell_2)\lesssim \sqrt{\log n}$ for any $n$-point metric space $\MM$. By Corollary~\ref{thm:ext} for $\d=1$ we see that if  $\MM$ is also $(s,\e)$-quasisymmetrically Hilbertian, then $\ae(\MM;\TT)\lesssim_{s,\e,\TT} \sqrt{\log n}$ for any metric space $(\TT,d_\TT)$ that belongs to the class $\TT\!(\ell_2)$ from Question~\ref{Q:targets ext}.

\begin{proof}[Proof of Corollary~\ref{thm:ext} assuming Theorem~\ref{thm:random zero}] Fix a metric space $(\SS,d_\SS)$ such that $\MM\subset \SS$ and the restriction of $d_\SS$ to $\MM\times \MM$ coincides with $d_\MM$. By combining Theorem~\ref{thm:random zero} with Theorem~\ref{thm:quote descent} there is a probability measure  $\prob$ on $2^{\MM}\setminus \{\emptyset\}$ such that the Fr\'echet embedding $\Phi=\Phi_{(\MM,d_\MM)}:\MM\to L_2(\prob)$  in~\eqref{eq:def Frechet embedding} satisfies
\begin{equation}\label{eq:in verse of frechet}
\forall x,y\in \MM,\qquad \|\Phi(x)-\Phi(y)\|_{L_2(\prob)}\gtrsim_{s,\e} \frac{1}{\sqrt{\log n}}d_\MM(x,y).
\end{equation}

The key observation now is that each of the coordinates of the Fr\'echet embedding is a distance from a nonempty subset of $\MM$, so one can define $\Phi$ on the super space $\SS$ and not just on $\MM$, and the $1$-Lipschitz condition~\eqref{eq:trivial frechet}, which is nothing more than an application of the triangle inequality for $d_\SS$, still holds.  Therefore, we may assume from now that $\Phi:\SS\to L_2(\prob)$ is $1$-Lipschitz and satisfies~\eqref{eq:in verse of frechet}.

Denote $m=\lceil \d \log n \rceil$. The (proof of the) Johnson--Lindenstrauss dimension reduction lemma~\cite{JL82} (see equation~(28) in~\cite{Nao18} for the version that we are using  here) implies that there exists a universal constant $C>0$ and  a $1$-Lipschitz function $g:\Phi(\MM)\to \ell_2^m$  that satisfies
\begin{equation}\label{eq:JL g}
\forall u,v\in \Phi(\MM),\qquad  \|g(u)-g(v)\|_2\ge e^{-\frac{C}{\d}}\|u-v\|_{L_2(\prob)}.
\end{equation}

By Kirszbraun's extension theorem~\cite{Kir34} (also e.g.~\cite[Chapter~1]{BL00}) there is a $1$-Lipschitz function $G:L_2(\prob)\to \ell_2^m$ that extends $g$. Denote $\f=G\circ\Phi:\SS\to \ell_2^m$. Then $\f$ is $1$-Lipschitz and for every $x,y\in \MM$,
$$
\|\f(x)-\f(y)\|_2=\|g(\Phi(x))-g(\Phi(y))\|_2\stackrel{\eqref{eq:JL g}}{\gtrsim} e^{-\frac{C}{\d}}\|\Phi(x)-\Phi(y)\|_{L_2(\prob)}\stackrel{\eqref{eq:in verse of frechet}}{\gtrsim}_{s,\e} \frac{e^{-\frac{C}{\d}} }{\sqrt{\log n}}d_\MM(x,y).
$$
It remains to apply Observation~\ref{observation:unravel definition} with $\NN=\ell_2^m=\ell_2^{\lceil \d \log n\rceil}$, $\alpha=1$ and $\beta\lesssim_{s,\e} e^{\frac{C}{\d}}  \sqrt{\log n}$, and then to take the supremum over all possible super-spaces $(\SS,d_\SS)$ of $(\MM,d_\MM)$ to arrive at the desired estimate~\eqref{eq:our extension bound}.
\end{proof}

\subsubsection{Sparsest Cut, Max-Cut, sparsification, and nearest neighbor search}\label{sec:alg}


Given $n\in \N$, the Sparsest Cut problem (with general capacities and demands) on $n$-vertices takes as its input two $n$-by-$n$ symmetric matrices with nonnegative entries $\sfC=(c_{ij}),\sfD=(d_{ij})\in \M_n([0,\infty))$  and aims to evaluate (or estimate) in polynomial time the following quantity:
\begin{equation}\label{eq:sparsest cut def}
\mathrm{SparsestCut}(\sfC,\sfD)\eqdef \min_{\emptyset \neq S\subsetneq [n]}\frac{\sum_{i\in S}\sum_{j\in [n]\setminus S}c_{ij}}{\sum_{i\in S}\sum_{j\in [n]\setminus S}d_{ij}}.
\end{equation}

In the mid-1990s Goemans and Linial introduced a semidefinite program (SDP) that computes  (with $o(1)$ precision) in polynomial a number $\SDP_{\GL}(\sfC,\sfD)\ge 0$ that satisfies:
\begin{equation*}\label{eq:relaxation}
\SDP_{\GL}(\sfC,\sfD)\le  \mathrm{SparsestCut}(\sfC,\sfD).
\end{equation*}
See e.g.~\cite{naor2018vertical} and the references therein for a relatively recent account of this extensively studied topic, as well as Section~\ref{sec:history}. The pertinent question is therefore to understand the growth rate as $n\to \infty$ of the integrality gap of the Goemans--Linial SDP for Sparsest Cut, which is defined to be the following quantity:
$$
\sup_{\substack{\sfC,\sfD\in \M_n([0,\infty))\\ \sfC,\sfD\ \mathrm{symmeric}}}\frac{\mathrm{SparsestCut}(\sfC,\sfD)}{\SDP_{\GL}(\sfC,\sfD)}.
$$
The algorithm that outputs $\SDP_{\GL}(\sfC,\sfD)$ is then guaranteed to estimate $\mathrm{SparsestCut}(\sfC,\sfD)$ within a factor that is at most this integrality gap.  No other algorithm is currently known to (or is conjectured to) perform asymptotically better then this algorithm of Goemans  and Linial.

\begin{theorem}[integrality gap of the Goemans--Linial SDP for Sparsest Cut]\label{thm:sparsest} For every $n\in \N$, the $n$-vertex integrality gap of the Goemans--Linial semidefinite program for the Sparsest Cut problem is bounded from above and from below by positive universal constant multiples of $\sqrt{\log n}$.
\end{theorem}

The  lower bound of $\Omega(\sqrt{\log n})$ in Theorem~\ref{thm:sparsest} on the  integrality gap of the Goemans--Linial SDP for Sparsest Cut is from~\cite{naor2018vertical}. Thus, our contribution  to Theorem~\ref{thm:sparsest} is in terms of algorithm design rather than proving an impossibility result, i.e., we derive an improved (sharp) upper bound on the integrality gap of this SDP, and hence we provide the best known polynomial time approximation algorithm  for Sparsest Cut. The fact that Theorem~\ref{thm:negl1} implies a $O(\sqrt{\log n})$   upper bound on the  integrality gap of the Goemans--Linial SDP for Sparsest Cut is a standard result that has been known for a long time (going back at least to~\cite{Goe97});  see~\cite[Lemma~4.5]{Nao10} for its detailed derivation (based on the duality argument in~\cite[Proposition~15.5.2]{Mat02}, which is attributed there to unpublished work of Rabinovich).

We focused above on the algorithmic task of efficiently approximating  the number $\mathrm{SparsestCut}(\sfC,\sfD)$, but the fact that our $O(\sqrt{\log n})$-embedding of any $n$-point metric space of negative type is actually (per Theorem~\ref{thm:neg}) into $\ell_2$   rather than ``merely'' into $\ell_1$ implies formally that there is also an algorithm which outputs a subset $\emptyset\neq S\subsetneq [n]$ that is a near-minimizer of the right hand side of~\eqref{eq:sparsest cut def}, up to the aforementioned $O(\sqrt{\log n})$ error tolerance. This deduction utilizes the observation from the seminal work~\cite{LLR95} of Linial, London and Rabinovich that optimal embeddings into a Hilbert space can themselves be found in polynomial time (unlike embeddings into $L_1$, see~\cite{IM04,DL10}), as this task itself can be cast as a semidefinite program; the (standard) deduction of this assertion is worked out in e.g.~\cite[Section~5]{ALN08}.

One can refine the above discussion for each $k\in \{2,\ldots,n\}$ to obtain a $O(\sqrt{\log k})$-factor approximation algorithm if the matrix $\sfD$ has support of size at most $k$, i.e., $|\{(i,j)\in [n]\times [n]: d_{ij}>0\}|\le k$. Also this fact is a formal consequence of what we have seen, as Theorem~\ref{thm:neg} is proved by substituting Theorem~\ref{thm:random zero} into Theorem~\ref{thm:quote descent}, which yields the stated distortion guarantee via the Fr\'echet embedding. That embedding can be automatically extended to any super-space  while remaining $1$-Lipschitz (as we have seen above in the  proof  of Corollary~\ref{thm:ext}), which is all that is needed in order to obtain the aforementioned  $O(\sqrt{\log k})$-factor approximation guarantee. Again, the standard deduction of this assertion is worked out in~\cite{ALN08}, though note that~\cite[Section~5]{ALN08} incorporates a step that is irrelevant for our purposes because the embedding of~\cite{ALN08} is not   Fr\'echet, so~\cite{ALN08} must justify why it could be extended.

It is worthwhile to summarize the above discussion as the following separate algorithmic statement:

\begin{theorem}[algorithm for finding an  approximate Sparsest Cut when there are $k$ demand pairs]\label{thm:k demands} There exists a polynomial time algorithm with the following property. Suppose that $n\in \N$ and $k\in \{2,\ldots,n\}$.  Let $\sfC=(c_{ij}),\sfD=(d_{ij})\in \M_n([0,\infty))$ be symmetric matrices such that $|\{(i,j)\in [n]\times [n]: d_{ij}>0\}|\le k$. Then, the aforementioned algorithm outputs a subset $\emptyset\neq S\subsetneq [n]$ that satisfies:
\begin{equation}\label{eq:k demand pairs}
\frac{\sum_{i\in S}\sum_{j\in [n]\setminus S}c_{ij}}{\sum_{i\in S}\sum_{j\in [n]\setminus S}d_{ij}}\lesssim \mathrm{SparsestCut}(\sfC,\sfD)\sqrt{\log k}.
\end{equation}
\end{theorem}

\medskip
The next  algorithmic application of Theorem~\ref{thm:random zero}  that we will describe in this section follows from the work of Makarychev, Makarychev and  Vijayaraghavan~\cite{MMV14} which beautifully relates the Sparsest Cut problem to  perturbation resilience of the Max-Cut problem (in the sense of Bilu--Linial~\cite{BL12}; see also the MaxCut-specific work~\cite{BDLS13}, and the survey~\cite{MM16}).

The Max-Cut problem on $n$-vertices takes as its input an $n$-by-$n$ symmetric matrix with nonnegative entries $\sfW=(w_{ij})\in \M_n([0,\infty))$, which one should think of as an edge-weighted graph whose vertex set is $[n]$,  and aims to compute (or approximate) in polynomial time the following quantity:
\begin{equation}\label{eq:def max cut}
\mathrm{MaxCut}(\sfW)\eqdef \max_{\emptyset \neq S\subsetneq [n]}\sum_{i\in S}\sum_{j\in [n]\setminus S} w_{ij}.
\end{equation}

Given $\gamma\ge 1$, a matrix $\sfW$ as above is   Bilu--Linial $\gamma$-stable ($\gamma$-stable, in short) for  Max-Cut if there is a unique  $\emptyset \neq S\subsetneq [n]$ such that $\mathrm{MaxCut}(\sfW)=\sum_{i\in S}\sum_{j\in [n]\setminus S} w_{ij}$, namely, the right hand side of~\eqref{eq:def max cut} has only one maximizer, and furthermore we have $\mathrm{MaxCut}(\sfW')=\sum_{i\in S}\sum_{j\in [n]\setminus S} w_{ij}'$ for every symmetric matrix $\sfW'=(w_{ij}')\in \M_n([0,\infty))$ that satisfies $w_{ij}\le w_{ij}'\le \gamma w_{ij}$ for every $i,j\in [n]$.

Makarychev, Makarychev and  Vijayaraghavan studied~\cite{MMV14} an SDP that can be evaluated (with $o(1)$ precision) in polynomial time and outputs a number $\SDP_{\MMV}(\sfW)\ge 0$ that satisfies
\begin{equation}\label{eq:max cut is a relaxation}
\SDP_{\MMV}(\sfW)\ge  \mathrm{MaxCut}(\sfW).
\end{equation}
The  Makarychev--Makarychev--Vijayaraghavan SDP for Max-Cut coincides with the Goemans--Williamson SDP for Max-Cut~\cite{GL94}  with the (by now standard) added squared-$\ell_2$ triangle inequality constraints, except for the ``twist'' that those constraints are imposed on the symmetrized version of the vector solution, namely, if the SDP outputs unit vectors $v_1,\ldots,v_n\in S^{n-1}$, then  Makarychev, Makarychev and Vijayaraghavan require that $\|u-v\|_2^2\le \|u-w\|_2^2+\|w-v\|_2^2$ for every $u,v,w\in \{v_1,\ldots,v_n,-v_1,\ldots,-v_n\}$.\footnote{See~\cite{Kar05,ACMM05} for earlier incarnations of this idea, along with demonstrations of its utility in other aspects of combinatorial optimization.}

The  Makarychev--Makarychev--Vijayaraghavan SDP for Max-Cut is said to be integral on the input $\sfW$ (see~\cite[Definition~2.3]{MMV14}) if for every optimal solution $v_1,\ldots,v_n\in S^{n-1}$ of this SDP there is a unit vector $e\in S^{n-1}$ such that $v_i\in \{e,-e\}$ for every $i\in [n]$. By~\eqref{eq:max cut is a relaxation}, $\SDP_{\MMV}(\sfW)=  \mathrm{MaxCut}(\sfW)$ when this occurs, i.e., the Makarychev--Makarychev--Vijayaraghavan algorithm outputs an exact solution for Max-Cut on the input $\sfW$, in contrast to the fact that (under standard complexity-theoretic assumptions)  one cannot hope to always obtain such an exact solution on general inputs~\cite{PY91,Has01,KKMO07,MOO10}.

By substituting Theorem~\ref{thm:negl1} into~\cite[Theorem~3.1]{MMV14}   and~\cite[Theorem~5.2]{MMV14}, we get the following evaluation of the growth rate of the critical $\gamma=\gamma(n)$ such that the Makarychev--Makarychev--Vijayaraghavan SDP for Max-Cut is integral on every $n$-vertex instance of Max-Cut that is $\gamma$-stable:

\begin{theorem}[stable integrality of  the Makarychev--Makarychev--Vijayaraghavan SDP for Max-Cut]\label{thm:MMV} There are universal constants $C>c>0$ with the following property for any integer $n\ge 2$. If $\gamma\ge C\sqrt{\log n}$, then the Makarychev--Makarychev--Vijayaraghavan SDP relaxation of Max-Cut is integral on any $n$-vertex input that is $\gamma$-stable for Max-Cut. If $1\le \gamma\le c\sqrt{\log n}$, then there is an $n$-vertex input that is $\gamma$-stable for Max-Cut on which the Makarychev--Makarychev--Vijayaraghavan SDP relaxation of Max-Cut is not integral.
\end{theorem}

\medskip

We will next combine Corollary~\ref{thm:ext}  with~\cite[Theorem~5.1]{MM16} to get an asymptotic improvement of~\cite[Theorem~5.11]{MM16}, which is an intriguing link between the (major) classical problem on Lipschitz extension that we recalled in Question~\ref{Q:ball}, and the vertex sparsification paradigm of Moitra~\cite{Moi09}.

Fix $n\in \N$ and $Q\ge 1$. Let $\sfA=(\alpha_{ij})\in \M_n([0,\infty))$ be a symmetric matrix (which, again, we think of as an edge-weighted graph whose vertex set is $[n]$). Following~\cite{Moi09}, given $U\subset [n]$ and a symmetric matrix
$$\sfB=(\beta_{ij})_{(i,j)\in U\times U}\in M_{U\times U}\big([0,\infty)\big),
$$
one says that $\sfB$ is a $Q$-quality vertex cut sparsifier of $\sfA$ if the following holds for every $\emptyset \neq S\subsetneq U$:
\begin{equation}\label{eq:moitra}
\min_{\substack{T\subset [n]\\ T\cap U=S}}\sum_{i\in T}\sum_{j\in [n]\setminus T}\alpha_{ij}\le\sum_{p\in S}\sum_{q\in U\setminus S} \beta_{pq}\le Q\min_{\substack{T\subset [n]\\ T\cap U=S}}\sum_{i\in T}\sum_{j\in [n]\setminus T}\alpha_{ij}.
\end{equation}

The combinatorial meaning of~\eqref{eq:moitra} is that if the weighted graph $\sfB$ on the (potentially much) smaller  subset $U$ of the vertices $[n]$ is a $Q$-quality vertex cut sparsifier of a weighted graph $\sfA$, then the size with respect to $\sfB$  of the edge-boundary  of any partition of $U$ into two sets $S$ and $U\setminus S$ is up to the specified error tolerance $Q$ the minimal size with respect to $\sfA$ of the edge-boundary of a bipartition of $[n]$ that separates $S$ and $U\setminus S$.  If $|U|$ and $Q$ are small, then this is a valuable tool for reducing the dependence on $n$ in combinatorial optimization problems; see~\cite{Moi09,LM10} for concrete examples of such applications.

For $k\in \N$, let $Q_k^{\mathrm{cut}}$ denote the infimum over all those $Q\ge 1$ such that for every $n\in \N$, for every symmetric matrix $\sfA=(\alpha_{ij})\in \M_n([0,\infty))$ and every $U\subset [n]$ with $|U|= k$ there exists a $Q$-quality cut sparsifier $\sfB=(\beta_{ij})_{(i,j)\in U\times U}\in M_{U\times U}([0,\infty))$. Moitra proved in~\cite{Moi09} that
\begin{equation}\label{eq:moitra upper}
Q_k^{\mathrm{cut}}\lesssim \frac{\log k}{\log \log k},
\end{equation}
which remains the best known upper bound despite substantial efforts~\cite{CLLM10,EGKRT-CT10,MM16}. Using the notation for Lipschitz extension moduli we recalled above, Theorem~5.1 in~\cite{MM16} states that
\begin{equation}\label{eq:MM identity}
Q_k^{\mathrm{cut}}=\ee(\ell_1;\ell_1).
\end{equation}
 By~\eqref{eq:MM identity}, Theorem~D.2 in~\cite{MM16} (which is due to Johnson and Schechtman, applying~\cite{FJS88}) gives
 \begin{equation}\label{eq:Qk lower}
 Q_k^{\mathrm{cut}}\gtrsim \frac{\sqrt{\log k}}{\log\log k},
  \end{equation}
  which is currently the best known lower bound on  $Q_k^{\mathrm{cut}}$.

  By substituting the identity~\eqref{eq:MM identity} into Corollary~\ref{thm:ext} we conclude that
\begin{equation}\label{eq:eq:reduce to ball}
Q^{\mathrm{cut}}_k\lesssim \ee_k\big(\ell_2^{\lceil\log k\rceil};\ell_1\big)\sqrt{\log k}.
\end{equation}
Thus, a lower bound on $Q^{\mathrm{cut}}_k$ that grows faster than a positive universal constant multiple of $\sqrt{\log k}$ would imply a negative answer to Question~\ref{Q:ball}; this would hold in particular if the old upper bound~\eqref{eq:moitra upper} were sharp.  Conversely, if Question~\ref{Q:ball} had a positive answer, then by~\eqref{eq:eq:reduce to ball} we would get that $Q_k^{\mathrm{cut}}\lesssim \sqrt{\log k}$, i.e., we would then have  an upper bound on $Q_k^{\mathrm{cut}}$ that almost matches the lower bound~\eqref{eq:Qk lower}.

\begin{remark}
 The best known upper bound~\cite{LN05} on $\ee(\ell_2^n,\ell_1)$, in fact, on   $\ee(\ell_2^n; \bZ)$ for any Banach space $\bZ$, is  $\ee(\ell_2^n;\bZ)\lesssim \sqrt{n}$. The best known lower bound is  $\ee(\ell_2^n;\bZ)\gtrsim \sqrt[4]{n}$ for a suitably chosen Banach space $\bZ$. This is due to~\cite{MN13}, building on Kalton's important contributions~\cite{Kal04,Kal12}; a different proof was found in~\cite{Nao21-ext}. In both cases the target space $\bZ$ is not $\ell_1$ and it is conceivable that an upper bound on $\ee(\ell_2^n;\ell_1)$ that is better than $O(\sqrt{n})$ is  available (indeed, Question~\ref{Q:ball}  might have a positive answer). The forthcoming work~\cite{BN24} can be viewed as some evidence that perhaps $\ee(\ell_2^n;\bZ)\lesssim \sqrt[4]{n}$ for every Banach space $\bZ$. If this were true, then its special case $\bZ=\ell_1$ together with~\eqref{eq:eq:reduce to ball} would imply that  $Q^{\mathrm{cut}}_k\lesssim (\log k)^{3/4}$.
\end{remark}

\medskip

Our final example of an application of Theorem~\ref{thm:random zero} (through its corollaries in Section~\ref{sec:math applications}) is to the Approximate Nearest Neighbor Search problem, which is of importance to both theory and practice; the relatively recent survey~\cite{AIR18} and the references therein are a good entry point to this active area.

Let $(\MM,d)$ be a metric space, $n\in \N$ and $D>1$. The goal of the  $D$-approximate nearest neighbor search problem is to preprocess $\sub$ so as to obtain a data structure with the following property. There is an algorithm that takes as input a query  point $x\in \MM\setminus \sub$, examines only the data structure,  and outputs a point $y\in \sub$ that is guaranteed to satisfy $d(x,y)\le Dd(x,\sub)$. Both the preprocessing and the query algorithm  are allowed to use randomness, in which case the aforementioned approximation guarantee has to hold with probability that is at least a positive universal constant (say, $2/3$). The pertinent parameters which one hopes to make small are the preprocessing time (how long it takes to construct the data structure), the size of the data structure (how much space is required to store it), and the query time (given the query point $x$, how long it takes the algorithm to output its approximate nearest neighbor $y$ in $\sub$).

The above formulation sets aside computational issues regarding how the metric space is given; these are important but secondary in the present context. It suffices to consider the metric in the black-box model, i.e.,  through  oracle calls to the distance function. In concrete examples (e.g.~if there is an extrinsic  coordinate structure) more realistic computational models  for the input metric are available, but those are case-specific while herein we wish to focus on more general principles.

By~\cite{HIM12}, the approximate nearest neighbor search problem can be reduced to solving $(\log n)^{O(1)}$ instances of its ``decision version,''  called the near neighbor search problem. Given $D>1$ and $r>0$, the $(D,r)$-near neighbor search problem aims to (randomly) preprocess the data set $\sub$ to build a data structure for which there is a (randomized) algorithm that takes as input a query point $x\in \MM\setminus \sub$, examines only the data structure  and outputs a point $y\in \sub$  such that if $d(x,\sub)\le r$, then with $\Omega(1)$ probability we have $d(x,y)\le Dr$. Due to this reduction, we will focus from now on near neighbor search.

The following result is a substitution of  Theorem~\ref{thm:thm:cutting mudulus} into~\cite[Theorem~4.1]{ANNRW18}:
\begin{theorem}[cell-probe data structure for near neighbor on quasisymmetrically Hilbertian spaces]\label{thm:cell probe alg} Suppose that $0<s,\e\le 1/2$ and that $K,n,w\ge 2$  are integers. Let $(\MM,d)$ be a finite metric space whose size satisfies  $\log |\MM|\lesssim \log \log n$ that is both $K$-doubling and $(s,\e)$-quasisymmetrically Hilbertian.  Fix  $\sub \subset \MM$ with $|\sub|=n$. Also, fix $r>0$ that can be specified using $w$ bits. For every $(\log \log n)/\log n\le \d\le 1$ one can randomly preprocess this datum (we make no claim about the preprocessing time)  and eventually store some memory as a (random) data structure consisting of a sequence of cells, each of which contains a single bit. The total number of cells (the space of the data structure) is  $O(w+n^{1+\d}\log |\MM|)$.  There is a randomized algorithm that takes as input a query point $x\in \MM$, probes at most $O(w+n^{\d}\log |\MM|)$  cells  of the data structure (possibly adaptively),  performs unbounded auxiliary computations and uses unbounded auxiliary memory, and outputs a point $y\in \sub$ such that  if $d(x,\sub)\le r$, then with probability $\Omega(1)$ we have
\begin{equation}\label{eq:near neighbor}
d(x,y)\lesssim_{s,\e} \frac{\sqrt{\log K}}{\d}r.
\end{equation}
\end{theorem}

Background on the  cell-probe model for data structures can be found in e.g.~\cite{Yao81,Mil99}. Theorem~\ref{thm:cell probe alg} plainly does not address all of the algorithmic issues that we mentioned above. It does not control the preprocessing time for constructing the data structure, and it only bounds   the number of cells  of the data structure that the algorithm examines without controlling its overall running time. As in~\cite{ANNRW18}, the value of Theorem~\ref{thm:cell probe alg} is that it serves as a ``proof of concept'' rather than a final result, or at least it can be thought of as
a barrier for proving impossibility of efficient approximate near neighbor  data structures with the parameters described in Theorem~\ref{thm:cell probe alg}. This is so because  all of the known unconditional data structure lower bounds proceed by proving a cell-probe lower bound~\cite{Mil99}.

Theorem~\ref{thm:cell probe alg} raises the challenge to make the iterative partitioning scheme of~\cite{ANNRW18} that is powered by the cutting modulus estimate of Theorem~\ref{thm:thm:cutting mudulus} computationally efficient, which would entail making our construction in the proof of Theorem~\ref{thm:random zero} computationally efficient. We expect that this goal is achievable (perhaps only for a subclass of quasisymmetrically Hilbertian metric spaces that are represented in a suitably succinct manner), but would require effort and new ideas, analogously to the transition  from~\cite{ANNRW18} to~\cite{ANNRW18-2}; we defer this interesting research direction  to future work.

By combining ~\cite[Theorem~4.1]{ANNRW18} with the facts that were recalled immediately after the statement of Theorem~\ref{thm:gap estimate}, we see that for arbitrary  $K$-doubling metric spaces  Theorem~\ref{thm:cell probe alg} holds mutatis mutandis, but with the factor $\sqrt{\log K}$ in~\eqref{eq:near neighbor} replaced by $\log K$. The available $D$-approximate near neighbor  data structure for general $K$-doubling metric spaces with approximation $D=O((\log K)^\theta)$ for some $0<\theta<1$ suffers from the ``curse of dimensionality'' (the intrinsic dimension in this context is naturally of order $\log K$), in the sense that its query time grows super-polynomially in  $\log K$. Specifically, the query time obtained in~\cite{Fil23} when $D=C(\log K)^\theta$ is at least $\exp(\Omega((\log K)^{1-\theta})/C)$, albeit in a tailor-made (nonstandard) computational model (when $D=1+\e$ for some $0<\e<1$, which is a very important range of parameters but not the one that concerns us here, the same issue occurs in the older works~\cite{KL05,HM06} that use a standard computational model). If the given metric space is also $(s,\e)$-quasisymmetrically Hilbertian and $n=\exp( o_{s,\e}(\sqrt{\log K}))$, then Theorem~\ref{thm:cell probe alg}  applies with $D\lesssim_{s,\e}\sqrt{\log K}/\d$ and  its query time is better than that of~\cite{Fil23}.

If the answer to Question~\ref{Q:main}  were such that the $\sqrt{\log K}$ term in~\eqref{eq:nonlinear gap quasi}  could be improved to $o(\sqrt{\log K})$, then we would correspondingly improve the dependence on $K$ in~\eqref{eq:near neighbor}. This is an intriguing possibility of potential practical significance in addition to its great theoretical interest. As an example of a consequence of this possiblity that does not relate to~\eqref{eq:nonlinear gap quasi}, we note that because for every $n\in \N$  any $n$-point metric space is (trivially) $n$-doubling, in combination with~\cite[Theorem~1.3]{Nao14}  this would imply in particular that  any $n$-point metric space of negative type embeds with average distortion $o(\sqrt{\log n})$ into $\ell_1$, thus improving over the main embedding result of~\cite{arora2009expander} and consequently yielding an asymptotic improvement over the best-known polynomial time approximation algorithm for the  Sparsest Cut problem with uniform capacities and demands, i.e., the algorithmic optimization problem that corresponds to the special case of~\eqref{eq:sparsest cut def} in which $c_{ij}\in \{0,1\}$ and $d_{ij}=1$ for every $i,j\in [n]$.

\section{History and proof overview}\label{sec:history} 

\noindent The precursor  to Theorem~\ref{thm:random zero}  appeared as Theorem~3.1 of~\cite{ALN08} (first announced in~\cite{ALN05-STOC}); using the terminology that we recalled as Definition~\ref{def:random zero},  it states that any $n$-point metric space $(\MM,d)$ of negative type admits a  random zero set that is $O(\sqrt{\log n})$-spreading with probability $\Omega(1)$, i.e., for some universal constant $c>0$ and any  $\tau>0$ there exists a probability measure $\prob=\prob^{\tau,\MM}$ on $2^\MM\setminus\{\emptyset\}$  such that
\begin{equation}\label{eq:sqrt log zero set-history section}
\forall x,y\in \MM, \qquad d(x,y)\ge \tau\implies \prob\bigg[ \emptyset \neq  \cZ\subset \MM:\ d(y,\cZ) \ge \frac{c\tau}{\sqrt{\log n}}  \quad  \mathrm{and}\  \quad x\in \cZ  \bigg]\gtrsim 1.
\end{equation}
The main result of~\cite{ALN05-STOC,ALN08} used~\eqref{eq:sqrt log zero set-history section} to prove that $\cc_2(\MM)\lesssim \sqrt{\log n}\log\log n$, and~\cite{ALN07} showed that the same bound on the Euclidean distortion of $\MM$ can be obtained via the Fr\'echet embedding. 

Evidently, conclusion~\eqref{eq:in main thm} of Theorem~\ref{thm:random zero} is  quantitatively stronger than~\eqref{eq:sqrt log zero set-history section}. However, its main contribution is   a qualitative   enhancement that allows one to exploit cancelation to derive the upper bounds\footnote{Recall that the new content of these theorems is their upper bounds, namely their positive embedding/algorithmic assertions; the matching impossibility results are  due to~\cite{Enf69,naor2018vertical}.} of Theorem~\ref{thm:l1}, Theorem~\ref{thm:neg}, Theorem~\ref{thm:negl1} and Theorem~\ref{thm:sparsest} without any unbounded lower order factors, thus resolving the respective open questions; the previously best known bounds here~\cite{ALN08} had a redundant $\log \log n$ factor, and correspondingly conclusion~\eqref{eq:k demand pairs} of Theorem~\ref{thm:k demands} was  known with the factor $\sqrt{\log k}$ replaced by $\sqrt{\log k}\log\log k$. The connection between the Sparsest Cut problem and perturbation resilience of the MaxCut problem was discovered in~\cite{MMV14}, where Theorem~\ref{thm:MMV} was obtained with the lower bound $\gamma\ge C\sqrt{\log n}$ replaced by $\gamma\ge C\sqrt{\log n}\log\log n$, as~\cite{MMV14} appealed to~\cite{ALN08}; the impact of Theorem~\ref{thm:random zero} in this context is to obtain the sharp rate of growth of the threshold for stable integrality.\footnote{Up to universal constant factors; perhaps here one could hope to discover the exact threshold. Obtaining sharp constants in Theorem~\ref{thm:quasisymmetric John}, Theorem~\ref{thm:l1}, Theorem~\ref{thm:neg}, Theorem~\ref{thm:negl1} and Theorem~\ref{thm:sparsest} seems very far beyond reach, likely an unrealistic goal.} The connection between Lipschitz extension and vertex sparsification was discovered in~\cite{MM16}, where the bound~\eqref{eq:eq:reduce to ball} was obtained with the term $\sqrt{\log k}$ replaced by~$\sqrt{\log k}\log\log k$, as~\cite{MM16} appealed to~\cite{ALN07}; the above derivation of Corollary~\ref{thm:ext} follows the approach of~\cite{MM16}. 

The phenomenon that Theorem~\ref{thm:observable} establishes was first broached in~\cite{naor2005quasisymmetric}. While~\cite{naor2005quasisymmetric} did not succeed to obtain the desired extremal property of the Euclidean sphere that is independent of dimension, Theorem~1.7 of~\cite{naor2005quasisymmetric} is a result towards this goal that is off by a logarithmic factor. Specifically, \cite[Theorem~1.7]{naor2005quasisymmetric} shows that for any increasing $\eta:[0,\infty)\to [0,\infty)$ with $\lim_{s\to 0}\eta(s)=0$ and any $n\in \N$, if  $(\MM,d)$ is a metric space of diameter $O(1)$ that has a modulus-$\eta$ quasisymmetric embedding into $\ell_2$ and $\mu$ is a normalized $e^{O(n)}$-doubling Borel probability measure on $\MM$, then there exists $0<\theta_\eta<1$ such that   
\begin{equation}\label{eq:sub optimal observable extremality}
\mathrm{ObsDiam}_\mu^{\MM}(\theta_\eta)\gtrsim_{\eta} \frac{1}{\sqrt{\log n}}\mathrm{ObsDiam}_{\sigma^{n-1}}^{S^{n-1}}(\theta_\eta).
\end{equation}
The other applications to doubling metric spaces that are obtained herein did not appear in the literature even up to lower-order factors, though for some of those applications we expect that one could suitably adapt the reasoning of~\cite{naor2005quasisymmetric,ALN08} to derive similarly non-sharp statements.

The utility of Theorem~\ref{thm:random zero} for removing altogether  any unbounded lower order factor from~\cite{ALN08} is obvious, as it is nothing more than a direct substitution of~\eqref{eq:in main thm} into   measured descent~\cite{KLMN05} (Theorem~\ref{thm:quote descent}  herein). This possibility is discussed in~\cite{ALN05-STOC}, though with skepticism that the statement of Theorem~\ref{thm:random zero}  could be valid due to the nonlocal nature of the (deep) ARV rounding algorithm~\cite{arora2009expander}, which is the central input to~\eqref{eq:sqrt log zero set-history section}. In contrast, a key feature of~\eqref{eq:in main thm} is that the ``performance'' of the random zero set at the given scale $\tau$ for each given pair of points $x,y\in \MM$  depends only the local ``snapshot'' $B(y,\kappa\beta \tau)$ of $\MM$ near $y$ at scale $O_\eta(\tau)$, and moreover that it depends only on the extent to which the measure of $B(y,\kappa\beta \tau)$ increases relative to the measure of the proportionally smaller snapshot $B(y,\beta \tau)$ of $\MM$. 

In absence of such locality, \cite{ALN08} starts out with~\eqref{eq:sqrt log zero set-history section}, from which point it does not make any further appeal to~\cite{arora2009expander} and instead it proceeds by  adapting measured descent  in order to obtain the aforementioned distortion bound, i.e., while incurring a lower order yet unbounded multiplicative loss. The  reasoning that is used in~\cite{ALN07} to show   that this can, in fact, be  achieved via the Fr\'echet embedding also takes~\eqref{eq:sqrt log zero set-history section} as a ``black box'' without any further appeal to~\cite{arora2009expander} and proceeds by yet another enhancement of measured descent. 

We do not see how to derive the aforementioned sharp embedding results  via the above route, which is purely metric/analytic in contrast to the more structural nature of ARV.    Instead, the present work ``flips'' the approach of~\cite{ALN08} by leaving measured decent untouched (it can now be simply quoted), and enhancing, as we will next outline, the structural insights that are provided by the ARV framework. 

After the appearance  of the ARV algorithm in~\cite{ARV04},  simplifications, refinements, extensions, and reformulations of it were developed in multiple works, including notably its full journal version~\cite{arora2009expander} and~\cite{CGR05,lee2006distance,naor2005quasisymmetric,ACMM05,arora2011reformulation,rothvoss2016lecture}. All of those contributions were valuable to us in the process of  developing the ensuing  enhancement of ARV.  Rothvoss' lecture notes~\cite{rothvoss2016lecture} stand out here since they creatively redo the setup and reasoning in a natural, and, as it turns out, more flexible way. In particular, we introduce Definition~\ref{def:compatibility} below of compatibility of a labelled graph with  a mapping into $\R^n$, which arose from our efforts to understand the extent to which the proof in~\cite{rothvoss2016lecture}  can be strengthened.  

\subsection{Directional Euclidean sparsification of graphs}
The above referenced versions of the ARV algorithm study (explicitly or implicitly) a natural way to ``sparsify'' certain (combinatorial) graphs, i.e., a procedure that removes in a meaningful way some of the edges of a given graph. Our proof of Theorem~\ref{thm:random zero} investigates  a more involved version of this procedure, applied to  graphs that encode more geometric information than the proximity graphs that were used in those works. This section is devoted to explaining these ideas in suitable generality.

The aforementioned literature considers a finite negative type metric space $(\MM,d)$   and a scale $\tau>0$, and studies the proximity graph $\sfG=(\MM,E)$ whose vertex set is $\MM$ and $\{x,y\}\in E$ if and only if  $d(x,y)\le \Delta$, where $\Delta=c\tau/\sqrt{\log |\MM|}$ for   some universal constant $c>0$. In the present setting, we are led to consider a certain  graph  (specified below) whose vertex set is still $\MM$, yet if a pair of points $x,y\in \MM$ are joined by an edge, then this  will encode information  that combines both their proximity and the local growth rate near $x,y$ of a given measure $\mu$ on $\MM$.\footnote{In fact, we will need to analyse separately the connected components of such a graph, but we will initially suppress that subtlety for the purpose of this overview.} The sparsification procedure  in our context will involve a pairwise thresholding criterion (we will soon explain  what this means)  that is nonconstant, while previous works considered a fixed threshold that is independent of the given pair of points in $\MM$.

Because the ensuing reasoning needs to consider multiple metrics simultaneously, including multiple metrics  on the same set (arising from both the original metric and the shortest-path metrics of graphs as above), it will be beneficial to use subscripts when denoting distances, balls, diameters.\footnote{In parts that treat only a single metric, such as  Section~\ref{sec:descent revisited} and some of the Introduction, we will drop this convention.} Thus, given a metric space $(\MM,d_\MM)$, we will write $\diam_\MM(A)=\sup_{a,b\in A} d_\MM(a,b)$ and $d_\MM(x, A)=\inf_{a\in A} d_\MM(x,a)$ for, respectively, the $d_\MM$-diameter of $\emptyset \neq A\subset \MM$ and the $d_\MM$-distance of  $x\in \MM$ from $A$. We will also write $B_\MM(x,r)=\{y\in  \MM:\ d_\MM(x,y)\le r\}$ for the $d_\MM$-ball centered at $x$ of radius $r\ge 0$.  Thus, given $n\in \N$, $p\ge 1$ and $x\in \R^n$, we will use the notation $B_{\ell_p^n}(x,r)=\{y\in \R^n:\ \|x-y\|_p\le r\}$, where $\|z\|_p=(|z_1|^p+\ldots+|z_n|^p)^{1/p}$ for $z=(z_1,\ldots,z_n)\in \R^n$.  The the standard scalar product on $\R^n$ will be denoted $\langle \cdot,\cdot\rangle :\R^n\times \R^n\to \R$, i.e., $\langle z,w\rangle=z_1w_1+\ldots +z_nw_n$ for $z=(z_1,\ldots,z_n),w=(w_1,\ldots,w_n)\in \R^n$. The   standard Gaussian measure on $\R^n$ will be denoted (as usual) by $\gamma_n$, i.e., the  density of $\gamma_n$ at $z\in \R^n$ is proportional to $\exp(-\|z\|_2^2/2)$.

Throughout what follows, all graphs will be tacitly assumed to be finite and will be allowed to have self-loops. Given  a (possibly disconnected) graph $\sfG=(V,E)$, denote the shortest-path/geodesic (extended) metric that it induces on its vertex set $V$ by $d_\sfG:V\times V\to [0,\infty]$, under the natural convention that  $d_\sfG(x,y)=\infty$ if and only if $(x,y)\in \Gamma\times \Gamma'$ for distinct connected components $\Gamma,\Gamma'\subset V$  of $\sfG$. For $r\ge 0$ and $x\in V$, the corresponding ball  in $(V,d_\sfG)$ will be denoted $B_\sfG(x,r)=\{y\in V:\ d_\sfG(x,y)\le r\}$. In particular,  $B_\sfG(x,1)=\{x\}\cup\{y\in V:\ \{x,y\}\in E\}$, which is also  denoted $N_\sfG(x)$, is the neighborhood in $\sfG$ of the vertex $x$.  Correspondingly, the neighborhood in $\sfG$ of a vertex subset $U\subset V$ will be denoted $N_\sfG(U)$, i.e., 
$$
N_\sfG(U)\eqdef \bigcup_{x\in U} N_\sfG(x)=\big\{y\in V:\ d_\sfG(y,U)\le 1\big\}.
$$

When a graph $\sfG=(V,E)$ is accompanied by a mapping $f:V\to \R^n$, which we think of as a geometric representation of $\sfG$, and an edge-labelling $\sigma:E\to \R$, which we think of as a thresholding criterion for determining which edges will be retained in the ensuing sparsification, for each vector $v\in \R^n$ we can consider the sub-graph of $\sfG$ that is obtained by deleting those  $\{x,y\}\in E$ with $|\langle f(x)-f(y),v\rangle|\le 4\sigma(\{x,y\})$:

\begin{definition}[directional Euclidean sparsification]\label{def:directional sparsification} Let $\sfG=(V,E)$ be a graph. If $n\in \N$ and $f:V\to \R^n$, then for every $\sigma:E\to \R$ and $v\in \R^n$  define\footnote{We decided to insert the  factor $4$ in the definition~\eqref{eq:sparsification definition} of Euclidean sparsification for convenience only, as this yields a slight simplification of expressions in the ensuing reasoning. This choice is, of course,  nothing more than a superficial normalization which could be removed if so desired by replacing throughout what follows the thresholding function $\sigma$ by $\sigma/4$.} 
\begin{equation}\label{eq:sparsification definition}
E(v;f,\sigma)\eqdef \big\{\{x,y\}\in E:\ |\langle f(x)-f(y),v\rangle|>4\sigma(\{x,y\})\big\}\subset E. 
\end{equation}
We thus obtain the following  subgraph\footnote{Observe that even though graphs are allowed herein to have self-loops, if $\sigma$ takes values in $[0,\infty)$, then the strict inequality in~\eqref{eq:sparsification definition} implies that $\sfG(v;f,\sigma)$ will never have self-loops, i.e., $\{x,x\}\notin E(v;f,\sigma)$ for every $x\in V$.} of $\sfG$, which we call the Euclidean sparsification of $\sfG$ in direction $v$ corresponding to its Euclidean representation $f$ and the thresholding function $\sigma$:
\begin{equation}\label{eq:def euclidean sparsification}
\sfG(v;f,\sigma)\eqdef \big(V,E(v;f,\sigma)\big).
\end{equation}
\end{definition}

Understanding typical properties of  $\sfG(v;f,\sigma)$ when $v$ is chosen randomly according to the Gaussian measure $\gamma_n$ is interesting in its own right. Here, we will investigate its matching number,  which is also what~\cite{rothvoss2016lecture} studies (as do other ARV-related works, usually implicitly), in the case when $\sigma\equiv 1/2$ is constant and $\sfG=(\MM,E)$ is the above ``vanilla'' proximity graph on a finite metric space $(\MM,d_\MM)$ of negative type, i.e., $\{x,y\}\in E$ if and only if  $d(x,y)\le \Delta$ for some fixed  $\Delta>0$. Recall that the matching number $\nu(\sfG)$ of a graph $\sfG=(V,E)$ is the maximum cardinality of a pairwise-disjoint collection of its edges.  We will need only rudimentary properties of this basic combinatorial notion (that will be recalled when they will arise in proofs), which are covered in e.g.~\cite{LP09}. 

Since both the graphs $\sfG= (V,E)$ that we will investigate herein, and their labelings $\sigma:E\to \R$ that we will use as thresholds for Euclidean sparsification, are more complicated than the aforementioned special case, it is beneficial to first study the above setting abstractly. We arrived at the following definition, that will have an important role below, by examining  the  elegant  proof in~\cite{rothvoss2016lecture}  with this goal in mind:

\begin{definition}[compatibility of a graph with its Euclidean realization and edge labeling]\label{def:compatibility} Fix $n\in \N$ and $C>0$. Let $\sfG=(V,E)$ be a graph, $f:V\to \R^n$ and $\sigma:E\to [0,\infty)$. We say that $\sfG$ is $C$-compatible with $f$ and $\sigma$ if there exist $\Delta:V\to [0,\infty)$ and $K:V\to \N$ that have the following three properties: 
\begin{enumerate}
\item For every $x\in V$ and $y\in B_\sfG(x,K(x)-1)$, if $z\in V$ is such that $\{y,z\}\in E$, then 
\begin{equation}\label{eq:4 Delta compatibility}
\Delta(x)\le \sigma(\{y,z\}).
\end{equation}
\item For every $x\in V$ and every $y\in N_\sfG(x)$ we have 
\begin{equation}\label{eq:second compatibility condition}
\int_{\R^n}\bigg(\max_{z\in B_\sfG\left(y,K(y)\right)}\langle f(z)-f(y),v\rangle\bigg)\ud\gamma_n(v)\le K(x)\Delta(y).  
\end{equation}
\item For every $x\in V$ we have
\begin{equation}\label{eq:inclusion condition in compatibility def}
f\Big(B_\sfG\big(x,K(x)\big)\Big)\subset B_{\ell_2^n}\Big(f(x),\frac{1}{C}\Delta(x)\Big). 
\end{equation}
\end{enumerate}
\end{definition}

Definition~\ref{def:compatibility}  is scale-invariant in the following sense: if $\sfG$ is $C$-compatible with $f$ and $\sigma$, then $\sfG$ is also $C$-compatible with $\lambda f$ and $\lambda \sigma$ for every $\lambda\ge 0$, as seen by considering  $\lambda \Delta$.  The only part of Definition~\ref{def:compatibility} that involves the parameter $C$ is the inclusion~\eqref{eq:inclusion condition in compatibility def}, which implies in particular that $C$-compatibility becomes a more stringent property as $C$ grows. The precise role of the properties that Definition~\ref{def:compatibility} requires from $K,\Delta$ will become apparent upon examining the details of how they are applied in the ensuing proofs. In (very) broad strokes, their significance is that $K,\Delta$ take as input a single vertex, which we think of as a consistent choice of ``local scales'' at that vertex (for, respectively, the domain and range of $f$),    yet they control  the pairwise interactions $\sigma$ through~\eqref{eq:4 Delta compatibility}, and the oscillations of $f$ through~\eqref{eq:second compatibility condition}; both of these controls occur on (combinatorial)  balls in $\sfG$ whose radius is determined by $K$.

Note the following feature of Definition~\ref{def:compatibility}: it stipulates the existence of $\Delta, K$ with the stated properties, but these are auxiliary objects that occur internally to the definition and are not part of the notion of $\sfG$ being $C$-compatible with $f$ and $\sigma$. Thus, any statement  about $C$-compatibility  will not refer to $K,\Delta$, i.e., they will only be used as tools within its proof. This is exemplified by the following theorem, which is central to our proof of Theorem~\ref{thm:random zero}. It asserts that $C$-compatibility for large $C>0$ implies that the expected matching number  $\nu(\sfG(v;f,\sigma))$  is small when $v$ is distributed according to $\gamma_n$. This result is a generalization of the ARV reasoning as it was recast by Rothvoss, and its proof (which appears in Section~\ref{sec:generalized chaining} below), is an adaptation of the strategy that he introduced in~\cite[Section~6.1]{rothvoss2016lecture}.

\begin{theorem}[expected matching number of Euclidean sparsification]\label{thm:matching from compatibility} Fix $C\ge 1$ and $n\in \N$. Suppose that $\sfG=(V,E)$ is a graph that is $C$-compatible with $f:V\to \R^n$ and $\sigma:E\to [0,\infty)$. Then,\footnote{As $V$ is finite, it is straightforward to check from the definition~\eqref{eq:sparsification definition} that the function $v\mapsto \nu(\sfG(v;f,\sigma))$ from $\R^n$ to $\N\cup\{0\}$ is Borel-measurable, i.e.,  $\{v\in \R^n:\ \nu(\sfG(v;f,\sigma))=m\}$ is a Borel subset of $\R^n$ for fixed $m\in \N\cup\{0\}$. So, the integral in~\eqref{expected fractional matching bound} is defined.} 
\begin{equation}\label{expected fractional matching bound}
\int_{\R^n}\nu\big(\sfG(v;f,\sigma)\big)\ud\gamma_n(v)< 6e^{-\frac{1}{4}C^2}|V|.
\end{equation}
\end{theorem}  

The proof of Theorem~\ref{thm:matching from compatibility} is the only part of our proof of Theorem~\ref{thm:random zero general}  that elaborates on the central innovation  of~\cite{arora2009expander}, namely, what is known today as the ``ARV chaining argument'' (through the approach to it that was devised in~\cite{rothvoss2016lecture}). We will suitably implement this argument in Section~\ref{sec:generalized chaining}  to prove Theorem~\ref{thm:matching from compatibility}. That implementation involves work that is of a more technical nature, and the  conceptually new contribution  here is the mere introduction of Definition~\ref{def:compatibility}, which makes Theorem~\ref{thm:matching from compatibility} possible.

\subsection{From Theorem~\ref{thm:matching from compatibility}  to Theorem~\ref{thm:random zero general}}

The rest of the proof of Theorem~\ref{thm:random zero general} does not relate to ARV chaining. We will next describe the steps that remain in the derivation of  Theorem~\ref{thm:random zero general}  from Theorem~\ref{thm:matching from compatibility}. Proposition~\ref{prob:graphical prob multiscale} below will be used for that purpose; its elementary proof appears in Section~\ref{sec:groundwork}. Like Theorem~\ref{thm:matching from compatibility}, Proposition~\ref{prob:graphical prob multiscale}  is about graph theory/probability: Notice that we have not yet used the assumption that $(\MM,d_\MM)$ is a metric space, or that it is  quasisymmetrically Hilbertian; these will occur only in subsequent stages of the reasoning.

\begin{prop}\label{prob:graphical prob multiscale} There is a universal constant $\kappa> 1$ with the following property. Fix $n\in \N$ and let $\sfG=(V,E)$ be a graph. Suppose that $\Lambda:V\to (0,\infty]$ is moderately varying\footnote{The constant factor $2$ in~\eqref{eq:moderate variation along edges} was chosen arbitrarily for simplicity as this suffices for the present purposes, but the proof herein of Proposition~\ref{prob:graphical prob multiscale}  shows that replacing it by any factor that is bigger than $1$ only impacts the value of $\kappa$ in~\eqref{eq:super gaussian omega}.  } along edges of $\sfG$ in the following sense:
\begin{equation}\label{eq:moderate variation along edges}
\forall \{x,y\}\in E,\qquad  \Lambda(y)\le 2\Lambda(x).
\end{equation} 
Assume also that we are given a mapping $f:V\to \R^n$  and a probability measure $\omega$  on $V\times V$ such that the following inequality holds for every $x,y\in V$ for which $(x,y)$ belongs to the support of $\omega$, i.e., $\omega(x,y)>0$, and also $x$ and $y$ belong to  same connected component of $\sfG$:
\begin{equation}\label{eq:min assumption in proposition}
\|f(x)-f(y)\|_2\ge \min\big\{\Lambda(x),\Lambda(y)\big\}.
\end{equation}
Then, for every $C\ge 1$ and every  $v\in \R^n$, there are $A(v)=A^{\omega,C}(v),B(v)=B^{\omega,C}(v)\subset V$ such that the mapping 
\begin{equation}\label{eq:which mapping is measurable}
(v\in \R^n)\mapsto \big(A(v),B(v)\big)\in 2^V\times 2^V 
\end{equation}
 is Borel-measurable,\footnote{Proposition~\ref{prob:graphical prob multiscale} must state the measurability of the mapping in~\eqref{eq:which mapping is measurable} in order to allow us to freely consider (at multiple junctures throughout the ensuing proofs) probabilities of various events, as well as expectations such as~\eqref{eq:super gaussian omega}. The classical  treatment of (Borel-)measurability of  set-valued mappings between topological spaces is covered in e.g.~the survey~\cite{Wag77}, but we do not need to recall it here: As $V$ is a finite set, the mapping in~\eqref{eq:which mapping is measurable} can take finitely many values, so its measurability in Proposition~\ref{prob:graphical prob multiscale} means that all of its fibers are Borel, i.e., for every $A,B\subset V$ the set $\{v\in \R^n:\ A(v)=A\ \mathrm{and}\ B(v)=B\}$ is a Borel subset of $\R^n$. This is also how one should interpret  below measurability statements about set-valued mappings that are defined on topological spaces that differ from $\R^n$; we will need to consider such situations   later in this article (all of the topological spaces that will occur below will be Polish, i.e., separable complete metric spaces). Measurability is thus a benign matter in the present finitary setting, but conceivably in the future one could need to consider variants for infinite spaces, in which case we expect that the nature of the concrete and simple constructions herein to easily lead to favorable measurability properties.} and the following point-wise estimate holds:
\begin{equation}\label{eq:separation within connected components}
\forall v\in \R^n,\ \forall (x,y)\in A(v)\times B(v),\qquad \{x,y\}\in E\implies |\langle v,f(x)-f(y)\rangle |>C \max\big\{\Lambda(x),\Lambda(y)\big\}.
\end{equation}
Furthermore, the $\gamma_n$-expected $\omega$-measure of the product 
$A(v)\times B(v)\subset V\times V$  satisfies 
\begin{equation}\label{eq:super gaussian omega}
\int_{\R^n} \omega\big(A(v)\times B(v)\big)\ud \gamma_n(v)\gtrsim e^{-\kappa C^2}.
\end{equation}
\end{prop}

The following  proposition is an especially important step in the proof of Theorem~\ref{thm:random zero general}:

\begin{prop}\label{thm:universally compatible compression} There is a universal constant $\zeta \ge 1$ with the following property.\footnote{An inspection of  the reasoning herein reveals that $\zeta=2$ would work for  Proposition~\ref{thm:universally compatible compression}, though this value is not sharp.} Let $(\MM,d_\MM)$ be a finite metric space, equipped with a nondegenerate measure $\mu$. For  $\tau,C>0$ define $\rho=\rho_{d_\MM,\mu,C,\tau}: \MM\to [1,\infty)$ by 
\begin{equation}\label{eq:our specific rho}
\forall x\in \MM,\qquad  \rho(x)\eqdef 1+\frac{\zeta}{ C} \sqrt{\log \frac{\mu (B_\MM(x,19\tau))}{\mu(B_\MM(x,\tau))}}.
\end{equation}
Then, there exists a mapping $q=q_{d_\MM,\mu,C,\tau}:\MM\to \MM$ satisfying 
\begin{equation}\label{eq:displacement condition per component}
\forall x\in \MM,\qquad d_\MM(q(x),x)\le 7\tau,
\end{equation}
with the following  properties. 

Let $\sfG=\sfG_{d_\MM,\mu,C,\tau}$ be the graph whose vertex set is $\MM$ and whose edge set $E=E_{d_\MM,\mu,C,\tau}$ is given by
\begin{equation}\label{eq:geometric edges specific rho}
\forall x,y\in \MM,\qquad \{x,y\}\in E \iff d_\MM(x,y)\le \frac{\tau}{\min\{\rho(x),\rho(y)\}}. 
\end{equation}
For every $x\in \MM$, let $\Gamma(x)\subset \MM$ denote the connected component of $x$ in $\sfG$. Then, $q(x)\in \Gamma(x)$ for every $x\in \MM$. Furthermore, for any $n\in \N$ and any mapping $\f:\MM\to \R^n$, the graph $\sfG$ is $C$-compatible with $\f\circ q:\MM\to \R^n$ and $\sigma:E\to [0,\infty)$, where  $\sigma=\sigma_{d_\MM,C,\f\circ q}:E\to [0,\infty)$ is defined by\footnote{This $\sigma$  is well-defined because if $\{x,y\}\in E$, then $x$ and $y$ belong to the same connected component of $\sfG$, i.e., $\Gamma(x)=\Gamma(y)=\Gamma$. Also, because $\rho\ge 1$ by~\eqref{eq:our specific rho},  definition~\eqref{eq:geometric edges specific rho} implies that if $\{x,y\}\in E$, then $d_\MM(x,y)\le \tau$, so the maximum in~\eqref{eq:def our sigma specific q} is over a nonempty subset of $\MM$ (e.g., one can take there $a=x\in \Gamma$ and $b=y\in \Gamma$).}
\begin{equation}\label{eq:def our sigma specific q}
\forall \{x,y\}\in E,\qquad \sigma(\{x,y\})\eqdef C\bigg(\max_{\substack{a\in B_\MM(x,2\tau)\cap B_\MM(y,2\tau)\cap \Gamma(x)\\ b\in B_\MM(a,2\tau)\cap \Gamma(x)}}\|\f\circ q(a)-\f\circ q(b)\|_2\bigg).
\end{equation}
\end{prop}

Before proceeding to the  rest of the steps of the proof of Theorem~\ref{thm:random zero general}, will next  discuss the significance of Proposition~\ref{thm:universally compatible compression}, whose proof, which appears in Section~\ref{sec:compression}, contains a key idea of the present work. 

Note that Proposition~\ref{thm:universally compatible compression} is the first time that the assumption that $(\MM,d_\MM)$ is a metric space is used   in the proof of Theorem~\ref{thm:random zero general}, but this proposition works for any metric space and we do not yet need to know that it is quasisymmetrically Hilbertian. Also, Proposition~\ref{thm:universally compatible compression} introduces  the type of graphs to which  the general combinatorial/probabilisitc statements of Theorem~\ref{thm:matching from compatibility}  and Proposition~\ref{prob:graphical prob multiscale} will be applied: their vertices are points in  $\MM$ and their edges are given by~\eqref{eq:geometric edges specific rho}.  The relevance of Proposition~\ref{thm:universally compatible compression} to Theorem~\ref{thm:matching from compatibility} is evident, as Proposition~\ref{thm:universally compatible compression} produces situations in which the $C$-compatibility assumption of Theorem~\ref{thm:matching from compatibility}  is satisfied. The link to  Proposition~\ref{prob:graphical prob multiscale}  will be made later, in a  subsequent step of the proof of Theorem~\ref{thm:random zero general}.

The crucial geometric contribution of  Proposition~\ref{thm:universally compatible compression} is constructing the mapping $q:\MM\to \MM$, which should be viewed as a way to ``compress'' a given metric space: its image $q(\MM)\subset \MM$ will typically be much smaller than  $\MM$, yet it will encode geometric properties of $\MM$ and $\mu$ that will be important for the purpose of working with ratios of measures of balls as in~\eqref{eq:our specific rho}, which is our main goal. In fact, Proposition~\ref{thm:universally compatible compression}  states that $q$ is a universally compatible compression scheme in the sense that $C$-compatibility arises upon composition with $q$ of any function whatsoever from $\MM$ to $\R^n$ (for a suitable choice of edge-labelling $\sigma$). 

Proposition~\ref{thm:universally compatible compression} treats the specific function $\rho$ in~\eqref{eq:our specific rho} because this  is what is needed  below, and also since this is what arises from  the use of the Gaussian measure in requirement~\eqref{eq:second compatibility condition} of Definition~\ref{def:compatibility}.  Nevertheless, the compression scheme in Section~\ref{sec:compression} works for any edge set as in~\eqref{eq:geometric edges specific rho} when $\rho:\MM\to [1,\infty)$ is arbitrary. The construction in  Section~\ref{sec:compression} is therefore more general than what is used for  Proposition~\ref{thm:universally compatible compression}, and hence it could be useful for other purposes (perhaps in settings to which appropriate non-Gaussian or non-Euclidean versions of compatibility are pertinent). The idea is to consider and suitably analyse hierarchically nested $(2\tau)$-nets in the level sets of $\rho$, arranged in increasing  order; see Construction~\ref{construction:compression}.

The fact that the function $q$ of Proposition~\ref{thm:universally compatible compression} preserves connected components of $\sfG$ will be used to apply Proposition~\ref{thm:universally compatible compression}  when $\f$ is the mapping from Definition~\ref{def:quasisym metric space} of what it means for $(\MM,d_\MM)$ to be $(s,\e)$-quasisymemtrically Hilbertian; recall~\eqref{eq:def quasi Hilbert}.   The point is that since by~\eqref{eq:geometric edges specific rho} we know that $\{x,y\}\in E$ implies that $d_\MM(x,y)\le \tau$, as $\rho\ge 1$ by~\eqref{eq:our specific rho},   any connected component $\Gamma$ of $\sfG$ is $\tau$-discretely path connected as a subset of $\MM$, i.e., for every $x,y\in \Gamma$  there exist $k\in \N$ and a discrete path $z_0=x,z_1,\ldots,z_k=y\in \Gamma$ joining $x$ to $y$ such that $d_\MM(z_i,z_{i-1})\le \tau$ for every $i\in [k]$. This will allow us  to iterate~\eqref{eq:def quasi Hilbert} within each connected component of $\sfG$, leading to the following proposition, which is the only  way that  the assumption that $(\MM,d_\MM)$ is quasisymmetrically Hilbertian will be used in the proof of Theorem~\ref{thm:random zero general}:

\begin{prop}\label{prop:use quasi to get good graph} Fix $0<\e,s\le \frac12$ and $r\ge 1$, and suppose that $\beta>0$ satisfies
\begin{equation}\label{eq:def our beta in prop}
\beta\le\beta(s,\e,r)\eqdef s^{\frac{3\log(8r)}{\e}}.
\end{equation}
For a finite $(s,\e)$-quasisymmetrically Hilbertian metric space  $(\MM,d_\MM)$,  a nondegenerate measure $\mu$  on $\MM$, and $\tau,C>0$, let $\sfG=(\MM,E)$ be the graph from Proposition~\eqref{thm:universally compatible compression} with $(C,\tau)$ replaced by $(rC,\beta\tau)$, i.e.,  
\begin{equation}\label{eq:def our E with beta}
\forall x,y\in \MM,\qquad \{x,y\}\in E \iff d_\MM(x,y)\le \frac{\beta\tau}{\min\{\rho(x),\rho(y)\}}, 
\end{equation}
where, with $\zeta\ge 1$ the universal constant from Proposition~\eqref{thm:universally compatible compression}, we denote 
\begin{equation}\label{eq:rewrite rho specific}
\forall x\in \MM,\qquad  \rho(x)\eqdef 1+\frac{\zeta}{r C} \sqrt{\log \frac{\mu (B_\MM(x,19\beta\tau))}{\mu(B_\MM(x,\beta\tau))}}.
\end{equation}
Then, there are $f:\MM\to \R^{|\MM|}$,  $\sigma:E\to [0,\infty)$, and $\Lambda:\MM\to (0,\infty]$ that have the following properties:
\begin{itemize}
\item $\sfG$ is $(rC)$-compatible with $f$ and $\sigma$;
\item If $x,y\in \MM$ belong to the same connected component of $\sfG$, then
\begin{equation}\label{eq:in same connected component}
d_\MM(x,y)\ge \tau\implies C\|f(x)-f(y)\|_2\ge \max\big\{\Lambda(x),\Lambda(y)\big\},
\end{equation}
and
\begin{equation}\label{eq:restate 1/r version}
\{x,y\}\in E\implies  \Lambda(y)\le 2\Lambda(x)\qquad \mathrm{and}\qquad 4\sigma(\{x,y\})\le \min\big\{\Lambda(x),\Lambda(y)\big\}.
\end{equation}
\end{itemize}
\end{prop}

The proof of Proposition~\ref{prop:use quasi to get good graph}  appears in Section~\ref{sec:structural}: it consists of an application of Proposition~\ref{thm:universally compatible compression}  to $\f$ as in~\eqref{eq:def quasi Hilbert}, together with studying some basic  implications of being $(s,\e)$-quasisymmetrically Hilbertian, in the spirit of (but not identical to) the foundational work in Section~2 of~\cite{TK80}. 

The upshot of  Proposition~\ref{prop:use quasi to get good graph} is that it provides the compatibility of the graph that we care about with $f$ and $\sigma$, which is the assumption of Theorem~\ref{thm:matching from compatibility}, but now we can ignore the specific choice of $\sigma$ from Proposition~\ref{thm:universally compatible compression} that appears  in~\eqref{eq:def our sigma specific q}, and instead   we know that $\sigma$ is controlled as in the second inequality in~\eqref{eq:restate 1/r version} by a function $\Lambda$ that satisfies the assumptions of Proposition~\ref{prob:graphical prob multiscale}. Thus,  Proposition~\ref{prop:use quasi to get good graph} will allow us to combine Theorem~\ref{thm:matching from compatibility} and Proposition~\ref{prob:graphical prob multiscale} to obtain the following theorem:

\begin{theorem}\label{thm:main separated piut together} There exists a universal constant $\alpha_0\ge 1$ with the following properties. Fix  $0<s,\e\le \frac12$ and $\alpha\ge \alpha_0$. Suppose  that $\beta>0$ satisfies
\begin{equation}\label{eq:beta with alpha constant}
\beta\le\beta_\alpha(s,\e)\eqdef s^{\frac{\alpha}{\e}}.
\end{equation}
Let $(\MM,d_\MM)$ be a finite   metric space that is $(s,\e)$-Hilbertian. Suppose that $0<\tau \le\diam(\MM)$ and that $\omega$ is a symmetric probability measure  on $\MM\times \MM$ whose support  is contained in  $\{(x,y)\in \MM\times \MM:\ d_\MM(x,y)\ge \tau\}$. In other words, $\omega(\MM\times \MM)=1$, for every $x,y\in \MM$ we have $\omega(x,y)=\omega(y,x)$, and $\omega(x,y)>0\implies d_\MM(x,y)\ge \tau$. Then, for every $C\ge 1$ and $v\in \R^{|\MM|}$ there exist nonempty subsets $A^*(v)=A^*_{\omega,C}(v), B^*(v)=B^*_{\omega,C}(v)$ of $\MM$ such that  the set-valued mapping $(v\in\R^n)\mapsto (A^*(v),B^*(v))$ is Borel-measurable, they satisfy 
\begin{equation}\label{eq:star separation actual distances}
\forall v\in \R^{|\MM|},\ \forall (x,y)\in A^*(v)\times B^*(v), \qquad d_\MM(x,y)> \frac{\beta\tau}{\min\{\rho(x),\rho(y)\}},  
\end{equation}
where $\rho:\MM\to [1,\infty)$ is defined by
\begin{equation}\label{eq:our rho with alpha const}
\forall x\in \MM,\qquad  \rho(x)\eqdef 1+\frac{1}{\alpha C} \sqrt{\log \frac{\mu (B_\MM(x,19\beta\tau))}{\mu(B_\MM(x,\beta\tau))}}, 
\end{equation}
and furthermore, if  $\kappa>1$ is the universal constant from Proposition~\ref{prob:graphical prob multiscale}, then 
\begin{equation}\label{eq:expected omega weight}
\int_{\R^n}\omega\big(A^*(v)\times B^*(v)\big)\ud\gamma_n(v)\gtrsim e^{-\kappa C^2}.
\end{equation}

\end{theorem}
  
The  conclusion of Theorem~\ref{thm:random zero general} about random zero sets is a simple formal consequence of the fact that  for every fixed probability measure $\omega$ as in Theorem~\ref{thm:main separated piut together}, there   exist random pairs of sets as in Theorem~\ref{thm:main separated piut together} that are separated per~\eqref{eq:star separation actual distances} and $\omega$-large per~\eqref{eq:expected omega weight}. The reason for this is mainly duality (minimax theorem), together with a simple scale gluing argument; the (short) details of this deduction appear in Section~\ref{sec:duality}. The idea to insert  into the ARV reasoning a ``weighting'' such as $\omega$ on pairs of points of $\MM$  is a key insight of~\cite{CGR05}, were it was introduced in order to prove that any $n$-point metric space of negative type embeds into $\ell_2$ with distortion  $O((\log n)^{3/4})$; this idea played the same role in~\cite{ALN05-STOC}, as well as herein.

The proof that Theorem~\ref{thm:main separated piut together} follows from Theorem~\ref{thm:matching from compatibility}, Proposition~\ref{prob:graphical prob multiscale}, and Proposition~\ref{prop:use quasi to get good graph}, which, as we explained above, is all that remains to complete the proof of  Theorem~\ref{thm:random zero general},  appears in Section~\ref{sec:use fractional}.  We will next conclude this section by  sketching in broad strokes the   reason why this works. 

Write $n=|\MM|$. Start by applying  Proposition~\ref{prop:use quasi to get good graph} with $r=\zeta\alpha$, where $\zeta$ is the universal constant in~\eqref{eq:rewrite rho specific}, thus  ensuring that~\eqref{eq:our rho with alpha const} coincides with~\eqref{eq:rewrite rho specific}. Henceforth in this sketch,  $\sfG=(\MM,E)$ will stand for  the graph from this application of Proposition~\ref{prop:use quasi to get good graph}, i.e., its edges are given by~\eqref{eq:def our E with beta} where $\rho$ is defined in~\eqref{eq:rewrite rho specific}. The first bullet point in the conclusion of Proposition~\ref{prop:use quasi to get good graph} makes it possible to apply Theorem~\ref{thm:matching from compatibility} to get that
\begin{equation}\label{expected fractional matching bound with r}
\int_{\R^n}\nu\big(\sfG(v;f,\sigma)\big)\ud\gamma_n(v)\lesssim e^{-\frac{r^2}{4}C^2}n=e^{-\frac{\zeta^2\alpha^2}{4}C^2}n.
\end{equation}

Think of the expectation estimate~\eqref{expected fractional matching bound with r} as expressing  the following structural information about the Euclidean sparsification  $\sfG(v;f,\sigma)$ of $\sfG$ in a  typical direction $v\in \R^n$: for such $v$ the graph  $\sfG(v;f,\sigma)$  is  ``clustered'' in the sense that it cannot have a large collection of disjoint edges, and hence there is a small set of vertices (i.e., a small subset of $\MM$) which is incident to all of the edges in $\sfG(v;f,\sigma)$. Even though this is not quite how we will use Theorem~\ref{thm:matching from compatibility} in Section~\ref{sec:use fractional}, namely, we will actually use a similar statement about fractional matchings of $\sfG(v;f,\sigma)$ that is a simple formal consequence of Theorem~\ref{thm:matching from compatibility}, for the purpose of intuitively understanding within the present sketch the reason why the deduction of Theorem~\ref{thm:main separated piut together} works, it suffices to initially consider  the above combinatorial implication of~\eqref{expected fractional matching bound with r}.

Conclusion~\eqref{eq:in same connected component} of Proposition~\ref{prop:use quasi to get good graph} is stronger than assumption~\eqref{eq:min assumption in proposition} of Proposition~\ref{prob:graphical prob multiscale} with $f$ replaced by $Cf$. Also, the first inequality in conclusion~\eqref{eq:restate 1/r version} of Proposition~\ref{prop:use quasi to get good graph} coincides with assumption~\eqref{eq:moderate variation along edges} of Proposition~\ref{prob:graphical prob multiscale}. We may therefore proceed to apply  Proposition~\ref{prob:graphical prob multiscale} with $f$ replaced by $Cf$ to  get subsets $A(v),B(v)\subset \MM$ for each $v\in \R^n$  such that~\eqref{eq:super gaussian omega} holds and, by canceling $C$ in~\eqref{eq:separation within connected components} with $f$ replaced by $Cf$, 
\begin{equation}\label{eq:cancel C sketch}
\forall v\in \R^n,\ \forall (x,y)\in A(v)\times B(v),\qquad \{x,y\}\in E\implies |\langle v,f(x)-f(y)\rangle |> \max\big\{\Lambda(x),\Lambda(y)\big\}.
\end{equation} 

The next observation is crucial. Recalling Definition~\ref{def:directional sparsification}, by the second inequality in conclusion~\eqref{eq:restate 1/r version} of Proposition~\ref{prop:use quasi to get good graph}, it follows from~\eqref{eq:cancel C sketch} that for every $v\in \R^n$, if  $(x,y)\in A(v)\times B(v)$ and  $\{x,y\}\in E$, then $\{x,y\}$ is also an edge of $\sfG(v;f,\sigma)$. Equivalently, if $(x,y)\in A(v)\times B(v)$ yet $\{x,y\}$ is not an edge of $\sfG(v;f,\sigma)$, then necessarily $\{x,y\}\notin E$, i.e., by the definition~\eqref{eq:def our E with beta} of $E$ the desired inequality in~\eqref{eq:star separation actual distances} holds.  

Per the above discussion, for typical $v\in \R^n$  there is a small subset of $\MM$ that is incident to all of the edges in $\sfG(v;f,\sigma)$, so by removing it we get large subsets $A^*(v)\subset A(v)$ and $B^*(v)\subset B^*(v)$ such that~\eqref{eq:star separation actual distances} holds. The notion of ``large'' here must be interpreted as the size of $A^*(v)\times B^*(v)$ with respect to the given measure $\omega$ on $\MM\times \MM$, since the input to this reasoning (supplied by Proposition~\ref{prob:graphical prob multiscale}) is~\eqref{eq:super gaussian omega};  this is why we will actually work with a weighted version of~\eqref{expected fractional matching bound with r} for fractional matchings, but, as we stated above, it is a simple formal consequence Theorem~\ref{thm:matching from compatibility} that is quickly deduced in  Section~\ref{sec:use fractional}. Finally, for the above reasoning to succeed, the lower bound in~\eqref{eq:super gaussian omega} needs to dominate the upper bound in~\eqref{expected fractional matching bound with r}; this is why in Theorem~\ref{thm:main separated piut together} $\alpha$ is assumed to be at least a sufficiently large universal constant $\alpha_0$ (given in~\eqref{eq:def r choice} below). The corresponding number-crunching is carried out in   Section~\ref{sec:use fractional}.


\section{Proof of Theorem~\ref{thm:matching from compatibility}}\label{sec:generalized chaining}

\noindent As we explained in Section~\ref{sec:history}, in this section we will prove Theorem~\ref{thm:matching from compatibility} by following the strategy in~\cite{rothvoss2016lecture}.  To start with, the following simple lemma is a straightforward generalization of~\cite[Lemma~10]{rothvoss2016lecture}:

\begin{lemma}\label{lem:lip from boundedness}
For every $x\in V$ and $L,R\ge 0$ define $F^{\sfG,f}_{x,R}:\R^n\to \R$ by setting 
\begin{equation}\label{eq:def our F}
\forall v\in \R^n,\qquad F^{\sfG,f}_{x,R}(v)\eqdef \max_{y\in B_\sfG(x,R)} \langle f(y)-f(x),v\rangle. 
\end{equation}
Then, $F^{\sfG,f}_{x,R}$ is $L$-Lipschitz (as a function from $\ell_2^n$ to $\R$) provided that the following inclusion holds: 
\begin{equation}\label{eq:containment assumption}
f\big(B_\sfG(x,R)\big)\subset B_{\ell_2^n}\big(f(x),L\big),
\end{equation}
\end{lemma}

\begin{proof} For every $u,v\in \R^n$ we have 

\begin{align}\label{eq:prove lip condition}
\begin{split}
F^{\sfG,f}_{x,R}(u)&= \max_{y\in B_\sfG(x,R)} \big(\langle f(y)-f(x),v\rangle+\langle f(y)-f(x),u-v\rangle\big)\\&\le \max_{y\in B_\sfG(x,R)} \langle f(y)-f(x),v\rangle+\max_{y\in B_\sfG(x,R)}  \langle f(y)-f(x),u-v\rangle\\&=F^{\sfG,f}_{x,R}(v)+\max_{y\in B_\sfG(x,R)}  \langle f(y)-f(x),u-v\rangle \\&\le F^{\sfG,f}_{x,R}(v)+\left(\max_{y\in B_\sfG(x,R)} \|f(x)-f(y)\|_2\right)\|u-v\|_2\\
&\le F^{\sfG,f}_{x,R}(v)+L\|u-v\|_2,
\end{split}
\end{align}
where we used the definition~\eqref{eq:def our F} in the first and third steps of~\eqref{eq:prove lip condition}, the fourth step of~\eqref{eq:prove lip condition} is an application of Cauchy--Schwarz, and the final step of~\eqref{eq:prove lip condition} is a restatement of the assumption~\eqref{eq:containment assumption}.  
\end{proof}

We will henceforth denote the $\gamma_n$-mean  of the function $F^{\sfG,f}_{x,R}$ that is given  in~\eqref{eq:def our F} by $\E_{x,R}^{\sfG,f}$, i.e., 
\begin{equation}\label{eq:def expectation operator}
\E_{x,R}^{\sfG,f}\eqdef \int_{\R^n} F^{\sfG,f}_{x,R}\ud\gamma_n. 
\end{equation}
Note in passing that the definition~\eqref{eq:def our F} implies the following point-wise monotonicity in the parameter $R$ that holds for every fixed vertex $x\in V$ and every fixed vector $v\in \R^n$:
\begin{equation}\label{eq:monotonicvity of F in R}
\forall 0\le R_1\le R_2, \qquad  F^{\sfG,f}_{x,R_1}(v)\le F^{\sfG,f}_{x,R_2}(v). 
\end{equation}
By integrating~\eqref{eq:monotonicvity of F in R} with respect to $\gamma_n$, we record for ease of later use the following basic property of~\eqref{eq:def expectation operator}:
\begin{equation}\label{eq:monotonicvity of E in R}
\forall x\in V,\ \forall 0\le R_1\le R_2, \qquad  \E_{x,R_1}^{\sfG,f}\le \E_{x,R_2}^{\sfG,f}. 
\end{equation}
We will also need some control on how~\eqref{eq:def expectation operator} varies as we change the vertex $x$ along edges of $\sfG$;  this appears in the following simple lemma, which is a straightforward generalization of~\cite[Lemma~11]{rothvoss2016lecture}:

\begin{lemma}\label{lem:E monotonicity along edges} For every $R\ge 0$, every $x\in V$, every $y\in N_\sfG(x)$, and every $v\in \R^n$ we have 
$$
F_{x,R+1}^{\sfG,f}(v)\ge F_{y,R}^{\sfG,f}(v)+\langle f(y)-f(x),v\rangle. 
$$
Consequently, $\E_{x,R+1}^{\sfG,f}\ge \E_{y,R}^{\sfG,f}$, as seen by integrating this point-wise inequality with respect to $\gamma_n$. 
\end{lemma}

\begin{proof} Simply observe that  $B_\sfG(y,R)\subset B_{\sfG}(x,R+1)$, as $y\in N_\sfG(x)$, and therefore
\begin{multline*}
F_{x,R+1}^{\sfG,f}(v)\stackrel{\eqref{eq:def our F}}{=}\max_{z\in B_\sfG(x,R+1)} \langle f(z)-f(x),v\rangle\ge \max_{z\in B_\sfG(y,R)} \langle f(z)-f(x),v\rangle\\= \langle f(y)-f(x),v\rangle+ \max_{z\in B_\sfG(y,R)} \langle f(z)-f(y),v\rangle\stackrel{\eqref{eq:def our F}}{=}\langle f(y)-f(x),v\rangle+ F_{y,R}^{\sfG,f}(v).\tag*{\qedhere}
\end{multline*}
\end{proof}

The following technical lemma contains a trichotomy that is crucial for a subsequent induction step:

\begin{lemma}\label{lem:for chaining} Fix $C>0$, $k\in \N\cup\{0\}$, and $n\in \N$. Let  $\sfG=(V,E)$ be a graph, equipped with mappings
$$
f:V\to \R^n,\qquad  \Delta:V\to [0,\infty),\qquad  K:V\to \N, \qquad \sigma:E\to [0,\infty),
$$  
for which 
\begin{equation}\label{eq:for Lip condition}
\forall x\in V,\qquad f\Big(B_\sfG\big(x,K(x)\big)\Big)\subset B_{\ell_2^n}\Big(f(x),\frac{1}{C}\Delta(x)\Big). 
\end{equation}
Assume that for every $v\in \R^n$ there exists a set of vertices $A_v\subset V$ and a one-to-one function $\phi_v:A_v\to V$ that satisfy the following properties. Firstly, we require that 
\begin{equation}\label{eq:Mv is 1-1}
\forall v\in \R^n,\ \forall x\in A_v,\qquad \{\phi_v(x),x\}\in E. 
\end{equation}
Secondly, we require that  
\begin{equation}\label{eq:increase condition along matching}
\forall v\in \R^n,\ \forall x\in A_v,\qquad \big\langle f(x)-f\big(\phi_v(x)\big),v\big\rangle> \Delta(x)+\Delta\big(\phi_v(x)\big)+2 \sigma\big(\{x,\phi_v(x)\}\big).
\end{equation}
 Thirdly, we require that for every $x\in V$ the set $\{v\in \R^n:\ x\in A_v\}$ is Lebesgue-measurable and it satisfies
\begin{equation}\label{eq:lower prob of inclusion in A}
\gamma_n\big(\{v\in \R^n:\ x\in A_v\}\big)\ge 3e^{-\frac14 C^2}.
\end{equation}
Then,  for every $\emptyset\neq U\subset V$ at least one of the following three (mutually exclusive) properties must hold: 
\begin{enumerate}[(1)]
\item\label{item:A} There exists $x\in N_\sfG(U)$ for which $K(x)\le k$;
\item\label{item:B} We have 
\begin{equation}\label{eq:neighborhood large}
|N_\sfG(U)|> e^{\frac{C^2}{4}}|U|\qquad \mathrm{and}\qquad \forall x\in N_\sfG(U),\qquad K(x)\ge k+1,
\end{equation}
and furthermore\footnote{The maximum in~\eqref{eq:when neighborhood large} is well-defined because $U \cap N_G (x)\neq \emptyset$ if (and only if) $x\in N_\sfG(U)$.} 
\begin{equation}\label{eq:when neighborhood large}
\forall x\in N_\sfG(U),\qquad \E_{x,k+1}^{\sfG,f}\ge \max_{y \in U \cap N_G (x)} \E_{y,k}^{\sfG,f}. 
\end{equation}
\item\label{item:C} We have 
\begin{equation}\label{eq:neighborhood small}
|N_\sfG(U)|\le e^{\frac{C^2}{4}}|U| \qquad \mathrm{and}\qquad \forall x\in N_\sfG(U),\qquad K(x)\ge k+1,
\end{equation}
and furthermore there exists a subset $\emptyset\neq W\subset N_\sfG(U)$ for which
\begin{equation}\label{eq:good subset W}
|W|> e^{-\frac{C^2}{4}}|U| \qquad \mathrm{and}\qquad \forall x\in W,\qquad \E_{x,k+1}^{\sfG,f}> \min_{\substack{y \in U\\ \{ x, y \} \in E}} \Big(\E_{y,k}^{\sfG,f} + 2\sigma(\{x, y\}) \Big).
\end{equation}
\end{enumerate}
\end{lemma}

\begin{proof} Suppose that Case~{\em \ref{item:A}}   does not hold, which means that (as $K$ takes values in $\N$) we have 
\begin{equation}\label{eq:K large on neighborhood k+1}
\forall x\in N_\sfG(U),\qquad K(x)\ge k+1. 
\end{equation}
In other words, the second requirement in~\eqref{eq:neighborhood large} and~\eqref{eq:neighborhood small} holds automatically if Case~{\em \ref{item:A}} does not hold.

If $|N_\sfG(U)|> e^{C^2/4}|U|$ in addition to~\eqref{eq:K large on neighborhood k+1}, then we claim that Case~{\em \ref{item:B}} holds. For this, we need to  verify~\eqref{eq:when neighborhood large}. Indeed, suppose that $x\in N_\sfG(U)$ and consider any $y\in U\cap N_\sfG(x)$. By Lemma~\ref{lem:E monotonicity along edges}  we have 
\begin{equation}\label{eq:pass from x to y}
\E_{x,k+1}^{\sfG,f}\ge \E_{y,k}^{\sfG,f}.
\end{equation}
Since $y$ is an arbitrary vertex in $U \cap N_\sfG (x)$, this proves \eqref{eq:when neighborhood large}.

It thus remains to assume that~\eqref{eq:neighborhood small} holds and to then demonstrate  the rest of Case~{\em \ref{item:C}}, i.e., to show that there exists $W\subset N_\sfG(U)$ that satisfies~\eqref{eq:good subset W}.  To this end, for each $v\in \R^n$ define $S_v\subset V$ by 

\begin{equation}\label{eq:def Sv}
S_v\eqdef\left\{x\in A_v\cap U:\ F^{\sfG,f}_{x,k}(v)\ge \E_{x,k}^{\sfG,f}-\Delta(x)\right\}.
\end{equation}
Observe that the following inclusion holds for every $x\in U$:
\begin{equation}\label{eq:first prob inclusion}
\big\{v\in \R^n:\ x\in A_v\big\}\subset \big\{v\in \R^n:\ x\in S_v\big\}\bigcup \left\{v\in \R^n:\ F^{\sfG,f}_{x,k}(v)<\E_{x,k}^{\sfG,f}-\Delta(x)\right\}.
\end{equation}
We therefore have 
\begin{equation}\label{eq:estimate Sv}
3e^{-\frac14 C^2}\stackrel{\eqref{eq:lower prob of inclusion in A}}{\le} \gamma_n\big(\{v\in \R^n:\ x\in A_v\}\big)\stackrel{\eqref{eq:first prob inclusion}}{\le} \gamma_n\big(\{v\in \R^n:\ x\in S_v\}\big)+\gamma_n\big(\{v\in \R^n:\ F^{\sfG,f}_{x,k}(v)<\E_{x,k}^{\sfG,f}-\Delta(x)\}\big).
\end{equation}

By the second part of our current assumption~\eqref{eq:neighborhood small}, every $x\in N_\sfG(U)$ satisfies $K(x)\ge k+1$. Hence
$$
\forall x\in N_\sfG(U),\qquad f\big(B_\sfG(x,k)\big)\subset f\big(B_\sfG(x,k+1)\big)\subset f\Big(B_\sfG\big(x,K(x)\big)\Big) \stackrel{\eqref{eq:for Lip condition}}{\subset} B_{\ell_2^n}\Big(f(x),\frac{1}{C}\Delta(x)\Big).
$$
By Lemma~\ref{lem:lip from boundedness} this implies that 
\begin{equation}\label{eq: Lip constant on W} 
\forall x\in N_\sfG(U),\qquad \big\|F^{\sfG,f}_{x,k}\big\|_{\Lip(\ell_2^n)}\le \frac{\Delta(x)}{C}\quad\mathrm{and}\quad \big\|F^{\sfG,f}_{x,k+1}\big\|_{\Lip(\ell_2^n)}\le \frac{\Delta(x)}{C}. 
\end{equation}
The Gaussian isoperimetric theorem~\cite{borell1975brunn,sudakov1978extremal} (see~\cite[Corollary~2.6]{Led01}) therefore implies\footnote{Formally,  the way Gaussian isoperimetry is typically stated in the literature implies~\eqref{eq:concentration on W} and~\eqref{eq:concentration on W k+1} when the upper bounds on the Lipschitz constants in~\eqref{eq: Lip constant on W}  are strictly positive, i.e., when  $\Delta(x)>0$. Nevertheless, when $\Delta(x)=0$ by ~\eqref{eq: Lip constant on W} we get that the functions $F^{\sfG,f}_{x,k}$ and $F^{\sfG,f}_{x,k+1}$ are constant, so the probabilities in the left hand sides of~\eqref{eq:concentration on W} and~\eqref{eq:concentration on W k+1} vanish.} that
\begin{equation}\label{eq:concentration on W}
\forall x\in N_\sfG(U),\qquad  \gamma_n\big(\{v\in \R^n:\ F^{\sfG,f}_{x,k}(v)<\E_{x,k}^{\sfG,f}-\Delta(x)\}\big)\le e^{-\frac12 C^2}, 
\end{equation}
and also
\begin{equation}\label{eq:concentration on W k+1}
\forall x\in N_\sfG(U),\qquad  \gamma_n\big(\{v\in \R^n:\ F^{\sfG,f}_{x,k+1}(v)>\E_{x,k+1}^{\sfG,f}+\Delta(x)\}\big)\le e^{-\frac12 C^2}. 
\end{equation} 
In particular, a substitution of~\eqref{eq:concentration on W} into~\eqref{eq:estimate Sv} gives the following statement: 
\begin{equation}\label{eq:lower boun on gaussian measure of Sv}
\forall x\in U,\qquad \gamma_n\big(\{v\in \R^n:\ x\in S_v\}\big)\ge 3e^{-\frac14 C^2}-e^{-\frac12 C^2} > 2e^{-\frac14 C^2}.
\end{equation}

Define $T_v\subset V$ for every $v\in \R^n$ by 
\begin{equation}\label{eq:def Tv}
T_v\eqdef \phi_v(S_v). 
\end{equation}
As $\phi_v:A_v\to V$ is assumed to be one-to-one, we have 
\begin{equation}\label{eq:S=T}
\forall v\in \R^n,\qquad |T_v|=|S_v|.
\end{equation}
The definition~\eqref{eq:def Sv} ensures that  $S_v\subset A_v$, so $T_v\subset N_\sfG(S_v)$ by~\eqref{eq:Mv is 1-1}. By~\eqref{eq:def Sv} also  $S_v\subset U$, so 
\begin{equation}\label{eq:Tv in neighborhood}
T_v\subset N_\sfG(U).
\end{equation}  
Define $W\subset N_\sfG(U)$ by
\begin{equation}\label{eq:def our W}
W\eqdef \Big\{x\in N_\sfG(U):\ \gamma_n\big(\{v\in \R^n:\ x\in T_v\}\big)> e^{-\frac12 C^2}\Big\}. 
\end{equation}
With this choice, we can reason as follows (the second step below is an application of Fubini):
\begin{eqnarray*}\label{eq:to combine with upper bound on neighborhood}
2e^{-\frac14 C^2}|U|&\stackrel{\eqref{eq:lower boun on gaussian measure of Sv}}{\le}& \sum_{x\in U} \gamma_n\big(\{v\in \R^n:\ x\in S_v\}\big)\\& =&\int_{\R^n} |S_v|\ud\gamma_n\\&\stackrel{\eqref{eq:S=T}}{=}& \int_{\R^n} |T_v|\ud\gamma_n \\&\stackrel{\eqref{eq:Tv in neighborhood}}{=}& \sum_{x\in N_\sfG(U)\setminus{W}} \gamma_n\big(\{v\in \R^n:\ x\in T_v\}\big) +\sum_{x\in {W}} \gamma_n\big(\{v\in \R^n:\ x\in T_v\}\big)\\&\stackrel{\eqref{eq:def our W}}{\le}& e^{-\frac12 C^2}|N_\sfG(U)\setminus W|+|W|\\
&=& e^{-\frac12 C^2}|N_\sfG(U)|+\left(1-e^{-\frac12 C^2}\right)|W|\\
&\stackrel{\eqref{eq:neighborhood small}}{\le}& e^{-\frac14 C^2}|U|+\left(1-e^{-\frac12 C^2}\right)|W|. 
\end{eqnarray*}
This simplifies to give the following justification of the first requirement in~\eqref{eq:good subset W} for our choice~\eqref{eq:def our W} of $W$: 
$$
|W|\ge \frac{e^{-\frac14 C^2}}{1-e^{-\frac12 C^2}}|U|>e^{-\frac14 C^2}|U|.
$$

To complete the proof of Lemma~\ref{lem:for chaining}, it  remains to derive the second requirement in~\eqref{eq:good subset W}; note that  we did not use thus far the remaining assumption~\eqref{eq:increase condition along matching}, but we will do so for this purpose. Consider $x\in W$. The definition~\eqref{eq:def our W} of $W$ ensures that also $x\in N_\sfG(U)$, and therefore the estimate in~\eqref{eq:concentration on W k+1} holds. Hence,
$$
\gamma_n\big(\{v\in \R^n:\ F^{\sfG,f}_{x,k+1}(v)>\E_{x,k+1}^{\sfG,f}+\Delta(x)\}\big)\stackrel{\eqref{eq:concentration on W k+1}}{\le} e^{-\frac12 C^2}\stackrel{\eqref{eq:def our W}}{<} \gamma_n\big(\{v\in \R^n:\ x\in T_v\}\big).
$$ 
This implies in particular that 
\begin{equation}\label{existance of good vector}
\exists v\in \R^n,\qquad x\in T_v\qquad \mathrm{and}\qquad F^{\sfG,f}_{x,k+1}(v)\le \E_{x,k+1}^{\sfG,f}+\Delta(x).
\end{equation}

From now, fix $v\in \R^n$ for which the two requirements in~\eqref{existance of good vector} hold. By the definition~\eqref{eq:def Tv} of $T_v$, there exists (a unique) $y\in S_v$ such that $x=\phi_v(y)$.  By the definition~\eqref{eq:def Sv} of $S_v$ we know that
$$
F^{\sfG,f}_{y,k}(v)\ge \E_{y,k}^{\sfG,f}-\Delta(y).
$$
Recalling the definition~\eqref{eq:def our F} of $F^{\sfG,f}_{y,k}$, it follows that 
\begin{equation}\label{eq:exists z in k neighborjhood}
\exists z\in B_\sfG(y,k),\qquad \langle f(z)-f(y),v\rangle\ge  \E_{y,k}^{\sfG,f}-\Delta(y).
\end{equation}
We will also fix from now any $z\in V$ that satisfies~\eqref{eq:exists z in k neighborjhood}. The definition~\eqref{eq:increase condition along matching} of $S_v$ ensures that $y\in A_v$, so we may apply the assumption~\ref{eq:increase condition along matching}  while recalling that $y$ was chosen to satisfy $x=\phi_v(y)$, to get the estimate 
\begin{equation}\label{eq:use increment assumption}
\langle f(y)-f(x),v\rangle> \Delta(y)+\Delta(x)+2\sigma(\{y,x\}).
\end{equation}
Consequently,
\begin{equation}\label{eq: to pass to F k+1}
\langle f(z)-f(x),v\rangle\stackrel{\eqref{eq:exists z in k neighborjhood}\wedge \eqref{eq:use increment assumption}}{>} \E_{y,k}^{\sfG,f}+\Delta(x)+2\sigma(\{y,x\}).
\end{equation}
But $z\in B_\sfG(x,k+1)$, as $z\in B_\sfG(y,k)$ and $\{x,y\}\in E$. So, by the definition~\eqref{eq:def our F} of $F^{\sfG,f}_{x,k+1}$, from~\eqref{eq: to pass to F k+1} we get 
\begin{equation}\label{eq:F big at a point, to pass to expectation}
F^{\sfG,f}_{x,k+1}(v)>\E_{y,k}^{\sfG,f}+\Delta(x)+2\sigma(\{y,x\}). 
\end{equation}
In combination with the second condition in~\eqref{existance of good vector}, it follows from~\eqref{eq:F big at a point, to pass to expectation} that
\begin{equation*}
\E_{x,k+1}^{\sfG,f} >\E_{y,k}^{\sfG,f}+2\sigma(\{y,x\}).
\end{equation*} 
Because $y \in S_v \subset U$ and $\{ x, y \} \in E$, this completes the proof of the second requirement in~\eqref{eq:good subset W}.
\end{proof}

In the proof of Theorem~\ref{thm:matching from compatibility} we will use the following convenient notation for every $x\in V$ and $R\ge 1$: 
\begin{equation}\label{eq:smooth sigma}
\mathcal{m}_\sigma(x,R)\eqdef \min_{\substack{\{y,z\}\in E\\ y\in B_\sfG(x,R-1)}}\sigma(\{y,z\})
\end{equation}
Also, for  $R<1$ we denote $\mathcal{m}_\sigma(x,R)=0$.  Observe that (by definition) we have
\begin{equation}\label{eq:monotonicity in R1R2}
\forall x\in V,\ \forall 0\le R_1\le R_2 , \qquad \mathcal{m}_\sigma(x,R_2) \le \mathcal{m}_\sigma(x,R_1)\qquad \mathrm{and}\qquad \forall \{x,y\}\in E,\qquad \mathcal{m}_\sigma(x,2)\le \sigma(\{x,y\}).
\end{equation}
Thus, in particular, we have 
\begin{equation}\label{eq:smooth sigma less than sigma}
\forall R\ge 2,\ \forall\{x,y\}\in E,\qquad \mathcal{m}_\sigma(x,R)\le \sigma(\{x,y\}). 
\end{equation}
Furthermore, the  definition~\eqref{eq:smooth sigma} immediately implies that  
\begin{equation}\label{eq:smooth sigma ALONG EDGES}
\forall R\ge 1, \ \forall x\in V,\ \forall y\in N_\sfG(x),\qquad \mathcal{m}_\sigma(x,R+1)\le \mathcal{m}_\sigma(y,R).
\end{equation}

\begin{remark}\label{rem:restate first condition of compatibility} It is worthwhile to observe for ease of later reference  that using the notions in~\eqref{eq:def expectation operator} and~\eqref{eq:smooth sigma} we can restate conditions~\eqref{eq:4 Delta compatibility} and~\eqref{eq:second compatibility condition} from Definition~\ref{def:compatibility} of $C$-compatibility as the following requirements:
$$
\forall x\in V,\  \forall y\in N_\sfG(x),\qquad  \Delta(x)\le  \mathcal{m}_\sigma\big(x,K(x)\big)\qquad\mathrm{and}\qquad \E_{x,K(y)}^{\sfG,f} \le K(x)\Delta(y).
$$
\end{remark}

We are now ready to prove Theorem~\ref{thm:matching from compatibility}:

\begin{proof}[Proof of Theorem~\ref{thm:matching from compatibility}] We will begin with quick setup that (straightforwardly) provides measurability and symmetry requirements that are needed for the ensuing reasoning. Definition~\eqref{eq:sparsification definition}  ensures  the mapping 
$$
E(\cdot;f,\sigma):\R^n\to 2^E
$$
is Borel measurable. Also, by definition  $E(-v;f,\sigma)=E(v;f,\sigma)$ for every $v\in \R^n$. For each $v\in \R^n$, let $\mathscr{M}(v)$ be the collection of all the maximal-size matchings of the graph $\sfG(v;f,\sigma)$. Since $V$ is finite, the mapping 
$$
\mathscr{M}:\R^n\to 2^{2^E}\subset 2^{2^{2^V}}
$$
is also Borel-measurable and satisfies $\mathscr{M}(-v)=\mathscr{M}(v)$ for every $v\in \R^n$. Fix any linear ordering $\prec$ on the (finite) set of matchings  of $\sfG$ and for each $v\in \R^n$ let $M(v)$ be the element of $\mathscr{M}(v)$ that is minimal with respect to $\prec$. The mapping $v\mapsto M(v)$ is Borel-measurable and satisfies $M(-v)=M(v)$ for every $v\in \R^n$. Furthermore, this construction ensures that for every fixed edge $e\in E$, the set $\{v\in \R^n:\ e\in M(v)\}$ is Borel.

Suppose for the purpose of obtaining a contradiction that 
\begin{equation}\label{eq:contra suumption all matchings big}
\int_{\R^n} |M(v)|\ud \gamma_n(v)\ge 6e^{-\frac14 C^2}|V|.
\end{equation}
We will start by performing the following iterative procedure. Denote $V_1=V$ and $M_1(v)=M(v)$ for every $v\in \R^n$. Suppose inductively that $V_j\subset V$ and $\{M_j(v)\}_{v\in \R^n}$ have been defined for some $j\in \N$, where for each $v\in \R^n$ we require that $M_j(v)\subset M(v)$ is a matching of the following graph:
\begin{equation}\label{eq: def GJ}
\sfG_j(v;f,\sigma)\eqdef \Big(V_j,E_j(v;f,\sigma)\eqdef \big\{e\in E(v;f,\sigma):\ e\subset V_j\big\}\Big).
\end{equation}
Furthermore, we require that the mapping $M_j$ is Borel-measurable, satisfies $M_j(-v)=M_j(v)$ for every $v\in \R^n$, and for every fixed edge $e\in E$, the set $\{v\in \R^n:\ e\in M_j(v)\}$ is Borel.

If there does not exist  $x_j\in V_j$ that satisfies
\begin{equation}\label{eq:probability too small matching j}
 \gamma_n\big(\{v\in \R^n:\ \exists e\in M_j(v)\ \mathrm{such\  that}\ x_j\in e\}\big)<  6e^{-\frac14 C^2},
\end{equation}
then terminate the construction. Otherwise, fix $x_j\in V_j$ for which~\eqref{eq:probability too small matching j} holds and denote 
$$
V_{j+1}\eqdef V_j\setminus \{x_j\}\qquad \mathrm{and}\qquad \forall v\in \R^n,\qquad M_{j+1}(v)\eqdef M_j(v)\setminus \big\{e\in M_j(v):\ x_j\in e\big\}. 
$$
The above measurability requirements for $M_{j+1}$ are immediate from this definition. 

Note that for every $v\in \R^n$, since $M(v)$  is a matching of $\sfG$ and $M_j(v)\subset M(v)$, there is at most one edge  $e\in M_j(v)$  for which $x_j\in e$. Consequently, we have the following point-wise estimate:
$$
\forall v\in \R^n,\qquad |M_{j+1}(v)|\ge |M_j(v)|-\1_{\{\exists e\in M_j(v)\ \mathrm{such\  that}\ x_j\in e\}}.
$$ 
By integrating this inequality with respect to $\gamma_n$ and using our choice of $x_j$, we conclude that 
\begin{align}\label{eq:how much Mj decreased}
\begin{split}
\int_{\R^n} |M_{j+1}(v)|\ud \gamma_n(v)&\ge \int_{\R^n} |M_{j}(v)|\ud \gamma_n(v) -  \gamma_n\big(\{v\in \R^n:\ \exists e\in M_j(v)\ \mathrm{such\  that}\ x_j\in e\}\big)\\&> \int_{\R^n} |M_{j}(v)|\ud \gamma_n(v)-6e^{-\frac14 C^2}.
\end{split}
\end{align}

Suppose that this procedure terminates at the $J\in \N$ step; this must occur because $V$ is finite and each step removes a vertex, whence $J\le |V|+1$. By applying~\eqref{eq:how much Mj decreased} inductively we see that
\begin{equation}\label{eq:on MJ}
\int_{\R^n} |M_{J}(v)|\ud \gamma_n(v)>\int_{\R^n} |M(v)|\ud \gamma_n(v)-6e^{-\frac14 C^2}(J-1)\ge \int_{\R^n} |M(v)|\ud \gamma_n(v)-6e^{-\frac14 C^2}|V|\stackrel{\eqref{eq:contra suumption all matchings big}}{\ge} 0.
\end{equation}
Thanks to the strict inequality in the first step of~\eqref{eq:on MJ}, it follows in particular that $V_J\neq \emptyset$.\footnote{By increasing the factor $6$ in~\eqref{eq:contra suumption all matchings big} to, say, $7$, and then repeating the above reasoning, we can ensure that the mean of $|M_{J}(v)|$ with respect to $\gamma_n$ is at least $\exp(-C^2/4)|V|$, and hence also  $|V_J|\gtrsim \exp(-C^2/4)|V|$, since $M_J(v)$ is a matching of a graph whose vertex set is $V_J$. We do not need such a  guarantee below, but it is worthwhile to note it in case this will occur in the future. } Furthermore, because the above procedure terminated at the $J\in \N$ step, we have 
\begin{equation}\label{eq:probability too small matching}
\forall x\in V_J,\qquad  \gamma_n\big(\{v\in \R^n:\ \exists e\in M_J(v)\ \mathrm{such\  that}\ x\in e\}\big)\ge   6e^{-\frac14 C^2}.
\end{equation}

We will next proceed to apply Lemma~\ref{lem:for chaining} (multiple times) to the following induced subgraph of $\sfG$:
$$
\sfG_J\eqdef \big(V_J,E_J\eqdef \{e\in E:\ e\subset V_J\}\big).
$$ 
Theorem~\ref{thm:matching from compatibility} assumes that $\sfG$ is $C$-compatible with $f$ and $\sigma$, so let $K:V\to \N$ and $\Delta:V\to [0,\infty)$ be mappings as in Definition~\ref{def:compatibility}. Henceforth, we will slightly abuse notation by also denoting the restrictions of $f,K,\Delta$ to $V_J$ by $f,K,\Delta$, and correspondingly denoting the restriction of $\sigma$ to $E_J$ by $\sigma$. Observe that because $\sfG_J$ is an induced subgraph of $\sfG$, we have $B_{\sfG_J}(x,R)\subset B_{\sfG}(x,R)$ for every $x\in V_J$ and $R\ge 0$. Thanks to these inclusions, an inspection of Definition~\ref{def:compatibility} reveals that as $\sfG$ is $C$-compatible with $f$ and $\sigma$, also $\sfG_J$ is $C$-compatible with $f$ and $\sigma$.

For each vector $v\in \R^n$ define $A_v\subset V_J$ by 
\begin{equation}\label{eq:def our Av}
A_v\eqdef \big\{x\in V_J:\ \exists \{x,y\}\in M_J(v)\ \mathrm{such\ that\ } \langle f(x)-f(y),v\rangle > 4\sigma(\{x,y\})\big\}.
\end{equation}
Because $M_J(v)$ is a matching of $\sfG$, if $x\in A_v$, then there is a unique $y\in V_J$ such that $\{x,y\}\in M_J(v)$. We can therefore define $\phi_v(x)=y$. The fact that $M_J(v)$ is a matching of $\sfG$ also implies that  $\phi_v:A_v\to V_J$ is  one-to-one, and by design this definition ensures that assumption~\eqref{eq:Mv is 1-1} of Lemma~\ref{lem:for chaining} is satisfied. 

We next claim that the following equality of events holds for every fixed $x\in V_J$:
\begin{equation}\label{eq:symmetry of events}
\big\{v\in \R^n:\ \exists e\in M_J(v)\ \mathrm{such\  that}\ x\in e\big\}= \big\{v\in \R^n:\ x\in A_v\big\}\cup \big\{v\in \R^n:\ x\in A_{-v}\big\}.
\end{equation}
Indeed, it is immediate from the definition~\eqref{eq:def our Av} of $\{A_v\}_{v\in \R^n}$ that the right hand side of~\eqref{eq:symmetry of events}  is contained in the left hand side of~\eqref{eq:symmetry of events}. The reverse inclusion follows from the fact that $M_J(v)$  is a matching of the graph $\sfG_J(v;f,\sigma)$ in~\eqref{eq: def GJ}, which is a subgraph of the Euclidean sparsification $\sfG(v;f,\sigma)$ of $\sfG$. So, if $v\in \R^n$ and there is $e\in M_J(v)$ with $x\in e$, then $e\in E(v;f,\sigma)$. By the definition~\eqref{eq:sparsification definition} of Euclidean sparsification, this means that if $y\in V_J$  is such that   $e=\{x,y\}$, then $|\langle f(x)-f(y),v\rangle|>4\sigma(\{x,y\})$. If $\langle f(x)-f(y),v\rangle>4\sigma(\{x,y\})$, then  $x\in A_v$, and otherwise $\langle f(x)-f(y),v\rangle<-4\sigma(\{x,y\})$, which means that $x\in A_{-v}$.

By~\eqref{eq:def our Av}, the sets in the right hand side of~\eqref{eq:symmetry of events} are disjoint unless $v=0$. Therefore, for every $x\in V_J$  
\begin{multline*}
6e^{-\frac14 C^2}\stackrel{\eqref{eq:probability too small matching}}{\le} \gamma_n\big(\{v\in \R^n:\ \exists e\in M_J(v)\ \mathrm{such\  that}\ x\in e\}\big) \\\stackrel{\eqref{eq:symmetry of events}}{=} \gamma_n\big(\{v\in \R^n:\ x\in A_v\}\big) +\gamma_n\big(\{v\in \R^n:\ x\in A_{-v}\}\big) =2\gamma_n\big(\{v\in \R^n:\ x\in A_v\}\big). 
\end{multline*}
Hence,
$$
\forall x\in V_J,\qquad \gamma_n\big(\{v\in \R^n:\ x\in A_v\}\big)\ge 3e^{-\frac14 C^2}.
$$
In other words, we checked that condition~\eqref{eq:lower prob of inclusion in A} of Lemma~\ref{lem:for chaining} is satisfied for $\sfG_J$ (the corresponding measurability requirement follows from the measurability that we ensured in the above construction). 

For every $v\in \R^n$ and $x\in A_v$, by (a special case of) the condition~\eqref{eq:4 Delta compatibility} of Definition~\ref{def:compatibility}, which is part of the $C$-compatibility assumption of Theorem~\ref{thm:matching from compatibility}, we know that 
\begin{equation}\label{eq:actual copmpatibility need 1}
2\sigma\big(\{x,\phi_v(x)\}\big)\ge \Delta(x)+\Delta\big(\phi_v(x)\big).
\end{equation}
Consequently, we deduce as follows that condition~\eqref{eq:increase condition along matching} of Lemma~\ref{lem:for chaining} is satisfied: 
$$
\big\langle f(x)-f\big(\phi_v(x)\big),v\big\rangle\stackrel{\eqref{eq:def our Av}}{>} 4\sigma\big(\{x,\phi_v(x)\}\big)\stackrel{\eqref{eq:actual copmpatibility need 1}}{\ge}  \Delta(x)+\Delta\big(\phi_v(x)\big)+2 \sigma\big(\{x,\phi_v(x)\}\big).
$$

Condition~\eqref{eq:for Lip condition} of Lemma~\ref{lem:for chaining} coincides with condition~\eqref{eq:inclusion condition in compatibility def} of Definition~\ref{def:compatibility} with $\sfG$ replaced by $\sfG_J$, so it holds thanks to the $C$-compatibility of $\sfG_J$ with $f$ and $\sigma$.   This completes the verification of all of the assumptions of Lemma~\ref{lem:for chaining}.

We will next  apply Lemma~\ref{lem:for chaining} iteratively as follows to obtain   $T\in \N$, as well as nonempty subsets $\emptyset \neq U_0,U_1,\ldots,U_T\subset V_J$ and indices  $a_1,\ldots,a_T\in \{2,3\}$, for which the following requirements hold. Firstly, at the start of the induction set $U_0=V_J$ and $a_1=3$. Secondly, at the end of the iteration we have
\begin{equation}\label{eq:termination T}
\min_{x\in N_{\sfG_J}(U_T)}K(x)\le T.
\end{equation}
Thirdly, for every $t\in [T]$ (so, excluding $t=0$), if we denote 
\begin{equation}\label{eq:def dk}
d_t\eqdef  \big|\{s\in [t]:\ a_s=3\}\big|,
\end{equation}
then
\begin{equation}\label{eq:Ut inclusion for induction}
\forall x\in U_t,\qquad K(x)\ge t\qquad\mathrm{and}\qquad  \E_{x,t}^{\sfG_J,f} > 2d_t\mathcal{m}_\sigma(x,t).
\end{equation}
Finally, we require that the following estimate holds:
\begin{equation}\label{eq:ratio control t}
\forall t\in [T],\qquad \frac{|U_{t}|}{|U_{t-1}|}>e^{\frac{C^2}{2} \left(\frac12-\1_{\{a_{t}=3\}}\right)}=\left\{\begin{array}{ll} e^{\frac14 C^2}&\mathrm{if}\ a_{t}=2,\\
e^{-\frac14 C^2}&\mathrm{if}\ a_{t}=3.\end{array}\right.
\end{equation}

Supposing for the moment   that the above construction has already been carried out, we will next see how it can be used to complete the proof of Theorem~\ref{thm:matching from compatibility}. By~\eqref{eq:termination T} we can fix $x\in N_{\sfG_J}(U_T)$, thus we can also  fix $y\in U_T\cap N_{\sfG_J}(x)$, for which $K(x)\le T$. As $y\in N_{\sfG_J}(x)$, by combining $K(x)\le T$ with condition~\eqref{eq:4 Delta compatibility} of Definition~\ref{def:compatibility}, per its formulation in Remark~\ref{rem:restate first condition of compatibility}, we see that 
\begin{equation}\label{eq:Kx less T terminmal}
\E_{y,K(y)}^{\sfG_J,f} \le T\Delta(y). 
\end{equation}

Because $y\in U_T$, the case $t=T$ of the first condition in~\eqref{eq:Ut inclusion for induction} gives the bound $K(y)\le T$, so by the (very simple) monotonicity properties that we recorded in~\eqref{eq:monotonicvity of E in R} and in~\eqref{eq:monotonicity in R1R2}, we have 
\begin{equation}\label{eq:use monotonicities y}
\E_{y,T}^{\sfG_J,f}\le \E_{y,K(y)}^{\sfG_J,f},\qquad\mathrm{and}\qquad  \mathcal{m}_\sigma\big(y,K(y)\big)\le \mathcal{m}_\sigma(y,T).
\end{equation}
Also, by Remark~\ref{rem:restate first condition of compatibility}, the (first part of the)  $C$-compatibility assumption of Theorem~\ref{thm:matching from compatibility} gives the estimate
\begin{equation}\label{eq:use compatibility 4D}
\Delta(y)\le \mathcal{m}_\sigma\big(y,K(y)\big). 
\end{equation}
By combining~\eqref{eq:Kx less T terminmal},  \eqref{eq:use monotonicities y} and~\eqref{eq:use compatibility 4D}, we conclude that 
\begin{equation}\label{eq:ET}
\E_{y,T}^{\sfG_J,f}\le T\mathcal{m}_\sigma(y,T). 
\end{equation}

At the same time, by applying~\eqref{eq:ratio control t} inductively we see that 
\begin{equation}\label{eq:telescoping porduct estimate}
\frac{|U_T|}{|U_0|}=\prod_{t=1}^{T}  \frac{|U_{t}|}{|U_{t-1}|}\stackrel{\eqref{eq:ratio control t}}{>} e^{\frac{C^2}{2}\left(\frac{T}{2}-\sum_{t=1}^T \1_{\{a_t=3\}}\right)}\stackrel{\eqref{eq:def dk}}{=}  e^{\frac{C^2}{2}\left(\frac{T}{2}-d_T\right)}.
\end{equation}
Since $U_T\subset V_J=U_0$, the left hand side of~\eqref{eq:telescoping porduct estimate} is at most $1$, so, since $C>0$, \eqref{eq:telescoping porduct estimate}   implies that
\begin{equation}\label{eq:DT is big}
d_T>\frac{T}{2}. 
\end{equation}
Because $y\in U_T$, a substitution of~\eqref{eq:DT is big} into the case $t=T$ of the second part of~\eqref{eq:Ut inclusion for induction} gives 
\begin{equation}\label{eq:ET lower}
\E_{y,T}^{\sfG_J,f} > T\mathcal{m}_\sigma(y,T).
\end{equation}
We thus arrive at the desired contradiction by contrasting~\eqref{eq:ET} with~\eqref{eq:ET lower}, i.e., the contrapositive assumption~\eqref{eq:contra suumption all matchings big}, which is the premise of the current discussion, cannot hold. This completes the proof of Theorem~\ref{thm:matching from compatibility} assuming that the aforementioned construction (of $T,U_0,U_1,\ldots,U_T, a_1,\ldots,a_T$ with the above specifications) can indeed be carried out; this is what we will justify next.  

Recall that we already defined $U_0=V_J$ and $a_1=3$. We will start by applying Lemma~\ref{lem:for chaining} with $U=U_0$ and $d=k=0$. Case~{\em \ref{item:A}}  of Lemma~\ref{lem:for chaining}  does not hold because $k=0$ and $K$ takes values in $\N$. Case~{\em \ref{item:B}}  of Lemma~\ref{lem:for chaining} does not hold as $U=U_0=V_J$ and $C>0$, so $|N_{\sfG_J}(V_J)|=|V_J|<e^{C^2/4}|V_J|$.  Lemma~\ref{lem:for chaining}  therefore ensures that its Case~{\em \ref{item:C}}  holds, thus producing a set $\emptyset \neq W\subset N_{\sfG_J}(V_J)=V_J$. We will then define $U_1=W$. With this notation, we thus know from the first inequality in~\eqref{eq:good subset W} that
 \begin{equation*}
 \frac{|U_1|}{|U_0|}>e^{-\frac14 C^2}.
 \end{equation*}
Because  $a_1=3$ (by definition), this coincides with the case $t=1$   of~\eqref{eq:ratio control t}.  Furthermore,  by~\eqref{eq:good subset W} we have 
\begin{equation*}\label{k=1 for induction}
\forall x\in U_1,\qquad \E_{x,1}^{\sfG_J,f} > 2 \min_{\substack{y\in V_J\\ \{x, y \} \in E_J}} \sigma(\{x, y\}) \stackrel{\eqref{eq:smooth sigma}}{=} 2\mathcal{m}_\sigma(x,1).
\end{equation*}
As $d_1=1$ thanks to~\eqref{eq:def dk}, this coincides with the case $t=1$ of the second requirement  in~\eqref{eq:Ut inclusion for induction}.  The case $t=1$ of the first requirement  in~\eqref{eq:Ut inclusion for induction} is automatic because $K$ takes values in $\N$.

Assume inductively that for  $t\in \N$ we already defined $\emptyset\neq U_0,U_1,\ldots,U_t\subset V_J$ and $a_1,\ldots,a_t\in \{2,3\}$ for which both~\eqref{eq:Ut inclusion for induction} and~\eqref{eq:ratio control t} hold, with $d_1,\ldots,d_t$ given by~\eqref{eq:def dk}. If there is $x\in N_{\sfG_J}(U_t)$ for which $K(x)\le t$, then define $T=t$, thus ensuring that~\eqref{eq:termination T} is satisfied. Otherwise, by Lemma~\ref{lem:for chaining} applied with $U=U_t$ and $d=d_t$,  either Case~{\em \ref{item:B}}  of Lemma~\ref{lem:for chaining} holds, or Case~{\em \ref{item:C}}  of Lemma~\ref{lem:for chaining} holds. We will encode this dichotomy by defining $a_{t+1}=2$ if Case~{\em \ref{item:B}}  holds, and defining $a_{t+1}=3$  if Case~{\em \ref{item:C}} holds.

If Case~{\em \ref{item:B}} holds, then set $U_{t+1}=N_{\sfG_J}(U_t)$.  By~\eqref{eq:def dk} we have $d_{t+1}=d_t$, as $a_{t+1}=2$, so~\eqref{eq:ratio control t} holds with $t$ replaced by $t+1$ thanks to the first condition in the conclusion~\eqref{eq:neighborhood large} of Lemma~\ref{lem:for chaining} in this case. The rest of~\eqref{eq:neighborhood large} is a restatement of the first part of~\eqref{eq:Ut inclusion for induction} with $t$ replaced by $t+1$, and the second part of~\eqref{eq:Ut inclusion for induction} (also with $t$ replaced by $t+1$) is derived as follows from the remaining conclusion~\eqref{eq:when neighborhood large} of Lemma~\ref{lem:for chaining} in Case~{\em \ref{item:B}}:
\begin{equation*}
    \E_{x,t+1}^{\sfG_J,f} \stackrel{\eqref{eq:when neighborhood large}}{\ge} \min_{y \in U_t \cap N_{\sfG_J} (x)} \E_{y,t}^{\sfG_J,f} \stackrel{\eqref{eq:Ut inclusion for induction}}{>} \min_{y \in U_t \cap N_{\sfG_J} (x)} 2d_t\mathcal{m}_\sigma(y,t) \stackrel{\eqref{eq:smooth sigma ALONG EDGES}}{\ge} 2d_{t+1}\mathcal{m}_\sigma(x,t+1).
\end{equation*}

If Case~{\em \ref{item:C}} holds,  then define $U_{t+1}=W$, where $W$ is the subset of $N_{\sfG_J}(U_t)$ that Case~$\eqref{item:C}$ produces.  Then, the second condition in~\eqref{eq:neighborhood small} shows that the first requirement in~\eqref{eq:Ut inclusion for induction} holds with $t$ replaced by $t+1$. To show that the second requirement in~\eqref{eq:Ut inclusion for induction} is satisfied with $t$ replaced by $t+1$, observe that by~\eqref{eq:def dk} we have $d_{t+1}=d_t+1$, as $a_{t+1}=3$, so by the second condition in~\eqref{eq:good subset W}, we see that 
\begin{equation*}
    \E_{x,t+1}^{\sfG_J,f} \stackrel{\eqref{eq:good subset W}}{>} \min_{\substack{y \in U_t\\ \{ x, y \} \in E_J}} \Big( \E_{y,t}^{\sfG_J,f} + 2\sigma(\{x, y\}) \Big) \stackrel{\eqref{eq:Ut inclusion for induction}}{>} \min_{\substack{y \in U_t\\ \{ x, y \} \in E_J}} \left(2d_t\mathcal{m}_\sigma(y,t) + 2\sigma(\{x, y\}) \right) \stackrel{\eqref{eq:smooth sigma less than sigma}\wedge \eqref{eq:smooth sigma ALONG EDGES}}{\ge} 2d_{t+1}\mathcal{m}_\sigma(x,t+1).
\end{equation*}
This completes the derivation of~\eqref{eq:Ut inclusion for induction} when Case~{\em \ref{item:C}} holds, thus completing the inductive step.

By the first condition in~\eqref{eq:Ut inclusion for induction}, the above iterative procedure cannot continue indefinitely, as it entails that $t\le \max_{x\in V} K(x)$; the desired construction is completed  when this iteration terminates. 
\end{proof}

\section{Probabilistic groundwork}\label{sec:groundwork}

\noindent Our goal here is to present the proof of Proposition~\ref{prob:graphical prob multiscale}, which consists of elementary probabilistic considerations.  We start with the following simple general lemma:

\begin{lemma}\label{lem:periodic distributional density estimate} Fix $\beta>0$. Let $\sfX$ be a random variable, defined on some probability space $(\Omega,\prob)$, which has a density $\f$ that is nondecreasing on $(-\infty,0]$ and nonincreasing on $[0,\infty)$. Suppose that $f:\R\to [0,\infty)$ is $\beta$-periodic, i.e, $f(t+\beta)=f(t)$ for every $t\in \R$, and also that $f$ is integrable on the interval $[0,\beta]$. Then,
$$
\E\big[f(\sfX)\big]\ge \bigg(\frac{1}{\beta}\int_0^\beta f(t)\ud t\bigg)\prob\big[|\sfX|\ge \beta\big].
$$
\end{lemma}

\begin{proof} Simply note that for every $k\in \N$ we have
\begin{align}\label{eq:k positive}
\begin{split}
\int_{(k-1)\beta}^{k\beta} f(x)\f(x)\ud x&\ge \int_{(k-1)\beta}^{k\beta} f(x)\f(k\beta)\ud x\\&= \bigg(\int_0^\beta f(t)\ud t\bigg) \f(k\beta) \ge 
\bigg(\int_0^\beta f(t)\ud t\bigg) \frac{1}{\beta}\int_{k\beta}^{(k+1)\beta} \f(x)\ud x,
\end{split}
\end{align}
and 
\begin{align}\label{eq:k negative}
\begin{split}
\int^{-(k-1)\beta}_{-k\beta} f(x)\f(x)\ud x&\ge \int^{-(k-1)\beta}_{k\beta} f(x)\f(-k\beta)\ud x\\&= \bigg(\int_0^\beta f(t)\ud t\bigg) \f(-k\beta) \ge 
\bigg(\int_0^\beta f(t)\ud t\bigg) \frac{1}{\beta}\int^{-k\beta}_{-(k+1)\beta} \f(x)\ud x,
\end{split}
\end{align}
where in~\eqref{eq:k positive} and~\eqref{eq:k negative} the first and third steps  use the monotonicity assumptions on $\f$, and the second steps  use the $\beta$-periodicity of $f$. By summing~\eqref{eq:k positive} and~\eqref{eq:k negative} over $k\in \N$ we conclude that 
\begin{equation*}
\E\big[f(\sfX)\big]=\int_{-\infty}^\infty f(x)\f(x)\ud x\ge \bigg(\frac{1}{\beta}\int_0^\beta f(t)\ud t\bigg)\int_{\R\setminus (-\beta,\beta)} \f(x)\ud x=\bigg(\frac{1}{\beta}\int_0^\beta f(t)\ud t\bigg)\prob\big[|\sfX|\ge \beta\big]. \tag*{\qedhere}
\end{equation*}
\end{proof}

For every $\theta\in [0,1]$ define two (periodic) subsets $L_\theta,R_\theta\subset \R$ of the real line by 
\begin{equation}\label{eq:def LtRt}
L_\theta\eqdef \left\{r\in \R:\  r-\theta\in \bigcup_{k\in \Z} \left[k,k+\frac14\right)\right\}\qquad\mathrm{and}\qquad R_\theta\eqdef \left\{r\in \R:\  r-\theta\in \bigcup_{k\in \Z} \left[k+\frac12,k+\frac34\right)\right\}.
\end{equation}
By definition, we then have
\begin{equation}\label{eq:an far}
\forall \theta\in [0,1],\ \forall (a,b)\in L_\theta\times R_\theta,\qquad |a-b|>\frac14. 
\end{equation}
Throughout what follows, we will denote the Lebesgue probability measure on $[0,1]$ by $\mathbb{U}$. Observe that 
\begin{equation}\label{eq:random shift probability 1}
\forall a\in \R,\qquad \mathbb{U}\big(\{\theta\in [0,1]:\ a\in L_\theta\}\big)=\mathbb{U}\big(\{\theta\in [0,1]:\ a\in R_\theta\}\big)=\frac14, 
\end{equation}
and, if we consider the periodic tent function $\tau:\R\to [0,\infty)$ that is given by
\begin{equation}\label{eq:tent tau}
\tau(s)\eqdef \max\left\{\frac14-\left|\frac12-(s-\lfloor s\rfloor)\right|,0\right\} =\left\{\begin{array}{ll} 0&\mathrm{if}\ 0\le  s-\lfloor s\rfloor\le \frac14,\\s-\lfloor s\rfloor -\frac14 &\mathrm{if}\ \frac14\le  s-\lfloor s\rfloor\le \frac12,\\ \frac34- (s-\lfloor s\rfloor)&\mathrm{if}\ \frac12\le  s-\lfloor s\rfloor\le \frac34,\\ 0&\mathrm{if}\ \frac34\le  s-\lfloor s\rfloor\le 1,  \end{array}\right.
\end{equation}
then
\begin{equation}\label{eq:random shift probability 2}
\forall (a,b)\in \R^2,\qquad \mathbb{U}\big(\{\theta\in [0,1]:\ (a,b)\in L_\theta\times R_\theta\}\big)=\tau(|a-b|).
\end{equation}
Both~\eqref{eq:random shift probability 1} and~\eqref{eq:random shift probability 2} follow from  straightforward elementary computations, which we omit.  

\begin{lemma}\label{lem:basic random union of strips}
Fix $C>0$. For every $v\in \R^n$ and $\theta\in [0,1]$ define $L_C(v,\theta),R_C(v,\theta)\subset \R^n$ by
\begin{equation}\label{eq:def random union of strips}
L_C(v,\theta)\eqdef \Big\{z\in \R^n:\  \frac{\langle v,z\rangle}{4C}\in L_\theta\Big\}\qquad \mathrm{and}\qquad R_C(v, \theta)\eqdef \Big\{z\in \R^n:\  \frac{\langle v,z\rangle}{4C}\in R_\theta\Big\},
\end{equation}
where $L_\theta,R_\theta\subset \R$ are as in~\eqref{eq:def LtRt}. Then,  both $(v,\theta)\mapsto L_C(v,\theta)$ and $(v,\theta)\mapsto R_C(v,\theta)$  are Borel-measurable set-valued mappings from (from $\R^n\times [0,1]$ to the Borel subsets of $\R^n$) that satisfy
\begin{equation}\label{eq:separated projection random}
\forall (v,\theta)\in \R^n\times [0,1],\ \forall (x,y)\in L_C(v,\theta)\times R_C(v,\theta), \qquad |\langle v,x-y\rangle |>C.
\end{equation}
Furthermore, recalling that $\mathbb{U}$ denotes the Lebesgue measure on $[0,1]$, we have 
\begin{equation}\label{eq:inclusion measure}
\forall v,z\in \R^n,\qquad \mathbb{U}\big(\{\theta\in [0,1]:\ z\in L_C(v,\theta)\} \big)=\mathbb{U}\big(\{\theta\in [0,1]:\ z\in R_C(v,\theta)\} \big) =\frac14, 
\end{equation}
and 
\begin{equation}\label{eq:simulataneously in strips}
\forall x,y\in \R^n,\qquad (\gamma_n\times \mathbb{U})\big(\{(v,\theta)\in \R^n\times [0,1]:\ (x,y)\in L_C(v,\theta)\times R_C(v,\theta)\}\big) \gtrsim e^{-\frac{9C^2}{\|x-y\|_2^2}}. 
\end{equation}
\end{lemma}

\begin{proof} The measurability assertion is immediate from the definitions~\eqref{eq:def LtRt} and~\eqref{eq:def random union of strips}. Requirement~\eqref{eq:separated projection random} follows directly from~\eqref{eq:def random union of strips} and~\eqref{eq:an far}.  Requirement~\eqref{eq:inclusion measure} follows directly from~\eqref{eq:def random union of strips} and~\eqref{eq:random shift probability 1}. To justify the remaining property~\eqref{eq:simulataneously in strips}, fix distinct $x,y\in \R^n$ and observe that 
\begin{align*}
(\gamma_n\times \mathbb{U})&\big(\{(v,\theta)\in \R^n\times [0,1]:\ (x,y)\in L_C(v,\theta)\times R_C(v,\theta)\}\big)\\&\stackrel{\eqref{eq:simulataneously in strips}}{=}\int_{\R^n} \mathbb{U}\Big(\big\{\theta\in [0,1]:\ \frac{\langle v,x\rangle}{4C}\in L_\theta \quad\mathrm{and}\quad \frac{\langle v,y\rangle}{4C} \in R_\theta\big\}\Big) \ud \gamma_n(v) \stackrel{\eqref{eq:random shift probability 2}}{=}  \int_{\R^n} \tau\left(\frac{|\langle v,x-y\rangle|}{4C}\right)\ud\gamma_n(v),
\end{align*}
where we recall the definition of $\tau$ in~\eqref{eq:tent tau}. By rotation-invariance, $\langle v,x-y\rangle/\|x-y\|_2$ is distributed over $\R$ according to $\gamma_1$ when  $v$ is distributed over $\R^n$ according to $\gamma_n$. By Lemma~\ref{lem:periodic distributional density estimate}, we therefore have 
$$
\int_{\R^n} \tau\left(\frac{|\langle v,x-y\rangle|}{4C}\right)\ud\gamma_n(v)\ge \bigg(\int_0^1\tau(t)\ud t\bigg)\sqrt{\frac{2}{\pi}}\int_{\frac{4C}{\|x-y\|_2}}^\infty e^{-\frac{s^2}{2}}\ud s\asymp \min\left\{\frac{\|x-y\|_2}{C},1\right\}e^{-\frac{8C^2}{\|x-y\|_2^2}} \gtrsim e^{-\frac{9C^2}{\|x-y\|_2^2}}, 
$$
where the second step is valid as the integral of $\tau$ over $[0,1]$ is a positive universal constant,  and the following standard asymptotic  identity holds (its elementary proof can be found in e.g.~\cite{Kom55,Dur19}):
\begin{equation*}
\int_a^\infty e^{-\frac{s^2}{2}}\ud s\asymp \min\left\{\frac{1}{a},1\right\} e^{-\frac{a^2}{2}}. \tag*{\qedhere}
\end{equation*}
\end{proof}

\begin{lemma}\label{lem:multiscale representation} Fix $\alpha,C>0$ and $n\in \N$. Suppose that $X$ is a finite set, equipped with mappings $f:X\to \R^n$ and $\Lambda:X\to (0,\infty]$. Then, there is a Polish probability space $(\Omega,\prob)$\footnote{Thus, $\Omega$ is a separable complete metric space and $\prob$ is a Borel probability measure on $\Omega$.} and for every $v\in \R^n$ and $\chi\in \Omega$ there are  
$$
A_C(v,\chi)=A_C^{f,\Lambda,\alpha}(v,\chi),B_C(v,\chi)=B_C^{f,\Lambda,\alpha}(v,\chi)\subset X,
$$
such that the set-valued mappings $(v,\chi)\mapsto A_C(v,\chi)$ and $(v,\chi)\mapsto B_C(v,\chi)$ from $\R^n\times \Omega$ to $2^X$ are Borel-measurable, and have the following properties. Firstly, for any $x,y\in X$, any $v\in \R^n$ and any $\chi\in \Omega$, if 
\begin{equation}\label{eq:Lambda smooth assumption}
(x,y)\in A_C(v,\chi)\times B_C(v,\chi)\qquad\mathrm{and}\qquad e^{-\alpha}\Lambda(x)\le \Lambda(y)\le e^{\alpha}\Lambda(x),
\end{equation}
then 
\begin{equation}\label{eq:Cmax goal}
|\langle v,f(x)-f(y)\rangle|>C\max\left\{\Lambda(x),\Lambda(y)\right\}. 
\end{equation}
Secondly, for every $v\in \R^n$ and $x\in X$  we have 
\begin{equation}\label{eq:1/6 prob equality}
\prob\big(\{\chi\in \Omega:\  x\in A_C(v,\chi)\}\big)=\prob\big(\{\chi\in \Omega:\  x\in B_C(v,\chi)\}\big)=\frac 16. 
\end{equation}
Finally, every  $x,y\in X$ for which $\min\{\Lambda(x),\Lambda(y)\}<\infty$ satisfy
\begin{equation}\label{eq:goal gamman P}
\gamma_n\times \prob\big(\{(v,\chi)\in \R^n\times \Omega:\ (x,y)\in A_C(v,\chi)\times B_C(v,\chi)\}\big) \gtrsim \exp\left(-9e^{4\alpha}\frac{\min\{\Lambda(x)^2,\Lambda(y)^2\}}{\|f(x)-f(y)\|_2^2}C^2 \right).
\end{equation}
Thus, $\gamma_n\times \prob\big(\{(v,\chi)\in \R^n\times \Omega:\ (x,y)\in A_C(v,\chi)\times B_C(v,\chi)\}\big) \gtrsim e^{-9e^{4\alpha}C^2}$ if $\|f(x)-f(y)\|_2\ge \min\{\Lambda(x),\Lambda(y)\}$. 
\end{lemma}

\begin{proof} Lemma~\ref{lem:multiscale representation} permits $\Lambda$ to take the value $\infty$  because this will be convenient when Lemma~\ref{lem:multiscale representation} will be applied, but it corresponds to a  degenerate situation. To isolate the main point of Lemma~\ref{lem:multiscale representation},  denote
\begin{equation}\label{eq:ded xinfty} 
X_\infty\eqdef\{x\in X:\ \Lambda(x)=\infty\}\qquad\mathrm{and}\qquad X_{<\infty}\eqdef X\setminus X_\infty.
\end{equation}
We will first treat the (more meaningful) subset $X_{<\infty}$ of $X$.

   For every $r\in [0,1]$ and $i\in \Z$ define $Y_i(r)=Y_i^\alpha(r)\subset X_{<\infty}$ by
\begin{equation}\label{eq:def Yir}
Y_i(r)\eqdef \Big\{x\in X_{<\infty}:\ e^{3\alpha(i+r)-2\alpha}\le \Lambda(x)<e^{3\alpha(i+r)}\Big\}=\bigg\{x\in X_{<\infty}:\ \frac1{3\alpha}\log \Lambda(x)-r\in \Big[i-\frac23,i\Big)\bigg\}.
\end{equation}
Note that $\{Y_i(r)\}_{i\in \Z}$ are pairwise disjoint subsets of $X_{<\infty}$. Writing also 
\begin{equation}\label{eq:def Yr}
Y(r)=Y^\alpha(r)\eqdef \bigcup_{i\in \Z} Y_i(r), 
\end{equation}
it is straightforward to compute that
\begin{equation}\label{eq:def 2/3}
\forall x\in X_{<\infty},\qquad \mathbb{U}\big(\{r\in [0,1]:\ x\in Y(r)\}\big)=\frac23. 
\end{equation}
Recalling the notation introduced in Lemma~\ref{lem:basic random union of strips},  every $\theta=(\theta_i)_{i\in \Z}\in [0,1]^\Z$ and $r\in [0,1]$, we now define 
\begin{align}\label{eq:def layered strip Lambda}
\begin{split}
E_{C}(v, \theta,r)&\eqdef \bigcup_{i\in \Z} E_C^i(v,\theta_i,r),\qquad\mathrm{where}\qquad E_C^i(v,\theta_i,r)\eqdef Y_i(r)\cap f^{-1}\big(L_{e^{3\alpha(i+r)}C}  (v,\theta_i)\big)\subset X_{<\infty},\\ F_C(v,\theta,r)&\eqdef \bigcup_{i\in \Z}F_C^i(v,\theta_i,r),\qquad\mathrm{where}\qquad  F_C^i(v,\theta_i,r)\eqdef Y_i(r)\cap f^{-1}\big(R_{e^{3\alpha(i+r)}C}  (v,\theta_i)\big)\subset X_{<\infty}.
\end{split}
\end{align}

Fix  $\theta=(\theta_i)_{i\in \Z}\in [0,1]^\Z$,  $v\in \R^n$, and  $r\in [0,1]$, as well as   $x,y\in X_{<\infty}$ that satisfy
\begin{equation}\label{eq:Lambda smooth assumption-later}
(x,y)\in E_C(v,\theta,r)\times F_C(v,\theta,r)\qquad\mathrm{and}\qquad e^{-\alpha}\Lambda(x)\le \Lambda(y)\le e^{\alpha}\Lambda(x),
\end{equation}
The first condition in~\eqref{eq:Lambda smooth assumption-later} and  the definitions~\eqref{eq:def layered strip Lambda} imply in particular that there are $i,j\in \Z$ such that 
\begin{equation}\label{eq:rewrite xy in product}
(x,y)\in Y_i(r)\times  Y_j(r)\qquad \mathrm{and}\qquad \big(f(x),f(y)\big)\in L_{e^{3\alpha(i+r)}C}  (v,\theta_i)\times R_{e^{3\alpha(j+r)}C}  (v,\theta_j).
\end{equation}
Together with the second condition in~\eqref{eq:Lambda smooth assumption-later}, the first condition in~\eqref{eq:rewrite xy in product} entails that necessarily $i=j$. Indeed, if $j\ge i+1$, then  we see from the definition~\eqref{eq:def Yir} that 
$$
\Lambda(y) \stackrel{y\in Y_j(r)}{>} e^{3\alpha(j+r)-2\alpha}\ge e^{3\alpha(i+1+r)-2\alpha} =e^\alpha \cdot e^{3\alpha(i+r)}\stackrel{x\in Y_i(r)}{\ge} e^\alpha \Lambda(x),
$$
in contradiction to~\eqref{eq:Lambda smooth assumption}.  Analogously, we cannot have $j\le i-1$. Thus necessarily $i=j$, and therefore 
\begin{equation}\label{eq:max is small}
\max\{\Lambda(x),\Lambda(y)\}<e^{3\alpha(i+r)},
\end{equation}
by the definition~\ref{eq:def Yir} of $Y_i(r)$. Using Lemma~\ref{lem:basic random union of strips}, the second condition in~\eqref{eq:rewrite xy in product}  now implies that 
\begin{equation}\label{eq:for < infty far proj}
|\langle v,f(x)-f(y)\rangle| \stackrel{\eqref{eq:separated projection random}}{>} e^{3\alpha(i+r)}C \stackrel{\eqref{eq:max is small}}{>}C\max\left\{\Lambda(x),\Lambda(y)\right\}.
\end{equation}
We have thus checked that the following implication holds:  
\begin{equation}
\forall x,y\in X_{<\infty},\qquad \eqref{eq:Lambda smooth assumption-later}\implies |\langle v,f(x)-f(y)\rangle| >C\max\left\{\Lambda(x),\Lambda(y)\right\}.
\end{equation}

Next, for every fixed $v\in \R^n$ and $x\in X_{<\infty}$ we have 
\begin{align}\label{eq:UZUA}
\begin{split}
(\mathbb{U}^\Z\times \mathbb{U})&\big(\{(\theta,r)\in [0,1]^\Z\times [0,1]:\  x\in E_C(v,\theta,r)\}\big)\\ &\stackrel{\eqref{eq:def layered strip Lambda}}{=} \sum_{i\in \Z} (\mathbb{U}^\Z\times \mathbb{U})\big(\{(\theta,r)\in [0,1]^\Z\times [0,1]:\ x\in Y_i(r)\ \mathrm{and}\ f(x)\in L_{e^{3\alpha(i+r)}C}  (v,\theta_i) \}\big)\\
&\ \ = \sum_{i\in \Z}  \int_0^1 \1_{\{x\in Y_i(r)\}}\bigg(\int_0^1 \1_{\{f(x)\in L_{e^{3\alpha(i+r)}C}  (v,\theta_i)\}} \ud \theta_i\bigg)\ud r\\
&\ \ = \sum_{i\in \Z}  \int_0^1 \1_{\{x\in Y_i(r)\}}\mathbb{U}\big(\{\theta\in [0,1]:\ f(x)\in L_{e^{3\alpha(i+r)}C}  (v,\theta)\}\big)\ud r\\
&\stackrel{\eqref{eq:inclusion measure}}{=} \sum_{i\in \Z} \frac14 \mathbb{U}\big(\{r\in [0,1]:\ x\in Y_i(r)\}\big)\\
&\stackrel{\eqref{eq:def Yr}}{=} \frac14 \mathbb{U}\big(\{r\in [0,1]:\ x\in Y(r)\}\big)\stackrel{\eqref{eq:def 2/3}}{=}\frac16. 
\end{split}
\end{align}
By the analogous reasoning,  for every $x\in X_{<\infty}$ we also have
\begin{equation}\label{eq:UZUB}
(\mathbb{U}^\Z\times \mathbb{U})\big(\{(\theta,r)\in [0,1]^\Z\times [0,1]:\  x\in F_C(v,\theta,r)\}\big)=\frac16.
\end{equation}

Finally, suppose that $i,j\in \Z$ satisfy $i\neq j$. If $x,y\in X_{<\infty}$, then for every $r\in [0,1]$ and $v\in \R^n$ we have 
\begin{align}\label{eq:1/16 indenity}
\begin{split}
\mathbb{U}^\Z\Big(\big\{\theta&\in [0,1]^\Z:\ (x,y)\in f^{-1}\big(L_{e^{3\alpha(i+r)}C}  (v,\theta_i)\big)\times f^{-1}\big( R_{e^{3\alpha(j+r)}C}  (v,\theta_j)\big)\big\}\Big)\\&= \mathbb{U} \big(\{\theta_i\in[0,1]:\ f(x) \in L_{e^{3\alpha(i+r)}C}  (v,\theta_i)\}\big)\mathbb{U} \big(\{\theta_j\in[0,1]:\ f(y) \in R_{e^{3\alpha(j+r)}C}  (v,\theta_j)\}\big)\stackrel{\eqref{eq:inclusion measure}}{=}\frac1{16}.
\end{split}
\end{align}
By multiplying~\eqref{eq:1/16 indenity} by $\1_{\{(x,y)\in Y_i(r)\times Y_j(r) \}}$ and then integrating with respect to $\gamma_n\times \mathbb{U}$, we see that 
\begin{align}\label{eq:when ij distinct}
\begin{split}
 (\gamma_n\times \mathbb{U}^\Z\times \mathbb{U})\big(\{(v,\theta,r)\in \R^n\times [0,1]^\Z\times [0,1]:\ &
(x,y)\in E_C^i(v,\theta,r)\times F_C^j(v,\theta,r) \}\big)\\
& \stackrel{\eqref{eq:def layered strip Lambda} \wedge\eqref{eq:1/16 indenity} }{=} 
\frac1{16} \mathbb{U}\big(\{r\in [0,1]:\ (x,y)\in Y_i(r)\times  Y_j(r)\}\big),
\end{split}
\end{align}
where we recall that the identity~\eqref{eq:when ij distinct} holds for every $x,y\in X_{<\infty}$ whenever $i,j\in \Z$ satisfy $i\neq j$.

At the same time, for every $i\in \Z$ and $r\in [0,1]$, if $x,y\in Y_i(r)$, then 
\begin{equation}\label{eq:use lower bound in Yir}
e^{3\alpha(i+r)}\stackrel{\eqref{eq:def Yir}}{\le} e^{2\alpha}\min\big\{\Lambda(x),\Lambda(y)\big\}.
\end{equation}
Therefore, for every $x,y\in X_{<\infty}$ and every fixed $r\in [0,1]$ we have 
\begin{align}\label{eq:on diagional ineq}
\begin{split}
&1_{\{x,y\in Y_i(r)\}}(\gamma_n\times \mathbb{U}^\Z)\Big(\big\{(v,\theta)\in \R^n\times [0,1]^\Z:\ (x,y)\in f^{-1}\big(L_{e^{3\alpha(i+r)}C}  (v,\theta_i)\big)\times f^{-1}\big( R_{e^{3\alpha(i+r)}C}  (v,\theta_i)\big)\big\}\Big)\\&=1_{\{x,y\in Y_i(r)\}}(\gamma_n\times \mathbb{U})\Big(\big\{(v,\theta_i)\in \R^n\times [0,1]:\ \big(f(x),f(y)\big)\in L_{e^{3\alpha(i+r)}C}  (v,\theta_i)\times  R_{e^{3\alpha(i+r)}C}  (v,\theta_i)\big\}\Big) \\&\stackrel{\eqref{eq:simulataneously in strips}}{\gtrsim} 1_{\{x,y\in Y_i(r)\}} \exp\bigg(-\frac{9e^{6\alpha(i+r)}C^2}{\|f(x)-f(y)\|_2^2}\bigg)\\&\stackrel{\eqref{eq:use lower bound in Yir}}{\ge} 1_{\{x,y\in Y_i(r)\}} \exp\bigg(-9e^{4\alpha}\frac{\min\{\Lambda(x)^2,\Lambda(y)^2\}}{\|f(x)-f(y)\|_2^2}C^2 \bigg). 
\end{split}
\end{align}
By integrating~\eqref{eq:on diagional ineq} with respect to $\mathbb{U}$, we conclude that for every $i\in \Z$ and every $x,y\in X_{<\infty}$, 
\begin{align}\label{eq:when i=j}
\begin{split}
 (\gamma_n\times \mathbb{U}^\Z\times \mathbb{U})&\big(\{(v,\theta,r)\in \R^n\times [0,1]^\Z\times [0,1]:\ 
(x,y)\in E_C^i(v,\theta,r)\times F_C^i(v,\theta,r) \}\big)\\
& \stackrel{\eqref{eq:def layered strip Lambda} \wedge\eqref{eq:on diagional ineq} }{\gtrsim } 
\exp\bigg(-9e^{4\alpha}\frac{\min\{\Lambda(x)^2,\Lambda(y)^2\}}{\|f(x)-f(y)\|_2^2}C^2 \bigg) \mathbb{U}\big(\{r\in [0,1]:\ x,y\in Y_i(r)\}\big).
\end{split}
\end{align}

By combining~\eqref{eq:when ij distinct} and~\eqref{eq:when i=j}, we see that the following estimate holds  for every $x,y\in X_{<\infty}$: 
\begin{align}\label{pairs before adding Xinfty}
\begin{split}
(\gamma_n\times &\mathbb{U}^\Z\times \mathbb{U})\big(\{(v,\theta,r)\in \R^n\times [0,1]^\Z\times [0,1]:\ (x,y)\in E_C(v,\theta,r)\times  F_C(v,\theta,r) \}\big)\\& \stackrel{\eqref{eq:def layered strip Lambda}}{=} \sum_{\substack {i,j\in \Z\\ i\neq j}}  (\gamma_n\times \mathbb{U}^\Z\times \mathbb{U})\big(\{(v,\theta,r)\in \R^n\times [0,1]^\Z\times [0,1]:\ 
(x,y)\in E_C^i(v,\theta,r)\times F_C^j(v,\theta,r) \}\big)\\&\qquad\qquad \quad +\sum_{i\in \Z} (\gamma_n\times \mathbb{U}^\Z\times \mathbb{U})\big(\{(v,\theta,r)\in \R^n\times [0,1]^\Z\times [0,1]:\ 
(x,y)\in E_C^i(v,\theta,r)\times F_C^i(v,\theta,r) \}\big)\\
&\!\!\!\!\!\!\! \!\! \! \stackrel{\eqref{eq:when ij distinct}\wedge \eqref{eq:when i=j}}{\gtrsim} \exp\bigg(-9e^{4\alpha}\frac{\min\{\Lambda(x)^2,\Lambda(y)^2\}}{\|f(x)-f(y)\|_2^2}C^2 \bigg)  \sum_{i,j\in \Z} \mathbb{U}\big(\{r\in [0,1]:\ (x,y)\in Y_i(r)\times  Y_j(r)\}\big)\\
&\stackrel{\eqref{eq:def Yr}}{=} \exp\bigg(-9e^{4\alpha}\frac{\min\{\Lambda(x)^2,\Lambda(y)^2\}}{\|f(x)-f(y)\|_2^2}C^2 \bigg)   \mathbb{U}\big(\{r\in [0,1]:\ x,y\in Y(r)\}\big)\\
&\ \ \ge \exp\bigg(-9e^{4\alpha}\frac{\min\{\Lambda(x)^2,\Lambda(y)^2\}}{\|f(x)-f(y)\|_2^2}C^2 \bigg)   \left(\mathbb{U}\big(\{r\in [0,1]:\ x\in Y(r)\}\big)+\mathbb{U}\big(\{r\in [0,1]:\ y\in Y(r)\}\big)-1\right)\\
&\stackrel{\eqref{eq:def 2/3}}{=}\frac13 \exp\bigg(-9e^{4\alpha}\frac{\min\{\Lambda(x)^2,\Lambda(y)^2\}}{\|f(x)-f(y)\|_2^2}C^2 \bigg).
\end{split}
\end{align}

The above considerations yield for  $E_C(\cdot),F_C(\cdot)$ the conclusions of Lemma~\ref{lem:multiscale representation} when $x,y\in X_{<\infty}$, but it is simple to extend those to every $x,y\in X$ as follows. Let $\bbQ$ be the probability measure on the $3$-point set  $\{1,2,3\}$ that is given by $\bbQ(1)=2/3$ and $\bbQ(2)=\bbQ(3)=1/6$.  Consider the  (Polish) probability space 
\begin{equation}\label{eq:def our Omega product polish}
(\Omega,\prob)\eqdef \big([0,1]^\Z\times [0,1]\times \{1,2,3\},\mathbb{U}^\Z\times \mathbb{U}\times \bbQ\big),
\end{equation}
and, recalling~\eqref{eq:ded xinfty}, define $A_C,B_C:\R^n\times \Omega\to 2^X$ by setting for every $(v,\theta,r,k)\in \R^n\times [0,1]^\Z\times [0,1]\times \{1,2,3\}$,
\begin{equation}\label{eq:include infinity}
 \big(A_C(v,\theta,r,k),B_C(v,\theta,r,k)\big)\eqdef \left\{\begin{array}{ll} \big(E_C(v,\theta,r),F_C(v,\theta,r)\big)&\mathrm{if}\ k=1,\\
\big(E_C(v,\theta,r)\cup X_\infty,F_C(v,\theta,r)\big)&\mathrm{if}\ k=2,\\
\big(E_C(v,\theta,r),F_C(v,\theta,r)\cup X_\infty\big)&\mathrm{if}\ k=3.
\end{array}\right.
\end{equation}
For these definitions, the measurability requirements of Lemma~\ref{lem:multiscale representation} are immediate\footnote{Recall that we assumed in  Lemma~\ref{lem:multiscale representation} that $X$ is  finite, but we did this only to remove the need to mention any assumption on $f$ and $\Lambda$; if $X$ is an infinite Polish space, then~\eqref{eq:include infinity} will define Borel-measurable set valued mappings if we impose the assumptions that $f$ and $\Lambda$ are Borel measurable.}. 

Suppose first that~\eqref{eq:Lambda smooth assumption} is satisfied for some  $x,y\in X$, $v\in \R^n$ and $\chi=(\theta,r,k)\in \Omega$. If $\{x,y\}\cap X_\infty\neq \emptyset$, then by the second requirement in~\eqref{eq:Lambda smooth assumption} we have $\Lambda(x)=\Lambda(y)=\infty$, i.e.,  $x,y\in X_\infty$, but then by~\eqref{eq:include infinity}    we see that the first requirement in~\eqref{eq:Lambda smooth assumption}  cannot hold. We therefore necessarily have $x,y\in X_{<\infty}$, in which case~\eqref{eq:Lambda smooth assumption-later} holds thanks to~\eqref{eq:include infinity} and~\eqref{eq:Lambda smooth assumption},  so the desired conclusion~\eqref{eq:Cmax goal} follows from~\eqref{eq:for < infty far proj}.

Next, to verify~\eqref{eq:1/6 prob equality} fix $v\in \R^n$ and $x\in X$. If $x\in X_{<\infty}$, then 
$$
\prob \big(\{(\theta,r,k)\in \Omega:\ x\in A_C(v,\theta,r,k)\}\big)\stackrel{\eqref{eq:def our Omega product polish}\wedge \eqref{eq:include infinity}}{=} (\mathbb{U}^\Z\times \mathbb{U})\big(\{(\theta,r)\in [0,1]^\Z\times [0,1]:\  x\in E_C(v,\theta,r)\}\big)\stackrel{\eqref{eq:UZUA}}{=}\frac16,
$$
and analogously thanks to~\eqref{eq:UZUB} we also have
$$
\prob \big(\{(\theta,r,k)\in \Omega:\ x\in B_C(v,\theta,r,k)\}\big)=\frac16. 
$$ 
On the other hand, if $x\in X_\infty$, then
$$
\prob \big(\{(\theta,r,k)\in \Omega:\ x\in A_C(v,\theta,r,k)\}\big)\stackrel{\eqref{eq:def our Omega product polish}\wedge \eqref{eq:include infinity}}{=}  \bbQ[k=2]=\frac16,
$$ 
and
$$
\prob \big(\{(\theta,r,k)\in \Omega:\ x\in B_C(v,\theta,r,k)\}\big)\stackrel{\eqref{eq:def our Omega product polish}\wedge \eqref{eq:include infinity}}{=}  \bbQ[k=3]=\frac16. 
$$ 
Altogether, this completes the verification of~\eqref{eq:1/6 prob equality}.  

To complete the proof of Lemma~\ref{lem:multiscale representation}, it remains to check that~\eqref{eq:goal gamman P} holds for every $x,y\in X$ such that $\min\{\Lambda(x),\Lambda(y)\}<\infty$, i.e., $|\{x,y\}\cap X_{<\infty}|\ge 1$. Indeed, if $x,y\in X_{<\infty}$, then   thanks to~\eqref{eq:def our Omega product polish} and~\eqref{eq:include infinity}, we see that~\eqref{eq:goal gamman P} coincides with~\eqref{pairs before adding Xinfty}. If $x\in X_{<\infty}$ and $y\in X_\infty$, then 
\begin{align*}
\gamma_n\times \prob\big(\{(v,\theta,r,k)\in \R^n\times \Omega:\ &(x,y)\in A_C(v,\theta,r,k)\times B_C(v,\theta,r,k)\}\big) \\
&\stackrel{\eqref{eq:include infinity}}{=} \gamma_n\times \prob\big(\{(v,\theta,r,k)\in \R^n\times \Omega:\  x\in E_C(v,\theta,r)\quad\mathrm{and}\quad k=3\}\big)\\
&\stackrel{\eqref{eq:def our Omega product polish}}{=} \frac16 \int_{\R^n} (\mathbb{U}^\Z\times \mathbb{U})\big(\{(\theta,r)\in [0,1]^\Z\times [0,1]:\  x\in E_C(v,\theta,r)\}\big) \ud\gamma_n(v)\\
&\stackrel{\eqref{eq:UZUA}}{=}\frac{1}{36}\gtrsim \exp\bigg(-9e^{4\alpha}\frac{\min\{\Lambda(x)^2,\Lambda(y)^2\}}{\|f(x)-f(y)\|_2^2}C^2 \bigg), 
\end{align*}
i.e., \eqref{eq:goal gamman P}  holds in this case. The remaining case $(x,y) \in X_{<\infty}\times X_\infty$ is derived analogously using~\eqref{eq:UZUB}.  
\end{proof}

\begin{proof}[Proof of Proposition~\ref{prob:graphical prob multiscale}] Set $\alpha=\log 2$, for assumption~\eqref{eq:moderate variation along edges} of Proposition~\ref{prob:graphical prob multiscale} to match assumption~\eqref{eq:Lambda smooth assumption} of Lemma~\ref{lem:multiscale representation}, which we will next apply.\footnote{The proof shows that if we replace the factor $2$ in~\eqref{eq:moderate variation along edges} by $e^\alpha$ for any $\alpha>0$, then the conclusions of Proposition~\ref{prob:graphical prob multiscale} hold with the universal constant $\kappa$ in~\eqref{eq:super gaussian omega} replaced by $9e^{4\alpha}$.} Let $V^1,\ldots,V^m\subset V$ be the connected components of $\sfG$. For each $s\in [m]$, let
$
A_C^s,B_C^s:\R^n\times \Omega^s\to 2^{V_s}\subset 2^V
$ 
the random subsets of $V^s$ that Lemma~\ref{lem:multiscale representation} provides (for $X=V^s$, the $C$ from the statement of Proposition~\ref{prob:graphical prob multiscale}, the above $\alpha$, and slightly abusing notation by identifying $f,\Lambda$ with their restrictions to $V^s$), where the underlying probability space is now denoted
 $(\Omega^s,\prob^s)$. 
 
Consider the probability space  
$$(\Omega,\prob)\eqdef  (\Omega^1\times \ldots\times \Omega^s,\prob^1\times \ldots\times \prob^s).$$ 
 Define random subsets $A_C,B_C:\R^n\times \prob\to 2^X$  of $V$ by setting
\begin{equation}\label{eq:random union}
\forall v\in \R^n,\ \forall \chi=(\chi_1,\ldots,\chi_m)\in \Omega,\qquad  A_C(v,\chi)\eqdef \bigcup_{s=1}^m A_C^s(v,\chi_s)\qquad\mathrm{and}\qquad B_C(v,\chi)\eqdef \bigcup_{s=1}^m B_C^s(v,\chi_s).
\end{equation}
In other words, for each fixed $v\in \R^n$ the random pairs of subsets 
$
(A_C^1(v,\cdot),B_C^1(v,\cdot)),\ldots,(A_C^m(v,\cdot),B_C^m(v,\cdot))$ are chosen independently, and then $A_C(v,\cdot)$ and $B_C(v,\cdot)$ are the  unions of their first and second coordinates, respectively.  We will proceed to check that with positive probability (indeed, in expectation), the sets in~\eqref{eq:random union} when $\chi$ is distributed according to $\prob$ satisfy the stated conclusions of Proposition~\ref{prob:graphical prob multiscale}. 

In fact, requirement~\eqref{eq:separation within connected components} holds for every $\chi\in \Omega$, since if $(x,y)\in A_C(v,\chi)\times B_C(v,\chi)$ and $\{x,y\}\in E$, then $x,y$ are in the same connected component of $\sfG$, so from~\eqref{eq:random union} we know that $(x,y)\in  A_C^s(v,\chi)\times B_C^s(v,\chi)$ for some $s\in [m]$. Also~\eqref{eq:Lambda smooth assumption} holds by~\eqref{eq:moderate variation along edges}, so conclusion~\eqref{eq:Cmax goal} of Lemma~\ref{lem:multiscale representation} for  $(\Omega^s,\prob^s)$ gives~\eqref{eq:separation within connected components}.

For each $x\in V$ let $s(x)$ denote the unique $s\in [m]$ to which $x$ belongs. We then have
\begin{align}\label{eq:pass to minimum}
\begin{split}
\int_{\Omega} &\left(\int_{\R^n} \omega\big(A_C(v,\chi)\times B_C(v,\chi)\big)\ud\gamma_n(v)\right)\ud \prob(\chi)\\&\ \ \ =\sum_{(x,y)\in V}\omega(x,y) (\gamma_n\times \prob)\big(\{(v,\chi)\in \R^n\times \Omega:\ (x,y)\in A_C(v,\chi)\times B_C(v,\chi)\}\big)\\
&\stackrel{\eqref{eq:random union}}{=}\sum_{(x,y)\in V}\omega(x,y) (\gamma_n\times \prob)\big(\{(v,\chi)\in \R^n\times \Omega:\ (x,y)\in A_C^{s(x)}(v,\chi_{s(x)})\times B_C^{s(y)}(v,\chi_{s(y)})\}\big)\\
&\ \ \ge \min_{x,y\in V} (\gamma_n\times \prob)\big(\{(v,\chi)\in \R^n\times \Omega:\ (x,y)\in A_C^{s(x)}(v,\chi_{s(x)})\times B_C^{s(y)}(v,\chi_{s(y)})\}\big),
\end{split}
\end{align}
where the last step of~\eqref{eq:pass to minimum} holds because $\omega$ is a probability measure. For  $x,y\in V$, if $s(x)=s(y)=s$, then 
\begin{align}\label{eq:diagonal lower bound}
\begin{split}
(\gamma_n\times \prob)&\big(\{(v,\chi)\in \R^n\times \Omega:\ (x,y)\in A_C^{s(x)}(v,\chi_{s(x)})\times B_C^{s(y)}(v,\chi_{s(y)})\}\big)\\& \stackrel{\eqref{eq:random union}}{=}(\gamma_n\times \prob^{s})\big(\{(v,\chi_{s})\in \R^n\times \Omega:\ (x,y)\in A_C^{s}(v,\chi_{s})\times B_C^{s}(v,\chi_{s})\}\big)  \stackrel{\eqref{eq:goal gamman P}\wedge \eqref{eq:min assumption in proposition}}{\gtrsim} e^{-9e^{4\alpha}C^2}.
\end{split}
\end{align}
On the other hand, if $s(x)\neq s(y)$, then   $A_C^{s(x)}(v,\chi_{s(x)})$ and $B_C^{s(y)}(v,\chi_{s(y)})$ are independent  random subsets of $X$ for each fixed $v\in \R^n$, so we have the following crude estimate: 
\begin{align}\label{eq:off diagoinal crude}
\begin{split}
(\gamma_n&\times \prob)\big(\{(v,\chi)\in \R^n\times \Omega:\ (x,y)\in A_C^{s(x)}(v,\chi_{s(x)})\times B_C^{s(y)}(v,\chi_{s(y)})\}\big)\\&\stackrel{\eqref{eq:random union}}{=} \int_{\R^n}\prob^{s(x)}\big(\{\chi_{s(x)}\in \Omega^{s(x)}:\ x\in A_C^{s(x)}(v,\chi_{s(x)})\}\big)\prob^{s(y)}\big(\{\chi_{s(y)}\in  \Omega^{s(y)}:\ y\in A_C^{s(y)}(v,\chi_{s(y)})\}\big)\ud\gamma_n(v)\\
&\stackrel{\eqref{eq:1/6 prob equality}}{=}\frac{1}{36}\gtrsim e^{-9e^{4\alpha}C^2}.
\end{split}
\end{align}
A substitution of~\eqref{eq:diagonal lower bound} and~\eqref{eq:off diagoinal crude} into~\eqref{eq:pass to minimum} shows that
$$
\int_{\Omega} \left(\int_{\R^n} \omega\big(A_C(v,\chi)\times B_C(v,\chi)\big)\ud\gamma_n(v)\right)\ud \prob(\chi)\gtrsim e^{-9e^{4\alpha}C^2}. 
$$
Hence there exists $\chi^\omega\in \Omega$ such that if we denote $A(v)=A_C(v,\chi^\omega)$ and $B(v)=B_C(v,\chi^\omega)$ for $v\in \R^n$, then 
$$
\int_{\R^n} \omega\big(A_C^\omega(v)\times B_C^\omega(v)\big)\ud \gamma_n(v) \gtrsim e^{-9e^{4\alpha}C^2}. 
$$
This completes the proof of~\eqref{eq:super gaussian omega}, with the (concrete, but far from optimal) value $\kappa=9e^{4\log2}=144$. 
\end{proof}

\section{Universally compatible compression of local growth  proximity graphs}\label{sec:compression}

 The following is a key construction that underlies our proof of Proposition~\ref{thm:universally compatible compression}:

\begin{construction}\label{construction:compression}Let $(\SS,d_\SS)$ be a finite metric space and fix a function $\vartheta:\SS\to \R$. For each $\xi\in \R$ denote 
\begin{equation}\label{eq:def level set}
S_\xi=S_\xi(\vartheta)\eqdef \vartheta^{-1}\big((-\infty,\xi]\big)=\big\{x\in \SS: \vartheta(x)\le \xi\big\}, 
\end{equation}
namely, $S_\xi$ is the $\xi$-sublevel set of $\vartheta$. Let $\ell \in \N$ and   $\xi_1<\xi_2<\ldots<\xi_\ell$ be the increasing rearrangement of the values of $\vartheta$, i.e., $\vartheta(\SS)=\{\xi_1,\ldots,\xi_\ell\}$. Thus, 
$$
S_{\xi_1}\subsetneq S_{\xi_2}\subsetneq \ldots\subsetneq S_{\xi_\ell}=\SS.$$ 
Fix $\tau>0$. For each $i\in [\ell]$ define inductively $\NN_{\xi_i}\subset S_{\xi_i}$ as follows. Let $\NN_{\xi_1}$ be a subset of $S_{\xi_1}$ that is maximal with respect to inclusion among all those  $\NN\subset S_{\xi_1}$ such that $d_\SS(a,b)>2\tau$ for every distinct $a,b\in \NN$. If $i\in \{2,\ldots,\ell\}$, then assume inductively that $\NN_{\xi_{i-1}}$  has already been defined and let $\NN_{\xi_i}$ be  a subset of $S_{\xi_i}$ that is maximal with respect to inclusion among all those  $\NN\subset S_{\xi_i}$ that contain $\NN_{\xi_{i-1}}$ and such that  $d_\SS(a,b)>2\tau$ for every distinct $a,b\in \NN$. We thus have (by design), 
\begin{equation}\label{eq:nets increasing}
\NN_{\xi_1}\subset \NN_{\xi_2}\subset \ldots \subset \NN_{\xi_\ell}.
\end{equation}
Observe also that  $\NN_{\xi_i}$ is $(2\tau)$-dense in $S_{\xi_i}$ for  every $i\in [\ell]$, i.e., $d_\SS(w,\NN_{\xi_i})\le 2\tau$ for every $w\in S_{\xi_i}$ (otherwise, we could add $w$ to $\NN_{\xi_i}$, in contradiction to its maximality with respect to inclusion).

Next, define a rounding function $q:\SS\to \SS$ as follows. Given $w\in \SS$, fix any minimizer $w_{\min}$ of $\vartheta$ on $B_\SS(w,5\tau)$, i.e.,  $w_{\min}\in B_\SS(w,5\tau)$ and   $\vartheta(w_{\min})\le \vartheta(z)$ for every $z\in B_\SS(w,5\tau)$. Because $w_{\min}\in S_{\vartheta(w_{\min})}$ and $\NN_{\vartheta(w_{\min})}$ is $(2\tau)$-dense in $S_{\vartheta(w_{\min})}$, we can define $q(w)$ to be any  point in $\NN_{\vartheta(w_{\min})}$ with $d_\SS(q(w),w_{\min})\le 2\tau$. 
\end{construction}


The following lemma derives basic properties of the objects  defined in Construction \ref{construction:compression}:

\begin{lemma}\label{lem:rounding function} Under   the definitions of Construction~\ref{construction:compression}, $q$ and $\{\NN_{\xi_i}\}_{i=1}^\ell$ have the following properties: 
\begin{equation}\label{eq:displacement of q}
\forall w\in \SS,\qquad d_\SS(q(w),w)\le 7\tau\qquad\mathrm{and}\qquad q\big(B_\SS(w,5\tau)\big)\subset \NN_{\vartheta(w)}.
\end{equation}
Consequently, 
\begin{equation}\label{eq:q inclusion}
\forall x,y\in \SS,\qquad d_\SS(x,y)\le 3\tau\implies  q \big(B_\SS (y,2\tau)\big) \subset B_\SS(y,9\tau)\cap \NN_{\vartheta(x)}.
\end{equation}
\end{lemma}

\begin{proof} For~\eqref{eq:displacement of q}, fix $w\in \SS$. The first part of~\eqref{eq:displacement of q} is a direct consequence of Construction~\ref{construction:compression} and the triangle inequality, since for every $w\in \SS$ we have $d_\SS(w_{\min},w)\le 5\tau$ and $d_\SS(q(w),w_{\min})\le 2\tau$, by definition. For the second part of~\eqref{eq:displacement of q}, observe that  for every $z\in B_\SS(w,5\tau)$ we have $w\in B_\SS(z,5\tau)$ and $z_{\min}$ is the minimizer of $\vartheta$ on $B_\SS(z,5\tau)$, so $\vartheta(z_{\min})\le \vartheta(w)$. Hence, $\NN_{\vartheta(z_{\min})}\subset \NN_{\vartheta(w)}$ by~\eqref{eq:nets increasing}. But $q(z)\in \NN_{\vartheta(z_{\min})}$ by definition, so $q(z)\in \NN_{\vartheta(w)}$. As $z$ is an arbitrary point in $B_\SS(w,5\tau)$, this proves  the second part of~\eqref{eq:displacement of q}.  

\eqref{eq:q inclusion} is a formal consequence of~\eqref{eq:displacement of q} and the triangle inequality,   since if $x,y\in \SS$ satisfy $d_\SS(x,y)\le 3\tau$, then $B_\SS(y,2\tau)\subset B_\SS(x,5\tau)$, whence $q(B_\SS(y,2\tau))\subset q(B_\SS(x,5\tau))\subset \NN_{\vartheta(x)}$ by the case $w=x$ of the second part of~\eqref{eq:displacement of q}, and for every $w\in B_\SS(y,2\tau)$ we have $d_\SS(q(w),y)\le d_\SS(q(w),w)+d_\SS(w,y)\le 7\tau+2\tau=9\tau$ using the first part of~\eqref{eq:displacement of q}, so we also have $q(B_\SS(y,2\tau))\subset B_\SS(y,9\tau)$. 
\end{proof}

The significance of incorporating the rounding function $q$ of Construction~\ref{construction:compression}  arises from the following lemma, which bounds the sizes of sets such as those that appear in the right hand side of~\eqref{eq:q inclusion}:

\begin{lemma}\label{lem:h inverse} Fix $\tau>0$. Let $(\SS,d_\SS)$ be a finite metric space. Suppose that  $\mathfrak{d}: \SS \to (0, \infty)$ and $\mathfrak{n}: \SS \to (0, \infty)$ are functions that  have  the following property. For every $x\in \SS$ and every subset $\TT$ of  $B_\SS(x,18\tau)$ with the property that every distinct $y,z\in \TT$ satisfy  $d_\SS(y,z)>2\tau$, we have\footnote{Thus, in particular, $\mathfrak{d}(x)\le \mathfrak{n}(x)$ for every $x\in \SS$, as seen by considering $\TT=\{x\}$.} 
\begin{equation}\label{eqn:measure monotone}
 \sum_{y\in \TT} \mathfrak{d}(y) \le \mathfrak{n}(x).
\end{equation}
Consider the setting of Construction~\ref{construction:compression} applied to $(\SS,d_\SS)$ and $\tau$, and with  $\vartheta=\vartheta_{\mathfrak{d},\mathfrak{n}}:\SS\to [1,\infty)$ given by 
\begin{equation}\label{eq:rho of h}
\forall x\in \SS,\qquad \vartheta (x)\eqdef  \frac{\mathfrak{n}(x)}{\mathfrak{d}(x)}.
\end{equation}
Then, for every $\xi\in \vartheta(\SS)$ and every  $\emptyset \neq \sub\subset \MM$ with $\diam_\SS(\sub)\le 18\tau$ we have 
\begin{equation}\label{eq:cardinality of special net}
|\sub\cap \NN_{\xi} |\le \xi.
\end{equation}
\end{lemma}

Before passing to the proof of Lemma~\ref{lem:h inverse}, we record in the following simple remark the  examples that we will use in the ensuing proof of Proposition~\ref{thm:universally compatible compression} of (numerator functions) $\mathfrak{n}: \SS \to (0, \infty)$  and  (denominator functions) $\mathfrak{d}: \SS \to (0, \infty)$ as  in~\eqref{eq:rho of h} that satisfy  assumption~\eqref{eqn:measure monotone} of Lemma~\ref{lem:h inverse}:

\begin{remark}\label{rem:basic example} In the setting of Lemma~\ref{lem:h inverse}, for every $\mathfrak{d}: \SS \to (0, \infty)$ and $x\in \SS$ define 
$\mathfrak{d}_*(x)$ to be the maximum of $\sum_{y\in \TT} \mathfrak{d}(y)$ over all possible $\TT\subset B_\SS(x,18\tau)$ that satisfy  $d_\SS(y,z)>2\tau$ for every distinct $y,z\in \TT$. Then, by design the smallest $\mathfrak{n}: \SS \to (0, \infty)$ that satisfies~\eqref{eqn:measure monotone} is $\mathfrak{d}_*:\SS\to (0,\infty)$, i.e., assumption~\eqref{eqn:measure monotone} of Lemma~\ref{lem:h inverse} holds if and only if $\mathfrak{n}\ge \mathfrak{d}_*$ point-wise. A simple way to achieve this occurs   when  $(\MM,d_\MM)$ is a metric space that contains $(\SS,d_\SS)$ isometrically, i.e., $\SS\subset \MM$ and the restriction of $d_\MM$ to $\SS\times \SS$ coincides with $d_\SS$. If $\mu$ is any nondegenerate measure on the super-space $\MM$, then consider the functions $\mathfrak{d},\mathfrak{n}: \SS \to (0, \infty)$  that are given by setting $\mathfrak{d}(x)=\mu(B_\MM(x,\tau))$ and $\mathfrak{n}(x)=\mu(B_\MM(x,19\tau))$ for each $x\in \SS$; note that we are evaluating here the measure $\mu$ on balls of the super-space $\MM$ which could contain points that do not belong to $\SS$ (this nuance will be needed later). For any $x\in \SS$, if $\TT\subset \SS$ is such that $d_\SS(y,z)>2\tau$ for every distinct $y,z\in \TT$, then by the triangle inequality for $d_\MM$ the balls $\{B_\MM(y,\tau)\}_{y\in \TT}$ (in the super-space $\MM$) are pairwise disjoint. If also $\TT\subset B_\SS(x,18\tau)$, then $\bigcup_{y\in \TT} B_\MM(y,\tau)\subset B_\MM(x,19\tau)$  by the triangle inequality for $d_\MM$, so~\eqref{eqn:measure monotone} holds. 
\end{remark}

\begin{proof}[Proof of Lemma~\ref{lem:h inverse}] As $\xi\in \vartheta(\SS)\subset [1,\infty)$, the conclusion of Lemma~\ref{lem:h inverse} is immediate if $|\sub\cap \NN_\xi|\le 1$. We may therefore assume that $\sub\cap \NN_\xi\neq \emptyset$. By Construction~\ref{construction:compression} we know that $d_\SS(y,z)>2\tau$ every distinct $y,z\in \NN_\xi$. Since $\diam_\SS(\sub)\le 18\tau$, we also have $\sub \subset B_\SS(x,18\tau)$  for every $x\in \sub$. We can therefore apply the assumption on $\mathfrak{d},\mathfrak{n}$ in Lemma~\ref{lem:h inverse} for $\TT=\sub\cap \NN_\xi$ to get that
\begin{equation}\label{eq:apply pacjing axiom}
\forall x\in \sub,\qquad \sum_{y\in \sub\cap \NN_\xi} \mathfrak{d}(y)\stackrel{\eqref{eqn:measure monotone}}{\le}  \mathfrak{n}(x). 
\end{equation}
But, Construction~\ref{construction:compression} ensures that $\NN_\xi\subset S_\xi$, so recalling~\eqref{eq:def level set} we have $\vartheta(x)\le \xi$ for every $x\in  \NN_\xi$. Thus,
\begin{equation}\label{eq:apply level set}
\forall x\in \NN_\xi,\qquad  \mathfrak{n}(x)\stackrel{\eqref{eq:rho of h}}{\le} \xi\mathfrak{d}(x). 
\end{equation}
Consequently,
\begin{equation}\label{eq:sum less than minimizer}
\forall x\in \sub\cap \NN_\xi,\qquad \sum_{y\in \sub\cap \NN_\xi} \mathfrak{d}(y)\stackrel{\eqref{eq:apply pacjing axiom}\wedge \eqref{eq:apply level set}}{\le}  \xi\mathfrak{d}(x). 
\end{equation}
Let $x_*\in \sub\cap\NN_\xi$ be the minimizer of $\mathfrak{d}$ on $\sub\cap \NN_\xi$, i.e., $\mathfrak{d}(x_*)\le \mathfrak{d}(y)$  for every $y\in \sub\cap \NN_\xi$. Then, \eqref{eq:sum less than minimizer} applied to  $x_*$ implies that  $|\sub\cap \NN_\xi|\mathfrak{d}(x_*) \le \xi\mathfrak{d}(x_*)$. As $\mathfrak{d}(x_*)>0$, we conclude that $|\sub\cap \NN_\xi|\le \xi$. 
\end{proof}

We will use Lemma~\ref{lem:h inverse}  via the following corollary of it which provides a bound on Gaussian processes that arise from the composition of arbitrary mappings $\f:\SS\to \R^m$ with the rounding function $q$:

\begin{corollary}\label{cor:process} Continuing with the notation and assumptions of Lemma~\ref{lem:h inverse}, fix $m\in \N$ and  $\f:\SS\to \R^m$. Then,  the following estimate holds for every $y\in \SS$ and every  $\UU \subset B_{\SS}(y,2\tau)$:
\begin{equation}\label{eq:int max in lemma}
\int_{\R^m} \Big(\max_{z\in \UU}|\langle v,\f\circ q(z)-\f\circ q(y)\rangle|\Big)\ud\gamma_n(v)\lesssim  \Big(\max_{z\in \UU}\|\f\circ q(z)-\f\circ q(y)\|_2\Big) \min_{w \in B_\SS (y,3\tau)} \sqrt{\log \frac{\mathfrak{n}(w)}{\mathfrak{d}(w)}}.
\end{equation}
\end{corollary}

\begin{proof} For every $y,w\in \SS$, by applying Lemma~\ref{lem:h inverse} with $\xi=\vartheta(w)$ and $\sub=B_\SS(y,9\tau)$, we see that  
\begin{equation}\label{eq:use fpor entire ball}
\forall y,w\in \SS,\qquad \big|B_\SS(y,9\tau)\cap \NN_{\vartheta(w)}\big| \stackrel{\eqref{eq:cardinality of special net}}{\le } \vartheta(w)
\stackrel{\eqref{eq:rho of h}}{=}\frac{\mathfrak{n}(w)}{\mathfrak{d}(w)}.
\end{equation}
By combining~\eqref{eq:use fpor entire ball} with conclusion~\eqref{eq:q inclusion} of Lemma~\ref{lem:rounding function}, we see that
\begin{equation}\label{eq:q of ball}
\forall y,w\in \SS,\qquad d_\SS(y,w)\le 3\tau\implies    \big|q\big(B_\SS (y,2\tau)\big)\big|\le \frac{\mathfrak{n}(w)}{\mathfrak{d}(w)}. 
\end{equation}
Fixing $y\in \SS$ and $\UU \subset B_{\SS}(y,2\tau)$ as in Corollary~\ref{cor:process}, we get from~\eqref{eq:q of ball} the following cardinality estimate: 
\begin{equation}\label{eq:phi circ}
|\f\circ q(\UU)|\le |q(\UU)|\le \big|q\big(B_\SS (y,2\tau)\big)\big|  \le \min_{w\in B_\SS(y,3\tau)}\frac{\mathfrak{n}(w)}{\mathfrak{d}(w)}. 
\end{equation}
By a simple bound on the expected maximum of a finite Gaussian process which appears in many places in the literature, including specifically in e.g.~\cite[Lemma~4.14]{Pis89}, we have\footnote{A mechanical inspection of the (elementary) proof of~\cite[Lemma~4.14]{Pis89} shows that the implicit universal constant in~\eqref{eq:maximum of gaussian process} can be taken to be less than $2$, and, in fact,  that it can be taken to be arbitrarily close to $\sqrt{2}$ as $|\f\circ q(\UU )|\to \infty$.  } 
\begin{equation}\label{eq:maximum of gaussian process}
\int_{\R^m} \Big(\max_{z\in \UU}|\langle v,\f\circ q(z)-\f\circ q(y)\rangle|\Big)\ud\gamma_n(v)\lesssim  \Big(\max_{z\in \UU}\|\f\circ q(z)-\f\circ q(y)\|_2\Big)\sqrt{\log |\f\circ q(\UU )|}.
\end{equation}
The desired conclusion~\eqref{eq:int max in lemma} follows from a substitution of~\eqref{eq:phi circ} into~\eqref{eq:maximum of gaussian process}.  
\end{proof}

Observe that the proof of Corollary~\ref{cor:process} makes the value of the compression (via composition with the rounding function $q$) that Construction~\ref{construction:compression} provides evident: If we would have considered $\f$ rather than $\f\circ q$ in~\eqref{eq:int max in lemma}, then  the same reasoning would have led to the factor $\min_{x \in B_\SS (y,3\tau)} \sqrt{\log (\mathfrak{n}(x)/\mathfrak{d}(x))}$ in~\eqref{eq:int max in lemma} being replaced by $\sqrt{\log |\UU|}\le \sqrt{\log |B_\SS(y,2\tau)|}$, which is insufficient for our proof of Theorem~\ref{thm:random zero}.

\begin{remark} If $h : [1, \infty) \to \R$ is monotone increasing, then one could consider Lemma~\ref{lem:h inverse} with the function $\vartheta$ in~\eqref{eq:rho of h} replaced by the function $(x\in \SS)\mapsto h(\mathfrak{n}(x)/\mathfrak{d}(x))$.  As Construction~\ref{construction:compression} only depends on the ordering of the values in $\vartheta$, it will output the same rounding function $\rho$, independently of $h$, and therefore Corollary~\ref{cor:process} will remain unchanged. In the overview of the proof that we presented in Section~\ref{sec:history}, we sketched the above compression step this way while referring to nets in the level sets of the function $\rho$ that is given in equation~\eqref{eq:our specific rho} of Proposition~\ref{thm:universally compatible compression}; using the notation of the statement Proposition~\ref{thm:universally compatible compression}, this corresponds to  considering  $h(t)=1+(\zeta/C)\sqrt{\log t}$ above,  with $\mathfrak{d}(x)=\mu(B_\MM(x,\tau))$     and $\mathfrak{n}(x)=\mu(B_\MM(x,19\tau))$. 
\end{remark}

We will soon explain how Construction~\ref{construction:compression} facilitates compatibility (per Definition~\ref{def:compatibility}) of proximity graphs of a metric space (as in e.g.~definition~\eqref{eq:geometric edges specific rho} of Proposition~\ref{thm:universally compatible compression}, though we will work in greater generality). The following lemma is the main way that we will use such proximity information:

\begin{lemma}\label{lem:ball comparison} Fix $\tau>0$. Suppose that $(\SS,d_\SS)$ is a finite metric space. Given $\rho:\SS\to [1,\infty)$, consider the graph $\sfG=\sfG_{d_\mathscr{S},\rho,\tau}=(\SS,E)$ whose vertex set is $\SS$ and whose edge set $E=E_{d_\SS,\rho,\tau}$ is defined by
\begin{equation}\label{eq:geometric edges general psi}
\forall x,y\in \SS,\qquad \{x,y\}\in E \iff d_\SS(x,y)\le \frac{\tau}{\min\{\rho(x),\rho(y)\}}. 
\end{equation}
Then,
\begin{equation}\label{eq:alpha=2}
\forall x\in \SS,\qquad B_{\sfG}\big(x,\widetilde{\rho}(x)+1\big)\subset B_\SS(x,2\tau)\qquad\mathrm{where}\qquad \widetilde{\rho}(x)\eqdef \min_{z\in B_\SS(x,2\tau)}\rho(z).
\end{equation}
\end{lemma}

\begin{proof} We will actually prove the following local Lipschitz property of the formal identity mapping from $(\SS,d_\sfG)$ to $(\SS,d_\SS)$, which holds for every $\alpha\ge 1$ and every $x,y\in \SS$:
\begin{equation}\label{eq:geometric graph bounds original metric}
 d_\sfG(x,y)\le (\alpha-1)\left(\min_{z\in B_\SS(x,\alpha\tau)}\rho(z)\right)+1\implies d_\SS(x,y)\le \frac{\tau }{\min_{z\in B_\SS(x,\alpha\tau)}\rho(z)}d_\sfG(x,y).
\end{equation}
This is stronger than the desired inclusion in~\eqref{eq:alpha=2} since if $\alpha=2$, then the minimum that appears in~\eqref{eq:geometric graph bounds original metric}  coincides with the function $\widetilde{\rho}$ that we introduced in~\eqref{eq:alpha=2}, so~\eqref{eq:geometric graph bounds original metric} implies for $\alpha=2$ that  
\begin{equation}\label{eq:deduce case alpha=2}
\forall x,y\in \SS,\qquad d_\sfG(x,y)\le \widetilde{\rho}(x)+1\le 2\widetilde{\rho}(x)\implies d_\SS(x,y)\le \frac{\tau }{\widetilde{\rho}(x)}d_\sfG(x,y)\le 2\tau.
\end{equation}
The second inequality in~\eqref{eq:deduce case alpha=2}uses the assumption  $\rho\ge 1$. 

If $d_\sfG(x,y)=0$, then $x=y$, so~\eqref{eq:geometric graph bounds original metric} holds. We will next prove~\eqref{eq:geometric graph bounds original metric} by induction on $d_{\sfG}(x,y)$. If 
\begin{equation}\label{eq:dG range for induction}
1\le d_{\sfG}(x,y)\le (\alpha-1)\left(\min_{z\in B_\SS(x,\alpha\tau)}\rho(z)\right)+1,\end{equation}
then there is $w\in V$ such that $d_\sfG(x,w)=d_{\sfG}(x,y)-1$ and $\{w,y\}\in E$. Hence, the induction hypothesis gives 
\begin{equation}\label{eq:induction d-1}
d_\SS(x,w)\le \frac{\tau }{\min_{z\in B_\SS(x,\alpha\tau)}\rho(z)}d_\sfG(x,w)=\frac{\tau }{\min_{z\in B_\SS(x,\alpha\tau)}\rho(z)}\big(d_\sfG(x,w)-1\big)\stackrel{\eqref{eq:dG range for induction}}{\le} (\alpha-1)\tau<\alpha\tau,
\end{equation}
and the definition~\eqref{eq:geometric edges general psi} of $E$ gives 
\begin{equation}\label{eq:wy edge}
d_\SS(w,y)\le \frac{\tau}{\min\{\rho(w),\rho(y)\}}\le \tau,
\end{equation}
where we used the assumption $\rho\ge 1$.  Hence, $d_\SS(x,y)\le d_\SS(x,w)+d_\SS(w,y)\le (\alpha-1)\tau+\tau=\alpha\tau$, by~\eqref{eq:induction d-1} and~\eqref{eq:wy edge}. Thus, $y\in B_\SS(x,\alpha \tau)$. Also $w\in B_\SS(x,\alpha \tau)$, by~\eqref{eq:induction d-1}. So, $\min\{\rho(w),\rho(y)\}\ge \min_{z\in B_\SS(x,\alpha\tau)}\rho(z)$. A substitution of this into~\eqref{eq:wy edge} gives 
\begin{equation}\label{eq:max psi estimated}
d_\SS(w,y)\le  \frac{\tau}{\min_{z\in B_\SS(x,\alpha\tau)}\rho(z)}.
\end{equation}
The inductive step now concludes as follows: 
\begin{multline*}
d_\SS(x,y)\le d_\SS(x,w)+d_\SS(w,y) \\\stackrel{\eqref{eq:induction d-1}\wedge \eqref{eq:max psi estimated}}{\le} \frac{\tau }{\min_{z\in B_\SS(x,\alpha\tau)}\rho(z)}\big(d_\sfG(x,w)-1\big)+\frac{\tau}{\min_{z\in B_\SS(x,\alpha\tau)}\rho(z)}=\frac{\tau }{\min_{z\in B_\SS(x,\alpha\tau)}\rho(z)}d_\sfG(x,w).\tag*{\qedhere}
\end{multline*}
\end{proof} 

We are now ready to perform the main purpose of this section, which is to prove Proposition~\ref{thm:universally compatible compression}:

\begin{proof}[Proof of Proposition~\ref{thm:universally compatible compression}] Take the universal constant $\zeta\ge 1$, which is part of the definition~\eqref{eq:our specific rho} of the function  $\rho:\MM\to [1,\infty)$ in Proposition~\ref{thm:universally compatible compression}, to be the implicit factor in conclusion~\eqref{eq:int max in lemma} of Corollary~\ref{lem:h inverse}. 

We will apply  Lemma~\ref{lem:ball comparison} separately to each of the connected component of $\Gamma$ of the graph $\sfG=(\MM,E)$ from Proposition~\ref{thm:universally compatible compression} whose edges are defined in\eqref{eq:geometric edges specific rho},  with $(\SS,d_\SS)=(\Gamma,d_\MM)$ and with $\rho:\Gamma\to [1,\infty)$ replaced by $\rho_\Gamma=\rho|_\Gamma$, the restriction of the function $\rho$ given in~\eqref{eq:our specific rho} to $\Gamma$.  In accordance with~\eqref{eq:alpha=2}, we have 
\begin{equation}\label{eq:rho gamm atilde}
\forall x\in \Gamma,\qquad \widetilde{\rho_\Gamma}(x)= \min_{z\in B_\Gamma(x,2\tau) }\rho(z),
\end{equation}
where we denote $B_\Gamma(x,r)\eqdef B_\MM(x,r)\cap \Gamma$ for every $x\in \Gamma$ and $r\ge 0$. 

The above application of Lemma~\ref{lem:ball comparison} to $(\Gamma,d_\MM)$ and $\rho_\Gamma$ produces a graph $\sfG_\Gamma$ whose vertex set is $\Gamma$. The definition of the edge set of  $\sfG_\Gamma$  coincides per~\eqref{eq:geometric edges general psi}with the restriction to $\Gamma$ of the edge set of the graph $\sfG$ on $\MM$ that is given in~\eqref{eq:geometric edges specific rho}, i.e., $E_{d_\Gamma,\rho_\Gamma,\tau}= \{\{x,y\}\in E_{d_\MM,\rho,\tau}:\ \{x,y\}\subset \Gamma\}$. Consequently, as $\Gamma$ is a connected component of $\sfG$, we have
\begin{equation}\label{eq:ball consistency connected component}
\forall r\in [0,\infty),\ \forall x\in \Gamma,\qquad B_\sfG(x,r)=B_\sfG(x,r)\cap \Gamma=B_{\sfG_\Gamma}(x,r). 
\end{equation}
So, the inclusion~\eqref{eq:alpha=2} from Lemma~\ref{lem:ball comparison} can be written as follows for every connected component $\Gamma$ of $\sfG$: 
\begin{equation}\label{eq:ba;; comparison wioth modified tilde in Gamma}
\forall x\in \Gamma,\qquad B_\sfG(x,\widetilde{\rho_{\Gamma}}(x)+1)=B_{\sfG_\Gamma}(x,\widetilde{\rho_{\Gamma}}(x)+1)\stackrel{\eqref{eq:alpha=2}}{\subset} B_\Gamma(x,2\tau). 
\end{equation}

Next, as in Remark~\ref{rem:basic example}, define $\mathfrak{d},\mathfrak{n} : \MM \to (0, \infty)$ by setting
\begin{equation}\label{eqn:def h, f, F}
\forall x\in \MM,\qquad      \mathfrak{d}(x) \eqdef  \mu\big(B_\MM (x, \tau)\big)\qquad\mathrm{and}\qquad \mathfrak{n}(x) \eqdef  \mu\big(B_\MM (x, 19\tau)\big).
\end{equation}
Let $\vartheta:\MM\to[1,\infty)$ be given as in~\eqref{eq:rho of h} for the choices of $\mathfrak{d},\mathfrak{n}$ in~\eqref{eqn:def h, f, F}. For every connected component $\Gamma$ of $\sfG$ denote the restrictions of $\mathfrak{d},\mathfrak{n},\vartheta$ to $\Gamma$  by $\mathfrak{d}_\Gamma,\mathfrak{n}_\Gamma,\vartheta_\Gamma$, respectively.  Apply Construction~\ref{construction:compression} to $(\Gamma,d_\MM)$ and  $\vartheta_\Gamma$, thus obtaining a rounding function $q_\Gamma:\Gamma\to \Gamma$. Let $q:\MM\to \MM$ be given by defining it to coincide with $q_\Gamma$ on each connected component $\Gamma$ of $\sfG$, i.e., 
\begin{equation}\label{eq:def our q union of Gammas}
\forall x\in \MM,\qquad q(x)=q_{\Gamma(x)}(x).
\end{equation}
The requirement $q(x)\in \Gamma(x)$ of Proposition~\ref{thm:universally compatible compression}   holds by design. Also, \eqref{eq:def our q union of Gammas}  ensures that requirement~\eqref{eq:displacement condition per component} of Proposition~\ref{thm:universally compatible compression} holds, by applying Lemma~\ref{lem:rounding function} separately to each connected component of $\sfG$.

Verifying Definition~\ref{def:compatibility}  of $C$-compatibility (for $\f\circ q$) requires introducing two auxiliary functions $K:\MM\to \N$ and $\Delta:\MM\to [0,\infty)$; recalling~\eqref{eq:rho gamm atilde}, we will define those functions as follows:
\begin{equation}\label{eq:def our Delta and K}
\forall x\in \MM,\qquad K(x)\eqdef \left\lceil \widetilde{\rho_{\Gamma(x)}}(x)\right\rceil \qquad\mathrm{and}\qquad \Delta(x)\eqdef C\left(\max_{z\in B_{\Gamma(x)}(x,2\tau)} \|\f\circ q(x)-\f\circ q (z)\|_2\right).
\end{equation}
With these choices (and the above preparation), it is quite simple to check that Definition~\ref{def:compatibility} is satisfied.

For requirement~\eqref{eq:4 Delta compatibility} of Definition~\ref{def:compatibility}, we are given $x\in \MM$ and  $y\in B_\sfG(x,K(x)-1)$, as well as $z\in \MM$ such that $\{y,z\}\in E$. The goal is to check that $\Delta(x)\le \sigma(\{y,z\})$,   with $\sigma$ given in~\eqref{eq:def our sigma specific q}. This indeed holds since
\begin{equation}\label{eq:BK in BGamma}
y,z\in B_\sfG\big(x,K(x)\big)\stackrel{\eqref{eq:def our Delta and K}}{\subset} B_\sfG\big(x,\widetilde{\rho_{\Gamma(x)}}(x)+1\big)\stackrel{\eqref{eq:ba;; comparison wioth modified tilde in Gamma}}{\subset} B_{\Gamma(x)}(x,2\tau).
\end{equation}
Therefore, $x\in B_\MM(y,2\tau)\cap B_\MM(z,2\tau)\cap \Gamma$, where $\Gamma\eqdef \Gamma(x)=\Gamma(y)=\Gamma(z)$, so we have
\begin{align*}
\Delta(x)\stackrel{\eqref{eq:def our Delta and K}}{=} &C\left(\max_{b\in B_\MM(x,2\tau)\cap \Gamma} \|\f\circ q(x)-\f\circ q (b)\|_2\right)\\&\le  C\bigg(\max_{\substack{a\in B_\MM(y,2\tau)\cap B_\MM(z,2\tau)\cap \Gamma\\ b\in B_\MM(a,2\tau)\cap \Gamma}}\|\f\circ q(a)-\f\circ q(b)\|_2\bigg)\stackrel{\eqref{eq:def our sigma specific q}}{=}\sigma(\{y,z\}).
\end{align*}

Requirement~\eqref{eq:second compatibility condition} of Definition~\ref{def:compatibility} is proved as follows.  If $x\in \MM$ and $y\in N_\sfG(x)$, then $\Gamma(x)=\Gamma(y)\eqdef \Gamma$, and also $d_\MM(x,y)\le \tau$ by~\eqref{eq:geometric edges specific rho}, as $\rho\ge 1$. By the triangle inequality we therefore have 
\begin{equation}\label{eq:23 inclusion}
B_\Gamma(x,2\tau)\subset B_\Gamma(y,3\tau). 
\end{equation}
Apply Corollary~\ref{cor:process}  to $(\Gamma,d_\MM)$,  the functions $\mathfrak{d}_\Gamma,\mathfrak{n}_\Gamma:\Gamma:\to (0,\infty)$,\footnote{Recall that this is how we denoted  the restrictions to $\Gamma$ of the functions given in~\eqref{eqn:def h, f, F}; as  explained in Remark~\ref{rem:basic example}, they   satisfy assumption~\eqref{eqn:measure monotone} of Lemma~\ref{lem:h inverse}.} and $\UU\eqdef B_\Gamma(y,2\tau)$. Recalling that at the start of the proof of  Proposition~\ref{thm:universally compatible compression} we defined $\zeta$ to be the implicit factor in the right hand side of~\eqref{eq:int max in lemma}, we thus obtain the following estimate: 
\begin{align}\label{eq:use U lemma}
\begin{split}
\int_{\R^m} \Big(\max_{z\in B_\Gamma(y,2\tau)}&|\langle v,\f\circ q(z)-\f\circ q(y)\rangle|\Big)\ud\gamma_n(v)\\
&\stackrel{\eqref{eq:int max in lemma}}{\le} \zeta   \Big(\max_{z\in B_\Gamma(y,2\tau)}\|\f\circ q(z)-\f\circ q(y)\|_2\Big) \min_{w \in B_\Gamma (y,3\tau)} \sqrt{\log \frac{\mu(B_\MM(w,19\tau))}{\mu(B_\MM(w,\tau))}}\\
&\stackrel{\eqref{eq:23 inclusion}}{\le} \zeta   \Big(\max_{z\in B_\Gamma(y,2\tau)}\|\f\circ q(z)-\f\circ q(y)\|_2\Big) \min_{w \in B_\Gamma(x,2\tau)} \sqrt{\log \frac{\mu(B_\MM(w,19\tau))}{\mu(B_\MM(w,\tau))}}
\\&\stackrel{\eqref{eq:def our Delta and K}}{=}\Delta(y)\min_{w\in B_\Gamma(x,2\tau)} \frac{\zeta}{C}\sqrt{\log \frac{\mu(B_\MM(w,19\tau))}{\mu(B_\MM(w,\tau))}}.
\end{split}
\end{align}
Consequently, 
\begin{align}\label{eq:get KD}
\begin{split}
\int_{\R^n} \bigg(\max_{z\in B_{\sfG}\left(y,K(y)\right)}|\langle v,\f\circ q(z)-\f\circ q(y)\rangle|\bigg)\ud\gamma_n(v)&\stackrel{\eqref{eq:BK in BGamma}}{\le}\int_{\R^n} \Big(\max_{z\in B_{\Gamma}(y,2\tau)}|\langle v,\f\circ q(z)-\f\circ q(y)\rangle|\Big)\ud\gamma_n(v)\\
&\!\!\!\!\!\!\!\!\!\stackrel{\eqref{eq:our specific rho}\wedge \eqref{eq:use U lemma}}{<}\Delta(y)\min_{w\in B_\Gamma(x,2\tau)}\rho(w)\\ &\stackrel{\eqref{eq:rho gamm atilde}}{=} \Delta(y)\widetilde{\rho_\Gamma}(x)\\&\stackrel{\eqref{eq:def our Delta and K}}{\le} \Delta(y)K(x).
\end{split}
\end{align}
Requirement~\eqref{eq:second compatibility condition} of Definition~\ref{def:compatibility} coincides with~\eqref{eq:get KD}. The second step of~\eqref{eq:get KD} is the only place in the proof of Proposition~\ref{thm:universally compatible compression} in which the 
specific choice of $\rho$ in \eqref{eq:our specific rho} is used.

Finally, the remaining requirement~\eqref{eq:inclusion condition in compatibility def} of Definition~\ref{def:compatibility} is established as follows for every $x\in \MM$:
\begin{align*}
\f\circ q\Big(B_\sfG\big(x,K(x)\big)\Big)&\stackrel{\eqref{eq:BK in BGamma}}{\subset} \f\circ q\big(B_{\Gamma(x)}(x,2\tau)\big)\\& \stackrel{\phantom{\eqref{eq:BK in BGamma}}}{\subset} B_{\ell_2^n} \Big(\f\circ q(x),\max_{z\in B_{\Gamma(x)}(x,2\tau)} \|\f\circ q(x)-\f\circ q (z)\|_2\Big)\\&\stackrel{\eqref{eq:def our Delta and K}}{=} B_{\ell_2^n} \Big(\f\circ q(x),\frac{1}{C}\Delta(x)\big)\Big).
\end{align*}
This concludes the proof of Proposition~\ref{thm:universally compatible compression}.
\end{proof}

 \section{A structural implication  of quasisymmetry}\label{sec:structural}

The ultimate goal of this section is to prove Proposition~\ref{prop:use quasi to get good graph}. We will start by studying some basic implications of the notion of a metric space $(\MM,d)$ being $(s,\e)$-quasisymmetrically Hilbertian (given in Definition~\ref{def:quasisym metric space}), though many of the ensuing elementary considerations do not require the target space to be a Hilbert space. Similar issues are covered in the foundational work~\cite[Section~2]{TK80}.

It is convenient to introduce the following notion as it arises naturally when one wishes to iterate~\eqref{eq:def quasi Hilbert}:

\begin{terminology}\label{terminology:non-discrete} Given $0<s\le 1$ and $\d\ge 0$, say that a metric space $(\MM,d)$ has $(s,\d)$-non-discrete distances if for  every $x,y\in \MM$ there exists $z\in \MM$ such that $sd(x,y)-\d\le d(x,z)\le sd(x,y)$.
\end{terminology}

The next lemma is a simple iterative application of a condition such as the requirement~\eqref{eq:def quasi Hilbert} of being quasisymmetrically Hilbertian.\footnote{Lemma~\ref{lem:iterate quasisymmetry} assumes that the spaces in question are metric spaces because this is the context of the present investigation, but the triangle inequality is never used in its proof.} 

\begin{lemma}\label{lem:iterate quasisymmetry} Fix $\alpha,\d\ge 0$ and $0<s< 1$. Suppose that $(\MM,d_\MM)$ is a metric space that has $(s,\d)$-non-discrete distances. Suppose furthermore that $(\NN,d_\NN)$ is a metric space and that $\f:\MM\to \NN$ satisfies 
\begin{equation}\label{eq:weak quasi MN}
\forall (x,y,z)\in \MM^3,\qquad d_\MM(x,y)\le s d_\MM(x,z)\implies  d_\NN\big(\f(x),\f(y)\big)\le \alpha d_\NN\big(\f(x),\f(z)\big).
\end{equation}
Then, for every $\ell\in \N$ we have 
\begin{equation}\label{eq:l-version of weak quasi MN}
\forall (x,y,z)\in \MM^3,\qquad d_\MM(x,y)\le s^\ell d_\MM(x,z)-\frac{s(1-s^{\ell-1})}{1-s}\d\implies  d_\NN\big(\f(x),\f(y)\big)\le \alpha^\ell d_\NN\big(\f(x),\f(z)\big).
\end{equation}
\end{lemma}

\begin{proof} Observe that~\eqref{eq:weak quasi MN} and~\eqref{eq:l-version of weak quasi MN} coincide when $\ell=1$. We will proceed by induction on $\ell$. Assuming that~\eqref{eq:l-version of weak quasi MN} holds for $\ell\in \N$,  suppose that  $x,y,z\in \MM$ satisfy
\begin{equation}\label{eq:distance assumption ell+1}
d_\MM(x,y)\le s^{\ell+1} d_\MM(x,z)-\frac{s(1-s^{\ell})}{1-s}\d.
\end{equation}
As $(\MM,d_\MM)$  is assumed to have $(s,\d)$-non-discrete distances, we can fix $w\in \MM$ such that 
\begin{equation}\label{eq:for induction sd}
sd_\MM(x,z)-\d\le d_\MM(x,w)\le sd_\MM(x,z).
\end{equation}
The second inequality in~\eqref{eq:for induction sd} makes it possible for us to apply~\eqref{eq:weak quasi MN} to the triple $(x,w,z)$, which gives 
\begin{equation}\label{eq:alpha version xwz}
d_\NN\big(\f(x),\f(w)\big)\le \alpha d_\NN\big(\f(x),\f(z)\big).
\end{equation}
Furthermore, observe that
\begin{align*}
s^\ell d_\MM(x,w)-\frac{s(1-s^{\ell-1})}{1-s}\d&\stackrel{\eqref{eq:for induction sd}}{\ge} s^\ell(sd_\MM(x,z)-\d)-\frac{s(1-s^{\ell-1})}{1-s}\d\\&\stackrel{\eqref{eq:distance assumption ell+1}}{\ge} d_\MM(x,y)+\frac{s(1-s^{\ell})}{1-s}\d-s^\ell\d-\frac{s(1-s^{\ell-1})}{1-s}\d=d_\MM(x,y).
\end{align*}
We can therefore apply the induction hypothesis~\eqref{eq:l-version of weak quasi MN} to the triple $(x,y,w)$ to get that
$$
d_\NN\big(\f(x),\f(y)\big)\le \alpha^\ell d_\NN\big(\f(x),\f(w)\big).
$$
Thanks to~\eqref{eq:alpha version xwz}, this implies $d_\NN(\f(x),\f(y))\le \alpha^{\ell+1}d_\NN(\f(x),\f(z))$, completing the induction step. 
\end{proof}

\begin{remark}\label{rem:path} Given $\d> 0$, recall that a metric space $(\MM,d)$ is said to be $\d$-discretely path connected if for every $x,y\in \MM$  there exist $k\in \N$ and a $k$-step discrete path $z_0=x,z_1,\ldots,z_k=y\in \MM$ joining $x$ to $y$ such that $d_\MM(z_i,z_{i-1})\le \d$ for every $i\in [k]$. Any such space has $(s,\d)$-non-discrete distances for every $0<s\le 1$. Indeed, if $sd(x,y)<\d$, then choose $z=x$ in Terminology~\ref{terminology:non-discrete}. If $sd(x,y)>\d$, then let $i$ be the largest element of $[k]\cup \{0\}$ for which $d(x,z_i)<sd(x,y)-\d$; there is such $i$  by the assumption $sd(x,y)-\d>0=d(x,z_0)$. Now, $i<k$ as $d(x,z_k)=d(x,y)\ge sd(x,y)> sd(x,y)-\d$. So, $d(x,z_{i+1})\ge sd(x,y)-\d$ by the maximality of $i$. Also, $d(x,z_{i+1})\le d(x,z_i)+d(z_i,z_{i+1})<sd(x,y)-\d+\d=sd(x,y)$. Thus $z=z_{i+1}$ works for Terminology~\ref{terminology:non-discrete}.

Furthermore, if $(\MM,d)$ is path connected, then by continuity it has  $(s,0)$-non-discrete distances for any $0<s< 1$. Lemma~\ref{lem:iterate quasisymmetry} therefore implies that if~\eqref{eq:weak quasi MN} holds, then for every $\ell\in \N$ and every $x,y,z\in \MM$ we have 
\begin{equation}\label{eq:s alpha}
d_\MM(x,y)\le s^\ell d_\MM(x,z)\implies  d_\NN\big(\f(x),\f(y)\big)\le \alpha^\ell d_\NN\big(\f(x),\f(z)\big).
\end{equation}
Suppose that $d_\MM(x,y)\le s d_\MM(x,z)$ and write $\ell\eqdef \lfloor  (\log(d(x,z)/d(x,y)))/\log(1/s)\rfloor\in \N$. By design we have $d_\MM(x,y)\le s^\ell d_\MM(x,z)$, so $d_\NN(\f(x),\f(y))\le \alpha^\ell d_\NN(\f(x),\f(z))$ by~\eqref{eq:s alpha}. If $0<\alpha<1$, then
$$
\alpha^\ell\le \alpha^{\frac{\log\frac{d(x,z)}{d(x,y)}}{\log\frac{1}{s}}-1} =\frac{1}{\alpha} \left(\frac{d(x,y)}{d(x,z)}\right)^{\frac{\log \frac{1}{\alpha}}{\log\frac{1}{s}}}. 
$$
Consequently, for every distinct $x,y,z\in \MM$ that satisfy $d_\MM(x,y)\le s d_\MM(x,z)$ we have 
$$
 \frac{d_\NN(\f(x),\f(y))}{d_\NN(\f(x),\f(z))}\le \eta \bigg(\frac{d_\MM(x,y)}{d_\MM(x,z)}\bigg),
$$
where $\eta=\eta_{\alpha,s}:[0,\infty)\to [0,\infty)$ is given by setting  $\eta(t)=t^{(\log(1/\alpha))/\log(1/s)}/\alpha$ for $t\ge 0$, so  $\lim_{t\to 0^+} \eta(t)=0$.

We have thus shown that for every path connected metric space $(\MM,d_\MM)$ 
and  $0<\alpha,s<1$, if $\f:\MM\to \NN$ satisfies~\eqref{eq:weak quasi MN}, 
then it also satisfies the condition~\eqref{eq:def quasi} of being quasisymmetric, except that now~\eqref{eq:def quasi} is required to hold for every distinct $x,y,z\in \MM$  that satisfy $d_\MM(x,y)\le s d_\MM(x,z)$, rather than for any triple of distinct points in $\MM$ whatsoever.  A straightforward inspection of the proof of the main result of~\cite{Nao12} reveals that even though its statement assumes the existence of a quasisymmetric embedding in the classical sense~\eqref{eq:def quasi}, it actually uses this weaker condition. Therefore, thanks to the above reasoning~\cite{Nao12} implies in particular that if $2<p<\infty$, then $\ell_p$ is not $(s,\e)$-quasisymmetrically Hilbertian for any $0<s,\e<1$.  
\end{remark}

The following simple lemma will also be used in the proof of Proposition~\ref{prop:use quasi to get good graph}:

\begin{lemma}\label{lem: for smoothness of labeling along edges} Fix $\tau,s,\alpha>0$. Let $(\MM,d_\MM)$ and $(\NN,d_\NN)$ be metric spaces. Suppose that $f:\MM\to \NN$ satisfies 
\begin{equation}\label{eq: two distance condition}
\forall x,y,z\in \MM,\qquad d_\MM(x,y)\le s\tau\le 2sd_\MM(x,z)\implies d_\NN\big(f(x),f(y)\big)\le \alpha d_\NN\big(f(x),f(z)\big). 
\end{equation}
Given a nonempty subset $\mathscr{E}$ of $\MM\times \MM$ with
\begin{equation}\label{eq:pairs in E are tau apart}
\forall (x,y)\in \mathscr{E},\qquad  d_\MM(x,y)\ge \tau,
\end{equation}
define $\Lambda:\MM\to [0,\infty)$ by setting
\begin{equation}\label{eq:def Lambda ion lemma}
\forall x\in \MM,\qquad \Lambda(x)\eqdef \inf_{(a,b)\in \mathscr{E}}\max\left\{d_\NN\big(f(x),f(a)\big),d_\NN\big(f(x),f(b)\big)\right\}.
\end{equation}
Then, $\Lambda(y)\le (\alpha+1)\Lambda(x)$ for every $x,y\in \MM$ with $d_\MM(x,y)\le s\tau$ (by symmetry, also $\Lambda(y)\ge \Lambda(x)/(\alpha+1)$). 
\end{lemma}

\begin{proof} Suppose that $x,y\in \MM$ satisfy $d_\MM(x,y)\le s\tau$. If $(w,z)\in \mathscr{E}$, then in particular  $d_\MM(w,z)\ge \tau$, by~\eqref{eq:pairs in E are tau apart}. By the triangle inequality we therefore have $\max\{d_\MM(x,w),d_\MM(x,z)\}\ge \tau/2$; without loss of generality we may assume that $d_\MM(x,z)\ge \tau/2$.   We can now use~\eqref{eq: two distance condition} to get that
\begin{equation}\label{eq: 2 poiont}
d_\NN\big(f(x),f(y)\big)\le \alpha d_\NN\big(f(x),f(z)\big).
\end{equation}
Consequently,
\begin{align*}
\Lambda(y)&\le \max\left\{d_\NN\big(f(y),f(w)\big),d_\NN\big(f(y),f(z)\big)\right\}\\&\le d_\NN\big(f(x),f(y)\big)+\max\left\{d_\NN\big(f(x),f(w)\big),d_\NN\big(f(x),f(z)\big)\right\}\\ &\le(\alpha+1)\max\left\{d_\NN\big(f(x),f(w)\big),d_\NN\big(f(x),f(z)\big)\right\},
\end{align*}
where the last step uses~\eqref{eq: 2 poiont}. By taking the infimum over $(w,z)\in \mathscr{E}$, we get that $\Lambda(y)\le (\alpha+1)\Lambda(x)$. 
\end{proof}

We can now prove Proposition~\ref{prop:use quasi to get good graph}, which, as explained in Section~\ref{sec:history}, is a structural implication of $(\MM,d_\MM)$ being $(s,\e)$-quasisymmetrically Hilbertian that is used in the proof of Theorem~\ref{thm:random zero general} (the assumption that $(\MM,d_\MM)$ is  $(s,\e)$-quasisymmetrically Hilbertian is not used elsewhere in the proof of Theorem~\ref{thm:random zero general}).

\begin{proof}[Proof of Proposition~\ref{prop:use quasi to get good graph}] Write $n=|\MM|$. The assumption that $(\MM,d_\MM)$ is $(s,\e)$-quasisymmetrically Hilbertian means that there is an injection $\f:\MM\to \R^n$ such that 
\begin{equation}\label{eq:our quasi in prop}
\forall x,y,z\in \MM,\qquad d_\MM(x,y)\le s d_\MM(x,z)\implies  \|\f(x)-\f(y)\|_{2}\le (1-\e)\|\f(x)-\f(z)\|_{2}.
\end{equation}

Apply Proposition~\ref{thm:universally compatible compression} with $\tau$ replaced by $\beta\tau$ and $C$  replaced by $rC$ to get $q:\MM\to \MM$ that satisfies 
\begin{equation}\label{eq:proximity q}
\forall x\in \MM,\qquad d_\MM(q(x),x)\le 7\beta \tau, 
\end{equation}
and such that, recalling  that the statement of Proposition~\ref{thm:universally compatible compression} denotes the connected component in $\sfG$ of each $x\in \MM$ by $\Gamma(x)$,  if we define, in accordance with~\eqref{eq:def our sigma specific q},
\begin{equation}\label{eq:defs of f and sigma}
f\eqdef \f\circ q\qquad \mathrm{and}\qquad \forall \{x,y\}\in E,\qquad \sigma(\{x,y\})\eqdef rC\bigg(\max_{\substack{a\in B_\MM(x,2\beta\tau)\cap B_\MM(y,2\beta\tau)\cap \Gamma(x)\\ b\in B_\MM(a,2\beta\tau)\cap \Gamma(x)}}\|f(a)-f(b)\|_2\bigg),
\end{equation}
then $\sfG$ is $(rC)$-compatible with $f$ and $\sigma$. So, the first requirement of Proposition~\ref{prop:use quasi to get good graph}   holds by construction. 

It remains to define $\Lambda$ and show that it satisfies the rest of the assertions of Proposition~\ref{prop:use quasi to get good graph}.  Firstly,
\begin{equation}\label{eq:def our Lambda main}
\forall x\in \MM,\qquad \diam_\MM\big(\Gamma(x)\big)\ge \tau\implies \Lambda(x)\eqdef C\bigg(\min_{\substack{w,z\in \Gamma(x)\\ d_\MM(w,z)\ge \tau}} \max\big\{\|f(x)-f(w)\|_2,\|f(x)-f(z)\|_2\big\}\bigg).
\end{equation}
The assumption $\diam_\MM(\Gamma(x))\ge \tau$  in~\eqref{eq:def our Lambda main} ensures that the minimum is over a nonempty set; we complement this definition to the case   $\diam_\MM(\Gamma(x))<\tau$, which is uninteresting for Proposition~\ref{prop:use quasi to get good graph}, as follows:
$$
\forall x\in \MM,\qquad \diam_\MM\big(\Gamma(x)\big)< \tau\implies \Lambda(x)\eqdef \infty. 
$$
This ensures that~\eqref{eq:in same connected component} holds (vacuously) and~\eqref{eq:restate 1/r version} holds (trivially) when $x,y\in \MM$  are both members of the same connected component of $\sfG$ whose diameter is strictly smaller than $\tau$.   

Passing  to the more meaningful case of $x,y\in \MM$ with $\Gamma(x)=\Gamma(y)\eqdef \Gamma$ and $\diam_\MM(\Gamma)\ge \tau$, we will first check  that $\Lambda(x)\neq 0$. If $w,z\in \Gamma$ satisfy $d_\MM(w,z)\ge \tau$, then $d_\MM(q(w),q(z))\ge \tau-14\beta\tau>0$,  thanks to~\eqref{eq:proximity q} and since $\beta<1/14$ by~\eqref{eq:def our beta in prop}. Hence $q(w)\neq q(z)$, so because $\f$ is injective and $f=\f\circ q$, either $f(x)\neq f(w)$ or $f(x)\neq f(z)$, which implies that $\Lambda(x)>0$ by~\eqref{eq:def our Lambda main}. If  $d_\MM(x,y)\ge \tau$, then we can consider $(w,z)=(x,y)$ in~\eqref{eq:def our Lambda main} to get $\Lambda(x)\le C\|f(x)-f(y)\|_2$. By symmetry, also  $\Lambda(y)\le C\|f(x)-f(y)\|_2$,  so~\eqref{eq:in same connected component} holds. 

It remains to verify~\eqref{eq:restate 1/r version}. We will first check that the assumptions of Lemma~\ref{lem: for smoothness of labeling along edges} are satisfied with $\alpha=1$,  $s$ replaced by $s/8$, $f$ replaced by $Cf$, and the target space $(\NN,d_\NN)$ taken to be $\ell_2^n$. So, suppose that $x,y,z\in \MM$ satisfy $d_\MM(x,y)\le s\tau/8$ and $d_\MM(x,z)\ge \tau/2$. Then, by~\eqref{eq:proximity q} and the triangle inequality for $d_\MM$, 
\begin{equation}\label{eq:pass to q for quasi condition}
d_\MM\big(q(x),q(y)\big)\le \frac{s}{8}\tau + 14\beta\tau\le \frac{s}{4}\tau \qquad \mathrm{and}\qquad d_\MM\big(q(x),q(z)\big)\ge \frac12\tau-14\beta\tau\ge \frac{1}{4}\tau,
\end{equation}
where the two last steps in~\eqref{eq:pass to q for quasi condition} hold since  $\beta\le s/112<1/56$ by~\eqref{eq:def our beta in prop}, using $0<s,\e\le \frac12$ and $r\ge 1$. By contrasting the two estimates in~\eqref{eq:pass to q for quasi condition} we see that $d_\MM(q(x),q(y))\le sd_\MM(q(x),q(z))$, so we may apply~\eqref{eq:our quasi in prop} with $(x,y,z)$ replaced by $(q(x),q(y),q(z))$ to get $\|f(x)-f(y)\|_2\le (1-\e)\|f(x)-f(z)\|_2\le \|f(x)-f(z)\|_2$, as  $f=\f\circ q$. This  establishes assumption~\eqref{eq: two distance condition} of Lemma~\ref{lem: for smoothness of labeling along edges} with the above parameters. 

If  $\{x,y\}\in E$, which implies  $\Gamma(x)=\Gamma(y)\eqdef\Gamma$,  and if $\diam_\MM(\Gamma)\ge \tau$, then as $\rho\ge 1$  by~\eqref{eq:rewrite rho specific}, definition~\eqref{eq:def our E with beta} of $E$ gives $d_\MM(x,y)\le \beta\tau\le (s/8)\tau$, using~\eqref{eq:def our beta in prop}. Applying Lemma~\ref{lem: for smoothness of labeling along edges} to $\mathscr{E}=\{(w,z)\in \Gamma\times \Gamma: d_\MM(w,z)\ge \tau\}$, notice that the restriction of~\eqref{eq:def Lambda ion lemma} to $\Gamma$ coincides with the restriction of~\eqref{eq:def our Lambda main} to $\Gamma$,\footnote{Recall that we chose above to apply Lemma~\ref{lem: for smoothness of labeling along edges} with $f$ replaced by $Cf$.} so $\Lambda(y)\le 2\Lambda(x)$. 

It remains to prove the second part of~\eqref{eq:restate 1/r version} when $\{x,y\}\in E$ and $\diam(\Gamma)\ge \tau$, where $\Gamma\eqdef \Gamma(x)=\Gamma(y)$. By~\eqref{eq:defs of f and sigma}, we need to show that if $a\in B_\MM(x,2\beta\tau)\cap B_\MM(y,2\beta\tau)\cap \Gamma$ and $b\in B_\MM(a,2\beta\tau)\cap \Gamma$, then 
\begin{equation}\label{eq:rewrite goal max Lambda}
4rC\|f(a)-f(b)\|_2\le \min\big\{\Lambda(x),\Lambda(y)\big\}.
\end{equation}
To see why~\eqref{eq:rewrite goal max Lambda} holds, fix from now such $a,b\in \MM$. Then, $d_\MM(x,b)\le d_\MM(x,a)+d_\MM(a,b)\le 4\beta \tau$, so
\begin{align}\label{eq:distances between the qs upper}
\begin{split}
\max\left\{d_\MM\big(q(x),q(a)\big),d_\MM\big(q(x),q(a)\big)\right\}\stackrel{\eqref{eq:proximity q}}{\le} & \max\big\{d_\MM(x,a),d_\MM(x,b)\big\}+14\beta\tau\\&\le 18\beta\tau\le\frac14s^\ell\tau-\frac{s}{1-s}\beta\tau<\frac14s^\ell\tau-\frac{s(1-s^{\ell-1})}{1-s}\beta\tau,  
\end{split}
\end{align}
where we introduce the notation
\begin{equation}\label{eq:choose our ell for lemma}
\ell=\ell(r,\e)\eqdef \left\lceil \frac{\log(8r)}{\e}\right\rceil\in \N,
\end{equation}
using which the penultimate inequality in~\eqref{eq:distances between the qs upper} is straightforward to verify thanks to the assumption~\eqref{eq:def our beta in prop} on $\beta$ and because $0<s,\e\le 1/2$ and $r\ge 1$.

Because $\Gamma$ is the connected component of $x$ (and $y$) in the graph $\sfG$, and thanks to~\eqref{eq:def our E with beta} and the fact that $\rho\ge 1$ by~\eqref{eq:rewrite rho specific}, the $d_\MM$-distance between the endpoints of every edge of $\sfG$ is at most $\beta\tau$, the metric space $(\Gamma,d_\MM)$ is $\beta\tau$-discretely path connected, and hence it has  $(s,\beta\tau)$-non-discrete distances (recall Terminology~\ref{terminology:non-discrete} and Remark~\ref{rem:path}). We can therefore apply Lemma~\ref{lem:iterate quasisymmetry} to the metric space $(\Gamma,d_\MM)$ with $s,\f$ as above, $\alpha=1-\e$ and $\d=\beta\tau$, in which case assumption~\eqref{eq:weak quasi MN} of Lemma~\ref{lem:iterate quasisymmetry} for $\NN=\ell_2^n$ coincides with~\eqref{eq:our quasi in prop}. 
 
The assumption $\diam(\Gamma)\ge \tau$ means that there are $w,z\in \Gamma$ for which $d_\MM(w,z)\ge \tau$. By the triangle inequality for $d_\MM$, it follows that $\max\{d_\MM(x,w),d_\MM(x,z)\}\ge \tau/2$, so we may assume without loss of generality that $d_\MM(x,z)\ge \tau/2$. Using~\eqref{eq:proximity q} and the triangle inequality for $d_\MM$, we therefore have 
\begin{equation}\label{eq:distance from qs lower}
d_\MM\big(q(x),q(z)\big)\ge \frac12\tau-14\beta \tau\ge \frac14\tau. 
\end{equation}
By combining~\eqref{eq:distances between the qs upper} and~\eqref{eq:distance from qs lower}, we see that 
\begin{equation}\label{eq:assumption of quasi iteration lemma for a,b}
\max\left\{d_\MM\big(q(x),q(a)\big),d_\MM\big(q(x),q(b)\big)\right\}<s^\ell d_\MM\big(q(x),q(z)\big)-\frac{s(1-s^{\ell-1})}{1-s}\beta\tau.
\end{equation}
Since $a,b,x,z\in \Gamma$ and Proposition~\ref{thm:universally compatible compression} asserts that $q(\Gamma)\subset \Gamma$, we have $q(a),q(b),q(x),q(z)\in \Gamma$. Because we checked that the assumptions of Lemma~\ref{lem:iterate quasisymmetry} hold for $(\Gamma,d_\MM)$ with the above parameters, and~\eqref{eq:assumption of quasi iteration lemma for a,b} corresponds to the assumption in~\eqref{eq:l-version of weak quasi MN}, then using that $f=\f\circ q$ we get  from Lemma~\ref{lem:iterate quasisymmetry} (applied to $\f$ twice, once to the triple  $(q(x),q(a),q(z))\in \Gamma^3$, and once to the triple $(q(x),q(b),q(z))\in \Gamma^3$) that 
\begin{align}\label{eq:before inf on wz}
\begin{split}
\max\big\{\|f(x)-f(a)\|_2,\|f(x)&-f(b)\|_2\big\}\le (1-\e)^\ell\|f(x)-f(z)\|_2\le e^{-\e\ell} \|f(x)-f(z)\|_2\\&\stackrel{\eqref{eq:choose our ell for lemma}}{\le} \frac{1}{8r}\|f(x)-f(z)\|_2\le \frac{1}{8r} \max\big\{\|f(x)-f(w)\|_2,\|f(x)-f(z)\|_2\big\}.
\end{split}
\end{align}
By taking the minimum over all  $w,z\in \Gamma$ with $d_\MM(w,z)\ge \tau$ while recalling~\eqref{eq:def our Lambda main}, we get from~\eqref{eq:before inf on wz} that
$$
\max\big\{\|f(x)-f(a)\|_2,\|f(x)-f(b)\|_2\big\}\le \frac{\Lambda(x)}{8rC}. 
$$
By the triangle inequality for $\ell_2^n$, this implies that $4rC\|f(a)-f(b)\|_2\le \Lambda(x)$. The symmetric reasoning shows that also $4rC\|f(a)-f(b)\|_2\le \Lambda(y)$. This completes the proof of the desired inequality~\eqref{eq:rewrite goal max Lambda}.
\end{proof}

\section{Proof of Theorem~\ref{thm:main separated piut together}}\label{sec:use fractional}

\noindent As Theorem~\ref{thm:main separated piut together} treats a general class of probability measures on $\MM\times \MM$, which is needed for the duality step that will be performed in Section~\ref{sec:duality}, a fractional version of Theorem~\ref{thm:matching from compatibility} is required. Prior to passing to the proof of Theorem~\ref{thm:main separated piut together}, we will start this section by explaining this  
formal consequence of Theorem~\ref{thm:matching from compatibility}.

Recall that $\nu(\sfG)$ denotes the matching number of a graph $\sfG=(V,E)$, i.e., it is the maximum size of a pairwise disjoint collection of edges in $E$. The fractional matching number of $\sfG$, denoted $\nu^*(\sfG)$, is the maximum of $\sum_{e\in E}\phi(e)$ over all possible $\phi:E\to [0,1]$ that satisfy the following constraints:
\begin{equation}\label{eq:fractional constrainyts no vertex weights}
\forall x\in V,\qquad \sum_{\substack{e\in E\\ x\in e}}\phi(e)\le 1. 
\end{equation}
Clearly $\nu^*(\sfG)\ge \nu(\sfG)$, since $\nu(\sfG)$ is the maximum of the same optimization problem as above if we add to it the further requirement that $\phi$ only takes values in $\{0,1\}$. Conversely, by~\cite{Lov74} we have  $\nu^*(\sfG)\le \frac32\nu(\sfG)$, 
with equality holding if and only if $\sfG$ is a disjoint union of triangles; see~\cite[Corollary~1]{Fur81} (also, e.g.~the recent exposition~\cite[Corollary~9]{CKO16}). We therefore have the following corollary of Theorem~\ref{thm:matching from compatibility}:

\begin{corollary}\label{cor:unweighted fractional}Fix $C\ge 1$ and $n\in \N$. Suppose that $\sfG=(V,E)$ is a graph that is $C$-compatible with $f:V\to \R^n$ and $\sigma:E\to [0,\infty)$. Then,\footnote{The function $(v\in \R^n) \mapsto \nu^*(\sfG(v;f,\sigma))$ in~\eqref{eq:nustar} is Borel-measurable because by~\eqref{eq:sparsification definition} the set $\{v\in \R^n:\ E(v;f,\sigma)=E'\}$ is Borel for every fixed $E'\subset E$, and there are finitely many such $E'$.}
\begin{equation}\label{eq:nustar}
\int_{\R^n}\nu^*\big(\sfG(v;f,\sigma)\big)\ud\gamma_n(v)< 9e^{-\frac{1}{4}C^2}|V|.
\end{equation}
\end{corollary}

We will, in fact, need below the following  vertex-weighted version of Corollary~\ref{cor:unweighted fractional}, which is a quick consequence of it thanks to the immediate fact that the notion of $C$-compatibility is preserved by blowups (we soon will explain what this means). Given a graph $\sfG=(V,E)$ and vertex weights $Q:V\to [0,\infty)$, let $\nu^*_Q(\sfG)$ be the maximum of   $\sum_{e\in E}\phi(e)$ over all possible $\phi:E\to [0,\infty)$ that satisfy the following constraints:
\begin{equation}\label{eq:Qfractional constrainyts no vertex weights}
\forall x\in V,\qquad \sum_{\substack{e\in E\\ x\in e}}\phi(e)\le Q(x). 
\end{equation}
Thus, $\nu^*(\sfG)$ corresponds to the special $Q\equiv 1$. 

Suppose that $Q(V)\subset \N\cup \{0\}$. Write  $V'\eqdef V\setminus Q^{-1}(0)$ and $E'\eqdef \{e\in E:\ e\subset V'\}$. The graph $\sfG'\eqdef (V',E')$  is obtained from $\sfG$ by deleting all the vertices on which $Q$ vanishes and  all of the edges that are incident to such a vertex. For every $x\in V'$ introduce $Q(x)$-copies $\{x_{1},\ldots, x_{Q(x)}\}$ of $x$, and define  
$$
U\eqdef \bigcup_{x\in V'}  \big\{x_{1},\ldots, x_{Q(x)}\big\}\qquad \mathrm{and}\qquad F\eqdef \bigcup_{\{x,y\}\in E'}\big\{\{x_i,y_j\}: (i,j)\in [Q(x)]\times [Q(y)]\big\}.
$$ 
Thus, $\sfH\eqdef (U,F)$ is the standard $Q|_{V'}$-blowup of  $\sfG'$ in which each vertex $x$ of $\sfG'$ is replaced by a ``cloud'' consisting of $Q(x)$ copies of $x$, and every edge of $\sfG'$ is replaced by the complete bipartite graph between the cloud of $x$ and the cloud of $y$. It is straightforward to check as follows  that  $\nu^*_Q(\sfG)=\nu^*(\sfH)$. 

Suppose that $\phi:E\to [0,\infty)$ satisfies the constraints~\eqref{eq:Qfractional constrainyts no vertex weights}. We ``lift'' $\phi$ to $\phi^*:F\to [0,1]$ by setting 
$$
\forall \{x,y\}\in E',\  \forall (i,j)\in [Q(x)]\times [Q(y)], \qquad \phi^*(\{x_i,y_j\})=\frac{\phi(\{x,y\})}{Q(x)Q(y)}.
$$ 
Then, $\sum_{f\in F} \phi^*(f)=\sum_{e\in E}\phi(e)$, and by~\eqref{eq:Qfractional constrainyts no vertex weights} we know that $\phi^*$ satisfies the analogue of the constraints~\eqref{eq:fractional constrainyts no vertex weights} for $\sfH$. This shows that $\nu^*_Q(\sfG)\le \nu^*(\sfH)$. In the reverse direction, for $\Phi:F\to [0,1]$ define  $\Phi_*:E\to [0,\infty)$ by 
$$
\forall \{x,y\}\in E,\qquad \Phi_*(\{x,y\})=\sum_{i=1}^{Q(x)}\sum_{j=1}^{Q(y)} \Phi(\{x_i,y_j\}).
$$ 
By definition, $\sum_{f\in F} \Phi(f)=\sum_{e\in E}\Phi_*(e)$, and if $\Phi$ satisfies the analogue of the constraints~\eqref{eq:fractional constrainyts no vertex weights} for $\sfH$, then $\Phi_*$ satisfies~\eqref{eq:Qfractional constrainyts no vertex weights}. Therefore,  $\nu^*(\sfH)\le \nu^*_Q(\sfG)$, so we indeed have $\nu^*_Q(\sfG)=\nu^*(\sfH)$.

Continuing with the above setting and notation, suppose that  $\sfG$ is $C$-compatible with $f:V\to \R^n$ and $\sigma:E\to [0,\infty)$ for some $n\in \N$ and $C>0$, per Definition~\ref{def:compatibility}. Lift $f$ to $f^*:U\to \R^n$ by setting $f^*(x_i)=f(x)$ for every $x\in V'$ and $i\in [Q(x)]$. Also lift $\sigma$ to $\sigma^*:F\to [0,\infty)$ by setting $\sigma^*(\{x_i,y_j\})=\sigma(\{x,y\})$ for every $\{x,y\}\in E$ and $(i,j)\in [Q(x)]\times [Q(y)]$. Then $\sfH$ is $C$-compatible with $f^*$ and $\sigma^*$, as seen immediately by lifting the mappings $K:V\to \R^n$ and $\Delta:V\to [0,\infty)$ from Definition~\ref{def:compatibility} in the same manner, i.e., defining $K^*(x_i)=K(x)$ and $\Delta^*(x_i)=\Delta(x)$ for every $x\in V'$ and $i\in [Q(x)]$. Therefore, by Corollary~\ref{cor:unweighted fractional} we have 
$$
\int_{\R^n}\nu^*\big(\sfH(v;f^*,\sigma^*)\big)\ud\gamma_n(v)< 9e^{-\frac{1}{4}C^2}|U|=9e^{-\frac{1}{4}C^2}\sum_{x\in V}Q(x).
$$
For each $v\in \R^n$, an  inspection of the definition of the Euclidean sparsification $\sfH(v;f^*,\sigma^*)$ reveals that  it coincides with the $Q|_{V'}$-blowup of $\sfG(v;f,\sigma)$. Consequently, $\nu^*(\sfH(v;f^*,\sigma^*))=\nu^*_Q(\sfG(v;f,\sigma))$. Hence,
\begin{equation}\label{eq:Q version 9}
\int_{\R^n}\upnu_Q^*\big(\sfG(v;f,\sigma)\big)\ud\gamma_n(v)< 9e^{-\frac{1}{4}C^2}\sum_{x\in V}Q(x).
\end{equation}

We derived~\eqref{eq:Q version 9} under the assumption that $Q$ takes values in the integers, but it follows formally from this that~\eqref{eq:Q version 9} holds for any $Q:V\to [0,\infty)$. Indeed, by considering an arbitrarily precise approximation of $Q$ by rational-valued vertex weights, we may assume that $Q$ takes values in $\mathbb{Q}$, and then by rescaling by the common denominator of these values, the general case follows from the case $Q(V)\subset \N\cup \{0\}$. 

 \begin{corollary}\label{cor:unweighted fractional-Q}Fix $C\ge 1$ and $n\in \N$. Suppose that  $\sfG=(V,E)$ is a graph that is $C$-compatible with $f:V\to \R^n$ and $\sigma:E\to [0,\infty)$. Then, for every $Q:V\to [0,\infty)$ we have 
\begin{equation}\label{eq:nustarQ}
\int_{\R^n}\nu_Q^*\big(\sfG(v;f,\sigma)\big)\ud\gamma_n(v)< 9e^{-\frac{1}{4}C^2}\sum_{x\in V}Q(x).
\end{equation}
\end{corollary}

We will also use  the following general lemma about vertex-weighted fractional matching numbers:

\begin{lemma}\label{lem:graph fractional} Suppose that $V$ is a finite set and that $\omega$ is a symmetric nonnegative measure  on $V\times V$. Define $Q=Q_\omega:V\to [0,\infty)$ by setting 
\begin{equation}\label{eq:def Q vertex weights}
\forall x\in V,\qquad Q(x)\eqdef \sum_{y\in V} \omega(x,y).
\end{equation}  
Let $L,R\subset V$ be disjoint subsets of $V$. Suppose that  $\sfG=(L\cup R,E)$  is a bipartite graph whose sides are $L,R$ (thus, $|e\cap L|=|e\cap R|=1$ for every $e\in E$). 
Then, there exist subsets $L_0^*\subset L$ and $R_0^*\subset R$ such that 
\begin{equation}\label{eq:pass to L0R0}
\omega(L_0^*\times R_0^*)\ge \omega(L\times R)-2\nu^*_{ Q|_{L\cup R}}(\sfG)\qquad\mathrm{and}\qquad \forall (x,y)\in L_0^*\times R_0^*,\qquad \{x,y\}\notin E.
\end{equation}
\end{lemma}
\begin{proof} We will slightly abuse notation by letting $Q$ also denote the restriction $Q|_{L\cup R}$ of $Q$ to $L\cup R$; thus,  the $Q$-weighted fractional matching number $\nu^*_{ Q|_{L\cup R}}(\sfG)$ that appears in~\eqref{eq:pass to L0R0} will be denoted $\nu^*_{Q}(\sfG)$  below  

Let $\mathsf{FM}_Q(\sfG)\subset \R^{E}$ be the $Q$-weighted fractional matching polytope of $\sfG$, i.e.,  
$$
\mathsf{FM}_Q(\sfG)\eqdef \Big\{\phi:E\to [0,\infty): \sum_{\substack{e\in E\\x\in e}} \phi(e)\le Q(x)\ \mathrm{for\ all\ }x\in V\Big\}.
$$
In other words, $\mathsf{FM}_Q(\sfG)$ consists of all of the  edge weights $\phi:E\to [0,\infty)$  that satisfy the constraints that appear in~\eqref{eq:Qfractional constrainyts no vertex weights}. By the compactness of $\mathsf{FM}_Q(\sfG)$, we may fix from now a  maximizer $\phi^*$ of the fractional matching problem for $\sfG$ with vertex weights $Q$,  i.e., $\phi^*:E\to [0,\infty)$
 belongs to $\mathsf{FM}_Q(\sfG)$ and  
\begin{equation}\label{eq:value of maximal fractional-AB}
\sum_{e\in E} \phi^*(e)=\nu^*_Q(\sfG)=\sup_{\phi\in \mathsf{FM}_Q(\sfG)} \sum_{e\in E} \phi(e).
\end{equation}

Define $Q^*:L\cup R\to [0,\infty)$ by setting
\begin{equation}\label{eq:def Q*}
\forall x\in L\cup R,\qquad Q^*(x)\eqdef \sum_{\substack{e\in E\\x\in e}} \phi^*(e)\le Q(x),
\end{equation}
where the last inequality in~\eqref{eq:def Q*} holds because $\phi^*\in \mathsf{FM}_Q(\sfG)$. Then, since $\sfG$ is bipartite we have 
\begin{equation}\label{eq:sum of sums}
\sum_{x\in L} Q^*(x)+\sum_{y\in R} Q^*(y)\stackrel{\eqref{eq:def Q*}}{=}2\sum_{e\in E}\phi^*(e)\stackrel{\eqref{eq:value of maximal fractional-AB}}{=}2\nu^*_Q(\sfG). 
\end{equation}
Next, define $L^*_0\subset L$ and $R^*_0\subset B$ as follows:
\begin{equation}\label{eq:def A*B*}
L_0^*\eqdef \big\{x\in L:\ Q^*(x)<Q(x)\big\}\qquad\mathrm{and}\qquad R_0^*\eqdef \big\{x\in R:\ Q^*(x)<Q(x)\big\}.
\end{equation}
In other words, thanks to the last inequality in~\eqref{eq:def Q*}, we can also write
\begin{equation}\label{eq:complements of star}
L\setminus L_0^*= \big\{x\in L:\ Q^*(x)=Q(x)\big\}\qquad\mathrm{and}\qquad  R\setminus R_0^*= \big\{x\in R:\ Q^*(x)=Q(x)\big\}.
\end{equation}

The second requirement in~\eqref{eq:pass to L0R0} is satisfied because if there were $x\in L_0^*$ and $y\in R_0^*$ with $\{x,y\}\in E$, then by~\eqref{eq:def A*B*}  there would be $\e>0$ such that \begin{equation}\label{eq:add eps to bad edge}
Q^*(x)+\e\le Q(x)\qquad\mathrm{and}\qquad  Q^*(y)+\e\le Q(y).
\end{equation}
 Hence, if we consider  $\phi:E\to [0,\infty)$ by setting
$$
\forall e\in E,\qquad \phi(e)\eqdef \left\{\begin{array}{ll}\phi^*(e)&\mathrm{if}\ e\in E\setminus \{x,y\},\\ \phi^*(e)+\e&\mathrm{if}\ e=\{x,y\},
\end{array}\right. 
$$
then $\phi\in \mathsf{FM}_Q(\sfG)$ thanks to the fact that $\phi^*\in \mathsf{FM}_Q(\sfG)$ and~\eqref{eq:add eps to bad edge}, yet $\sum_{e\in E}\phi(e)=\sum_{e\in E} \phi^*(e)+\e>\nu^*_Q(\sfG)$ by the first equality in~\eqref{eq:value of maximal fractional-AB}. This is in contradiction to the second equality in~\eqref{eq:value of maximal fractional-AB}. 

Finally, the first requirement in~\eqref{eq:pass to L0R0} is justified as follows:
\begin{align*}
\begin{split}
\omega(L_0^*\times R_0^*)&\stackrel{\phantom{\eqref{eq:sum of sums}}}{=}\omega(L\times R)-\sum_{x\in L}\sum_{y\in R\setminus R_0^*} \omega(x,y)
-\sum_{x\in L\setminus L_0^*}\sum_{y\in R_0^*} \omega(x,y)\\
&\stackrel{\phantom{\eqref{eq:sum of sums}}}{\ge}\omega(L\times R)-\sum_{y\in R\setminus R_0^*} \sum_{x\in V} \omega(x,y)
- \sum_{x\in L\setminus L_0^*}\sum_{y\in V} \omega(x,y)\\
&\stackrel{\phantom{\eqref{eq:sum of sums}}}{=} \omega(L\times R)-\sum_{y\in R\setminus R_0^*} \sum_{x\in V} \omega(y,x)
- \sum_{x\in L\setminus L_0^*}\sum_{y\in V} \omega(x,y)\\
&\stackrel{\eqref{eq:def Q vertex weights}}{=} \omega(L\times R)-\sum_{y\in R\setminus R_0^*} Q(y)-\sum_{x\in L\setminus L_0^*}Q(x)\\
&\stackrel{\eqref{eq:complements of star}}{=} \omega(L\times R\big)-\sum_{y\in R\setminus R_0^*} Q^*(y)-\sum_{x\in L\setminus L_0^*}Q^*(x)\\
&\stackrel{\phantom{\eqref{eq:sum of sums}}}{\ge} \omega(L\times R)-\sum_{y\in R} Q^*(y)-\sum_{x\in L}Q^*(x)\\
&\stackrel{\eqref{eq:sum of sums}}{=} \omega(L\times R)-2\nu^*_Q(\sfG),
\end{split}
\end{align*}
where the third step above is where the symmetry assumption on $\omega$ is used. 
\end{proof}

With the above preparation at hand, we can now prove Theorem~\ref{thm:main separated piut together}:

\begin{proof}[Proof of Theorem~\ref{thm:main separated piut together}] Denote $n\eqdef   |\MM|$. Let $\alpha_0\ge 1$ be a  universal constant whose  value will be specified at the end of the ensuing reasoning, to satisfy constraints that will arise as we go along; see~\eqref{eq:def r choice}. Suppose that $\alpha\ge \alpha_0$ and, letting $\zeta\ge 1$   be the universal constant that appears in~\eqref{eq:rewrite rho specific}, define 
\begin{equation}\label{eq:choose our alpha}
r\eqdef \zeta \alpha\ge 1. 
\end{equation}

The statement of Theorem~\ref{thm:main separated piut together} assumes that $\beta>0$ satisfies~\eqref{eq:beta with alpha constant}. Our next step will be to apply Proposition~\ref{prop:use quasi to get good graph} with this choice of $\beta$, for which we need to know that the assumption~\eqref{eq:def our beta in prop} of Proposition~\ref{prop:use quasi to get good graph}  holds. This requirement simplifies to $\alpha \ge 3\log (8r)=3\log (8\zeta\alpha)$ using~\eqref{eq:choose our alpha}, and since our only assumption on  $\alpha$ here is that $\alpha\ge \alpha_0$, for this to be satisfied it suffices to impose the following first restriction on $\alpha_0$:
 \begin{equation}\label{first restriction on r}
 \alpha_0\ge 6 \log(48\zeta).
 \end{equation}

So, apply Proposition~\ref{prop:use quasi to get good graph}  to  $(\MM,d_\MM)$ with the datum $s,\e,r,\tau,C,\beta$. This produces a graph $\sfG=(\MM,E)$ that is given by~\eqref{eq:def our E with beta}, and the function $\rho:\MM\to [1,\infty)$ that appears in~\eqref{eq:def our E with beta} is defined in~\eqref{eq:rewrite rho specific}. Proposition~\ref{prop:use quasi to get good graph}  also produces $f:\MM\to \R^n$, $\sigma:E\to [0,\infty)$ and $\Lambda:\MM\to (0,\infty]$, such that $\sfG$ is $(rC)$-compatible with $f$ and $\sigma$,  and the requirements in~\eqref{eq:in same connected component} and~\eqref{eq:restate 1/r version} are satisfied for every $x,y\in \MM$ that belong to the same connected component of $\sfG$.  The  choice of $r$ in~\eqref{eq:choose our alpha} was made to ensure that the function $\rho$ in~\eqref{eq:rewrite rho specific} coincides with the function $\rho$ in~\eqref{eq:our rho with alpha const}, i.e., the output of Proposition~\ref{prop:use quasi to get good graph} is consistent with the statement of Theorem~\ref{thm:main separated piut together}.

Part of the above consequences of Proposition~\ref{prop:use quasi to get good graph}  makes it possible for us to use Proposition~\ref{prob:graphical prob multiscale} as follows. The first requirement in~\eqref{eq:restate 1/r version} coincides with assumption~\eqref{eq:moderate variation along edges} of Proposition~\ref{prob:graphical prob multiscale}. Also, because the measure $\omega$ of Theorem~\ref{thm:main separated piut together} is assumed to be supported on $\{(x,y)\in \MM\times\MM:\ d_\MM(x,y)\ge \tau\}$ and the second inequality in~\eqref{eq:in same connected component} holds under the assumption $d_\MM(x,y)\ge \tau$, we have   
$$
\forall x,y\in \MM,\qquad \omega(x,y)>0\implies   \|Cf(x)-Cf(y)\|_2\ge \max\big\{\Lambda(x),\Lambda(y)\big\}.
$$
Therefore, the function $Cf$ satisfies\footnote{In fact, we get here that $Cf$ satisfies a stronger form of~\eqref{eq:min assumption in proposition} with $\max\{\cdot,\cdot\}$ in place of $\min\{\cdot,\cdot\}$ in the right hand side, but we do not need to take advantage of this herein.} assumption~\eqref{eq:min assumption in proposition} of Proposition~\ref{prob:graphical prob multiscale}.

We may thus apply Proposition~\ref{prob:graphical prob multiscale} to obtain for every $v\in \R^n$ subsets $A(v),B(v)$ of $\MM$ that satisfy the stated (rudimentary) measurability requirements, and for which we have, by rewriting conclusion~\eqref{eq:separation within connected components} of  Proposition~\ref{prob:graphical prob multiscale} with $f$ replaced by $Cf$, the following point-wise inequality: 
\begin{equation}\label{eq:restated separation within connected components}
\forall v\in \R^n,\ \forall (x,y)\in A(v)\times B(v),\qquad \{x,y\}\in E\implies |\langle v,f(x)-f(y)\rangle |> \max\big\{\Lambda(x),\Lambda(y)\big\}.
\end{equation}
Conclusion~\eqref{eq:super gaussian omega} of  Proposition~\ref{prob:graphical prob multiscale} also holds, i.e.,  for some universal constants   $0<\theta\le 1\le \kappa$, we have
\begin{equation}\label{eq:restate super gaussian omega}
\int_{\R^n} \omega\big(A(v)\times B(v)\big)\ud \gamma_n(v)\ge \theta e^{-\kappa C^2}.
\end{equation}

Observe that $A(v)$ and $B(v)$ are disjoint for every $v\in \R^n$. Indeed, the definition~\eqref{eq:def our E with beta} of $E$ implies that  $\{x,x\}\in E$ for every $x\in \MM$ (i.e., by  definition $\sfG$ has a self-loop at each of its vertices), so if  it were the case that $x\in A(v)\cap B(v)$, then we would get a contradiction to the strict inequality in~\eqref{eq:restated separation within connected components} in its special case $x=y$. We can therefore consider the bipartite graph $\sfG_v=(A(v)\cup B(v),E_v)$ that is defined by 
\begin{equation}\label{eq:def GAB}
\forall x,y\in A(v)\cup B(v),\qquad \{x,y\}\in E_{v}\iff \{x,y\}\in E\ \mathrm{and}\ (x,y)\in \big(A(v)\times B(v)\big)\cup \big(B(v)\times A(v)\big).
\end{equation}

We will next make the following observation about  $\sfG_v$   that will be especially important to the ensuing reasoning. For every $v\in \R^n$, the graph $\sfG_v$   is a subgraph of the Euclidean sparsification $\sfG(v;f,\sigma)$ per its definition in~\eqref{eq:sparsification definition} and~\eqref{eq:def euclidean sparsification}, for the  $f$ and $\sigma$ that were obtained above from  Proposition~\ref{prop:use quasi to get good graph}. Indeed, by~\eqref{eq:def GAB}, if $\{x,y\}\in E_v$, then $\{x,y\}\in E$ and either $(x,y)\in A(v)\times B(v)$ or $(y,x)\in A(v)\times B(v)$.  Hence, $|\langle v,f(x)-f(y)\rangle|>\max\{\Lambda(x),\Lambda(y)\}$ by~\eqref{eq:restated separation within connected components}.  But, $\max\{\Lambda(x),\Lambda(y)\}\ge \min\{\Lambda(x),\Lambda(y)\}\ge 4\sigma(\{x,y\})$ using the second inequality in conclusion~\eqref{eq:restate 1/r version} of Proposition~\ref{prob:graphical prob multiscale}. So, $|\langle v,f(x)-f(y)\rangle|>4\sigma(\{x,y\})$, which, per~\eqref{eq:sparsification definition}, is precisely the requirement for $\{x,y\}\in E$ to be an edge of the Euclidean sparsification $\sfG(v;f,\sigma)$.  

Consider the function $Q=Q_\omega:\MM\to [0,\infty)$ that is given by setting $Q(x)=\sum_{y\in \MM} \omega(x,y)$ for every $x\in \MM$. The fact that $\sfG_v$ is a subgraph of $\sfG(v;f,\sigma)$ for every $v\in \R^n$ implies in particular the following point-wise estimate on its maximal fractional matching with respect to the vertex weights  $Q$: 
\begin{equation}\label{eq:poin-wise fraction number bound}
\forall v\in \R^n,\qquad \nu^*_{Q|_{A(v)\cup B(v)}}(\sfG_v)\le \nu^*_Q\big(\sfG(v;f,\sigma)\big).
\end{equation}
Both sides of the inequality in~\eqref{eq:poin-wise fraction number bound} are Borel-measurable functions of $v$ since there are finitely many subgraphs $\sfH$ of $\sfG$ and the sets $\{v\in \R^n:\  \sfG_v=\sfH\}$ and $\{v\in \R^n:\  \sfG(v;f,\sigma)=\sfH\}$ are Borel subsets of $\R^n$ by the  measurability of $v\mapsto (A(v), B(v))$ in Proposition~\ref{prob:graphical prob multiscale}, and the definition of $\sfG(v;f,\sigma)$. Integrate~\eqref{eq:poin-wise fraction number bound} with respect to $\gamma_n$ while substituting  the fact  that $\sfG$ is $(rC)$-compatible with $f$ and $\sigma$ into Corollary~\ref{cor:unweighted fractional-Q}, to get
\begin{equation}\label{eq:use expected fractional}
\int_{\R^n} \nu^*_{Q|_{A(v)\cup B(v)}}(\sfG_v)\ud \gamma_n(v)\stackrel{\eqref{eq:nustarQ}\wedge \eqref{eq:poin-wise fraction number bound}}{<}  9 e^{-\frac14 r^2C^2}\sum_{x\in \MM}Q(x)\stackrel{\eqref{eq:def Q vertex weights}}{=}9e^{-\frac14 r^2C^2} \omega(\MM\times \MM)=9e^{-\frac14 r^2C^2}.
\end{equation}

Next, for each $v\in \R^n$ apply Lemma~\ref{lem:graph fractional} to $\sfG_v$. This yields $A_0^*(v)\subset A(v)$ and $B_0^*(v)\subset B(v)$ such that
\begin{equation}\label{eq:pass to L0R0-later}
\omega\big(A_0^*(v)\times B_0^*(v)\big)\ge \omega\big(A(v)\times B(v)\big)-2\nu^*_{ Q|_{A(v)\cup B(v)}}(\sfG_v),
\end{equation}
and 
\begin{equation}\label{eq:we through away the bad edges-new}
\forall (x,y)\in A_0^*(v)\times B_0^*(v),\qquad \{x,y\}\notin E_v.
\end{equation}
The mapping $(v\in \R^n)\mapsto (A_0^*(v),B_0^*(v))$ is Borel-measurable because it is the composition of the mappings $v\mapsto (A(v),B(v))$ and $(A(v),B(v))\mapsto (A_0^*(v),B_0^*(v))$, the  first of which is Borel-measurable by Proposition~\ref{prob:graphical prob multiscale} and the second of which  is a mapping between finite sets.

By the definition~\eqref{eq:def our E with beta} of $E$, it follows from~\eqref{eq:def GAB} and~\eqref{eq:we through away the bad edges-new} that 
\begin{equation}\label{eq:on star product points are far apart}
\forall (x,y)\in A_0^*(v)\times B_0^*(v), \qquad d_\MM(x,y)> \frac{\beta\tau}{\min\{\rho(x),\rho(y)\}}. 
\end{equation}
Furthermore,
\begin{align}\label{eq:substitute our r}
\begin{split}
\int_{\R^n}\omega\big(A_0^*(v)\times B_0^*(v)\big)\ud\gamma_n(v)\stackrel{\eqref{eq:pass to L0R0-later}}{\ge} \int_{\R^n}\omega\big(A(v)\times  B(v)\big)\ud\gamma_n&(v)- 2\int_{\R^n}\nu^*_{Q|_{A(v)\cup B(v)}}(\sfG_v)\ud\gamma_n(v)\\&\stackrel{\eqref{eq:restate super gaussian omega}\wedge \eqref{eq:use expected fractional}}{\ge} \theta e^{-\kappa C^2}-9e^{-\frac14 r^2C^2}\ge \frac{\theta}{2}e^{-\kappa C^2},
\end{split}
\end{align}
provided  $r\ge  2\sqrt{\kappa+\log(18/\theta)}$, so that the penultimate step of~\eqref{eq:substitute our r} is valid (this is where  the assumption $C\ge 1$ is used). Recalling the dependence~\eqref{eq:choose our alpha} of $r$ on $\alpha$, and that we are assuming  $\alpha\ge \alpha_0$, this will be  satisfied if $\alpha_0\ge (2/\zeta)\sqrt{\kappa+\log(18/\theta)}$. Earlier we also required that the restriction~\eqref{first restriction on r} on $\alpha_0$ holds, so altogether the entirety of the  requirements that we impose on $\alpha_0$ will be satisfied by the following choice: 
\begin{equation}\label{eq:def r choice}
\alpha_0\eqdef  \max\bigg\{\frac{2}{\zeta}\sqrt{\kappa+\log\frac{18}{\theta}},6 \log(48\zeta)\bigg\}.
\end{equation}

We have thus almost established the assertions of Theorem~\ref{thm:main separated piut together}, except that  $A_0^*(v)$ or $B_0^*(v)$ could possibly be empty for some $v\in \R^n$, while the subsets of $\MM$ that  Theorem~\ref{thm:main separated piut together} outputs are required (for convenience in subsequent applications) to be nonempty. This further requirement can be ensured since we are assuming in Theorem~\ref{thm:main separated piut together} that $\tau\ge \diam(\MM)$, so we can fix $x_0,y_0\in \MM$ with $d_\MM(x_0,y_0)\ge \tau$ and define 
$$
\forall v\in \R^n,\qquad \big(A^*(v),B^*(v)\big)\eqdef \left\{\begin{array}{ll} \big(A^*_0(v),B^*_0(v)\big) & \mathrm{if}\ \emptyset \notin \big\{A^*_0(v),B^*_0(v)\big\},\\ \big(\{x_0\},\{y_0\}\big),&\mathrm{if}\ \emptyset \in \big\{A^*_0(v),B^*_0(v)\big\}.\end{array}\right.
$$
Then, the mapping $(v\in \R^n)\mapsto (A^*(v),B^*(v))$ is Borel-measurable as it is the composition of the mappings $v\mapsto (A_0^*(v),B_0^*(v))$ and $(A_0^*(v),B_0^*(v))\mapsto (A^*(v),B^*(v))$, the first of which we checked is measurable  and the second of which is between finite sets. 
Because $0<\beta\le 1$ and $\rho\ge 1$, it follows from~\eqref{eq:on star product points are far apart} that the first desired conclusion~\eqref{eq:star separation actual distances} of Theorem~\ref{thm:main separated piut together} holds, and it follows from~\eqref{eq:substitute our r} that   the second desired conclusion~\eqref{eq:expected omega weight} of Theorem~\ref{thm:main separated piut together} holds. 
\end{proof}

\section{Duality}\label{sec:duality}

\noindent Our goal here is to deduce Theorem~\ref{thm:random zero general} from Theorem~\ref{thm:main separated piut together}, which is primarily straightforward duality: 

\begin{lemma}\label{lem:duality} Fix $0\le \mathfrak{p} \le 1$. Suppose that $X$ is a finite set. Let $\phi:X\times X\to \R$ be a  symmetric real-valued function, and let $S$ be a nonempty subset of $X\times X$. Fix also any function  $\psi:X\to \R$ and define
\begin{equation}\label{eq:def calF}
\mathcal{F}\eqdef \Big\{(A,B)\in \big(2^X\setminus\{\emptyset\}\big)\times  \big(2^X\setminus\{\emptyset\}\big):\ \forall (x,y)\in A\times B,\quad \phi(x,y)\ge \max\big\{\psi(x),\psi(y)\big\}\Big\}.
\end{equation}
Assume that for every probability measure $\omega$  on $X\times X$ whose support is contained in $S$ there exists a probability measure $\mu_\omega$ on $\mathcal{F}$ that satisfies:
\begin{equation}\label{eq:average weight assumption}
\int_\mathcal{F} \frac{\omega(A\times B)+\omega(B\times A)}{2}\ud \mu_\omega (A,B)\ge \mathfrak{p}.
\end{equation}
Then, there exists a probability measure  $\prob$ on $2^{X}\setminus\{\emptyset\}$  such that
\begin{equation}\label{eq:p lower prob}
\forall (x,y)\in S,\qquad \prob\Big[\emptyset \neq \cZ\subset X:\ x\in \cZ\quad \mathrm{and} \quad \min_{z\in \cZ} \phi(y,z)\ge \psi(y)\Big]\ge \mathfrak{p}.
\end{equation}
\end{lemma}

\begin{proof} For a finite set $U$ we denote the simplex of all the probability measures on $U$ by $\Delta_U$. Observe that
$$
\min_{\omega\in \Delta_{S}}\max_{\mu\in \Delta_\cF} \int_\mathcal{F} \frac{\omega(A\times B)+\omega(B\times A)}{2}\ud \mu(A,B)\ge \min_{\omega\in \Delta_{S}}\int_\mathcal{F} \frac{\omega(A\times B)+\omega(B\times A)}{2}\ud \mu_\omega(A,B)\stackrel{\eqref{eq:average weight assumption}}{\ge} \mathfrak{p}.
$$
By the minimax theorem~\cite{vNe28,Sim95}, we therefore have
\begin{align*}
p&\le \max_{\mu\in \Delta_\cF} \min_{\omega\in \Delta_{S}} \int_\mathcal{F}  \frac{\omega(A\times B)+\omega(B\times A)}{2}\ud \mu(A,B) \\
&= \max_{\mu\in \Delta_\cF} \min_{\omega\in \Delta_{S}} \left( \frac12 \sum_{(A,B)\in \cF}\sum_{(x,y)\in A\times B}\omega(x,y)\mu(A,B)+\frac12 \sum_{(A,B)\in \cF}\sum_{(x,y)\in B\times A}\omega(x,y)\mu(A,B)\right)\\&= \max_{\mu\in \Delta_\cF} \min_{\omega\in \Delta_{S}} \sum_{(x,y)\in S}\omega(x,y)\frac{\mu \big(\{(A,B)\in \mathcal{F}:\ (x,y)\in A\times B\}\big)+ \mu \big(\{(A,B)\in \mathcal{F}:\ (y,x)\in A\times B\}\big)}{2}\\&= \max_{\mu\in \Delta_\cF}  \min_{(x,y)\in S} \frac{\mu \big(\{(A,B)\in \mathcal{F}:\ (x,y)\in A\times B\}\big)+ \mu \big(\{(A,B)\in \mathcal{F}:\ (y,x)\in A\times B\}\big)}{2}.
\end{align*}
Thus, there exists a probability measure $\mu$ on $\cF$ that satisfies
\begin{equation}\label{eq:on S times S}
\forall (x,y)\in S,\qquad \frac12 \mu \big(\{(A,B)\in \mathcal{F}:\ (x,y)\in A\times B\}\big)+ \frac12\mu \big(\{(A,B)\in \mathcal{F}:\ (y,x)\in A\times B\}\big)\ge \mathfrak{p}.
\end{equation}

Suppose that $(\cA,\cB)$ is a random element of $\mathcal{F}$ that is distributed according to $\mu$. Let $\e$ be a standard $\{0,1\}$-valued Bernoulli random variable (thus, the probabilities that $\e=0$ and $\e=1$ are both equal $1/2$) that is independent of $(\cA,\cB)$. Denote the probability space on which $(\cA,\cB)$ and $\e$ are defined by $(\Omega,\bbP)$, and let $\cZ$ be the random nonempty subset of $X$ which is the following function of $\Omega$: if $\e=0$, then $\cZ=\cA$ and if $\e=1$, then $\cZ=\cB$. This definition means that for every fixed $x,y\in X$ we have 
\begin{align}\label{eq:condition Z on eps}
\begin{split}
\bbP\big[x\in \cZ\quad&\mathrm{and}\quad \min_{z\in \cZ} \phi(y,z)\ge \psi(y)\big]\\&=\frac12 \bbP\big[x\in \cA\quad\mathrm{and}\quad \min_{z\in \cA}\phi(y,z)\ge \psi(y)\big]+\frac12 \bbP\big[x\in \cB\quad\mathrm{and}\quad \min_{z\in \cB}\phi(y,z)\ge \psi(y)\big].
\end{split}
\end{align}
As $\phi$ is a symmetric function, by the definition~\eqref{eq:def calF} of $\cF$ the following inclusions of events hold:
\begin{equation}\label{eq:two inclusions AB}
\left\{\begin{array}{ll}\{x\in \cA\quad\mathrm{and}\quad y\in  \cB\} \subset \big\{x\in \cA\ \quad\mathrm{and}\quad \min_{z\in \cA}\phi(y,z)\ge \psi(y)\big\},\\
\{x\in \cB\quad\mathrm{and}\quad y\in  \cA\} \subset \big\{x\in \cB\ \quad\mathrm{and}\quad \min_{z\in \cB}\phi(y,z)\ge \psi(y)\big\}.\end{array}\right.
\end{equation}
Therefore,
\begin{align}\label{eq:lower prob Z all xy}
\begin{split}
\bbP\big[x\in \cZ\quad\mathrm{and}\quad \min_{z\in \cZ} \phi(y,z) \ge& \psi(y)\big]
\stackrel{\eqref{eq:condition Z on eps} \wedge \eqref{eq:two inclusions AB} }{\ge} \frac12 \bbP\big[x\in \cA\quad\mathrm{and}\quad y\in  \cB\big]+\frac12 \bbP\big[x\in \cB\quad\mathrm{and}\quad y\in  \cA\big]\\
&=\frac12 \mu \big(\{(A,B)\in \mathcal{F}:\ (x,y)\in A\times B\}\big)+ \frac12\mu \big(\{(A,B)\in \mathcal{F}:\ (y,x)\in A\times B\}\big).
\end{split}
\end{align}
The desired probabilistic estimate~\eqref{eq:p lower prob} now follows by combining~\eqref{eq:on S times S} and~\eqref{eq:lower prob Z all xy}.
\end{proof}

We can now explain how to deduce Theorem~\ref{thm:random zero general} from Theorem~\ref{thm:main separated piut together}; as Theorem~\ref{thm:main separated piut together} was already proven, this will complete the proof of Theorem~\ref{thm:random zero general}. The reasoning below will freely use the notations and assumptions from the statements of   Theorem~\ref{thm:main separated piut together}.

\begin{proof}[Proof of Theorem~\ref{thm:random zero general}] We will use  the notations and assumptions of Theorem~\ref{thm:random zero general} and  Theorem~\ref{thm:main separated piut together}, with the following choice of universal constant $\alpha$ in the definition $\beta=s^{\alpha/\e}$ in Theorem~\ref{thm:random zero general}, which uses the universal constant $\kappa>1$ from Proposition~\ref{prob:graphical prob multiscale}; this is consistent with requirement~\eqref{eq:beta with alpha constant}  of Theorem~\ref{thm:main separated piut together}, as $\alpha\ge \alpha_0$: 
\begin{equation}\label{eq:choose alpha constant}
\alpha\eqdef\max\left\{2e\sqrt{2\kappa}, \alpha_0\right\}.
\end{equation}

Our goal is to apply Lemma~\ref{lem:duality} with $X=\MM$, $\phi=d_\MM$ and $S=\{(x,y)\in \MM\times \MM:\ d_\MM(x,y)\ge \tau\}$; note that the conlusion of Theorem~\ref{thm:random zero general} is vacuous if $\diam(\MM)<\tau$, so we may assume from now that $S\neq\emptyset$. We will also choose $\psi=\beta\tau/\rho$, where   $\rho$ is from~\eqref{eq:our rho with alpha const}. With these choices, the set $\mathcal{F}$ in~\eqref{eq:def calF} becomes
\begin{equation}\label{eq: rewrite calF in our case}
\mathcal{F}=\bigg\{(A,B)\in \big(2^\MM\setminus \{\emptyset\}\big)\times  \big(2^\MM\setminus\{\emptyset\}\big):\ \forall (x,y)\in A\times B,\quad d_\MM(x,y)\ge  \frac{\beta\tau}{\min\{\rho(x),\rho(y)\}}\bigg\}.\end{equation}

Fix $C\ge 1$ and a probability measure $\omega$ on $\MM\times \MM$ whose support is contained in $S$, which by the definition of $S$ means that  every $x,y\in \MM$ with $\omega(x,y)>0$ satisfy $d_\MM(x,y)\ge \tau$.  We can thus apply Theorem~\ref{thm:main separated piut together} to the symmetrization of $\omega$, i.e., to the probability measure that assigns the mass $(\omega(x,y)+\omega(y,x))/2$ to each $(x,y)\in \MM\times \MM$, to get for every $v\in \R^n$ nonempty subsets $A^*(v),B^*(v)$ of $\MM$. Let $\mu_{\omega,C}$ be the law of the pair $(A^*(v),B^*(v))$ as $v\in \R^n$ is distributed according to $\gamma_n$. By~\eqref{eq: rewrite calF in our case} and~\eqref{eq:star separation actual distances} we know that $(A^*(v),B^*(v))\in \mathcal{F}$ for all $v\in \R^n$, i.e., $\mu_{\omega,C}$ is supported on $\mathcal{F}$. By~\eqref{eq:expected omega weight} applied to the symmetrization of $\omega$, 
\begin{equation}\label{eq:average weight assumption-in proof}
\int_\mathcal{F} \frac{\omega(A\times B)+\omega(B\times A)}{2}\ud \mu_{\omega,C} (A,B)\gtrsim e^{-\kappa C^2}.
\end{equation}

We are thus in position to use Lemma~\ref{lem:duality} to get a probability measure $\prob^C$ on $2^\MM\setminus\{\emptyset\}$ that satisfies 
\begin{equation*}\label{eq:p lower prob alpha C}
\forall (x,y)\in S,\qquad  \prob^C\big[\emptyset \neq \cZ\subset \MM:\ d_\MM(x,\cZ)\ge \psi(y)\quad \mathrm{and} \quad y\in \cZ\big]\gtrsim e^{-\kappa C^2}.
\end{equation*}
Recalling our choice of $S$ and $\psi$, as well as the notation for $\rho$ in~\eqref{eq:our rho with alpha const}, this conclusion coincides with the requirement that every $x,y\in \MM$ with $d_\MM(x,y)\ge \tau$ satisfy
\begin{equation}\label{eq:p lower prob alpha C}
 \prob^C\bigg[\emptyset \neq \cZ\subset \MM:\  d_\MM(x,\cZ)\ge \frac{\beta \tau}{1+\frac{1}{\alpha C} \sqrt{\log \frac{\mu (B_\MM(x,19\beta\tau))}{\mu(B_\MM(x,\beta\tau))}}}\quad \mathrm{and} \quad y\in \cZ\bigg]\gtrsim e^{-\kappa C^2}.
\end{equation}

Observe that the special case $C=O(1)$ of~\eqref{eq:p lower prob alpha C} already implies Theorem~\ref{thm:random zero}. To deduce Theorem~\ref{thm:random zero general}, let $\cZ$ be the following random nonempty subset of $\MM$:  First choose $k\in \N$ with probability $2^{-k}$, and then select a random subset of $\MM$ according to $\prob^C$ for $C=e^{k-1}$. Denote  the law of this random subset $\cZ$ by $\prob$. 

Fix $x,y\in \MM$ with $d_\MM(x,y)\ge \tau$. Suppose first that 
\begin{equation}\label{eq:lambda not too crazy}
\frac{\alpha}{\sqrt{\log \frac{\mu (B_\MM(x,19\beta\tau))}{\mu(B_\MM(x,\beta\tau))}}}\le \lambda\le \frac12. 
\end{equation}
As $1/\lambda-1>0$, 
$$
\frac{\sqrt{\log \frac{\mu (B_\MM(x,19\beta\tau))}{\mu(B_\MM(x,\beta\tau))}}}{\alpha\left(\frac{1}{\lambda }-1\right)}\ge \frac{\lambda}{\alpha} \sqrt{\log \frac{\mu (B_\MM(x,19\beta\tau))}{\mu(B_\MM(x,\beta\tau))}}\stackrel{\eqref{eq:lambda not too crazy}}{\ge} 1.
$$
Hence, there exists $k\in \N$ such that 
\begin{equation}\label{eq:ek range}
\frac{\sqrt{\log \frac{\mu (B_\MM(x,19\beta\tau))}{\mu(B_\MM(x,\beta\tau))}}}{\alpha\left(\frac{1}{\lambda }-1\right)}\le   e^{k-1}  < \frac{e\sqrt{\log \frac{\mu (B_\MM(x,19\beta\tau))}{\mu(B_\MM(x,\beta\tau))}}}{\alpha\left(\frac{1}{\lambda }-1\right)}\le  \frac{2e\lambda}{\alpha}\sqrt{\log \frac{\mu (B_\MM(x,19\beta\tau))}{\mu(B_\MM(x,\beta\tau))}}, 
\end{equation}
where the third inequality in~\eqref{eq:ek range} is valid because $0<\lambda\le 1/2$, so $1/\lambda-1\ge 1/(2\lambda)$.  Now,
\begin{align}\label{eq:use for C=ek-1}
\begin{split}
\prob \big[d_\MM&(x,\cZ)\ge \lambda\beta\tau\quad\mathrm{and}\quad y\in \cZ\big]\ge 2^{-k}\prob^{e^{k-1}} \big[d_\MM(x,\cZ)\ge \lambda\beta\tau\quad\mathrm{and}\quad y\in \cZ\big]\\&\ge 2^{-k}\prob^{e^{k-1}}\bigg[ d_\MM(x,\cZ)\ge \frac{\beta \tau}{1+\frac{1}{\alpha e^{k-1}} \sqrt{\log \frac{\mu (B_\MM(x,19\beta\tau))}{\mu(B_\MM(y,\beta\tau))}}}\quad \mathrm{and} \quad y\in \cZ\bigg]\gtrsim e^{-\left(k\log 2+\kappa e^{2k-2}\right)},
\end{split}
\end{align}
where the first step of~\eqref{eq:use for C=ek-1} is a direct consequence of the definition of $\prob$, the second step of~\eqref{eq:use for C=ek-1} is valid by the  first inequality in~\eqref{eq:ek range},  which can be rewritten as
$$
\lambda\le \frac{1}{1+\frac{1}{\alpha e^{k-1}} \sqrt{\log \frac{\mu (B_\MM(x,19\beta\tau))}{\mu(B_\MM(x,\beta\tau))}}},
$$
and the third step of~\eqref{eq:use for C=ek-1}  is an application of~\eqref{eq:p lower prob alpha C} with $C=e^{k-1}\ge 1$.  Next, because  $\kappa,k\ge 1$, it is straightforward to check that $k\log 2+\kappa e^{2k-2}\le 2\kappa e^{2k-2}$. Consequently, 
\begin{align*}
\prob \big[d_\MM(x,\cZ)\ge \lambda\beta\tau\quad\mathrm{and}\quad y\in \cZ\big]\stackrel{\eqref{eq:use for C=ek-1}}{\gtrsim}  e^{-2\kappa e^{2k-2}}\stackrel{\eqref{eq:ek range}}{\ge} \left(\frac{\mu (B_\MM(x,19\beta\tau))}{\mu(B_\MM(x,\beta\tau))}\right)^{-\frac{8\kappa e^2}{\alpha^2}\lambda^2} \ge \left(\frac{\mu (B_\MM(x,19\beta\tau))}{\mu(B_\MM(x,\beta\tau))}\right)^{-\lambda^2},
\end{align*}
where the last step holds as $\alpha\ge 2e\sqrt{2\kappa}$ by~\eqref{eq:choose alpha constant}. This proves~\eqref{eq:in main thm-general} for $\lambda$ satisfying~\eqref{eq:lambda not too crazy}. The validity of~\eqref{eq:in main thm-general} in the entire range $0<\lambda\le 1/2$  follows formally from its validity when $\lambda$ belongs to the restricted range~\eqref{eq:lambda not too crazy}:   if $0<\lambda<\alpha/\sqrt{\log (\mu (B_\MM(x,19\beta\tau))/\mu(B_\MM(x,\beta\tau))}$, then by using the endpoint case of what we just proved,
\begin{equation*}
\prob \big[d_\MM(x,\cZ)\ge \lambda\beta\tau\ \ \ \mathrm{and}\ \  \  y\in \cZ\big]\ge \prob \bigg[d_\MM(y,\cZ)\ge \frac{\alpha}{\sqrt{\log \frac{\mu (B_\MM(x,19\beta\tau))}{\mu(B_\MM(x,\beta\tau))}}}\beta\tau\ \ \ \mathrm{and}\ \ \  x\in \cZ\bigg]\gtrsim e^{-\alpha^2}\asymp 1.\tag*{\qedhere}
\end{equation*}
\end{proof}

\section{From a random  zero set to a ``no measure is  concentrated phenomenon''}\label{sec:obs}

Our goal here is to show how Theorem~\ref{thm:observable} and Theorem~\ref{thm:no-super} follow from Theorem~\ref{cor:spreading doubling}. We will, in fact, obtain the following more precise result, from which  it is quick to deduce Theorem~\ref{thm:observable} and Theorem~\ref{thm:no-super}:

\begin{theorem}\label{thm:iso p} For every $0<s,\e<1$ there exists $c=c(s,\e)>0$ with the following property. Given $n\in \N$, let  $(\MM,d)$ be a $5^n$-doubling $(s,\e)$-quasisymmetrically Hilbertian metric. Suppose that $\mu$ is a Borel probability measure on $\MM$ and fix $0<\fp\le 1$ and $\fd>0$ for which 
\begin{equation}\label{eq:fdfp}
(\mu\times \mu) \big(\{(x,y)\in \MM\times \MM:\ d(x,y)\ge \fd\}\big)\ge \fp.
\end{equation}
 Then, for every $\sqrt{\kappa/n}\le \phi \le 1$, where $\kappa>1$ is the universal constant from Theorem~\ref{cor:spreading doubling} , we have 
\begin{equation}\label{eq:iso lower p}
I_\mu^\MM(c\phi\fd)\ge \fp e^{-n\phi^2}.
\end{equation}
\end{theorem}

Prior to proving Theorem~\ref{thm:iso p}, we will next assume its validity and explain how it implies Theorem~\ref{thm:observable} and Theorem~\ref{thm:no-super}. If the assumptions and notation of Theorem~\ref{thm:iso p} hold, then 
\begin{equation}\label{iso in full range with twidle}
\forall 0\le \phi\le 1,\qquad I_\mu^\MM(c\phi\fd)\gtrsim \fp e^{-n\phi^2}.
\end{equation}
Indeed, conclusion~\eqref{eq:iso lower p} of Theorem~\ref{thm:iso p} implies~\eqref{iso in full range with twidle} if $\sqrt{\kappa/n}\le \phi \le 1$ (with the implicit constant in~\eqref{iso in full range with twidle}  equal to $1$).  If $0<\phi\le \sqrt{\kappa/n}$, then since per~\eqref{eq:def:isoperimetric function} the isoperimetric function is non-increasing, 
$$
I_\mu^\MM(c\phi\fd)\ge I_\mu^\MM\Big(c\sqrt{\frac{\kappa}{n}}\fd\Big)\stackrel{\eqref{eq:iso lower p}}{\ge}\fp e^{-\kappa}\gtrsim_\kappa \fp e^{-n\phi^2}. 
$$

The normalization assumption of Theorem~\ref{thm:observable} and Theorem~\ref{thm:no-super} is that~\eqref{eq:fdfp} holds for $\fd=1$ and $\fp=\frac12$, as we say that $\mu$ is normalized if the median of $d(x,y)$ when $(x,y)\in \MM\times \MM$ is distributed according to $\mu\times \mu$ is equal to $1$. Thus, Theorem~\ref{thm:no-super} is a special case of~\eqref{iso in full range with twidle}. 

To prove  (a similar strengthening of) Theorem~\ref{thm:observable}, suppose first that $\fp/(2e^n)\le\theta \le \min\{\fp/(2e^\kappa),\fp^2/4\}$, 
where we are continuing to reason under the assumptions and notation of Theorem~\ref{thm:iso p}. Define
\begin{equation}\label{eq:choose our phi}
\phi_0=\phi_0(\fp,\theta,n)\eqdef \sqrt{\frac{\log\frac{\fp}{2\theta}}{n}}\asymp \sqrt{\frac{\log\frac{1}{\theta}}{n}},
\end{equation}
where the last equivalence in~\eqref{eq:choose our phi} holds as  $\theta\le \fp^2/4$. The rest of the above assumptions on $\theta$ are equivalent to  $\sqrt{\kappa/n}\le \phi_0\le 1$, so~\eqref{eq:iso lower p} holds, i.e., $I_\mu^\MM(c\phi_0\fd)\ge 2\theta>\theta$.  Therefore, by the first implication in~\eqref{eq:relations obs iso},
\begin{equation}\label{eq:interestting range theta}
\forall \frac{\fp}{2e^n}\le\theta \le \min\Big\{\frac{\fp}{2e^\kappa},\frac{\fp^2}{4}\Big\},\qquad \mathrm{ObsDiam}_\mu^{\MM}(\theta)\ge c\phi_0\fd\stackrel{\eqref{eq:choose our phi}}{\asymp} c\fd\sqrt{\frac{\log\frac{1}{\theta}}{n}}.
\end{equation}
However, the following estimate on the observable diameter of the Euclidean sphere is a standard consequence of L\'evy's spherical isoperimetric theorem~\cite{Lev51} (see e.g.~equation~(1.21) in~\cite{Led01}):
\begin{equation}\label{eq:quote observable sphere}
\forall \theta>0,\qquad \mathrm{ObsDiam}_{\sigma^{n-1}}^{S^{n-1}}(\theta)\lesssim \sqrt{\frac{\log\frac{1}{\theta}}{n}}. 
\end{equation}
We therefore conclude from~\eqref{eq:interestting range theta} and~\eqref{eq:quote observable sphere} that
\begin{equation*}\label{eq:restate interesting range theta}
\forall \frac{\fp}{2e^n}\le\theta \le \min\Big\{\frac{\fp}{2e^\kappa},\frac{\fp^2}{4}\Big\},\qquad \frac{1}{\fd}\mathrm{ObsDiam}_\mu^{\MM}(\theta)\gtrsim_{s,\e} \mathrm{ObsDiam}_{\sigma^{n-1}}^{S^{n-1}}(\theta).
\end{equation*}

If $0<\theta\le \fp/(2e^n)$, then because the definition of the $\theta$-observable diameter (of any metric probability space) immediately implies that it is non-increasing in $\theta$, the special case $\theta=\fp/(2e^n)$ of~\eqref{eq:interestting range theta} gives  
$$
\frac{1}{\fd}\mathrm{ObsDiam}_\mu^{\MM}(\theta)\ge
\frac{1}{\fd}\mathrm{ObsDiam}_\mu^{\MM}\Big(\frac{\fp}{2e^n}\Big)\stackrel{\eqref{eq:interestting range theta}}{\gtrsim}_{s,\e} 1\asymp\diam_{\ell_2^n}\big(S^{n-1}\big)\gtrsim  \mathrm{ObsDiam}_{\sigma^{n-1}}^{S^{n-1}}(\theta).
$$
Altogether, we have verified that the assumptions of Theorem~\ref{thm:iso p} imply that 
\begin{equation}\label{eq:interestting range theta up to zero}
\forall 0<\theta \le \min\Big\{\frac{\fp}{2e^\kappa},\frac{\fp^2}{4}\Big\},\qquad \frac{1}{\fd}\mathrm{ObsDiam}_\mu^{\MM}(\theta)\gtrsim_{s,\e} \mathrm{ObsDiam}_{\sigma^{n-1}}^{S^{n-1}}(\theta).
\end{equation}
Theorem~\ref{thm:observable} is the special case $\fp=\frac12$ and $\fd=1$ of~\eqref{eq:interestting range theta up to zero},  so we get~\eqref{eq:observable lower}   with $\theta_0=\min\{1/(4e^\kappa),1/16\}$. 

Note that~\cite{naor2005quasisymmetric} proved the (suboptimal, dimension-dependent) estimate~\eqref{eq:sub optimal observable extremality} under the assumption that $\diam(\MM)<\infty$ and the following expectation bound holds for some $\alpha>0$:
\begin{equation}\label{eq:expectation rather than median}
\iint_{\MM\times \MM} d(x,y)\ud\mu(x)\ud\mu(y)\ge \alpha \diam(\MM). 
\end{equation}
This setting is covered by the above derivation of~\eqref{eq:interestting range theta up to zero} assuming the probability bound~\eqref{eq:fdfp} because  
$$
\alpha \diam(\MM)\stackrel{\eqref{eq:expectation rather than median}}{\le} (\mu\times \mu) \Big(\big\{(x,y)\in \MM\times \MM:\ d(x,y)\ge \frac{\alpha}{2}\diam(\MM)\big\}\Big) \diam(\MM) +\frac{\alpha}{2}\diam(\MM).
$$
So, \eqref{eq:fdfp} holds with $\fp=\alpha/2$ and $\fd=\alpha \diam(\MM)/2$, in which case~\eqref{eq:interestting range theta up to zero} becomes
$$
\forall 0<\theta \le \min\Big\{\frac{\alpha}{4e^\kappa},\frac{\alpha^2}{16}\Big\},\qquad \frac{\mathrm{ObsDiam}_\mu^{\MM}(\theta)}{\diam(\MM)}\gtrsim_{s,\e} \alpha \mathrm{ObsDiam}_{\sigma^{n-1}}^{S^{n-1}}(\theta),
$$ 
which is the form of the conclusion of~\cite[Theorem~1.7]{naor2005quasisymmetric}, though now it is dimension-independent. 

Passing to the proof of Theorem~\ref{thm:iso p}, we will first record Lemma~\ref{lem:separated sets} below. It shows that the existence of large well-separated subsets yields a lower bound on the isoperimetric function; a proof of this simple fact is implicit in the proof of~\cite[Proposition~2.26]{Shi16}. Recall that, as in~\eqref{eq:def:isoperimetric function}, the Borel subsets of a metric space $(\MM,d)$ are denoted $\mathcal{Bor}(\MM,d)$. The distance between $\emptyset \neq \mathscr{A},\mathscr{B}\subset \MM$ is denoted (as usual) by $$d(\mathscr{A},\mathscr{B})\eqdef \inf\big\{d(a,b):\ (a,b)\in \mathscr{A}\times \mathscr{B}\big\}.$$

\begin{lemma}\label{lem:separated sets} For every metric space $(\MM,d)$, every Borel probability measure $\mu$ on $\MM$ satisifes
$$
\forall t>0,\qquad I_\mu^\MM(t)\ge \sup_{\substack{\mathscr{A},\mathscr{B}\in \mathcal{Bor}(\MM,d)\setminus\{\emptyset\}\\d(\mathscr{A},\mathscr{B})\ge 2t}} \min\big\{\mu(\mathscr{A}),\mu(\mathscr{B})\big\}.
$$
\end{lemma}

\begin{proof} Fix $t>0$. Let $\mathscr{A},\mathscr{B}$ be nonempty Borel subsets of $\MM$ with $d(\mathscr{A},\mathscr{B})\ge 2t$. By the triangle inequality, $2t\le d(\mathscr{A},\mathscr{B})\le d(x,\mathscr{A})+d(x,\mathscr{B})\le 2\max\{d(x,\mathscr{A}),d(x,\mathscr{B})\}$ for every $x\in \MM$. Thus, $\max\{d(x,\mathscr{A}),d(x,\mathscr{B})\}\ge t$ for every $x\in \MM$, i.e.,   $\MM=\{x\in \MM:\ d(x,\mathscr{A})\ge t\}\cup \{x\in \MM:\ d(x,\mathscr{B})\ge t\}$. As $\mu$ is a probability measure, this implies that $\max\{\mu(\{x\in \MM:\ d(x,\mathscr{A})\ge t\}),\mu(\{x\in \MM:\ d(x,\mathscr{B})\ge t\})\}\ge \frac12$. If  $\mu(\{x\in \MM:\ d(x,\mathscr{A})\ge t\})\ge \frac12$, then denote $\sub=\{x\in \MM:\ d(x,\mathscr{A})\ge t\}$ and observe that (by design) we have  $\{x\in \MM:\ d(x,\sub)\ge t\}\supseteq \mathscr{A}$. So,
$$
I_\mu^\MM(t)\stackrel{\eqref{eq:def:isoperimetric function}}{\ge} \mu\big(\{x\in \MM:\ d(x,\sub)\ge t\}\big)\ge \mu(\mathscr{A}). 
$$ 
The symmetric reasoning gives $I_\mu^\MM(t)\ge \mu(\mathscr{B})$ in the remaining case $\mu(\{x\in \MM:\ d(x,\mathscr{B})\ge t\})\ge \frac12$. 
\end{proof}

The following  lemma shows that a spreading random zero set is an obstruction to isoperimetry: 

\begin{lemma}\label{lem:zero to iso} If a finite metric space $(\MM,d)$ admits a  random zero set which is $\zeta$-spreading with probability $\d$ (per Definition~\ref{def:random zero}) for some  $\zeta>0$ and $0<\d\le 1$, then every probability measure $\mu$ on $\MM$ satisfies
\begin{equation}\label{eq:goal no super concentration}
\forall t>0,\qquad I_\mu^\MM(t)\ge \d (\mu\times \mu)\big(\{(x,y)\in\MM\times \MM:\ d(x,y)\ge 2t\zeta\}\big).
\end{equation}
\end{lemma}

\begin{proof} Fix $t>0$ and let $\prob^{2t\zeta}$ be as in Definition~\ref{def:random zero} with $\tau=2t\zeta$, i.e., 
\begin{equation}\label{eq:def  zero set-rescaled}
\forall x,y\in \MM, \qquad d(x,y)\ge 2t\zeta\implies \prob^{2t\zeta}\big[ \emptyset \neq  \cZ\subset \MM:\  d(y,\cZ) \ge 2t  \quad  \mathrm{and}\quad  x\in \cZ   \big]\ge \d.
\end{equation}

Then,
\begin{align*}
\int_{2^\MM\setminus\{\emptyset\}}&\mu\big(\{y\in \MM:\ d(y,\cZ)\ge 2t\}\big) \mu(\cZ) \prob^{2t\zeta}(\cZ) \\&=\iiint_{(2^\MM\setminus\{\emptyset\})\times \MM\times \MM}\1_{\left\{d(y,\cZ)\ge 2t\ \ \mathrm{and}\ \  x\in \cZ\right\}} \ud\mu(x)\ud\mu(y)\ud\prob^{2t\zeta}(\cZ)\\&=\iint_{\MM\times \MM} \prob^{2t\zeta}\big[ \emptyset \neq  \cZ\subset \MM:\  d(y,\cZ) \ge 2t \quad  \mathrm{and}\quad  x\in \cZ   \big]\ud\mu(x)\ud\mu(y)\\ 
&\ge \iint_{\{(x,y)\in\MM\times \MM:\ d(x,y)\ge 2t\zeta\}} \prob^{2t\zeta}\big[ \emptyset \neq  \cZ\subset \MM:\  d(y,\cZ) \ge 2t  \quad  \mathrm{and}\quad  x\in \cZ   \big]\ud\mu(x)\ud\mu(y)\\
&\!\!\!\!\!\stackrel{\eqref{eq:def  zero set-rescaled}}{\ge} \d (\mu\times \mu)\big(\{(x,y)\in\MM\times \MM:\ d(x,y)\ge 2t\zeta\}\big).
\end{align*}
Therefore, there exists $\emptyset\neq \cZ=\cZ_\mu\subset \MM$ such that 
\begin{align}\label{eq:choose Zmu}
\begin{split}
\min\big\{ \mu\big(\{y\in \MM:\ d(y,\cZ)\ge 2t\}\big), \mu(\cZ)\big\}&\ge  \mu\big(\{y\in \MM:\ d(y,\cZ)\ge 2t\}\big) \mu(\cZ)\\&\ge \d (\mu\times \mu)\big(\{(x,y)\in\MM\times \MM:\ d(x,y)\ge 2t\zeta\}\big). 
\end{split}
\end{align}

The desired estimate~\eqref{eq:goal no super concentration} holds vacuously if its right hand side vanishes, so we may assume that $(\mu\times \mu)\big(\{(x,y)\in\MM\times \MM:\ d(x,y)\ge 2t\zeta\}\big)>0$, whence $\{y\in \MM:\ d(y,\cZ)\ge 2t\}, \cZ\neq \emptyset$ by~\eqref{eq:choose Zmu}. Furthermore, $d(\{y\in \MM:\ d(y,\cZ)\ge 2t\},\cZ)\ge 2t$, so~\eqref{eq:goal no super concentration} follows by combining  Lemma~\ref{lem:separated sets} with~\eqref{eq:choose Zmu}. 
\end{proof}

Lemma~\ref{lem:trivial approx by net} below is an approximation step for passing from Lemma~\ref{lem:zero to iso}  to the analogous statement for infinite spaces; we include its proof for completeness as we did not locate it  in the literature, though it is straightforward and  one could take it an easy exercise rather than reading it.

\begin{lemma}\label{lem:trivial approx by net} Let $(\MM,d)$ be a metric space in which every ball is totally bounded, and let $\mu$ a Borel probability measure on $\MM$. For any $0<\sigma<1$ and $t,D>0$ there exists a finitely supported probability measure $\nu$ on $\MM$ such that 
\begin{equation}\label{eq:still median}
(\mu\times \mu)\big(\{(x,y)\in \MM\times \MM:\  d(x,y)>D\}\big)\le (\nu\times \nu) \big(\{(x,y)\in \MM\times \MM:\ d(x,y)> D\}\big)+\sigma,
\end{equation}
and,
\begin{equation}\label{eq:discretized isoperimetric function}
I_\mu^\MM(t)\ge I_\nu^\MM( t+\sigma )-\sigma.
\end{equation}
\end{lemma}

Prior to justifying Lemma~\ref{lem:trivial approx by net}, we will  show how it quickly implies Theorem~\ref{thm:iso p}:

\begin{proof}[Proof of Theorem~\ref{thm:iso p} assuming Lemma~\ref{lem:trivial approx by net}] We may assume without loss of generality that $0<s,\e\le 1/2$. Since balls in  doubling metric spaces are totally bounded,  by Lemma~\ref{lem:trivial approx by net} it suffices to prove Theorem~\ref{thm:iso p}  when $(\MM,d)$ is a finite metric space. Because $\phi\ge \sqrt{\kappa/n}$, we may apply Theorem~\ref{cor:spreading doubling} with $p=n\phi^2/\kappa\ge 1$ and $K=5^n$. We have $p\le \log K$ as $\kappa>1\ge \phi$, so  $(\MM,d)$ has a random zero set that is $\zeta$-spreading with probability $\d$, where $\d=e^{-\kappa p}=e^{-n\phi^2}$ and $\zeta=s^{-\kappa/\e}\sqrt{(n\log 5)/p}$. Next, apply Lemma~\ref{lem:zero to iso}  with $t=c\phi\fd$, where $c=c(s,\e)=s^{\kappa /\e}/(2\sqrt{\kappa\log 5})$. For the above choices of parameters we have $2t\zeta=\fd$, so~\eqref{eq:goal no super concentration} implies the desired estimate~\eqref{eq:iso lower p} by~\eqref{eq:fdfp}.  
\end{proof}

\begin{proof}[Proof of Lemma~\ref{lem:trivial approx by net}] Fix any $x_0\in \MM$. As $\lim_{R\to \infty}\mu\big(B(x_0,R)\big)=\mu(\MM)=1$, we can also fix $R>t$ such that  
\begin{equation}\label{eq:big ball R}
\mu\big(B(x_0,R-t)\big)\ge 1-\frac{1}{2}\sigma.
\end{equation}
Finally, as $\lim_{\omega\to 0^+} (\mu\times \mu)(\{(x,y)\in \MM\times \MM:\  d(x,y)\ge D+3\omega\})= (\mu\times \mu)(\{(x,y)\in \MM\times \MM:\  d(x,y)> D\})$ we may fix  $0<\omega\le \sigma/2$ for which
\begin{equation}\label{eq:limit in alpha}
(\mu\times \mu)\big(\{(x,y)\in \MM\times \MM:\  d(x,y)> D\}\big) \le (\mu\times \mu)\big(\{(x,y)\in \MM\times \MM:\  d(x,y)\ge D+3\omega\}\big)+\frac{1}{4}\sigma^2.
\end{equation}
Observe in passing that the above choices imply in particular the following estimate:
\begin{align}\label{eq:plus sigma}
\begin{split}
(\mu\times \mu)&\big(\{(x,y)\in \MM\times \MM:\  d(x,y)> D\}\big) \\&\stackrel{\eqref{eq:limit in alpha}}{\le} (\mu\times \mu)\big(\{(x,y)\in B(x_0,R)\times B(x_0,R):\ d(x,y)\ge D+3\omega\}\big)+\frac14 \sigma^2\\&\qquad\qquad +(\mu\times \mu)\Big((\MM\times \MM)\setminus \big(B(x_0,R)\times B(x_0,R)\big)\Big)\\&\stackrel{\eqref{eq:big ball R}}{\le} (\mu\times \mu)\big(\{(x,y)\in B(x_0,R)\times B(x_0,R):\ d(x,y)\ge D+3\omega\}\big)+\sigma.
\end{split}
\end{align}

As every ball in $\MM$ is assumed to be totally bounded, there are $n\in \N$ and $x_1,\ldots,x_n\in B(x_0,R)$ such that 
\begin{equation}\label{eq:alpha dense}
B(x_0,R)\subset \bigcup_{i=1}^n B(x_i,\omega).
\end{equation} 
We will next consider the Voronoi tessellation $\sV_1,\ldots,\sV_n\subset B(x_0,R)$  of $B(x_0,R)$ that is induced by the (ordered) $\omega$-dense subset $\{x_1,\ldots,x_n\}$ of $B(x_0,R)$, which is defined (as usual) by setting 
$$
\sV_1\eqdef \big\{x\in B(x_0,R):\ d(x,x_1)=d(x,\{x_1,\ldots,x_n\})\big\},
$$ 
and then proceeding inductively by setting 
$$
\forall i\in \{1,\ldots,n-1\},\qquad \sV_{i+1}\eqdef \big\{x\in B(x_0,R):\ d(x,x_{i+1})=d(x,\{x_1,\ldots,x_n\})\big \}\setminus \bigcup_{j=1}^i \sV_j.
$$ 
Thus, $d(x,x_i)\le \omega$ for every $i\in [n]$ and every $x\in \sV_i$, thanks to~\eqref{eq:alpha dense}, and because, by design, $\{\sV_i\}_{i\in [n]}$ is a (Borel) partition of $B(x_0,R)$,  the following defines a finitely supported probability measure on $\MM$:   
\begin{equation}\label{eq:discretized measure}
\nu\eqdef \sum_{i=1}^n \frac{\mu(\sV_i)}{\mu\big(B(x_0,R)\big)}\bd_{\! x_i}.
\end{equation}

To check~\eqref{eq:still median}, denote  
\begin{equation}\label{eq:defJ big dist}
J\eqdef \Big\{(i,j)\in [n]\times [n]:\ d(x_i,x_j)\ge D+\omega\Big\}, 
\end{equation}
so that 
\begin{align}\label{eq:nu times nu}
\begin{split}
(\nu\times \nu) \big(\{(x,y)\in \MM\times \MM:\ d(x,y)> D\}\big)\ge (\nu&\times  \nu) \big(\{(x,y)\in \MM\times \MM:\ d(x,y)\ge D+\omega\}\big)\\&\stackrel{\eqref{eq:discretized measure}\wedge \eqref{eq:defJ big dist}}{=}\sum_{(i,j)\in J} \frac{\mu(\sV_i)\mu(\sV_j)}{\mu\big(B(x_0,R)\big)^2}
\ge\mu\bigg(\bigcup_{(i,j)\in J} \sV_i\times \sV_j\bigg),
\end{split}
\end{align}
where the last  step of~\eqref{eq:nu times nu} is valid as $\{\sV_i\times \sV_j\}_{(i,j)\in [n]\times [n]}$ are pairwise disjoint. Also, we will show that the following inclusion holds: 
\begin{equation}\label{eq:J inclusion}
\big\{(x,y)\in B(x_0,R)\times B(x_0,R):\ d(x,y)\ge D+3\omega\big\}\subset \bigcup_{(i,j)\in J}\sV_i\times \sV_j,
\end{equation}
from which we will get that
\begin{equation}\label{eq:on ball less thyan nu}
\begin{split}
(\mu\times \mu)\big(\{(x,y)\in &B(x_0,R)\times B(x_0,R):\ d(x,y)\ge D+3\omega\}\big)\\&\stackrel{\eqref{eq:nu times nu}\wedge\eqref{eq:J inclusion}}{\le} (\nu\times \nu) \big(\{(x,y)\in \MM\times \MM:\ d(x,y)> D\}\big).
\end{split}
\end{equation}
To verify~\eqref{eq:J inclusion}, if $x,y\in B(x_0,R)$, then fix $i,j\in [n]$ such that $(x,y)\in \sV_i\times \sV_j$, whence $d(x,x_i)\le \omega$ and $d(y,x_j)\le \omega$. If also $d(x,y)\ge D+3\omega$, then $d(x_i,x_j)\ge d(x,y)-d(x,x_i)-d(y,y_j)\ge D+\omega$, so $(i,j)\in J$ by~\eqref{eq:defJ big dist}, thus proving~\eqref{eq:J inclusion}. The desired bound~\eqref{eq:still median}  now follows by substituting~\eqref{eq:on ball less thyan nu} into~\eqref{eq:plus sigma}.  

It remains to prove~\eqref{eq:discretized isoperimetric function}. As $\nu$ is supported on the finite set $\{x_1,\ldots,x_n\}$, we can fix $K\subset [n]$ such that 
\begin{equation}\label{maximizer for voronoi measure}
\nu\big(\{x_k\}_{k\in K}\big)\ge \frac12 \qquad\mathrm{and}\qquad I_\nu^\MM(t+\sigma)=\nu\big(\{x\in \MM:\ d(x,\{x_k\}_{k\in K})\ge t+\sigma\}\big).
\end{equation}
If we introduce the notation 
\begin{equation}\label{eq:def C0S0}
\sub_0\eqdef \bigcup_{k\in K} \sV_k\qquad \mathrm{and}\qquad \mathcal{S}_0\eqdef \bigcup_{\ell\in L} \sV_\ell,\quad\mathrm{where}\quad  L\eqdef\big\{\ell\in [n]:\ d(x_\ell,\{x_k\}_{k\in K})\ge t+\sigma\big\},
\end{equation}
then by the definition~\eqref{eq:discretized measure} of $\nu$ we see that~\eqref{maximizer for voronoi measure} becomes
\begin{equation}\label{eq:sigma/3}
\mu(\sub_0)\ge\frac12 \mu\big(B(x_0,R)\big)\qquad\mathrm{and}\qquad \mu(\mathcal{S}_0)=I_\nu^\MM(t+\sigma) \mu\big(B(x_0,R)\big).
\end{equation}

Consider the Borel subsets $\sub,\cS$ of  $\MM$ that are given by
\begin{equation}\label{def:new CS}
\sub\eqdef \sub_0\cup \big(\MM\setminus B(x_0,R)\big)\qquad\mathrm{and}\qquad \cS\eqdef \cS_0\cap B(x_0,R-t). 
\end{equation}
Observe that the following simple inclusion holds:
\begin{equation}\label{eq:S in far from C}
\cS\subset \big\{x\in \MM: d(x,\sub)\ge t \big\}.
\end{equation}
Indeed, take $x\in \cS$ and $y\in \sub$. If $y\in \MM\setminus B(x_0,R)$, then  $d(x,y)\ge d(y,x_0)-d(x,x_0)>t$ since $x\in B(x_0,R-t)$ by~\eqref{def:new CS}. Otherwise $y\in \sub_0$ and $x\in \cS_0$  by~\eqref{def:new CS}, so by~\eqref{eq:def C0S0} there are $\ell\in L$ and $k\in K$ such that $x\in \sV_\ell$ and $y\in \sV_k$, whence $d(x,x_\ell) \le \omega$ and $d(y,x_k)\le \omega$. Because $d(x_\ell,x_k)\ge t+\sigma$ by the definition~\eqref{eq:def C0S0} of $L$, it follows that $d(x,y)\ge d(x_\ell,x_k)-d(x,x_\ell)-d(y,x_k)\ge t+\sigma-2\omega\ge t$, since $\omega\le \sigma/2$. This proves~\eqref{eq:S in far from C}.  

Because $\sub_0\subset B(x_0,R)$, the two sets in the union that defines $\sub$ in~\eqref{eq:def C0S0} are disjoint, so
$$
\mu(\sub)= \mu(\sub_0)+1-\mu\big(B(x_0,R)\big)\stackrel{\eqref{eq:sigma/3}}{\ge} 1-\frac12 \mu\big(B(x_0,R)\big)\ge \frac12.
$$ 
We may therefore apply the definition~\eqref{eq:def:isoperimetric function} of $I_\mu^\MM(t)$ to $\sub$ and conclude the proof of~\eqref{eq:discretized isoperimetric function} as follows:
\begin{multline*}
I_\mu^\MM(t)\ge \mu\big(\{x\in \MM: d(x,\sub)\ge t\} \big)\stackrel{\eqref{eq:S in far from C}}{\ge} \mu(\cS)\stackrel{\eqref{def:new CS}}{\ge} \mu(\cS_0)-\mu\big(\MM\setminus B(x_0,R-t)\big)\\ \stackrel{\eqref{eq:sigma/3}}{=} I_\nu^\MM(t+\sigma) + \big(\mu\big(B(x_0,R)\big)-1\big)I_\nu^\MM(t+\sigma) -1+\mu\big(B(x_0,R-t)\big)\stackrel{\eqref{eq:big ball R}}{\ge} I_\nu^\MM(t+\sigma)-\sigma.\tag*{\qedhere} 
\end{multline*} 
\end{proof}

\begin{remark}\label{rem:expander obs} Fix $n\in \N$ and a $4$-regular graph $\sfG=(V,E)$  with $|V|=5^n$ such the second-largest eigenvalue of the transition matrix of its standard random walk is at most $1-\gamma$, where $\gamma>0$ is a universal constant; the survey~\cite{HLW06} discusses such expander graphs, including their existence. Let $\mu$ be the uniform probability measure on $V$ and let $\mathfrak{m}>0$ be a median of $d_{\sfG}(x,y)$ with respect to $(x,y)\in V\times V$ distributed uniformly at random. Then $\mathfrak{m}\asymp n$ (one direction of this asymptotic equivalence follows from~\cite{Chu89} and the other direction is a consequence of a simple counting argument, as in e.g.~\cite{Mat97}).  By~\cite{GM83} (its discrete counterpart which we are using here is treated in~\cite{AM85,Alo86,BHT00}), we have
\begin{equation}\label{eq:expander iso}
\forall t>0,\qquad I_\mu^{\left(V,\frac{1}{\mathfrak{m}} d_\sfG\right)}(t)\lesssim e^{-O(n)t}.
\end{equation}
By e.g.~\cite[Proposition~1.12]{Led01},  \eqref{eq:expander iso} implies the following estimate on the observable diameter:
$$
\forall 0<\theta\le \frac12,\qquad \mathrm{ObsDiam}_\mu^{\left(V,\frac{1}{\mathfrak{m}} d_\sfG\right)}(\theta)\lesssim \frac{\log \frac{1}{\theta}}{n}. 
$$
The metric space $(V,\mathfrak{m}^{-1} d_\sfG)$ is trivially $5^n$-doubling, so this demonstrates that some assumption must be imposed on metric a space  so that conclusion~\eqref{eq:observable lower} of Theorem~\ref{thm:observable}  will hold (with the implicit dependence on $s,\e$ replaced by dependence on the aforementioned assumption). Finding ``useful minimal assumptions'' that are needed here  is an interesting research direction (purposefully formulated somewhat vaguely). 
\end{remark}

\section{Multi-scale stochastic mixing of random zero sets}\label{sec:descent revisited}


\noindent The main result of this section is the following theorem, which generalizes part of~\cite{KLMN05}. It describes a way to obtain a multi-scale stochastic mixture of distributions over random subsets of a metric space.

\begin{theorem}\label{thm:restatet decent with constants that we get} Let $(\MM,d)$ be a finite metric space and let $\mu$ be a nondegenerate measure on $\MM$. Denote
\begin{equation}\label{eq;aspect ratio notation}
\Phi=\Phi_\mu\eqdef \frac{\mu(\MM)}{\min_{x\in \MM}\mu(x)}.
\end{equation}
Suppose that for every $n\in \Z$ we are given  a probability measure $\prob^n$ on $2^{\MM}\setminus \{\emptyset\}$.  Then, for every $a>b$ there exists a probability measure $\prob$ on  $2^{\MM}\setminus \{\emptyset\}$ such that for every $x,y\in \MM$, every $\ell,u\ge 0$, and every $n\in \Z$,\footnote{Note that the definition~\eqref{eq;aspect ratio notation} of the aspect ratio $\Phi$ of $\mu$ implies that $\Phi\ge |\MM|$. In particular, $\Phi\ge 2$ per our convention that all metric spaces are not singletons, so the denominator in the right hand side of~\eqref{eq:mixed probabilties} does not vanish.}
\begin{align}\label{eq:mixed probabilties}
\begin{split}
\prob \big[d(x,\cZ)&\ge \min\{2^{b+n-2},\ell\}\ \ \mathrm{and}\ \ d(y,\cZ)\le u\big]\\&\gtrsim \frac{\left\lfloor\log \frac{\mu(B(x,2^{a+n}))}{\mu(B(x,2^{b+n}))} \right\rfloor}{(a-b+1)\log \Phi} \prob^n \big[d(x,\cZ)\ge \ell\ \ \mathrm{and}\ \ d(y,\cZ)\le u\big].
\end{split}
\end{align}
\end{theorem}

Prior to proving Theorem~\ref{thm:restatet decent with constants that we get}, we will use it to deduce the following theorem, which is a restatement of  Theorem~\ref{thm:quote descent}, except that the implicit dependence on the parameters  in~\eqref{eq:quote descent} is now stated explicitly: 

\begin{theorem}\label{thm:redo descent with explicit constant} Fix $\alpha\ge \beta>0$ and $0<\e,\d,\theta\le 1$. Let $(\MM,d)$ be a finite metric space such that for every $\tau>0$ there is a probability measure $\prob^\tau$ on $2^\MM\setminus\{\emptyset\}$ that satisfies the following estimate: 
\begin{equation}\label{eq:cadinality version of zero set}
\forall x,y\in \MM,\qquad \tau \le d(x,y)\le (1+\theta)\tau\implies \prob^\tau\Bigg[d(x,\cZ) \ge \frac{\e\tau}{
\sqrt{\log\frac{e|B(x,\alpha\tau)|}{|B(x,\beta\tau)|}}}\quad  \mathrm{and}\quad      y\in \cZ\Bigg]\ge \d.
\end{equation}
Then,
\begin{equation}\label{eq:quote descent with explicit constants}
\cc_2(\MM)\lesssim \frac{1}{\min\{\beta,\e\}}\Bigg(\frac{1+\log \frac{\max\{\alpha,1\}}{\min\{\beta,1\}}}{\d\theta}\Bigg)^\frac12\sqrt{\log |\MM|}.
\end{equation}
Furthermore, the bound~\eqref{eq:quote descent with explicit constants} on the Euclidean distortion of $\MM$ is obtained via the Fr\'echet embedding.
\end{theorem}

\begin{proof}[Proof of Theorem~\ref{thm:redo descent with explicit constant} assuming Theorem~\ref{thm:restatet decent with constants that we get}] Since the assumption~\eqref{eq:cadinality version of zero set} of Theorem~\ref{thm:redo descent with explicit constant}  becomes weaker if we increase $\alpha$ and decrease $\beta$,  we may assume from now without loss of generality that $\alpha\ge 2$ and $\beta< 1$. 

For each $n\in \Z$, define as follows a probability measure $\bbQ^n$ on $2^\MM\setminus \{\emptyset\}$. Let  $\tau$ be chosen uniformly at random from  $\{2^n,(1+\theta)2^n,\ldots ,(1+\theta)^{\lceil 1/\theta\rceil}2^n\}$, and then choose $\cZ\subset \MM$ according to $\prob^\tau$. We will denote  the law of the resulting random subset $\cZ$ by $\bbQ^n$. If $x,y\in \MM$ satisfy $2^n\le d(x,y)\le 2^{n+1}$, then with probability that is at least a positive universal constant multiple of $\theta$ we have $\tau \le d(x,y)\le (1+\theta)\tau$, which implies in particular that $2^{n-1}\le \tau\le 2^{n+1}$, as $2^n\le d(x,y)\le 2^{n+1}$ and $0<\theta\le 1$,  so we may use~\eqref{eq:cadinality version of zero set} to deduce that 
\begin{equation}\label{eq:bbq version}
\forall x,y\in \MM,\qquad 2^n\le d(x,y)\le 2^{n+1}\implies \bbQ^n \Bigg[d(x,\cZ) \ge \frac{\e2^{n-1}}{
\sqrt{\log\frac{e|B(x,\alpha2^{n+1})|}{|B(x,\beta 2^{n-1})|}}}\quad  \mathrm{and}\quad      y\in \cZ\Bigg]\gtrsim  \d\theta. 
\end{equation}
Writing  $2\alpha=2^{a}$ and $\beta/2=2^{b}$, so $a\ge 2$ and $b< -1$ and $1+\log(\alpha/\beta)\asymp a-b+1 \asymp a-b$, \eqref{eq:bbq version} becomes
\begin{equation}\label{eq:bbq version2}
\forall x,y\in \MM,\qquad 2^n\le d(x,y)\le 2^{n+1}\implies \bbQ^n \Bigg[d(x,\cZ) \ge \frac{\e2^{n-1}}{
\sqrt{\log\frac{e|B(x,2^{a+n})|}{|B(x,2^{b+n})|}}}\quad  \mathrm{and}\quad      y\in \cZ\Bigg]\gtrsim  \d\theta. 
\end{equation}
Of course, by symmetry we also have
\begin{equation}\label{eq:bbq version3}
\forall x,y\in \MM,\qquad 2^n\le d(x,y)\le 2^{n+1}\implies \bbQ^n \Bigg[d(y,\cZ) \ge \frac{\e2^{n-1}}{
\sqrt{\log\frac{e|B(y,2^{a+n})|}{|B(y,2^{b+n})|}}}\quad  \mathrm{and}\quad      x\in \cZ\Bigg]\gtrsim  \d\theta. 
\end{equation}

Apply Theorem~\ref{thm:restatet decent with constants that we get} to the measures $\{\bbQ^n\}_{n\in \Z}$ and with $\mu$ being the counting measure on $\MM$ (thus, recalling~\eqref{eq;aspect ratio notation}, we have $\Phi=|\MM|$ in this case) to get a probability measure $\prob$ on $2^{\MM}\setminus \{\emptyset\}$ that satisfies  the following inequality for every $x,y\in \MM$, every $n\in \Z$, and every $\ell,u\ge 0$:
\begin{align}\label{eq:mixed probabilties restate for application}
\begin{split}
\prob \big[d(x,\cZ)&\ge \min\{2^{b+n-2},\ell\}\ \ \mathrm{and}\ \ d(y,\cZ)\le u\big]\\&\gtrsim \frac{\left\lfloor\log \frac{\mu(B(x,2^{a+n}))}{\mu(B(x,2^{b+n}))} \right\rfloor}{(a-b)\log |\MM|} \bbQ^n \big[d(x,\cZ)\ge \ell\ \ \mathrm{and}\ \ d(y,\cZ)\le u\big].
\end{split}
\end{align}
We will proceed to bound the distortion into $L_2(\prob)$ of the Fr\'echet embedding $\Phi_{(\MM,d)}$ that is given in~\eqref{eq:def Frechet embedding}.

Fixing distinct $x,y\in \MM$, define $n=n(x,y)\in \Z$ by
\begin{equation}\label{eq:nxy choice log2}
n=n(x,y)\eqdef \left\lfloor \log_2 d(x,y)\right\rfloor\iff   2^n\le d(x,y)< 2^{n+1}.
\end{equation}
Apply~\eqref{eq:mixed probabilties restate for application} with the parameters $\ell=  \e2^{n-1}/
\sqrt{\log\frac{e|B(x,2^{a+n})|}{|B(x,2^{b+n})|}}$ and $u=0$ to get, using~\eqref{eq:bbq version2}, that
\begin{align}\label{eqLdxy version 1}
\prob \Bigg[d(x,\cZ)\ge \min\Bigg\{2^{b-2},\frac{\e}{
2\sqrt{\log\frac{e|B(x,2^{a+n})|}{|B(x,2^{b+n})|}}}\Bigg\}2^n\ \ \mathrm{and}\ \ y\in \cZ\Bigg]\gtrsim \frac{\left\lfloor\log \frac{|B(x,2^{a+n})|}{|B(x,2^{b+n})|} \right\rfloor}{(a-b)\log |\MM|} \d\theta.
\end{align}
In the same vein, using~\eqref{eq:bbq version3} we also get the symmetric estimate 
\begin{align*}
\prob \Bigg[d(y,\cZ)\ge \min\Bigg\{2^{b-2},\frac{\e}{
2\sqrt{\log\frac{e|B(y,2^{a+n})|}{|B(y,2^{b+n})|}}}\Bigg\}2^n\ \ \mathrm{and}\ \ x\in \cZ\Bigg]\gtrsim \frac{\left\lfloor\log \frac{|B(y,2^{a+n})|}{|B(y,2^{b+n})|} \right\rfloor}{(a-b)\log |\MM|} \d\theta.
\end{align*}

Next, observe that because $a\ge 2$ and $b<-1$, 
\begin{equation}\label{eq:unions and ingtersections of balls}
B\big(x,2^{b+n}\big)\cup B\big(y,2^{b+n}\big)\subset B\big(x,2^{a+n}\big)\cap B\big(y,2^{a+n}\big)\qquad\mathrm{and}\qquad B\big(x,2^{b+n}\big)\cap B\big(y,2^{b+n}\big)=\emptyset.
\end{equation}
Indeed, if $z\in B(x,2^{b+n})$, then 
$
d(z,y)\le d(z,x)+d(x,y)\le 2^{b+n}+d(x,y)<2^{b+n}+2^{n+1}< 5\cdot 2^{n-1}< 2^{a+n},
$ 
where the third step uses~\eqref{eq:nxy choice log2}, the penultimate step uses $b<-1$ and the final step uses $a\ge 2$. Thus, $B(x,2^{b+n})\subset B(y,2^{a+n})$. By symmetry also $B(y,2^{b+n})\subset B(x,2^{a+n})$, so the first assertion in~\eqref{eq:unions and ingtersections of balls} holds.  For the second assertion of~\eqref{eq:unions and ingtersections of balls}, if there were $w\in B(x,2^{b+n})\cap B(y,2^{b+n})$, then  we get the  contradiction 
$
d(x,y)\le d(x,w)+d(w,y)\le 2^{b+1+n}<2^n\le  d(x,y), 
$
where the penultimate step uses $b<-1$ and the final step uses~\eqref{eq:nxy choice log2}. Having verified that~\eqref{eq:unions and ingtersections of balls} is indeed satisfied, we deduce from it that
\begin{align*}
2\min\big\{|B(x,2^{b+n})|, |B(y,2^{b+n})|\big\}\le|B(x,2^{b+n})|+ |B(y,2^{b+n})|\le \min\big\{|B\big(x,2^{a+n})|, |B(y,2^{a+n})|\big\}.
\end{align*}
By interchanging  the roles of $x$ and $y$, if necessary, we may therefore assume without loss of generality that  $|B(x,2^{a+n})|/|B(x,2^{b+n})| \ge 2$, in which case~\eqref{eqLdxy version 1} implies  (recalling that $2^b=\beta/2$) that
\begin{equation}\label{eq:after relabeling for ratio}
\prob \Bigg[d(x,\cZ)\ge \frac1{8} \min\Bigg\{\beta,\frac{\e}{
\sqrt{\log\frac{|B(x,2^{a+n})|}{|B(x,2^{b+n})|}}}\Bigg\}2^n\ \ \mathrm{and}\ \ y\in \cZ\Bigg]\gtrsim \frac{\log \frac{|B(x,2^{a+n})|}{|B(x,2^{b+n})|}}{(a-b)\log |\MM|} \d\theta.
\end{equation}
Consequently,  the Fr\'echet embedding $\Phi_{(\MM,d)}$  yields the distortion bound~\eqref{eq:quote descent with explicit constants} into $L_2(\prob)$ because 
\begin{align*}
\|&\Phi_{(\MM,d)}(x)-\Phi_{(\MM,d)}(y)\|_{L_2(\prob)}=\|d(x,\cZ)-d(y,\cZ)\|_{L_2(\prob)}\\&\ \ \ \ \ \stackrel{\eqref{eq:after relabeling for ratio}}{\gtrsim} \Bigg(\frac{\log \frac{|B(x,2^{a+n})|}{|B(x,2^{b+n})|}}{(a-b)\log |\MM|} \d\theta\Bigg)^\frac12 \min\Bigg\{\beta,\frac{\e}{
\sqrt{\log\frac{|B(x,2^{a+n})|}{|B(x,2^{b+n})|}}}\Bigg\}2^n\stackrel{\eqref{eq:nxy choice log2}}{\gtrsim} \min\{\e,\beta\} \Bigg(\frac{\d\theta}{(a-b)\log|\MM|}\Bigg)^\frac12d(x,y).\tag*{\qedhere}
\end{align*}

\end{proof}

In preparation for the proof of Theorem~\ref{thm:restatet decent with constants that we get}, we introduce the following notation for a metric space $(\MM,d)$ and a nondegenerate measure $\mu$ on $\MM$:
\begin{equation}\label{eq:def ck}
\forall x\in \MM,\ \forall t\in \R,\qquad \ck(x,t)= \ck_{\mu,d}(x,t)\eqdef \max\Big\{k\in \Z:\ \mu\big(B(x,2^k)\big)\le e^t\Big\}.
\end{equation}
In other words, if we define 
\begin{equation}\label{eq:def ti}
\forall x\in \MM,\ \forall \theta\in \R,\qquad s_\theta(x)=s_{\theta,\mu,d}(x)\eqdef \log \mu \big(B(x,2^\theta)\big),
\end{equation}
then for every $x\in \MM$ the function $t\mapsto \ck(x,t)$ is  piecewise constant and given by
\begin{equation}\label{eq:piecewise}
\forall x\in \MM,\ \forall i\in \Z,\qquad  s_{i}(x) \le t<s_{i+1}(x)\implies \ck(x,t)=i. 
\end{equation}
We record for ease of later use the following simple variant of an observation from~\cite{KLMN05}:

\begin{lemma}\label{obs:klmn} For every $x\in \MM$ and $t\in \Z$ we have 
\begin{equation}\label{eq:ck almost constant}
\forall y\in B\big(x,2^{\ck(x,t)-1}\big),\qquad \ck(y,t)\in \big\{\ck(x,t)-1,\ck(x,t),\ck(x,t)+1\big\}.
\end{equation}
\end{lemma}

\begin{proof} Denote $k=\ck(x,t)$. If $y\in B(x,2^{k-1})$, then $B(y,2^{k-1})\subset B(x,2^k)$. Because $\mu(B(x,2^k))\le e^t$ by the definition~\eqref{eq:def ck} of $k=\ck(x,t)$, it follows that $\mu(B(y,2^{k-1}))\le e^t$, i.e., $k(y,t)\in \{k-1,k,\ldots\}$ by~\eqref{eq:def ck}. Similarly, $B(x,2^{k+1})\subset B(y,2^{k+1}+2^{k-1})\subset B(y,2^{k+2})$ while $\mu(B(x,2^{k+1}))>e^t$ by the definition~\eqref{eq:def ck} of $k=\ck(x,t)$, so $\mu(B(y,2^{k+2}))>e^t$, which implies by~\eqref{eq:def ck} that $\ck(y,t)\in \{\ldots,k,k+1\}$.   
\end{proof}

We can now prove Theorem~\ref{thm:restatet decent with constants that we get}:

\begin{proof}[Proof of Theorem~\ref{thm:restatet decent with constants that we get}] Because Theorem~\ref{thm:restatet decent with constants that we get} only discusses  ratios of values of the measure $\mu$, we may assume without loss of generality that $\mu$ is normalized so that $\min_{x\in \MM}\mu(x)=1$, in which case $\mu(\MM)=\Phi$.

Fix $a>b$. Define two subsets $I,T$ of $\Z$ by
\begin{equation}\label{eq:def IT}
I\eqdef \big\{\lceil b\rceil,\ldots, \lceil a\rceil \big\} \qquad\mathrm{and}\qquad T\eqdef \big\{0,\ldots, \lceil \log \Phi\rceil-1\big\}.
\end{equation}
The ensuing  construction involves auxiliary random variables which we will next introduce. Firstly, let $\si$ be a random variable that is distributed uniformly over $I$, and let $\st$ be a random variable that is distributed uniformly over $T$. Secondly,  for every $i\in \Z$ let $\sigma_i$ and $\eta_i$ be random variables that are distributed uniformly over $\{0,1\}$ and $\{0,1,2\}$, respectively. Finally, for every $n\in \Z$ let $\cZ_n$ be a random subset of $\MM$ that is distributed according to $\prob^n$. We also require that all of the above random variables/subsets, namely  
\begin{equation}\label{eq:random datum for mixing}
\si,\st, \{\sigma_i\}_{i\in \Z}, \{\eta_i\}_{i\in \Z},\{\cZ_n\}_{n\in \Z},
\end{equation}
 are independent, and we denote the (product) probability space on which they are defined by $(\Omega,\prob)$. 
 
 Consider the following random subset $\cZ$ of $\MM$, which is a function of the random variables in~\eqref{eq:random datum for mixing}:
 \begin{equation}\label{eq:def our Z mixture}
\cZ\eqdef \big\{z\in \MM:\ z\in \cZ_{\ck(z,\st)-\si+\eta_{\ck(z,\st)-\si}}\big\}\cup \big\{z\in \MM:\ \sigma_{\ck(z,\st)-\si}=1 \big\},
\end{equation}
where we recall the notation~\eqref{eq:def ck}. We will  prove that for every $x,y\in \MM$, every $\ell,u\ge 0$, and every $n\in \Z$,
 \begin{align}\label{eq:mixed probabilties-later}
\begin{split}
\prob \big[\cZ\neq\emptyset \ \ \mathrm{and}\ \ d(x,\cZ)&\ge \min\{2^{b+n-2},\ell\}\ \ \mathrm{and}\ \ d(y,\cZ)\le u\big]\\&\gtrsim \frac{\left\lfloor\log \frac{\mu(B(x,2^{a+n}))}{\mu(B(x,2^{b+n}))} \right\rfloor}{(a-b+1)\log \Phi} \prob^n \big[d(x,\cZ)\ge \ell\ \ \mathrm{and}\ \ d(y,\cZ)\le u\big].
\end{split}
\end{align}
 This will imply the desired conclusion~\eqref{eq:mixed probabilties} of Theorem~\ref{thm:restatet decent with constants that we get} by considering the  random subset $\cZ$ in~\eqref{eq:def our Z mixture} conditioned on the event $\{\cZ\neq\emptyset\}$, i.e., the restriction of $\cZ$ to the  subset $\{\cZ\neq \emptyset\}$ of $\Omega$, equipped with the probability measure $\prob[\cZ\neq\emptyset]^{-1}\prob$ (note that~\eqref{eq:mixed probabilties-later} implies in particular that  $\prob[\cZ\neq\emptyset]>0$).

Given $n\in \Z$, define two (cylinder) events $ C_n\subset C_n^* \subset \Omega$ by
\begin{equation}\label{eq:good event 3}
C_n^*\eqdef \big\{ (\eta_{n-2},\eta_{n-1},\eta_n)=(2,1,0)\big\}\quad\mathrm{and}\quad  C_n \eqdef \big\{ (\sigma_{n-2},\sigma_{n-1},\sigma_n,\eta_{n-2},\eta_{n-1},\eta_n)=(0,0,0,2,1,0)\big\}.
\end{equation}
Thus, 
\begin{equation}\label{eq:cylinder has big prob}
\Pr[C_n^*]=\frac{1}{3^3}\asymp 1\qquad\mathrm{and}\qquad \Pr[C_n]=\frac{1}{2^33^3}\asymp 1.
\end{equation}
Also, recalling the notation~\eqref{eq:def ck}, for every $n\in \Z$ and  $w\in \MM$ define events $E_n(w),F_n(w),G(w)\subset \Omega$ by
\begin{equation}\label{eq:def events EFG}
E_n(w)\eqdef \big\{w\in \cZ_n\big\},\qquad F_n(w)\eqdef \big\{\ck(w,\st)-\si\in \{n-2,n-1,n\}\big\},\qquad G(w)\eqdef \big\{\sigma_{\ck(w,\st)-\si}=1\big\}.
\end{equation}
\begin{claim}\label{claim:second inclusion of events} For every $n\in \Z$,  $y\in \MM$, and $u\ge 0$, the following inclusion of events holds:
\begin{equation}\label{eq:second event inclusion}
\big\{\cZ\neq\emptyset \ \ \mathrm{and}\ \ d(y,\cZ)\le u\big\}\supseteq \bigg(\bigcup_{w\in B(y,u)} \Big(E_n(w)\cap\big(F_n(w)\cup G(w)\big)\Big)\bigg)\cap C_n^*.
\end{equation}
\end{claim}

\begin{proof}
  If the event in the right-hand side of~\eqref{eq:second event inclusion} occurs, then the event $C_n^*$ occurs and, recalling the notations in~\eqref{eq:def events EFG}, there exists $w\in \cZ_n$ with $d(w,y)\le u$ such that  $\ck(w,\st)-\si\in \{n-2,n-1,n\}$ or $\sigma_{\ck(w,\st)-\si}=1$. If $\ck(w,\st)-\si\in \{n-2,n-1,n\}$, then $\ck(w,\st)-\si+\eta_{\ck(w,\st)-\si}=n$ thanks to the assumed occurrence of the event $C_n^*$ that is given in~\eqref{eq:good event 3}, so $w\in  \cZ$ by~\eqref{eq:def our Z mixture}. On the other hand, if $\sigma_{\ck(w,\st)-\si}=1$, then~\eqref{eq:def our Z mixture} implies directly that  $w\in \cZ$. In both cases $w\in \cZ$, whence $\cZ\neq\emptyset$ and $d(y,\cZ)\le d(y,w)\le u$, as required. 
\end{proof}

Next, recalling the notation~\eqref{eq:def ti}, for every $n\in \Z$ and $x\in \MM$ define an event $H_n(x)\subset \Omega$ by
\begin{equation}\label{eq:def Hnx}
H_n(x)\eqdef \big\{s_{\si+n-1}(x)\le \st< s_{\si+n}(x)\big\}.
\end{equation}

\begin{claim}\label{claim:compute interval prob} For every $m\in \Z$ and $x\in \MM$ we have 
\begin{equation}\label{eqLprob Hnx}
\prob[H_m(x)]\gtrsim \frac{\left\lfloor\log \frac{\mu(B(x,2^{a+m}))}{\mu(B(x,2^{b+m}))} \right\rfloor}{(a-b+1)\log \Phi}.
\end{equation}
\end{claim}
\begin{proof} Simply compute that, since $\si$ is distributed uniformly over $I$ and $\st$ is distributed uniformly over $T$, where $I,T$ are given in~\eqref{eq:def IT}, we have
\begin{align}\label{eq:compute prob E1}
\begin{split}
\prob[&H_m(x)]\stackrel{\eqref{eq:def Hnx}}{=}\frac{1}{\lceil a\rceil-\lceil b\rceil+1} \sum_{i=\lceil b\rceil}^{\lceil a\rceil} \frac{|\{t\in T:\ s_{i+m-1}(x)\le t< s_{i+m}(x)\}|}{\lceil \log \Phi\rceil}\\&= \frac{|T\cap\{\lceil s_{\lceil b\rceil +m-1}(x)\rceil,\ldots, \lceil s_{\lceil a\rceil+m}(x)\rceil-1 \}|}{(\lceil a\rceil -\lceil b\rceil +1)\lceil \log \Phi\rceil}= \frac{\lceil s_{\lceil a\rceil+m}(x)\rceil-\lceil s_{\lceil b\rceil +m-1}(x)\rceil}{(\lceil a\rceil -\lceil b\rceil +1)\lceil \log \Phi\rceil}\gtrsim \frac{\left\lfloor\log \frac{\mu(B(x,2^{a+m}))}{\mu(B(x,2^{b+m}))} \right\rfloor}{(a-b+1)\log \Phi},
\end{split}
\end{align}
where the penultimate step of~\eqref{eq:compute prob E1} is valid thanks to our normalization $\min_{z\in \MM}\mu(z)=1$, which ensures that $T\supseteq \{\lceil s_{\lceil b\rceil +m-1}(x)\rceil,\ldots, \lceil s_{\lceil a\rceil+m}(x)\rceil-1 \}$, and the last step of~\eqref{eq:compute prob E1} holds since $\lceil a\rceil -\lceil b\rceil +1\asymp a-b+1$ and $\lceil \log \Phi\rceil\asymp \log \Phi$, as $a>b$ and $\Phi\ge 2$, and furthermore
\begin{align*}
\lceil s_{\lceil a\rceil+m}(x)\rceil-\lceil s_{\lceil b\rceil +m-1}(x)\rceil&\gtrsim \lfloor s_{\lceil a\rceil+m}(x)- s_{\lceil b\rceil +m-1}(x)\rfloor\\&\stackrel{\eqref{eq:def ti}}{=} \left\lfloor \log \frac{\mu(B(x,2^{\lceil a\rceil +m}))}{\mu(B(x,2^{\lceil b\rceil -1+m}))}\right\rfloor\ge \left\lfloor \log \frac{\mu(B(x,2^{a+m}))}{\mu(B(x,2^{b+m}))}\right\rfloor.\tag*{\qedhere}
\end{align*}
\end{proof}

For the next (and final) claim, we define for every $n\in \Z$, $x\in \MM$ and $\ell\ge 0$ an event $L_n(x,\ell)\subset \Omega$ by
\begin{equation}\label{eq:def Lnxell}
L_n(x,\ell)\eqdef \big\{d(x,\cZ_n)\ge \ell\big\}.
\end{equation}

\begin{claim}\label{claim:first inclusion of events} For every $n\in \Z$,  $x\in \MM$, and $\ell\ge 0$, the following inclusion of events holds:
\begin{equation}\label{eq:event for x}
\big\{\cZ\neq\emptyset \ \ \mathrm{and}\ \  d(x,\cZ)\ge \min\{2^{b+n-2},\ell\}\big\}\supseteq  H_n(x)\cap  L_n(x,\ell)\cap C_n\cap\big\{\cZ\neq \emptyset \big\}.
\end{equation}

\end{claim}
\begin{proof} Suppose that $z\in \MM$ satisfies  $d(z,x)<\min\{2^{b+n-2},\ell\}$, and also that the event in the right hand side of~\eqref{eq:event for x} occurs. Our goal is to deduce  that $z\notin \cZ$. Indeed, $d(z,\cZ_n)\ge d(x,\cZ_n)-d(z,x)>d(x,\cZ_n)-\ell\ge 0$, because part of our assumption on $z$ is that $d(x,z)<\ell$ and  the occurrence of the event in the right hand side of~\eqref{eq:event for x} ensures in particular that $L_n(x,\ell)$ occurs, i.e.,  $d(x,\cZ_n)\ge \ell$. Thus, $z\not \in \cZ_n$. By~\eqref{eq:def our Z mixture}, we will therefore be done if we will show that the occurrence of the  event in the right hand side of~\eqref{eq:event for x} implies that $\ck(z,\st)-\si+\eta_{\ck(z,\st)-\si}=n$ and $\sigma_{\ck(z,\st)-\si}=0$. This indeed holds since $s_{\si+n-1}(x)\le \st< s_{\si+n}(x)$ as $H_n(x)$ occurs, so  $\ck(x,\st)=\si+n-1$ by~\eqref{eq:piecewise}. But $\si\ge  \lceil b\rceil\ge b$ since $\si\in I$, where we recall that $I\subset \Z$ is given in~\eqref{eq:def IT}, so $\ck(x,\st)\ge b+n-1$. Hence, $d(z,x)<2^{b+n-2}\le 2^{\ck(x,\st)-1}$ using the rest of our assumption on $z$, so by Lemma~\ref{obs:klmn} we see that $\ck(z,\st)-\si\in \{n-2,n-1,n\}$. As the occurrence of the right hand side of~\eqref{eq:event for x}  ensures that  the event $C_n$ in~\eqref{eq:good event 3} occurs, it follows that  $\sigma_{\ck(z,\st)-\si}=0$ and $\ck(z,\st)-\si+\eta_{\ck(z,\st)-\si}=n$, as required.
\end{proof}

Next, fix $n\in \Z$, $x,y\in \MM$ and $\ell,u\ge 0$.  Denote $|B(y,u)|=p$ and write 
\begin{equation}\label{eq:enumerate B(y,u)}
B(y,u)=\{w_1,\ldots,w_p\},
\end{equation} i.e., we are fixing an arbitrary linear order on $B(y,u)$. For every $q\in [p]$ define an event $A_q\subset \Omega$ by
\begin{equation}\label{eq:def Aq}
A_q\eqdef \bigg(E_n(w_q)\setminus \bigcup_{r=1}^{q-1}E_n(w_r)\bigg)\cap\big(F_n(w_q)\cup G(w_q)\big)\cap H_n(x)\cap  L_n(x,\ell)\cap C_n.
\end{equation}
Thus, $A_1,\ldots, A_p$ are pairwise disjoint and by combining Claim~\ref{claim:first inclusion of events}  and Claim~\ref{claim:second inclusion of events} we see that
$$
\big\{\cZ\neq\emptyset \ \ \mathrm{and}\ \  d(x,\cZ)\ge \min\{2^{b+n-2},\ell\}\ \ \mathrm{and}\ \ d(y,\cZ)\le u\big\}\supseteq \bigcup_{q=1}^p A_q. 
$$
Consequently,
\begin{equation}\label{eq:probability lower bound sum Aq}
\prob\big[\cZ\neq\emptyset \ \ \mathrm{and}\ \  d(x,\cZ)\ge \min\{2^{b+n-2},\ell\}\ \ \mathrm{and}\ \ d(y,\cZ)\le u\big]\ge \sum_{q=1}^p \prob[A_q].
\end{equation}

In order to understand the right hand side of~\eqref{eq:probability lower bound sum Aq}, define for every $q\in [p]$ events $A_q',A_q''\subset \Omega$ by
\begin{equation}\label{eq:def Aq'}
A_q'\eqdef \bigg(E_n(w_q)\setminus \bigcup_{r=1}^{q-1}E_n(w_r)\bigg)\cap F_n(w_q) \cap H_n(x)\cap  L_n(x,\ell)\cap C_n,
\end{equation}
and,
\begin{equation}\label{eq:def Aq''}
A_q''\eqdef \bigg(E_n(w_q)\setminus \bigcup_{r=1}^{q-1}E_n(w_r)\bigg)\cap \big(\Omega\setminus F_n(w_q)\big) \cap G(w_q)  \cap H_n(x)\cap  L_n(x,\ell)\cap C_n\stackrel{\eqref{eq:def Aq}\wedge\eqref{eq:def Aq'}}{=}A_q\setminus A_q'. 
\end{equation}
We therefore have 
\begin{equation}\label{eq:break Aq}
\forall q\in [p],\qquad \prob[A_q]=\prob[A_q']+\prob[A_q'']. 
\end{equation}

Recalling~\eqref{eq:def events EFG} and~\eqref{eq:def Hnx}, the event $(\Omega\setminus F_n(w_q))\cap   H_n(x)$ in the right hand side of~\eqref{eq:def Aq''} means that $s_{\si+n-1}(x)\le \st<s_{\si+n}(x)$ and $\ck(w_q,\st)-\si\notin \{n-2,n-1,n\}$. So, by integrating with respect to $\si$ and $\st$ we get 
\begin{equation}\label{eq:prob aq''}
\prob[A_q'']=\frac{1}{|I|\cdot|T|}\sum_{\substack{(i,t)\in I\times T\\ s_{i+n-1}(x)\le t<s_{i+n}(x)\\ \ck(w_q,t)-i\notin \{n-2,n-1,n\}}}\prob \bigg[\Big(E_n(w_q)\setminus \bigcup_{r=1}^{q-1}E_n(w_r)\Big)\cap\big\{\sigma_{\ck(w_q,t)-i}=1\big\}\cap  L_n(x,\ell)\cap C_n\bigg],
\end{equation}
where we used the facts that the events $E_n(w_1),\ldots,E_n(w_p)$ and $L_n(x,\ell)$, defined in~\eqref{eq:def events EFG} and~\eqref{eq:def Lnxell}, respectively,  depend only on $\cZ_n$, the event $C_n$ from~\eqref{eq:good event 3} depends only on $(\sigma_{n-2},\sigma_{n-1},\sigma_n,\eta_{n-2},\eta_{n-1},\eta_n)$, and the event $G(w_q)$, that does depend on  $\si,\st$  per~\eqref{eq:def events EFG}, is  equal to $\{\sigma_{\ck(w_q,\st)-\si}=1\}$. The crucial point to observe now is that if $(i,t)\in I\times T$ are such that $\ck(w_q,t)-i\notin \{n-2,n-1,n\}$, then $\sigma_{\ck(w_q,t)-i}$ and $C_n$ are independent and $\prob[\sigma_{\ck(w_q,t)-i}=1]=1/2$, so each of the summands in the right hand side~\eqref{eq:prob aq''}  satisfy

\begin{align*}
\prob \bigg[\Big(E_n(w_q)\setminus \bigcup_{r=1}^{q-1}E_n(w_r)\Big)\cap&\big\{\sigma_{\ck(w_q,t)-i}=1\big\}\cap  L_n(x,\ell)\cap C_n\bigg]\\&\stackrel{\eqref{eq:cylinder has big prob}}{=}\frac{1}{2^43^3}\prob \bigg[\Big(E_n(w_q)\setminus \bigcup_{r=1}^{q-1}E_n(w_r)\Big)\cap  L_n(x,\ell)\bigg].
\end{align*}
Thus,
\begin{equation}\label{eq:prob aq'' summands rewritten}
\prob[A_q'']=\frac{1}{|I|\cdot|T|}\sum_{\substack{(i,t)\in I\times T\\ s_{i+n-1}(x)\le t<s_{i+n}(x)\\ \ck(w_q,t)-i\notin \{n-2,n-1,n\}}}\frac{1}{2^43^3}\prob \bigg[\Big(E_n(w_q)\setminus \bigcup_{r=1}^{q-1}E_n(w_r)\Big)\cap  L_n(x,\ell)\bigg].
\end{equation}

The analogous consideration shows mutatis mutandis that if for $q\in [p]$ we define  $A_q'''\subset \Omega$ by 
\begin{equation}\label{eq:Aq'''}
A_q'''\eqdef \bigg(E_n(w_q)\setminus \bigcup_{r=1}^{q-1}E_n(w_r)\bigg)\cap \big(\Omega\setminus F_n(w_q)\big) \cap H_n(x)\cap  L_n(x,\ell)\cap C_n,
\end{equation}
then 
\begin{equation}\label{eq:1/2 equality}
\Pr[A_q''']=\frac{1}{|I|\cdot|T|}\sum_{\substack{(i,t)\in I\times T\\ s_{i+n-1}(x)\le t<s_{i+n}(x)\\ \ck(w_q,t)-i\notin \{n-2,n-1,n\}}}\frac{1}{2^33^3}\prob \bigg[\Big(E_n(w_q)\setminus \bigcup_{r=1}^{q-1}E_n(w_r)\Big)\cap  L_n(x,\ell)\bigg]\stackrel{\eqref{eq:prob aq'' summands rewritten}}{=}2\prob[A_q'']. 
\end{equation}
Consequently, for every $q\in [p]$ we have 
\begin{equation}
\prob[A_q]\stackrel{\eqref{eq:break Aq}\wedge \eqref{eq:1/2 equality}}{\ge} \prob[A_q']+\frac12\prob[A_q''']\ge \frac12\big(\prob[A_q']+\prob[A_q''']\big)=\frac12\prob[A_q''''], 
\end{equation}
where we introduce the notation
\begin{equation}\label{eq:def Aq''''}
A_q''''\eqdef \bigg(E_n(w_q)\setminus \bigcup_{r=1}^{q-1}E_n(w_r)\bigg)\cap H_n(x)\cap  L_n(x,\ell)\cap C_n,
\end{equation}
and observe that by~\eqref{eq:def Aq'}, \eqref{eq:Aq'''} and~\eqref{eq:def Aq''''}, the events $A_q', A_q'''$ are disjoint and they satisfy $A_q'\cup A_q'''=A_q''''$.

Finally, the events $\{A_q''''\}_{q=1}^p$ are pairwise disjoint, so 
\begin{equation}\label{eq:probability lower bound sum Aq''''}
\prob\big[\cZ\neq\emptyset \ \ \mathrm{and}\ \  d(x,\cZ)\ge \min\{2^{b+n-2},\ell\}\ \ \mathrm{and}\ \ d(y,\cZ)\le u\big]\stackrel{\eqref{eq:probability lower bound sum Aq}\wedge \eqref{eq:def Aq''''}}{\ge} \frac12\sum_{q=1}^p \prob[A_q'''']=\frac12 \prob\bigg(\bigcup_{q=1}^p A_q''''\bigg).
\end{equation}
Observing the following simple identity of events,
\begin{multline*}
\bigcup_{q=1}^p A_q''''\stackrel{\eqref{eq:def Aq''''}}{=}\bigg(\bigcup_{q=1}^{p}E_n(w_q)\bigg)\cap H_n(x)\cap  L_n(x,\ell)\cap C_n \stackrel{\eqref{eq:def events EFG}\wedge \eqref{eq:enumerate B(y,u)}}{=}\big\{B(y,u)\cap \cZ_n\neq \emptyset \big\}\cap H_n(x)\cap  L_n(x,\ell)\cap C_n\\
\stackrel{\eqref{eq:def Lnxell}}{=} H_n(x)\cap \big\{d(x,\cZ_n)\ge \ell\ \ \mathrm{and}\ \ d(y,\cZ_n)\le u\big\}\cap C_n,
\end{multline*}
we get from the independence of the events $H_n(x), \{d(x,\cZ_n)\ge \ell\ \ \mathrm{and}\ \ d(y,\cZ_n)\le u\},C_n$ (recalling~\eqref{eq:good event 3} and~\eqref{eq:def Hnx}, they depend on $(\si,\st), \cZ_n, (\sigma_{n-2},\sigma_{n-1},\sigma_n,\eta_{n-2},\eta_{n-1},\eta_n)$, respectively) and~\eqref{eq:probability lower bound sum Aq''''} that
\begin{align*}
\prob\big[\cZ\neq\emptyset \ \ \mathrm{and}\ \  d(x,\cZ)\ge &\min\{2^{b+n-2},\ell\}\ \ \mathrm{and}\ \ d(y,\cZ)\le u\big]\\&\gtrsim\prob \Big[H_n(x)\cap \big\{d(x,\cZ_n)\ge \ell\ \ \mathrm{and}\ \ d(y,\cZ_n)\le u\big\}\cap C_n\Big]
\\&= \prob [H_n(x)]\cdot \prob[d(x,\cZ_n)\ge \ell\ \ \mathrm{and}\ \ d(y,\cZ_n)\le u]\cdot\prob [C_n]
\\&\gtrsim\frac{\left\lfloor\log \frac{\mu(B(x,2^{a+n}))}{\mu(B(x,2^{b+n}))} \right\rfloor}{(a-b+1)\log \Phi} \prob[d(x,\cZ_n)\ge \ell\ \ \mathrm{and}\ \ d(y,\cZ_n)\le u],
\end{align*}
where the last step is an application of~\eqref{eq:cylinder has big prob} and conclusion~\eqref{eqLprob Hnx} of Claim~\ref{claim:compute interval prob}. This proves~\eqref{eq:mixed probabilties-later}. 
\end{proof}

\section{Proof of Theorem~\ref{thm:general doubling exponential}}\label{sec:prove loose ends}

We preferred to state Theorem~\ref{thm:general doubling exponential} for finite metric spaces to avoid the need to consider measurability issues about distributions over random subsets. Nevertheless, all that one needs here is that the event that appears in the left hand side of~\eqref{eq:super exponential general metric}  is $\prob^\tau$-measurable, and it is possible to ensure that the ensuing proof of Theorem~\ref{thm:general doubling exponential} achieves this under mild assumptions on a  separable metric space $(\MM,d)$ and any nondegenerate Borel measure $\mu$ on $\MM$ with $\mu(\MM)<\infty$; such a treatment appears in~\cite[Section~3.3]{Nao24}.

The ensuing proof of Theorem~\ref{thm:general doubling exponential} generalizes the proof of the main result of~\cite{MN07};   the incorporation of the selectors $\{\sigma_t\}_{t=1}^\infty$ (see the paragraph after equation~\eqref{eq:tail of stopping time})  originates from~\cite{Rao99}.

\begin{proof}[Proof of Theorem~\ref{thm:general doubling exponential}]As the right hand side of~\eqref{eq:super exponential general metric} involves only a ratio of two values of $\mu$, by normalizing  we may assume that $\mu$ is a probability measure. Let $\{\sfZ_t\}_{t=1}^\infty$ be i.i.d.~points distributed according to $\mu$, and denote their law by $\prob^\mu=\mu^{\otimes \N}$.  For every $r>0$ and $x\in \MM$ consider the following random variable:
\begin{equation}\label{eq:def stoppiong time}
T(x,r)\eqdef \inf\big\{t\in \N:\ d(\sfZ_t,x)\le r\big\},
\end{equation}
with the convention that $T(x,r)=\infty$ if there is no $t\in \N$ for which $d(\sfZ_t,x)\le r$. Observe that because $\mu$ is assumed to be a nondegenerate measure, $T(x,r)$ is finite almost everywhere. Indeed,
\begin{align*}\label{eq:check stopping time is finite}
\begin{split}
\prob^\mu\big[T(x,r)=\infty \big]&= \prob^\mu\big[\forall t\in \N,\ d(\sfZ_t,x)> r) \big]\\&=\lim_{t\to \infty} \prob^\mu\Big[\bigcap_{s=1}^t \big\{\sfZ_s\in \MM\setminus B(x,r)\big\}\Big]=\lim_{t\to \infty} \big(1-\mu(B(x,r))\big)^t=0,
\end{split}
\end{align*}
where in the penultimate step uses the independence of $\{\sfZ_s\}_{s=1}^\infty$ and the final step holds as $\mu(B(x,r))>0$.

Next, we claim that the following inclusion of events holds:
\begin{equation}\label{eq:event inclusion as in ramsey}
\big\{\forall w\in B(x,\lambda \tau),\ T(w,r)=t\big\}\supseteq \bigg(\bigcap_{s=1}^{t-1} \big\{\sfZ_s\in \MM\setminus B(x,r+\lambda\tau)\big\}\bigg)\cap \big\{\sfZ_t\in B(x,r-\lambda\tau)\big\}.
\end{equation}
 Indeed, suppose that the event in the right hand side of~\eqref{eq:event inclusion as in ramsey} occurs, i.e., for every $s\in \{1,\ldots,t-1\}$ we have $d(\sfZ_s,x)>r+\lambda\tau$  and also $d(\sfZ_t,x)\le r-\lambda\tau$. Therefore, if  $w\in B(x,\lambda\tau)$ and $s\in \{1,\ldots,t-1\}$, then $$d(\sfZ_s,w) \ge d(\sfZ_s,x)-d(x,w)>(r+\lambda\tau)-\lambda\tau=r,$$ which means that $T(w,r)\ge t$,  by the definition~\eqref{eq:def stoppiong time}. At the same time, $$d(\sfZ_t,w)\le d(\sfZ_t,x)+d(x,w)\le (r-\lambda\tau)+\lambda\tau=r,$$
  so, in fact, $T(w,r)=t$. As this holds for every  $w\in B(x,\lambda\tau)$, we thus checked that the event in the left hand side of~\eqref{eq:event inclusion as in ramsey} occurs. Using~\eqref{eq:event inclusion as in ramsey} and the independence of $\{\sfZ_s\}_{s=1}^t$, we therefore get the following bound:
  \begin{equation}\label{eq:tail of stopping time}
 \prob^\mu \big[\forall w\in B(x,\lambda \tau),\ T(w,r)=t\big]\ge \Big(1-\mu\big(B(x,r+\lambda\tau)\big)\Big)^{t-1}\mu\big(B(x,r-\lambda\tau)\big).
  \end{equation}

Let $\sfR$ be a random variable that is distributed uniformly over the interval $(\tau/4,\tau/2)$. Also, let $\{\upsigma_t\}_{t=1}^\infty$ be standard Bernoulli random variables, i.e., $\Pr[\upsigma_t=1]=\Pr[\upsigma_t=0]=1/2$ for every $t\in \N$. We will require below that the random variables $\{\sfZ_t\}_{t=1}^\infty,\sfR,\{\upsigma_t\}_{t=1}^\infty$ are independent. Denote by $\prob^\tau$ their joint law (note that the dependence on the scale $\tau$ here arises only from the presence of the  random variable $\sfR$). Finally, define a random subset $\cZ\subset \MM$ as follows:
\begin{equation}\label{eq:def CZ in general}
\cZ\eqdef \big\{x\in \MM:\ \upsigma_{T(x,\sfR)}=1\big\}.
\end{equation}
We will next demonstrate that~\eqref{eq:super exponential general metric} holds, thus proving Theorem~\ref{thm:general doubling exponential} by conditioning on the  event $\{\cZ\neq \emptyset\}$.

Fix $x,y\in \MM$ with $d(x,y)\ge \tau$ and observe the following inclusion of events:
\begin{align}\label{eq:st inclusion}
\begin{split}
\big\{d(y,\cZ) \ge \lambda \tau &\quad  \mathrm{and}\quad      x\in \cZ\big\}\\
&\supseteq \bigcup_{\substack{(s,t)\in \N\\s\neq t}} \big(\{T(x,\sfR)=s\}\cap \{\upsigma_s=1\}\cap \{\forall w\in B(y,\lambda \tau),\ T(w,\sfR)=t\}\cap\{\upsigma_t=0\}\big).
\end{split}
\end{align}
To check~\eqref{eq:st inclusion}, if the right hand side of~\eqref{eq:st inclusion} occurs,  then  $x\in \cZ$ and $B(y,\lambda r)\subset \MM\setminus\cZ$ by~\eqref{eq:def CZ in general}. The latter inclusion implies that $d(y,\cZ)\ge \lambda r$, i.e., the left hand side of~\eqref{eq:st inclusion} occurs.   As the events in right hand side of~\eqref{eq:st inclusion}, namely, $\{\{T(x,\sfR)=s\}\cap \{\upsigma_s=1\}\cap \{\forall w\in B(y,\lambda \tau),\ T(w,\sfR)=t\}\cap\{\upsigma_t=0\}\}_{s,t=1}^\infty$ are pairwise disjoint (as membership in each of them implies $s=T(x,\sfR)$ and $t=T(y,\sfR)$), we deduce from~\eqref{eq:st inclusion} that
\begin{equation}\label{eq:inclusion conclusion2}
 \prob^\tau\big[d(y,\cZ) \ge \lambda \tau \quad  \mathrm{and}\quad      x\in \cZ\big]\ge
 \sum_{\substack{(s,t)\in \N\\s\neq t}} \frac14 \prob^\tau\big[T(x,\sfR)=s\quad\mathrm{and}\quad \forall w\in B(y,\lambda \tau),\ T(w,\sfR)=t\big],
\end{equation}
where we used the fact that the random variables $T(x,\sfR),T(y,\sfR)$ depend only on $\{\sfZ_k\}_{k=1}^\infty$ and $\sfR$, so they are independent of $\upsigma_s,\upsigma_t$ for each $s,t\in \N$, and  $\upsigma_s,\upsigma_t$ are independent if $s\neq t$. For every $t\in \N$, the events $\{T(x,\sfR)=t\}$ and $\{\forall w\in B(y,\lambda \tau),\ T(w,\sfR)=t\}\subset \{T(y,\sfR)=t\}$ are disjoint, since if $T(x,\sfR)=T(y,\sfR)=t$, then by the definition~\eqref{eq:def stoppiong time} of  $T(\cdot,\cdot)$ we would have $d(\sfZ_t,x),d(\sfZ_t,y)\le \sfR<\tau/2$, so $d(x,y)<\tau$ by the triangle inequality, in contradiction to our assumption. Thus, \eqref{eq:inclusion conclusion2}  can be rewritten as follows:
\begin{align}\label{eq:put diagonal back}
\begin{split}
 \prob^\tau\big[d(y,\cZ) \ge \lambda \tau \quad  \mathrm{and}\quad      x\in \cZ\big]&\ge
  \frac14\sum_{s=1}^\infty \sum_{t=1}^\infty \prob^\tau\big[T(x,\sfR)=s\quad\mathrm{and}\quad \forall w\in B(y,\lambda \tau),\ T(w,\sfR)=t\big]\\&=
 \frac14\sum_{t=1}^\infty \prob^\tau\big[\forall w\in B(y,\lambda \tau),\ T(w,\sfR)=t\big]\\&= \frac14\sum_{t=1}^\infty \frac{4}{\tau}\int_{\frac14\tau}^{\frac12\tau} \prob^\mu\big[\forall w\in B(y,\lambda \tau),\ T(w,r)=t\big]\ud r,
 \end{split}
\end{align}
where the last step of~\eqref{eq:put diagonal back} uses the definition of  $\sfR$ and its independence from $\{\sfZ_k\}_{k=1}^\infty$.

We can now conclude the justification~\eqref{eq:in main thm-general} as follows:
\begin{align}
\nonumber\prob^\tau\big[d(y,\cZ) &\ge \lambda \tau \quad  \mathrm{and}\quad      x\in \cZ\big]\\&\ge \frac{1}{\tau}\int_{\frac14\tau}^{\frac12\tau} \bigg(\sum_{t=1}^\infty \big(1-\mu(B(x,r+\lambda\tau))\big)^{t-1}\mu(B(x,r-\lambda\tau))\bigg) \ud r\label{eq:combine the steps}\\ \nonumber
 &= \frac{1}{\tau}\int_{\frac14\tau}^{\frac12\tau} e^{\log \mu(B(x,r-\lambda\tau))-\log \mu(B(x,r+\lambda\tau)) }\ud r\\ \label{eq:use Jensen}
 &\ge \frac14 \exp\left(\frac{4}{\tau}\bigg(\int_{\frac14\tau}^{\frac12\tau}\log \mu(B(x,r-\lambda\tau))\ud r-\int_{\frac14\tau}^{\frac12\tau}\log \mu(B(x,r+\lambda\tau))\ud r\bigg)\right)\\
 &= \frac14 \exp\left(\frac{4}{\tau}\bigg(\int_{\left(\frac14-\lambda\right)\tau}^{\left(\frac14+\lambda\right)\tau}\log \mu(B(x,\rho))\ud \rho-\int_{\left(\frac12-\lambda\right)\tau}^{\left(\frac12+\lambda\right)\tau}\log \mu(B(x,\rho))\ud \rho\bigg)\right)\nonumber\\
 &\ge \frac14 \exp\left(8\lambda \Big(\log \mu\big(B(x,\frac18\tau)\big)-\log \mu\big(B(x,\frac58\tau)\big)\Big)\right)\nonumber \\
 &= \frac14\left(\frac{\mu(B(y,\frac58\tau))}{\mu(B(y,\frac18\tau))}\right)^{-8\lambda},\nonumber
\end{align}
where~\eqref{eq:combine the steps} is a combination of~\eqref{eq:tail of stopping time} and~\eqref{eq:put diagonal back}, in~\eqref{eq:use Jensen} we used Jensen's inequality, and the rest of the steps above hold because $\lambda\le 1/8$. This completes the proof of Theorem~\ref{thm:general doubling exponential}.
\end{proof}

\section{Impossibility results for average distortion}\label{sec:Rn into line}


\noindent Fact~\ref{prop: Rn} is the special case $q=2$ of the following proposition:

\begin{prop}\label{prop:Rn more genral} Fix  $p\ge 1$ and $n\in \N$. If $1\le q\le 2$, then  the smallest $D\ge 1$ for which $\ell_q^n$ embeds into $\R$ with $p$-average distortion $D$ is bounded from above and below by positive  universal constant multiples of $\sqrt{\max\{1,n/p\}}$.  If $q\ge 2$, then the smallest $D\ge 1$ for which $\ell_q^n$ embeds into $\R$ with $p$-average distortion $D$ is at most a universal constant multiple of $(\max\{1,n/p\})^{1-1/q}$, and it is at least a positive  universal constant multiple of the following quantity:
\begin{equation}\label{eq:cases weak lower bounds}
\left\{ \begin{array}{ll} \frac{q}{p}&\ \mathrm{if}\ 1\le p\le \frac{q}{n^{\frac{1}{q}}}\ \mathrm{and}\ q\le \log n,\\
n^{\frac{1}{q}}&\ \mathrm{if}\ \max\Big\{1,\frac{q}{n^{\frac{1}{q}}}\Big\}\le p\le q\le \log n,\\
\frac{\log n}{p} &\ \mathrm{if}\ p\le \log n\le q,\\ 1&\ \mathrm{if}\ \log n\le p\le q\ \mathrm{or}\ p\ge \max\left\{q,\frac{n}{e^q}\right\},\\
\big(\frac{n}{p}\big)^{\frac{1}{q}}&\ \mathrm{if}\ q\le p\le \frac{n}{e^q}.\end{array}\right.
\end{equation}
\end{prop}

Since $\ell_q$ embeds quasisymmetrically into $\ell_2$ when $1\le q\le 2$ (specifically, by~\cite{BDK65} its $\frac{q}{2}$-snowflake is isometric to a subset of $\ell_2$), for this range of $q$ the upper bound on $D$ in Proposition~\ref{prop:Rn more genral} is a special case of Theorem~\ref{coro:doubling average embedding}. But, Theorem~\ref{coro:doubling average embedding}   has a much simpler proof in its special case $\MM=\ell_q^n$, which  we will include  below.  When $q> 2$ the upper and lower bounds on $D$ in Proposition~\ref{prop:Rn more genral} do not match, and they do not belong to the framework that we study herein because $\ell_q$ does not admit a quasisymmetric mebedding into a Hilbert space~\cite{Nao12}. Nevertheless, it is beneficial to derive below the best bounds that  we currently have for $q$ in this range, as a small step towards the following independently interesting open question that arises naturally from the above discussion:

\begin{question}\label{Q:q>2} Given $n\in \N$ and $p\ge 1$, what is the growth rate as $n\to \infty$  of the smallest $D\ge 1$ such that $\ell_\infty^n$ embeds with $p$-average distortion $D$ into $\R$? We currently do not know the answer to this question even in the (most interesting) case $p=2$. More generally, for every $q>2$, what  is the order of magnitude (up to universal constant factors) of the smallest $D\ge 1$ such that $\ell_q^n$ embeds with $p$-average distortion $D$ into $\R$?
\end{question}

By~\cite{Nao14}, when $q>2$ the smallest $D\ge 1$  such that $\ell_q$ embeds into $\ell_2$ with quadratic\footnote{By~\cite[Proposition~6]{Nao21}, this statement formally implies bounds on the $p$-average distortion of $\ell_q$ into a Hilbert space.} average distortion $D$ is bounded from above and from below by positive universal constant multiples of $q$ (the proof of this fact in~\cite{Nao14} is nonconstructive; in~\cite{KNT21} an algorithmic proof was found). Proposition~\ref{prop:Rn more genral} shows that the situation is markedly different if one wishes to study the average distortion of $\ell_q$ in $\R$, and Question~\ref{Q:q>2}  aims to understand this phenomenon, with possible algorithmic ramifications, depending on the answer. Regardless, it is likely that a substantially new idea will be needed to answer Question~\ref{Q:q>2}.

\begin{proof}[Proof of Proposition~\ref{prop:Rn more genral}]  We will first derive the asserted upper bounds on $D$. For $q> 1$ let $\sfH=\sfH^{(q)}$ be the symmetric real-valued random variable whose density at each $s\in \R$ is equal to
\begin{equation*}\label{eq:dual distribution}
\frac{1}{2\Gamma\left(2-\frac{1}{q}\right)}e^{-|s|^{\frac{q}{q-1}}}.
\end{equation*}
We extend this notation to $q=1$ by letting $\sfH=\sfH^{(1)}$ be a symmetric Bernoulli random variable, namely, $\Pr[\sfH=1]=\Pr[\sfH=-1]=1/2$. If we take $\sfH_1,\ldots,\sfH_n$ to be i.i.d.~copies of $\sfH$ and consider the random vector
$$\sfU\eqdef (\sfH_1,\ldots,\sfH_n)\in \R^n,
$$
then the random vector $\sfU/\|\sfU\|_{q/(q-1)}$ takes (by design) values in the unit sphere of the dual $\ell_{q/(q-1)}^n$ of $\ell_q^n$, and furthermore the random variable $\|\sfU\|_{q/(q-1)}$ is independent of $\sfU/\|\sfU\|_{q/(q-1)}$; these useful probabilistic facts are due to~\cite{SZ90,RR91} (generalizations that could pertain to future extensions of the ensuing reasoning were found in~\cite[Section~2]{BLMN06} and~\cite[Lemma~156]{Nao24}).\footnote{Note that when $q=1$ the random vector $\sfU$ is distributed uniformly over $\{-1,1\}^n$ and $\|\sfU\|_{\infty}=1$ is constant. If one does not mind obtaining worse bounds on the universal constants that the ensuing proof provides, then in the range $1\le q\le 2$ one could work instead with the random vector $\sfU^{(1)}/n^{(q-1)/q}$ when $\sfU^{(1)}$ is distributed uniformly over $\{-1,1\}^n$.}

By e.g.~equation~(6.48) in~\cite{Nao24}, in combination with Stirling's formula, we have
\begin{equation}\label{eq:moment of denominator}
\left(\E\Big[\|\sfU\|_{\frac{q}{q-1}}^p\Big]\right)^{\frac{1}{p}}=\Bigg(\frac{\Gamma
\big(\frac{(n+p)(q-1)}{q}\big)}
{\Gamma\big(\frac{n(q-1)}{q}\big)}\Bigg)^{\frac{1}{p}}\asymp \big(\max\{p,n\}\big)^{1-\frac{1}{q}}.
\end{equation}
Also, by~\cite[Proposition~7]{BLMN06} (whose key input is~\cite{GK95}; for $q=1$ it suffices to use~\cite{Hit93} here), for every $a_1,\ldots,a_n\in \R$, if we let $a_1^*\ge a_2^*\ge\ldots\ge a_n^*\ge 0$ be the decreasing rearrangement of $|a_1|,\ldots,|a_n|$, then
\begin{equation}\label{eq:moment of numerator}
\Bigg(\E\bigg[\Big|\sum_{i=1}^n a_i\sfH_i\Big|^p\bigg]\Bigg)^{\frac{1}{p}}\asymp p^{1-\frac{1}{q}}\left(\sum_{i=1}^{\lfloor p\rfloor} (a_i^*)^{q}\right)^{\frac{1}{q}}+\sqrt{p} \left(\sum_{i=\lfloor p\rfloor}^n (a_i^*)^2\right)^{\frac12}\gtrsim \min\left\{p^{1-\frac{1}{q}},\frac{\sqrt{p}}{n^{\max\left\{0,\frac{1}{q}-\frac12\right\}}}\right\}\|a\|_q,
\end{equation}
where the last step of~\eqref{eq:moment of numerator} follows from a straightforward application of H\"older's inequality. Hence,
\begin{equation}\label{eq:use independence ratio}
\Bigg(\E\bigg[\Big|\big\langle a, \frac{1}{\|\sfU\|_{\frac{q}{q-1}}}\sfU\big\rangle\Big|^p\bigg]\Bigg)^{\frac{1}{p}}=\Bigg(\frac{\E \big[|\langle a,\sfU\rangle|^p\big]}{\E\big[\|\sfU\|_{\frac{q}{q-1}}^p\big]}\Bigg)^{\frac{1}{p}}\gtrsim \bigg(\min\Big\{1,\frac{n}{p}\Big\}\bigg)^{1-\frac{1}{\max\{2,q\}}},
\end{equation}
where in~\eqref{eq:use independence ratio} we continue to denote the standard scalar product on $\R^n$ by $\langle\cdot,\cdot\rangle:\R^n\times \R^n\to \R$, the first step of~\eqref{eq:use independence ratio} holds because   $\|\sfU\|_{q/(q-1)}$ is independent of $\sfU/\|\sfU\|_{q/(q-1)}$, and the last step of~\eqref{eq:use independence ratio} consists of a substitution of~\eqref{eq:moment of denominator} and~\eqref{eq:moment of numerator} and simplification of the resulting expression.

Thanks to~\eqref{eq:use independence ratio}, there is a universal constant $C>0$ such that if we consider the random linear functional $\sfF:\ell_q^n\to \R$ that is given by setting for every $x\in \R^n$,
$$
\sfF(x)\eqdef \frac{C\big(\max\big\{1,\frac{n}{p}\big\}\big)^{1-\frac{1}{\max\{2,q\}}}}{\|\sfU\|_{\frac{q}{q-1}}}\langle x,\sfU\rangle,
$$
then the Lipschitz constant of $\sfF$ with respect to the $\ell_q^n$ metric is at most  $C(\max\{1,n/p\})^{1-1/\max\{2,q\}}$, and
\begin{equation}\label{eq:random functional}
\forall x,y\in \R^n,\qquad \Big(\E \big[\|\sfF(x)-\sfF(y)\|_q^p\big]\Big)^{\frac{1}{p}}\ge \|x-y\|_q.
\end{equation}
By Fubini, it follows from~\eqref{eq:random functional} that for every Borel probability measure $\mu$ on $\R^n$ we have
$$
\E\left[\iint_{\R^n\times \R^n} |\sfF(x)-\sfF(y)|^p\ud\mu(x)\ud\mu(y)\right]\ge \iint_{\R^n\times \R^n} \|x-y\|_q^p\ud\mu(x)\ud\mu(y).
$$
Hence, with positive probability the $C(\max\{1,n/p\})^{1-1/\max\{2,q\}}$-Lipschitz function $\sfF$ satisfies
$$
\left(\iint_{\R^n\times \R^n} |\sfF(x)-\sfF(y)|^p\ud\mu(x)\ud\mu(y)\right)^{\frac{1}{p}}\ge \left(\iint_{\R^n\times \R^n} \|x-y\|_q^p\ud\mu(x)\ud\mu(y)\right)^{\frac{1}{p}},
$$
thus proving the upper bound that is asserted in Proposition~\ref{prop:Rn more genral} (in the entire range $q\ge 1$) on the smallest $D\ge 1$ for which $\ell_q^n$ embeds with $p$-average distortion $D$ into $\R$.

It remains to prove the lower bounds on the average distortion that are asserted in Proposition~\ref{prop:Rn more genral}. Observe first that it suffices to treat only the case $q\ge 2$. Indeed, if $1\le q\le 2$ then by~\cite{FLM77}   there exists an integer $m\asymp n$ such that $\ell_2^n$ embeds with (bi-Lipschitz) distortion $O(1)$ into $\ell_q^n$. Hence, if $\ell_q^n$ embeds with $p$-average distortion $D$ into $\R$, then also $\ell_2^n$ embeds with $p$-average distortion $O(D)$ into $\R$, which implies that $D\gtrsim \sqrt{\max\{1,n/p\}}$  by the (yet to be proven) case $q=2$ of Proposition~\ref{prop:Rn more genral}. Thus, suppose that $q\ge 2$ and that $\ell_q^n$ embeds with $p$-average distortion $D\ge 1$ into $\R$. The goal is to show that necessarily
\begin{equation}\label{eq:D lower desired cases}
D\gtrsim \left\{ \begin{array}{ll} \frac{q}{p}&\ \mathrm{if}\ 1\le p\le \frac{q}{n^{\frac{1}{q}}}\ \mathrm{and}\ q\le \log n,\\
n^{\frac{1}{q}}&\ \mathrm{if}\ \max\Big\{1,\frac{q}{n^{\frac{1}{q}}}\Big\}\le p\le q\le \log n,\\
\frac{\log n}{p} &\ \mathrm{if}\ p\le \log n\le q,\\ 1&\ \mathrm{if}\ \log n\le p\le q\ \mathrm{or}\ p\ge \max\left\{q,\frac{n}{e^q}\right\},\\
\big(\frac{n}{p}\big)^{\frac{1}{q}}&\ \mathrm{if}\ q\le p\le \frac{n}{e^q},\end{array}\right.
\end{equation}

Let $k=k(n)\in \N$ be the largest integer such that $k(k-1)/2\le n$. Thus $k\asymp \sqrt{n}$, so $\log k\asymp \log n$. By combining part~(i) of Theorem~1 in~\cite{Mat97} with Proposition~1 in~\cite{Bal90}, every  $k$-point metric space embeds into $\ell_q^n$ with (bi-Lipschitz) distortion $O(\max\{1,(\log k)/q\})\asymp \max\{1,(\log n)/q\}$. At the same time, by the {\em proof of} part~(ii) of Theorem~1 in~\cite{Mat97} there exists a $k$-point metric space $\MM$ (namely, any bounded degree $k$-vertex $\Omega(1)$-expander graph works here)  such that every embedding of $\MM$ into $\R$ incurs $p$-average distortion that is at least a positive universal constant multiple of $\max\{1,(\log k)/p\} \asymp \max\{1,(\log n)/p\}$. Our assumption on $D$ therefore implies that
\begin{equation}\label{eq:first cases D partial expander}
D\gtrsim \max\Bigg\{1,\frac{\max\big\{1,\frac{\log n}{p}\big\}}{\max\big\{1,\frac{\log n}{q}\big\}}\Bigg\}\asymp \left\{ \begin{array}{ll} \frac{q}{p} &\ \mathrm{if}\ 1\le p,q\le \log n,\\ 1 &\ \mathrm{if}\ p,q\ge \log n\ \mathrm{or}\ q\le \log n\le p, \\ \frac{\log n}{p}&\ \mathrm{if}\ 1\le p\le \log n\le q.
\end{array}\right.
\end{equation}

Next, denote the unit sphere of $\ell_q^n$ by $B_q^n=\{x\in \R^n:\ \|x\|_q\le 1\}$. Let $\kappa_q^n$ be the cone measure~\cite{GM87} on the sphere $\partial B_q^n$, i.e., $\kappa_q^n(E)=\vol_n([0,1]E)$ for every Lebesgue measurable $E\subset \partial B_q^n$. By our assumption on $D$, there exists $f:\R^n\to \R$ that is $D$-Lipschitz with respect to the $\ell_q^n$-metric such that
\begin{align}\label{average distortion for cone measure}
\begin{split}
\Bigg(\iint_{\partial B_q^n\times \partial B_q^n} |f(x)-f(y)|^p&\ud\kappa_q^n(x)\ud\kappa_q^n(y)\Bigg)^{\frac{1}{p}}\ge  \left(\iint_{\partial B_q^n\times \partial B_q^n} \|x-y\|_q^p\ud\kappa_q^n(x)\ud\kappa_q^n(y)\right)^{\frac{1}{p}}\\&\ge \left(\int_{\partial B_q^n} \bigg\|x-\int_{\partial B_q^n} y\ud \kappa_q^n(y)\bigg\|_q^p\ud \kappa_q^n(x)\right)^{\frac{1}{p}}=\left(\int_{\partial B_q^n} \|x\|_q^p\ud \kappa_q^n(x)\right)^{\frac{1}{p}}=1
 \end{split}
\end{align}
where the second step of~\eqref{average distortion for cone measure} uses Jensen's inequality. As $f$ is $D$-Lipschitz with respect to the $\ell_q^n$-metric and $q\ge 2$, by the Gromov--Milman concentration inequality~\cite{GM87} (see~\cite{ABV98} for a different simple proof; the case $p=2$ follows from L\'evy's classical isoperimetric theorem, as in e.g.~\cite[Chapter~2]{Led01}) combined with the estimate on the modulus of uniform convexity of $\ell_q$ in~\cite{Cla36}, we have
$$
\forall s\ge 0,\qquad \kappa_q^n \Big[\big\{x\in \partial B_q^n:\ |f(x)-M_f|\ge s\big\}\Big] \lesssim e^{-n\left(\frac{s}{CD}\right)^q},
$$
where $M_f$ is the median of $f$ and $C>0$ is a universal constant. Consequently,
\begin{equation}\label{eq:f concentrated}
\forall s\ge 0,\qquad \big(\kappa_q^n\times \kappa_q^n\big) \Big[\big\{(x,y)\in \partial B_q^n\times \partial B_q^n:\ |f(x)-f(y)|\ge s\big\}\Big] \lesssim e^{-n\left(\frac{s}{CD}\right)^q}.
\end{equation}
Hence,
\begin{align}\label{eq:use fubini p}
\begin{split}
 1&\le \left(\iint_{\partial B_q^n\times \partial B_q^n} |f(x)-f(y)|^p\ud\kappa_q^n(x)\ud\kappa_q^n(y)\right)^{\frac{1}{p}}\\&= \left(\int_0^\infty ps^{p-1} \big(\kappa_q^n\times \kappa_q^n\big)  \Big[\big\{(x,y)\in \partial B_q^n\times \partial B_q^n:\ |f(x)-f(y)|\ge s\big\}\Big]\ud s\right)^{\frac{1}{p}}
\\& \lesssim  \left(\int_0^\infty ps^{p-1} e^{-n\left(\frac{s}{CD}\right)^q}\ud s\right)^{\frac{1}{p}}= C\frac{\Gamma\left(1+\frac{p}{q}\right)^{\frac{1}{p}}}{n^{\frac{1}{q}}}D\asymp\left(\frac{p}{n}\right)^{\frac{1}{q}}D,
\end{split}
\end{align}
where the first step of~\eqref{eq:use fubini p} uses~\eqref{eq:f concentrated}, the third step of~\eqref{eq:use fubini p} uses~\eqref{average distortion for cone measure}, and the final step of~\eqref{eq:use fubini p} is an application of Stirling's formula.   We therefore conclude that $D$ must obey the following lower estimate:
\begin{equation}\label{eq:second D lower}
D\gtrsim \left(\frac{n}{p}\right)^{\frac{1}{q}}.
\end{equation}
The desired estimate~\eqref{eq:D lower desired cases} follows by combining~\eqref{average distortion for cone measure} with~\eqref{eq:second D lower}, and a straightforward case analysis.
\end{proof}


Next, the case $r=2$ of the following theorem coincides with Theorem~\ref{eq:type 2 p dependence}:

\begin{theorem}\label{eq:avwerage distortion into type} Let $(\bX,\|\cdot\|)$ be a Banach space of type $1\le r\le 2$. Given $n\in \N$ and $p\ge 1$, if $\ell_1^n$  embeds into $\bX$ with $p$-average distortion $D\ge 1$, then necessarily
\begin{equation}\label{eq:lower D l1 into type}
D\gtrsim \left\{\begin{array}{ll}\frac{n^{1-\frac{1}{r}}}{T_r(\bX)}&\mathrm{if}\ 1\le p\le T_r(\bX)^2n^{\frac{2}{r}-1},\\
\sqrt{\frac{n}{p}}&\mathrm{if}\ T_r(\bX)^2n^{\frac{2}{r}-1}\le p\le n.\end{array}\right.
\end{equation}
\end{theorem}

Theorem~\ref{eq:K functional for nonlinear type} below is the main  ingredient in the proof of Theorem~\ref{eq:avwerage distortion into type}. Its formulation uses the convention (to which we will adhere later as well) that all the expectations of Banach space-valued functions whose domain is a finite set are with respect to the uniform probability measure on that set.   Thus, if $S$ is a finite set and $(\bX,\|\cdot\|_\bX)$ is a Banach space, then for every $f:S\to \bX$ we will use the notations
$$
\E[f]=\frac{1}{|S|}\sum_{s\in S} f(s)\in \bX\qquad\mathrm{and}\qquad \forall p\ge 1,\quad \|f\|_{L_p(S;\bX)}=\Big( \E\big[\|f\|^p_\bX\big]\Big)^{\frac{1}{p}}=\bigg(\frac{1}{|S|}\sum_{s\in S} \|f(s)\|^p_\bX\bigg)^{\frac{1}{p}}\ge 0.
$$
Given $n\in \N$ and a Banach space $(\bX,\|\cdot\|_\bX)$, for $f:\{-1,1\}^n \to \bX$ and $i\in [n]$ define $\partial_if:\{-1,1\}^n\to \bX$ by
\begin{equation}\label{eq:def partial}
\forall \e\in \{-1,1\}^n,\qquad \partial_if(\e)\eqdef f(\e_1,\e_2,\ldots,\e_{i-1},-\e_i,\e_{i+1},\ldots,\e_n)-f(\e).
\end{equation}

\begin{theorem}\label{eq:K functional for nonlinear type} Fix $p\ge 2\ge r\ge 1$ and $n\in \N$.  Suppose that $(\bX,\|\cdot\|_\bX)$ is a Banach space of type $r$. Then, every $f:\{-1,1\}^n\to \bX$ with $\E[f]=0$ satisfies the following estimate:
\begin{equation}\label{eq:mixture conjectured optimal'}
\|f\|_{L_p(\{-1,1\}^n;\bX)}\lesssim T_r(\bX)\bigg(\sum_{i=1}^n \|\partial_if\|_{L_r(\{-1,1\}^n;\bX)}^r\bigg)^{\frac{1}{r}}+\sqrt{p}\bigg(\sum_{i=1}^n \|\partial_if\|_{L_p(\{-1,1\}^n;\bX)}^2\bigg)^{\frac{1}{2}}.
\end{equation}
\end{theorem}

The proof of Theorem~\ref{eq:K functional for nonlinear type} was found by Alexandros Eskenazis, answering positively a question that we  posed in an earlier draft of the present work. We are grateful to him for allowing us to include it here.

Prior to proving Theorem~\ref{eq:K functional for nonlinear type}, we will next show how it implies Theorem~\ref{eq:avwerage distortion into type}:

\begin{proof}[Proof of Theorem~\ref{eq:avwerage distortion into type} assuming Theorem~\ref{eq:K functional for nonlinear type}] It suffices to prove Theorem~\ref{eq:avwerage distortion into type} when  $p\ge 2$, because its conclusion when $1\le p\le 2$ follows formally from its special case $p=2$ by~\cite[equation~(105)]{Nao21}.

So, assume from now that $p\ge 2$ and suppose that $f:\{-1,1\}^n\to \bX$ satisfies\footnote{Note that the last equivalence in the assumption~\eqref{eq:cube averages} has  a quick justification: For one direction use  the triangle inequality in $L_p(\{-1,1\}^n;\ell_1^n)$, and the other direction use Jensen's inequality for the averaging over $\d$. }
\begin{equation}\label{eq:cube averages}
\Big(\E\big[\|f(\e)-f(\d)\|_\bX^p\big]\Big)^{\frac{1}{p}}\ge \Big(\E\big[\|\e-\d\|_1^p\big]\Big)^{\frac{1}{p}}\asymp n,
\end{equation}
and that also for some $D\ge 1$ we have
\begin{equation}\label{eq:lip condition on hamming}
\forall \e,\d\in \{-1,1\}^n,\qquad \|f(\e)-f(\d)\|_\bX\le D\|\e-\d\|_1=2D|\{i\in [n]:\ \e_i\neq \d_i\}|,
\end{equation}
i.e., $f$ is $D$-Lipschitz with respect to the metric on $\{-1,1\}^n$ that is inherited from $\ell_1^n$ and the metric on $\bX$ that is induced by $\|\cdot\|_\bX$.  Theorem~\ref{eq:avwerage distortion into type} will be proven once we will show that $D$ must satisfy~\eqref{eq:lower D l1 into type}, which is a simple consequence of~\eqref{eq:mixture conjectured optimal'}. Indeed, by translating $f$ we may assume that $\E[f]=0$. By Theorem~\ref{eq:K functional for nonlinear type}, 
\begin{multline*}
n\stackrel{\eqref{eq:cube averages}}{\lesssim} \Big(\E\big[\|f(\e)-f(\d)\|_\bX^p\big]\Big)^{\frac{1}{p}}\le 2\|f\|_{L_p(\{-1,1\}^n;\bX)}
\\\stackrel{\eqref{eq:mixture conjectured optimal'}}{\lesssim}T_r(\bX)\bigg(\sum_{i=1}^n \|\partial_if\|_{L_p(\{-1,1\}^n;\bX)}^r\bigg)^{\frac{1}{r}}+\sqrt{p}\bigg(\sum_{i=1}^n \|\partial_if\|_{L_p(\{-1,1\}^n;\bX)}^2\bigg)^{\frac{1}{2}}
\stackrel{\eqref{eq:def partial}\wedge \eqref{eq:lip condition on hamming}}{\le} 2D\left(T_r(\bX)n^{\frac{1}{r}}+\sqrt{pn}\right).
\end{multline*}
Therefore,
\begin{equation*}
D\gtrsim \frac{n}{T_r(\bX)n^{\frac{1}{r}}+\sqrt{pn}}\asymp \left\{\begin{array}{ll}\frac{n^{1-\frac{1}{r}}}{T_r(\bX)}&\mathrm{if}\ 1\le p\le T_r(\bX)^2n^{\frac{2}{r}-1},\\
\sqrt{\frac{n}{p}}&\mathrm{if}\ T_r(\bX)^2n^{\frac{2}{r}-1}\le p\le n.\end{array}\right.\tag*{\qedhere}
\end{equation*}
\end{proof}

The proof of Theorem~\ref{eq:K functional for nonlinear type} is a quick concatenation of the (substantial) results of~\cite{BL08} and~\cite{IvHV20}:

\begin{proof}[Proof of Theorem~\ref{eq:K functional for nonlinear type}] The main result of~\cite{IvHV20} is well-known to yield the following estimate:
\begin{equation}\label{eq:quote ivHV}
 \|f\|_{L_r(\{-1,1\}^n;\bX)}\lesssim T_r(\bX)\bigg(\sum_{i=1}^n \|\partial_if\|_{L_r(\{-1,1\}^n;\bX)}^r\bigg)^{\frac{1}{r}}.
\end{equation}
(Specifically, deduce~\eqref{eq:quote ivHV} from~\cite{IvHV20} by applying~\cite[Theorem~1.2]{IvHV20} to the  function $\Phi(x)=\|x\|_\bX^r$, and then combining the resulting estimate with the second and third displayed equations on page~674 of~\cite{IvHV20}.)
Furthermore, by the (real-valued) Poincar\'e inequality of~\cite[Theorem~1.1(1)]{BL08}, we have
\begin{equation}\label{eq:quote ben efraim lust piquard}
\big\|\|f\|_\bX-\|f\|_{L_1(\{-1,1\}^n;\bX)}\big\|_{L_p(\{-1,1\}^n)}\lesssim \sqrt{p} \bigg\|\Big(\sum_{i=1}^n \big(\partial_i\|f\|_\bX\big)^2\Big)^{\frac12}\bigg\|_{L_p(\{-1,1\}^n)}.
\end{equation}
Since $\|f\|_{L_1(\{-1,1\}^n;\bX)}\le \|f\|_{L_r(\{-1,1\}^n;\bX)}$ and by the triangle inequality for $\|\cdot\|_\bX$ we have $\partial_i\|f\|_\bX(\e)\le \|\partial_if(\e)\|_\bX$  for every $\e\in \{-1,1\}^n$ and every $i\in [n]$,  it follows from~\eqref{eq:quote ivHV} and~\eqref{eq:quote ben efraim lust piquard} that
\begin{align}\label{eq:square function version}
\begin{split}
\|f\|_{L_p(\{-1,1\}^n;\bX)}&\le \big\|\|f\|_\bX-\|f\|_{L_1(\{-1,1\}^n;\bX)}\big\|_{L_p(\{-1,1\}^n)}+\|f\|_{L_1(\{-1,1\}^n;\bX)} \\&\lesssim T_r(\bX)\bigg(\sum_{i=1}^n \|\partial_if\|_{L_r(\{-1,1\}^n;\bX)}^r\bigg)^{\frac{1}{r}}+ \sqrt{p} \bigg\|\Big(\sum_{i=1}^n \big(\|\partial_i f\|_\bX\big)^2\Big)^{\frac12}\bigg\|_{L_p(\{-1,1\}^n)}, 
\end{split}
\end{align}
where the first step of~\eqref{eq:square function version} is an application of the triangle inequality in $L_p(\{-1,1\}^n)$. It remains to note that final term in~\eqref{eq:square function version} is at most the second term in~\eqref{eq:mixture conjectured optimal'} by the triangle inequality in $L_{\frac{p}{2}}(\{-1,1\}^n)$. 
\end{proof}

\bibliographystyle{alphaabbrvprelim}
\bibliography{local-growth}

\end{document}